\documentclass[12pt,oneside,openany,article]{memoir}
\usepackage{mempatch}
\pdfoutput=1

\nouppercaseheads
\usepackage[amsthm,thmmarks,hyperref]{ntheorem}

\usepackage{etex}
\usepackage{todonotes}
\usepackage[english]{babel}
\usepackage[leqno]{mathtools}

\usepackage{mathrsfs}
\usepackage{upgreek}
\usepackage[compress,square,comma,numbers]{natbib}
\usepackage{hypernat} 

\usepackage{sfmath}

\usepackage{microtype}

\usepackage{tikz}
\usetikzlibrary{patterns}

\usepackage{subfig}
\usepackage{wrapfig}
\usepackage{array}

\usepackage{graphicx,color}
\definecolor{gray}{gray}{0}
\pagecolor{white}

\usepackage{hyperxmp}
\usepackage{xr-hyper}
\usepackage{nameref}
\usepackage[pdftex,bookmarks,pdfnewwindow,plainpages=false,unicode,pdfencoding=auto]{hyperref}

\usepackage{bookmark}

\usepackage{enumitem}

\usepackage{amssymb} 

\hypersetup{
colorlinks=true,
linkcolor=black,
citecolor=black,
urlcolor=blue,
pdfauthor={Victor Ivrii},
pdftitle={Spectral Asymptotics for Magnetic Schroedinger Operator With the Strong Magnetic Field. 2D case},
pdfsubject={Sharp Spectral Asymptotics},
pdfkeywords={Microlocal Analysis,  Sharp Spectral Asymptotics, Periodic flows},
bookmarksdepth={4}
}

\numberwithin{equation}{chapter}

\theoremstyle{plain}
\newtheorem{theorem}{Theorem}[chapter]
\newtheorem{lemma}[theorem]{Lemma}
\newtheorem{proposition}[theorem]{Proposition}

\theoremstyle{definition}

\newtheorem{Problem}[theorem]{Problem}
\theoremstyle{remark}
\newtheorem{remark}[theorem]{Remark}
\newtheorem{example}[theorem]{Example}
\newtheorem{problem}[theorem]{Problem}

\makeatletter
\newtheoremstyle{plainfoot}%
  {\item[\hskip\labelsep \theorem@headerfont ##1\ ##2\,\footnotemark\theorem@separator]}%
  {\item[\hskip\labelsep \theorem@headerfont ##1\ ##2\ (##3)\, \footnotemark\theorem@separator]}
\makeatother

\theoremstyle{plainfoot}
\theoremseparator{.}
\newtheorem{theorem-foot}[theorem]{Theorem}
\newtheorem{lemma-foot}[theorem]{Lemma}
\newtheorem{proposition-foot}[theorem]{Proposition}
\newtheorem{corollary-foot}[theorem]{Corollary}
\newtheorem{conjecture-foot}[theorem]{Conjecture}
\newtheorem{condition-foot}[theorem]{Condition}

\theoremstyle{plainfoot}
\theoremseparator{.}
\theorembodyfont{}
\newtheorem{definition-foot}[theorem]{Definition}
\newtheorem{Problem-foot}[theorem]{Problem}

\theoremstyle{plainfoot}
\theoremseparator{.}
\theorembodyfont{}
\theoremheaderfont{\slshape}
\newtheorem{remark-foot}[theorem]{Remark}         
\newtheorem{example-foot}[theorem]{Example}
\newtheorem{problem-foot}[theorem]{Problem}

\DeclareMathAlphabet{\mathpzc}{OT1}{pzc}{m}{it}

\newcommand{\cA}{\mathcal{A}}

 \newcommand{\cL}{\mathcal{L}}
 
 \newcommand{\cN}{\mathcal{N}}

 \newcommand{\cT}{\mathcal{T}}

 \newcommand{\cX}{\mathcal{X}}
 
 \newcommand{\cZ}{\mathcal{Z}}

 \newcommand{\sC}{\mathscr{C}}

 \newcommand{\sL}{\mathscr{L}}

 \newcommand{\sS}{\mathscr{S}}

 \newcommand{\scl}{\mathsf{scl}}

\newcommand{\D}{{\mathsf{D}}}

\newcommand{\n}{{\mathsf{n}}}
\newcommand{\g}{{\mathsf{g}}}
\newcommand{\MW}{{\mathsf{MW}}}
\newcommand{\N}{{\mathsf{N}}}
\newcommand{\F}{{\mathsf{F}}}

\newcommand{\J}{{\mathsf{J}}}

\newcommand{\A}{{\mathsf{A}}}

\newcommand{\W}{{{\mathsf{W}}}}

\newcommand{\y}{{\mathsf{y}}}

\newcommand{\const}{{\mathsf{const}}}
\newcommand{\dist}{{{\mathsf{dist}}}}

\newcommand{\eff}{{{\mathsf{eff}}}}

\newcommand{\ess}{{{\mathsf{ess}}}}

\newcommand{\new}{{{\mathsf{new}}}}

\newcommand{\bC}{{\mathbb{C}}}
\newcommand{\bH}{{\mathbb{H}}}
\newcommand{\bK}{{\mathbb{K}}}

\newcommand{\bR}{{\mathbb{R}}}
\newcommand{\bS}{{\mathbb{S}}}

\newcommand{\bZ}{{\mathbb{Z}}}

\newcommand{\fW}{{\mathfrak{W}}}

\newcommand{\fD}{{\mathfrak{D}}}

\newcommand{\fZ}{{\mathfrak{Z}}}
\newcommand{\fz}{{\mathfrak{z}}}

\def\1{\boldsymbol {|}}
\newcommand{\blangle}{{\boldsymbol{\langle}}}
\newcommand{\brangle}{{\boldsymbol{\rangle}}}

\newcommand{\3}{{|\!|\!|}}

\newcommand{\Hess}{\operatorname{Hess}}

\newcommand{\Ker}{\operatorname{Ker}}
\newcommand{\mes}{\operatorname{mes}}

\newcommand{\rank}{\operatorname{rank}}
\newcommand{\Spec}{\operatorname{Spec}}

\newcommand{\Specess}{\operatorname{Spec_\ess}}
\newcommand{\supp}{\operatorname{supp}}

\newcommand{\Tr}{\operatorname{Tr}}

\newenvironment{claim}[1][{\textup{(\theequation)}}]{\refstepcounter{equation}\vglue10pt
\begin{trivlist}
\item[{\hskip\labelsep#1}]}{\vglue10pt\end{trivlist}}

\newenvironment{claim*}[1][{}]{\vglue10pt
\begin{trivlist}
\item[{\hskip\labelsep#1}]}{\vglue10pt\end{trivlist}}

\newenvironment{phantomequation}[1][]{\refstepcounter{equation}}{}
\newcounter{note}

\DeclareTextCommand{\textbeta}{PU}{\83\262}
\DeclareTextCommand{\textmu}{PU}{\80\265}
\DeclareTextCommand{\texttau}{PU}{\83\304}

\DeclareTextCommand{\textlesssim}{PU}{\9042\162}
\DeclareTextCommand{\textgtrsim}{PU}{\9042\163}

\DeclareTextCommand{\textpartial}{PU}{\9042\002}
\DeclareTextCommand{\texttwosuperior}{PU}{\80\262}
\DeclareTextCommand{\textGamma}{PU}{\83\223}
\DeclareTextCommand{\textxinferior}{PU}{\9040\223}
\DeclareTextCommand{\textiinferior}{PU}{\9035\142}
\DeclareTextCommand{\textjinferior}{PU}{\9054\174}

\DeclareTextCommand{\textge}{PU}{\9042\145}
\DeclareTextCommand{\textle}{PU}{\9042\144}
\DeclareTextCommand{\texthat}{PD1}{\136}

\begin{document}
\title{Spectral Asymptotics for Magnetic Schr\"odinger Operator with the Strong Magnetic Field}
\author{Victor Ivrii}

\maketitle
{\abstract%
In this article we obtain eigenvalue asymptotics for $2\D$ and $3\D$-Schr\"{o}dinger, Schr\"{o}\-dinger-Pauli  and Dirac operators in the situations in which the role of the magnetic field is important.  We have seen in Chapters~\ref{book_new-sect-13}  and~\ref{book_new-sect-17} of \cite{futurebook} that these operators are essentially different and there is a significant difference between $2\D$ and $3\D$-operators.
\endabstract}

\setcounter{tocdepth}{0}

\tableofcontents*
\setlength{\marginparwidth}{3.5cm}

\chapter{$2\D$-case. Introduction}
\label{sect-23-1}

In this article we obtain eigenvalue asymptotics for $2\D$-Schr\"{o}dinger, Schr\"{o}\-dinger-Pauli  and Dirac operators in the situations in which the role of the magnetic field is important.  We have seen in Chapters~\ref{book_new-sect-13} and~\ref{book_new-sect-17} of \cite{futurebook} that these operators are essentially different and they also differ significantly from the corresponding $3\D$-operators.

While we are trying to emulate results of Chapter~\ref{book_new-sect-11} of \cite{futurebook}, we find ourselves now in the very different situation. Indeed, for operators we study the remainder estimates in the local spectral asymptotics under non-degeneracy assumptions are better than for similar operators with the magnetic field. However, as we seen in Chapters~\ref{book_new-sect-14} and~15 of \cite{futurebook} these remainder estimates deteriorate if the magnetic field degenerates or there is a boundary. This significantly limits our ability to consider the cases when magnetic field asymptotically degenerates at the point of singularity (finite or infinite) along some directions, or domains with the boundary.

We start from Section~\ref{sect-23-2} in which we consider the case when the spectral parameter is fixed ($\tau=\const$) and study asymptotics with respect to $\mu,h$ exactly like in Section~\ref{book_new-sect-11-1} of \cite{futurebook} we considered asymptotics with respect to $h$. However, since now we have two parameters, we need to consider an interplay between them: while always $h\to +0$, we cover
$\mu\to +0$, $\mu $ remains disjoint from $0$ and $\infty$ and $\mu\to \infty$, which in turn splits into subcases $\mu h\to 0$, $\mu h$  remains disjoint from $0$ and $\infty$ and $\mu h\to \infty$.

In Section~\ref{sect-23-3} we consider asymptotics with $\mu=h=1$ and with the spectral parameter $\tau$ tending to $+\infty$ for the Schr\"odinger and Schr\"odinger-Pauli operators and to $\pm\infty$ for the Dirac operator. We consider bounded domains with the singularity at some point and external domains with the singularity at infinity. In the latter case the specifics of the $2\D$ magnetic Schr\"odinger and Schr\"odinger-Pauli operators manifest itself in the better remainder estimate (and in the larger principal part for the Schr\"odinger-Pauli operator) than in the non-magnetic case. Furthermore, in contrast to the non-magnetic case there are  non-trivial results for the Dirac operator. This happens only in even dimensions.

In Section~\ref{sect-23-4} we consider asymptotics with the singularity at infinity and $\mu=h=1$ and with $\tau$ tending to $+0$ for the Schr\"odinger and Schr\"odinger-Pauli operators and to $\pm (M-0)$ for the Dirac operator.
Again the specifics of the $2\D$ magnetic Schr\"odinger and Schr\"odinger-Pauli operators manifest itself in the better remainder estimate (and in the larger principal part for the Schr\"odinger-Pauli operator) than in non-magnetic case.

It  includes the most interesting case (see Subsection~\ref{sect-23-4-1}) when magnetic field is either constant or stabilizes fast at infinity and potential decays at infinity. As we know if magnetic field and potential were constant then the operator would have purely point spectrum of infinite multiplicity and each eigenvalue (Landau level) would be disjoint from the rest. Now we have a sequences of eigenvalues tending to the Landau level either from below, or from above, or from both sides and we are interested in their asymptotics.
In contrast to the rest of the section we consider multidimensional case as well. In contrast to Section~\ref{book_new-sect-11-6} of \cite{futurebook} there are non-trivial results for fast decaying potentials as well.

In Section~\ref{sect-23-5} we consider asymptotics with respect to
$\mu, h, \tau$, like in Section~\ref{book_new-sect-11-7} of \cite{futurebook} again with significant differences mentioned above.

Finally, in Appendix~\ref{sect-23-A-3} the self-adjointness of the $2\D$-Dirac operator with a very singular magnetic field is proven.

The Schr\"{o}dinger operator theory is more extensive than the Dirac operator theory: there are many more meaningful problems and questions for the Schr\"{o}dinger operator than for the Dirac operator.  Also, the
$2\D$-theory is more extensive and provides more accurate remainder estimates than the $3\D$-theory.  These circumstances are not due to the technical difficulties but are instead due to the fact that the spectrum of the Schr\"{o}dinger operator is discrete more often than the spectrum of the Dirac operator in dimension $d=2$ the spectrum is
discrete more often than in dimension $d=3$.%

\chapter{$2\D$-case. Asymptotics with fixed spectral parameter}
\label{sect-23-2}

In this section we consider asymptotics with a fixed spectral parameter for $2$-dimensional magnetic Schr\"odinger, Schr\"odinger-Pauli and Dirac operators and discuss some of the generalizations\footnote{\label{foot-23-1} Mainly to higher dimensions with full-rank magnetic field.}.

As in Chapters \ref{book_new-sect-9} of \cite{futurebook} we will introduce a semiclassical zone and a singular zone, where $\rho\gamma \ge h$ and $\rho\gamma\le h$ respectively. In the semiclassical zone we apply asymptotics of Chapters~\ref{book_new-sect-13}--\ref{book_new-sect-22}  (but mainly of \ref{book_new-sect-13} and \ref{book_new-sect-19} of \cite{futurebook}--in the multidimensional case). In the singular zone we need to apply estimates for a number of eigenvalues; usually it would be sufficient to use non-magnetic estimate\footnote{\label{foot-23-2} With $V$ modified accordingly; for example, for the Schr\"odinger and Schr\"odinger-Pauli operators $V_-$ is replaced by
$C\bigl((1-\epsilon) V - C_\epsilon \mu^2|\vec{V}|^2\bigl)_-$.}
for number of eigenvalues which trivially follows from standard one but if needed one can use more delicate estimates.

\section{Schr\"odinger operator}
\label{sect-23-2-1}

\subsection{Estimates of the spectrum}
\label{sect-23-2-1-1}

Consider first the Schr\"{o}dinger operator (\ref{book_new-13-1-1}) of \cite{futurebook})
\begin{equation}
A=\sum_{j,k} P_jg^{jk}P_k+V,\qquad \text{with\ \ } P_j=hD_j-\mu V_j
\label{23-2-1}
\end{equation}
where $g^{jk},V_j,V$ satisfy (\ref{book_new-13-1-2}) and (\ref{book_new-13-1-4}) of \cite{futurebook}) i.e.
\begin{equation}
\epsilon |\xi |^2\le \sum_{j,k} g^{jk}\xi _j\xi _k\le c|\xi |^2 \qquad
\forall \xi \in\bR^d.
\label{23-2-2}
\end{equation}
Without any loss of the generality we can fix $\tau =0$ and then in the important function $V_\eff F_\eff ^{-1}$ the parameters $\mu $ and $h$ enter as factors. Thus, we treat the operator (\ref{book_new-13-1-1}) of \cite{futurebook}) assuming that it is self-adjoint.

We make assumptions similar to those of Chapter~\ref{book_new-sect-9} of \cite{futurebook}
\begin{phantomequation}\label{23-2-3}\end{phantomequation}
\begin{multline}
|D^\alpha g^{jk}|\le c\gamma ^{-|\alpha |},\quad
|D^\alpha F_{jk}|\le c\rho_1 \gamma ^{-|\alpha |},\quad
|D^\alpha V|\le c\rho ^2\gamma ^{-|\alpha |}
\tag*{$\textup{(\ref*{23-2-3})}_{1-3}$}\label{23-2-3-*}
\end{multline}
where scaling function $\gamma(x)$ and weight functions $\rho(x),\rho_1(x)$ satisfy the standard assumptions $\textup{(\ref{book_new-9-1-6})}_{1,2}$ of \cite{futurebook}. Then
\begin{phantomequation}\label{23-2-4}\end{phantomequation}
\begin{equation}
\mu_\eff= \mu\rho_1\gamma\rho^{-1}, \qquad  h_\eff = h\rho^{-1}\gamma^{-1}.
\tag*{$\textup{(\ref*{23-2-4})}_{1,2}$}\label{23-2-4-*}
\end{equation}

Let us introduce a \emph{semiclassical zone\/}
\begin{align}
&X'=\{ x\colon  \rho \gamma \ge h\}
\label{23-2-5}\\
\shortintertext{and a \emph{singular zone\/}}
&X''=\{ x\colon  \rho \gamma \le 2h\}.
\label{23-2-6}
\end{align}

Further, let us introduce two other overlapping zones
\begin{align}
&X'_1=\{x\in X_\scl\colon \mu \rho _1 \le 2c\rho \gamma ^{-1}\}
\label{23-2-7}\\
\shortintertext{and}
&X'_2=\{x\in X'\colon \mu \rho _1\ge c\rho \gamma ^{-1}\}
\label{23-2-8}
\end{align}
where the magnetic field $\mu_\eff=\mu\rho_1\rho^{-1}\gamma$  is \emph{normal\/} ($\mu_\eff \le 2c$) and where it is \emph{strong\/} ($\mu_\eff \ge c$) respectively. We assume that
\begin{equation}
|F|\ge \epsilon \rho _1\qquad \text{in \ \ } X'_2
\label{23-2-9}
\end{equation}
where $F_{jk}$ and $F$ are the tensor and scalar intensities of the magnetic field.

Moreover, let us assume that
\begin{phantomequation}\label{23-2-10}\end{phantomequation}
\begin{gather}
B\bigl(x,\gamma (x)\bigr)\subset X\qquad \forall x\in X'_{2-},
\tag*{$\textup{(\ref*{23-2-10})}_1$}\label{23-2-10-1}\\
\shortintertext{and}
u|_{\partial X\cap B\bigl(x,\gamma (x)\bigr)}=0\qquad
\forall x\in X'\quad \forall u\in \fD(A)
\tag*{$\textup{(\ref*{23-2-10})}_2$}\label{23-2-10-2}
\end{gather}
where
\begin{align}
&X'_{2-}=\{x\in X'_2\colon  V+\mu h F \ge \epsilon \rho^2\}
\label{23-2-11}\\
\shortintertext{and}
&X'_{2+}=\{x\in X'\colon  V+\mu h F \le 2\epsilon \rho^2\}.
\label{23-2-12}
\end{align}

Finally, let the standard boundary regularity condition  be fulfilled:
\begin{claim}\label{23-2-13}
For every $y\in X$,
$\partial X\cap B(y,\gamma (y))=\{x_k=\phi _k(x_{\hat{k}})\}$ with
\begin{equation*}
|D ^\alpha \phi _k|\le c\gamma ^{-|\alpha |}
\end{equation*}
and $k=k(y)$.
\end{claim}

Due to \ref{23-2-10-1} we need this condition only as $y\in X'_1 \cup X'_{2-}$. Finally, in $X'_{2-}$ let one of the following non-degeneracy conditions
\begin{equation}
|V+(2n+1)\mu hF| + |\nabla VF^{-1}|\gamma \ge  \epsilon \rho ^2
\qquad \forall n \in \bZ^+,
\label{23-2-14}
\end{equation}
\begin{multline}
|V+(2n+1)\mu hF| + |\nabla VF^{-1}|\gamma \le  \epsilon \rho ^2 \epsilon \implies \\[3pt]
\det \Hess (VF^{-1})\ge \epsilon \rho ^4\rho _1^{-2}\gamma ^{-4}
\qquad \forall n \in \bZ^+,
\label{23-2-15}
\end{multline}
\begin{multline}
|V+(2n+1)\mu hF| + |\nabla VF^{-1}|\gamma \le  \epsilon \rho ^2 \epsilon \implies \\[3pt]
|\det \Hess (VF^{-1})|\ge \epsilon \rho ^4\rho _1^{-2}\gamma ^{-4}
\qquad \forall n \in \bZ^+,
\label{23-2-16}
\end{multline}
be fulfilled.%

Recall that according to Chapter~\ref{book_new-sect-13} of \cite{futurebook} the contribution of the partition element $\psi\in \sC_0^K(B(y,\frac{1}{2}\gamma(y))$ to the principal part of asymptotics is
\begin{equation}
\cN^-(\mu, h)= \cN^{\MW\,-}(\mu, h)\coloneqq   h^{-2}\int \cN^\MW (x,\mu h)\psi(x)\,dx
\label{23-2-17}
\end{equation}
with $\cN^\MW (x,\mu h)$ given by \textup{(\ref{book_new-13-1-9}) of \cite{futurebook})} with $d=2$.

On the other hand, its contribution to the remainder  does not exceed
$Ch^{-1} \rho\gamma$ if $\mu\rho_1 \gamma\le c\rho $ \underline{and}
it does not exceed $C\mu^{-1}h^{-1} \rho^2\rho_1^{-1}$ if
$\mu\rho_1 \gamma\ge c \rho $ but   $\mu h\rho_1\le \rho^2$, $y\in X'_+$
and assumption (\ref{23-2-15}) is fulfilled\footnote{\label{foot-23-3} It does not exceed the same expression with an extra logarithmic factor under assumption (\ref{23-2-16}) but logarithmic factor could be skipped if we add corrections at the points with negative $\det\Hess(VF^{-1})$.}, \underline{and}
it does not exceed $ C(\mu ^{-s}\rho_1^{-s}\gamma^{-2s})$ if
$C\mu\rho_1 \gamma\ge \rho $, $\mu h\rho_1\le \rho^2$, $y\in X'_{2-}$.

Then we get estimate of $\N^- $ from below by the magnetic Weyl approximation $\cN^-(\mu, h)$ minus corresponding remainder, and also from above  by magnetic Weyl approximation plus corresponding remainder, provided $X=X'$ (so, there is no singular zone $X''=\emptyset$):
\begin{multline}
h^{-d} \int _{X'} \cN^\MW (x,\mu h)\,dx - CR_1  \le \N^- (0) \le\\
h^{-d} \int _{X'} \cN^\MW (x,\mu h)\,dx + CR_1 +C'R_2
\label{23-2-18}
\end{multline}
with
\begin{gather}
R_1= \mu^{-1}h^{1-d} \int_{X'_+} \rho^d\rho_1^{-1}\gamma^{-2}\,dx , \label{23-2-19}\\
R_2= \mu h^{s-d}\int_{X'_-} \rho_1 \rho^{d-s-1}\gamma^{1-s}\,dx
\label{23-2-20}
\shortintertext{provided}
\mu \rho_1 \gamma \ge \rho
\label{23-2-21}
\end{gather}
where the latter condition could be assumed without any loss of the generality, $C'$ depens also on $s$ and $\epsilon$.

We leave to the reader the following not very challenging set of problems:

\begin{Problem}\label{Problem-23-2-1}
\begin{enumerate}[label=(\roman*), wide, labelindent=0pt]
\item\label{Problem-23-2-1-i}
Consider the case $X''\ne \emptyset$ and prove the estimate from above with an extra term $C_1R_0$ like in estimate (\ref{book_new-9-1-29})  (Theorem~\ref{book_new-thm-9-1-7} of \cite{futurebook}) with $R_0$ defined in the same way as in Section~\ref{book_new-sect-9-1} of \cite{futurebook}.

\item\label{Problem-23-2-1-ii}
Consider multidimensional case; then we need to impose more sophisticated non-degeneracy assumptions (see Chapter~\ref{book_new-sect-19} of \cite{futurebook}).

\item\label{Problem-23-2-1-iii}
Incorporate results of Chapters~\ref{book_new-sect-14}, 15, \ref{book_new-sect-18}, \ref{book_new-sect-19} (in non-smooth settings), \ref{book_new-sect-21} and \ref{book_new-sect-22} of \cite{futurebook}.

\item\label{Problem-23-2-1-iv}
Using arguments and methods of Chapter~\ref{book_new-sect-10}  and results of Chapter~\ref{book_new-sect-17} of \cite{futurebook} consider Dirac operator (in the full-rank case).
\end{enumerate}
\end{Problem}

\subsection{Basic results}
\label{sect-23-2-1-2}

In what follows $h\to +0$ and the semiclassical zone $X'$ expands to $X$ while $\mu$ is either bounded (then we can assume that the zone of the strong magnetic field $X'_2$ is fixed) or tends to $\infty$ (then $X'_2$ expands to $X$). We assume that all conditions of the previous subsection are fulfilled with $\mu=h=1$ but we will assume them fulfilled in the corresponding zones.

The other important question is whether  $\mu h\to 0$, remains bounded and disjoint from $0$ or tends to $\infty$.

Finally, we should consider the singular zone $X''$.  In order to avoid this task we assume initially that
\begin{equation}
\rho _1\gamma ^2+\rho \gamma \ge \epsilon.
\label{23-2-22}
\end{equation}

Then we obtain the following assertion from the arguments of the previous Subsection~\ref{sect-23-2-1-1}.1:

\begin{theorem}\label{thm-23-2-2}
Let $d=2$ and let the Schr\"{o}dinger operator $A$ satisfy conditions \textup{(\ref{23-2-15})}, \textup{(\ref{23-2-2})}, \ref{23-2-3-*}, \textup{(\ref{23-2-9})}  in the corresponding regions where
$\rho ,\rho _1,\gamma $ satisfy $\textup{(\ref{book_new-9-1-6})}_{1-4}$ of \cite{futurebook}, $\textup{(\ref{23-2-10})}_{1,2}$, \textup{(\ref{23-2-22})}.

Let $\rho _1\ge \epsilon \rho \gamma ^{-1}$ and
\begin{gather}
\rho ^2\rho _1^{-1}\gamma^{-2}\in
\sL^1\bigl(X\cap \{V+tF\le \epsilon \rho ^2\}\bigr),
\label{23-2-23}\\
\rho _1^{-s}\gamma ^{-2s-d}\in \sL^1(X)
\label{23-2-24}
\end{gather}
with some $s\ge 0$ and $t\ge 0$.  Then for $h\to+0$, $1\le \mu $ such that
$\mu h\le t$ the ``standard'' asymptotics
\begin{gather}
\N^-(\mu ,h)=\cN ^- (\mu ,h)+O(\mu ^{-1}h^{-1})
\label{23-2-25}
\shortintertext{holds with}
\cN^- (\mu ,h)\coloneqq   h^{-d}\int \cN^\MW (x,\mu h)\,dx.
\label{23-2-26}
\end{gather}
\end{theorem}

\subsection{Power singularities}
\label{sect-23-2-1-3}

\begin{example}\label{example-23-2-3}
\begin{enumerate}[label=(\roman*), wide, labelindent=0pt]
\item\label{example-23-2-3-i}
Let $0$ be an \emph{inner singular point\/}\footnote{\label{foot-23-4} I.e. $0\in \bar{X}$ is an isolated point of $\bR^2\setminus X$.}  and let conditions of Theorem~\ref{thm-23-2-2} be fulfilled  with  $\gamma=\epsilon  |x|$,
$\rho =|x|^m$, $\rho _1=|x|^{m_1}$ and let
$m_1<\min (m-1, 2m)$\,\footnote{\label{foot-23-5} One can construct such potential $(V_1,V_2)$ easily; f.e. $V_1=-x_2|x|^{m_1}$, $V_2=x_1|x|^{m_1}$ for $m_1\ne 2$; for $m_1=-2$ one needs to multiply $V_1,V_2$ by $\log |x|$.}\footnote{\label{foot-23-6} To have the non-degeneracy condition (\ref{23-2-15})  fulfilled in the vicinity of $0$ we assume that
\begin{equation}
|\nabla VF^{-1}|\ge C\rho^2\rho_1^{-1}\gamma^{-1} \qquad \text{as\ \ } |x|\le \epsilon;
\label{23-2-27}
\end{equation}
in Statement~\ref{example-23-2-3-ii} one should replace $|x|\le \epsilon$ by $|x|\ge c$.}.

Let $V+\mu h F\ge \epsilon \rho^2$ on
$(\partial X\setminus 0)\cup \{x\colon |x|\ge c\}$.

Then conditions (\ref{23-2-23}), (\ref{23-2-24}) are fulfilled automatically and asymptotics (\ref{23-2-25})--(\ref{23-2-26}) holds for $\mu$ disjoint from $0$ and $h\to +0$.

Further,
\begin{equation}
\cN ^- (\mu ,h)=\left\{\begin{aligned}
&O(h^{-2})  &&m>-1,\\[2pt]
&O\bigl(h^{-2}(|\log (\mu h)| +1)\bigr) &&m=-1,\\[2pt]
&O\bigl(h^{-2}(\mu h)^{2(m+1)/(2m-m_1)}\bigr) &&m<-1.
\end{aligned}\right.
\label{23-2-28}
\end{equation}
Furthermore, one can replace ``$=O$''with ``$\asymp$" if \underline{either} $m>-1$, $\mu h\le t$ and
\begin{gather}
\{X \setminus 0, V\le -tF-\epsilon \}\ne \emptyset
\label{23-2-29}\\
\shortintertext{\underline{or} $m\le -1$ and}
V\le -\epsilon \rho ^2\qquad
\text{in \ \ } \Gamma \cap\{|x|\le \epsilon \} \subset X
\label{23-2-30}
\end{gather}
where $\Gamma $ is an open non-empty sector (cone) with vertex at $0$, and
$\mu h\le t$ with a small enough constant $t>0$.

\item\label{example-23-2-3-ii}
Let infinity be an \emph{inner singular point\/}\footnote{\label{foot-23-7} I.e. $\bR^2\setminus X$ is  compact.}  and let conditions of Theorem~\ref{thm-23-2-2} be fulfilled  with $\gamma=\epsilon  \langle x\rangle$,
$\rho \langle x\rangle^m$, $\rho _1=\langle x\rangle^{m_1}$ and
let $m_1>\max (m-1,2m)$\,\footref{foot-23-5}\footref{foot-23-6}. Let $V+\mu h F\ge \epsilon \rho^2$ on $\partial X$.

Then conditions (\ref{23-2-23}) and (\ref{23-2-24}) are fulfilled automatically and asymptotics (\ref{23-2-25})--(\ref{23-2-26}) holds for $\mu$ disjoint from $0$ and $h\to +0$.

Further,
\begin{equation}
\cN ^- (\mu ,h)=\left\{\begin{aligned}
&O(h^{-2})  &&m<-1,\\[2pt]
&O\bigl(h^{-2}(|\log (\mu h)| +1)\bigr) &&m=-1,\\[2pt]
&O\bigl(h^{-2}(\mu h)^{2(m+1)/(2m-m_1) }\bigr) &&m>-1.
\end{aligned}\right.
\tag*{$\textup{(\ref*{23-2-28})}^\#$}\label{23-2-28-*}
\end{equation}
Furthermore, one can replace ``$=O$'' by ``$\asymp$" if \underline{either} $m<-1$, $\mu h\le t$ and \textup{(\ref{23-2-29})} is fulfilled \underline{or} $m\le -1$,
\begin{equation}
V\le -\epsilon \rho ^2\qquad \text{in \ \ } \Gamma \cap\{|x|\ge c \} \subset X
\tag*{$\textup{(\ref*{23-2-30})}^\#$}\label{23-2-30-*}
\end{equation}
where $\Gamma $ is an open non-empty sector (cone) with vertex at $0$, and
$\mu h\le t$ with a small enough constant $t>0$.

\item\label{example-23-2-3-iii}
One can easily see that for $m>-1$ in \ref{example-23-2-3-i}, $m<-1$ in \ref{example-23-2-3-ii}
\begin{equation}
\cN ^- (\mu ,h)=\cN ^{\W\,-} (\mu ,h)+
O\bigl(h^{-2}(\mu h)^{(2m+2)/(2m-m_1) }\bigr)
\label{23-2-31}
\end{equation}
under condition (\ref{23-2-15}) where $\cN^{\W\,-} $ is the standard Weyl expression\footnote{\label{foot-23-8} Rather than the magnetic Weyl expression $\cN^{\MW\,-}$.}.  Therefore for $\mu \le h^p$ with $p=(2-m_1)/(4m-m_1+2) $ the asymptotics remain true with $\cN ^-$ replaced by $\cN ^{\W\,-}$.

On the other hand, under condition (\ref{23-2-30}) or \ref{23-2-30-*} (respectively)
$\cN ^{\W\,-} -\cN ^{\MW\,-} \ge \epsilon h^{-2}(\mu h)^{(2m+2)/(2m-m_1)}$.  Thus, for $\mu >h^p$ one cannot replace $\cN ^-$ by $\cN ^{\W\,-}$ and preserve the remainder estimates.  Note that $p$ is not necessarily negative in our conditions!  Therefore, due to the singularity, the magnetic field can be essential even for a fixed $\mu $. In what follows we leave this type of analysis to the reader.
\end{enumerate}
\end{example}

\begin{example}\label{example-23-2-4}
\begin{enumerate}[label=(\roman*), wide, labelindent=0pt]
\item\label{example-23-2-4-i}
Assume now that $m_1\ge \min (m-1,2m)$, $m_1\ne 2m$\,\footnote{\label{foot-23-9} Otherwise we could not satisfy(\ref{23-2-27}).},  while all other assumptions of Example~\ref{example-23-2-3}\ref{example-23-2-3-i} are fulfilled. Since we want to have a finite $\cN^-$ and (\ref{23-2-30}) to be fulfilled we need to assume that $m>-1$. Then $\cN\asymp h^{-2}$. Let us calculate the remainder estimate.

\begin{enumerate}[label=(\alph*), wide, labelindent=0pt]
\item\label{example-23-2-4-ia}
Assume first that $m-1\le m_1 < 2m$; then $\int \rho_1^{-1}\rho^2 \,dx<\infty$ and then contribution of the semiclassical zone
$X'=\{x\colon  |x|\ge r_0= h^{1/(m+1)}\}$ to the remainder is $O(\mu^{-1}h^{-1})$ and we need to estimate the contribution of the singular zone
$X''=\{x\colon  |x|\le r_0\}$. Without any loss of the generality we can assume that
$|\vec{V}|\le C\rho_1\gamma$\,\footnote{\label{foot-23-10} One can prove it easily taking $V_1=-\partial_2 \phi$, $V_2=\partial_1\phi$ with $\phi$ solving $\Delta \phi = F$.}. Then LCR\footnote{\label{foot-23-11} Sure, LCR does not hold for $d=2$ but we can use more complicated Rozenblioum's estimate exactly like in Section~\ref{book_new-sect-11-1} of \cite{futurebook}.} implies that the contribution of $X''$ to the asymptotics does not exceed
$Ch^{-2}\int _{X''} (\rho^2 + \mu^2 \rho_1^2\gamma^2)\,dx\asymp C$.
\item\label{example-23-2-4-ib}
Let now $m_1> 2m$; then we need to consider zones
\begin{align*}
&X'_2=\{x\colon |x|\ge r_1=\mu^{-1/(m_1+1-m)}\}\\
\shortintertext{where $\mu_\eff \ge 1$,}
&X'_1=\{x\colon  r_0\le |x|\le r_1\}
\end{align*}
where $\mu_\eff \le 1$ and $h_\eff\le 1$ and $X''$. Contributions of $X'_2$, $X'_1$ to the remainder do not exceed
$K_1 \coloneqq   C\mu^{-1}h^{-1}\int_{X'_2} \rho_1^{-1}\rho^2\,dx$ and
$K_1\coloneqq   Ch^{-1}\int_{X'_1}\rho \gamma^{-1}\,dx$ respectively.

Obviously,
$K_1\asymp K_2 \asymp  Ch^{-1}r_1^{m+1}\asymp Ch^{-1}\mu^{-(m+1)/(m_1+m-1)}$ for $m_1>2m$.

Finally, contribution of $X''$ to the asymptotics is $O(1)$.
\end{enumerate}

Thus as $h\to +0$, $\mu$ is disjoint from $0$ and $\mu h$ is disjoint from infinity, we have asymptotics
\begin{equation}
\N^- (\mu,h)= \cN^- (\mu,h)+ \left\{\begin{aligned}
& O(h^{-1}\mu^{-1}) && m_1<2m,\\
& O(h^{-1}\mu^{-(m+1)/(m_1+1-m)}) && m_1 >2m.
\end{aligned}\right.
\label{23-2-32}
\end{equation}

\item\label{example-23-2-4-ii}
Assume now that $m_1\le \max (m-1,2m)$  while all other assumptions of Example~\ref{example-23-2-3}\ref{example-23-2-3-ii} are fulfilled and $m<-1$. Again, considering cases
\begin{enumerate}[label=(\alph*), wide, labelindent=0pt]
\item\label{example-23-2-4-iia}
$2m < m_1 < m-1$ and
\item\label{example-23-2-4-iib}
$m_1< 2m$
\end{enumerate}
we arrive to the asymptotics
\begin{equation}
\N^- (\mu,h) = \cN^- (\mu,h) + \left\{\begin{aligned}
& O(h^{-1}\mu^{-1} ) && m_1>2m,\\
& O(h^{-1}\mu^{-(m+1)/(m_1+1-m)}) && m_1 <2m.
\end{aligned}\right.
\tag*{$\textup{(\ref*{23-2-32})}^\#$}\label{23-2-32-*}
\end{equation}
\end{enumerate}
\end{example}

Consider now fast increasing $\mu $ so that $\mu h\to \infty$. We will get non-trivial results only when domain defined by $\mu_\eff h_\eff \le C_0$ shrinks but remains non-empty which happens only in the frameworks of subcases (b) of Example~\ref{example-23-2-4}.

\begin{example}\label{example-23-2-5}
\begin{enumerate}[label=(\roman*), wide, labelindent=0pt]
\item\label{example-23-2-5-i}
In the framework of Example~\ref{example-23-2-4}\ref{example-23-2-4-i} with $m_1>2m$  consider  $\mu h\to \infty$. Then the  allowed domain is
\begin{equation}
\{x\colon  |x|\lesssim r_2 = (\mu h)^{-1/(m_1-2m)}\}
\label{23-2-33}
\end{equation}
and we have $r_0\le r_1\le r_2$ if $\mu \lesssim h^{-(m_1+1-m)/(m+1)}$ while for
$\mu \gtrsim h^{-(m_1+1-m)/(m+1)}$ inequalities go in the opposite direction.

Therefore as $h\to +0$, $ch^{-1}\le \mu \le h^{-(m_1+1-m)/(m+1)}$ asymptotics (\ref{23-2-32}) holds and one can see easily that
$\cN^- (\mu,h)\asymp h^{-2}r_2^{2m+2}$:
\begin{equation}
\cN^- (\mu,h) \asymp \mu ^{-2(m+1)/(m_1-2m)}h^{-2(m_1+ 1-m)/(m_1-2m)}.
\label{23-2-34}
\end{equation}

\item\label{example-23-2-5-ii}
Similarly, in the framework of Example~\ref{example-23-2-4}\ref{example-23-2-4-ii} with $m_1<2m$  asymptotics \ref{23-2-32-*} and (\ref{23-2-34}) hold as $h\to +0$,
$ch^{-1}\le \mu \le h^{-(m_1+1-m)/(m+1)}$.
\end{enumerate}
\end{example}

Let us consider  $\mu \to \mu_0$ where $\mu_0\ge 0$ is fixed.

\begin{example}\label{example-23-2-6}
\begin{enumerate}[label=(\roman*), wide, labelindent=0pt]
\item\label{example-23-2-6-i}
Let us consider singularity at $0$. In this case we are in the framework of Section~\ref{book_new-sect-11-1} of \cite{futurebook} provided $m>-1$, $m_1>-2$. So we need to consider the case when either $m\le -1$ or $m_1\le -2$ \underline{and} $m_1\ne 2m$.
\begin{enumerate}[label=(\alph*), wide, labelindent=0pt]
\item\label{example-23-2-6-ia}
The contribution of the semiclassical zone to the remainder is $O(h^{-1})$ only if $m > -1$.

\item\label{example-23-2-6-ib}
Assume now that $m\le -1$; then we need to assume that $m_1<2m$ and we have a normal magnetic field zone $X'_1=\{x\colon  |x|\ge r_1=\mu^{-1/(m_1-m+1)}\}$,
strong magnetic field zone
$X'_2=\{x\colon  r_0=(\mu h)^{-1/(m_1-2m)}\le |x|\le r_1\}$, and forbidden zone
$X'_3=\{x\colon |x|\le r_1\}$. One can check that $r_0<r_1$. Then the contribution to the remainder of $X'_1$ and $X'_2$ are both
$O(h^{-1}r_1^{m+1})= O(h^{-1}\mu^{-(m+1)/(m_1-m+1)})$ as $m<-1$ while for $m=-1$ contribution of $X'_1$ is $O(h^{-1}(|\log \mu|+1))$ and of $X'_2$ is $O(h^{-1})$. Contribution of $X'_3$ is smaller. So we get a remainder estimate
\begin{equation}
\N^- (\mu,h)=\cN^- (\mu,h) + \left\{\begin{aligned}
&O(h^{-1}), && m>-1,\\[2pt]
&O(h^{-1}(|\log \mu|+1)) && m=-1,\\[2pt]
&O(h^{-1}\mu^{-(m+1)/(m_1-m+1)})  && m<-1.
\end{aligned}\right.
\label{23-2-35}
\end{equation}
Meanwhile, one can prove easily that
\begin{multline}
\cN^- (\mu, h) \asymp \left\{\begin{aligned}
&h^{-2} &&m>-1,\\[2pt]
&h^{-2}(|\log \mu h|+1) && m=-1,\\[2pt]
&h^{-2(m_1-m+1)/(m_1-2m)}\mu ^{-2(m+1)/(m_1-2m)}  && m<-1.
\end{aligned}\right.
\label{23-2-36}
\end{multline}
under assumption (\ref{23-2-30}).
\end{enumerate}

\item\label{example-23-2-6-ii}
Similarly, consider the singularity at infinity. Then
\begin{enumerate}[label=(\alph*), wide, labelindent=0pt]
\item\label{example-23-2-6-iia}
Let $m<-1$; then the remainder estimate is $O(h^{-1})$.

\item\label{example-23-2-6-iib}
Let $m\ge -1$, $m_1> 2m$. Then
\begin{multline}
\N^- (\mu,h) =\cN^- (\mu,h) + \left\{\begin{aligned}
&O(h^{-1}) &&m<-1,\\[2pt]
&O(h^{-1}(|\log \mu|+1)) && m=-1,\\[2pt]
&O(h^{-1}\mu^{-(m+1)/(m_1-m+1)})  && m>-1.
\end{aligned}\right.
\tag*{$\textup{(\ref*{23-2-35})}^\#$}\label{23-2-35-*}
\end{multline}
and
\begin{multline}
\cN^- (\mu,h)\asymp\!  \left\{\begin{aligned}
&h^{-2} &&M<-1,\\[2pt]
&h^{-2}(|\log \mu h|+1) && m=-1,\\[2pt]
&h^{-2(m_1-m+1)/(m_1-2m)}\mu^{-2(m+1)/(m_1-2m)}  && m>-1.
\end{aligned}\right.
\tag*{$\textup{(\ref*{23-2-36})}^\#$}\label{23-2-36-*}
\end{multline}
under assumption \ref{23-2-30-*}.
\end{enumerate}
\end{enumerate}
\end{example}

\subsection{Improved remainder estimates}
\label{sect-23-2-1-4}
Let us improve remainder estimates (\ref{23-2-32}), \ref{23-2-32-*}, (\ref{23-2-35}), \ref{23-2-35-*} under certain non-periodicity-type assumptions.

\begin{example}\label{example-23-2-7}
\begin{enumerate}[label=(\roman*), wide, labelindent=0pt]
\item\label{example-23-2-7-i}
In the framework of Example~\ref{example-23-2-6}\ref{example-23-2-6-i} with $m>-1$ the contribution to the remainder of the zone
$\{x\colon  |x|\le \varepsilon\}$ does not exceed $\sigma h^{-1}$ with $\sigma=\sigma(\varepsilon)\to 0$ as $\varepsilon\to +0$. Then the standard arguments imply that under the standard non-periodicity assumption for Hamiltonian billiards\footnote{\label{foot-23-12} On the energy level $0$.} with the Hamiltonian
\begin{gather}
a (x,\xi,\mu_0)  =
\sum _{j,k} g^{jk}(\xi_j-\mu_0V_j)(\xi_k -\mu_0V_k)+V(x)
\label{23-2-37}\\
\intertext{the improved asymptotics}
\N^- (\mu,h) =\cN^- (\mu,h) + \kappa_1 h^{-1} +o(h^{-1})
\label{23-2-38}
\end{gather}
holds as $h\to +0$, $\mu\to \mu_0$ where $\kappa_1h^{-1}$ is the contribution of $\partial X$ calculated as $\mu=\mu_0$. However, in the general case we cannot replace  $\cN^{\MW\,-}$ by $\cN^{\W\,-}$ even if $\mu_0=0$.

\item\label{example-23-2-7-ii}
Similarly in the framework of Example~\ref{example-23-2-6}\ref{example-23-2-6-ii} with $m<-1$  under the standard non-periodicity assumption for Hamiltonian billiards\footref{foot-23-12}  with the Hamiltonian (\ref{23-2-37})  asymptotics (\ref{23-2-38}) holds as $h\to +0$, $\mu\to \mu_0$.

\item\label{example-23-2-7-iii}
In the framework of Example~\ref{example-23-2-6}\ref{example-23-2-6-i} with $m<-1$ the contributions to the remainder of the zones
$\{x\colon  |x|\le \varepsilon r _1\}$ and $\{x\colon  |x|\ge \varepsilon^{-1} r _1\}$ do not exceed $\sigma h^{-1} r_1^{m+1}$ with $\sigma=\sigma(\varepsilon)\to 0$ as $\varepsilon\to +0$. After scaling $x\mapsto xr_1^{-1}$ etc the magnetic field in the zone  $\{x\colon  \varepsilon r _1 \le |x|\le \varepsilon ^{-1}r _1\}$ becomes disjoint from $0$ and $\infty$.

Assume that  $g^{jk},V_j, V$ stabilize to positively homogeneous of degrees $0, m_1+1, 2m$ functions $g^{jk0}, V_j^0$, $V^0$ as $x\to 0$:
\begin{phantomequation}\label{23-2-39}\end{phantomequation}
\begin{align}
&D ^\alpha (g^{jk}-g^{jk0})=o\bigl(|x|^{-|\alpha |}\bigr),
\tag*{$\textup{(\ref*{23-2-39})}_1$}\label{23-2-39-1}\\
&D ^\alpha (V_j-V_j^0)=o\bigl(|x|^{m_1+1-|\alpha |}\bigr),
\tag*{$\textup{(\ref*{23-2-39})}_2$}\label{23-2-39-2}\\
&D ^\alpha (V-V^0)=o\bigl(|x|^{2m-|\sigma |}\bigr)
&&\forall \alpha : |\alpha |\le 1.
\tag*{$\textup{(\ref*{23-2-39})}_3$}\label{23-2-39-3}
\end{align}
Then the standard arguments imply that under the standard non-periodicity assumption for Hamiltonian trajectories\footref{foot-23-12}\footnote{\label{foot-23-13} In $T^*(\bR^2\setminus 0)$.} with the Hamiltonian
\begin{gather}
a^0(x,\eta )=\sum_{j,k} g^{jk0}(\eta _j-V_j^0)(\eta _k-V_k^0)+V_0
\label{23-2-40}\\
\intertext{the improved asymptotics}
\N^-(\mu ,h)=\cN (\mu ,h)+
o\bigl(h^{-1}\mu ^{-(m+1)/(m_1+1-m)}\bigr)
\label{23-2-41}
\end{gather}
holds as $h\to +0$, $\mu\to 0$.

\item\label{example-23-2-7-iv}
Similarly in the framework of Example~\ref{example-23-2-6}\ref{example-23-2-6-ii} with $m>-1$ let stabilization conditions $\textup{(\ref{23-2-39})}^\#_{1-3}$ be fulfilled. Then under the standard non-periodicity assumption for Hamiltonian trajectories\footref{foot-23-12}\footref{foot-23-13}  with the Hamiltonian (\ref{23-2-40})  asymptotics (\ref{23-2-41}) holds as $h\to +0$, $\mu\to 0$.

\item\label{example-23-2-7-v}
In the framework of Example~\ref{example-23-2-6}\ref{example-23-2-6-i} with $m=-1$ the contributions to the remainder of the zone $\{x: |x|\le r_1\}$ does not exceed $Ch^{-1}$, while the contributions to the remainder of the zones
$\{x: r_1\le |x|\le r_1\mu^{-\delta}\}$ and  $\{x: |x|\ge \mu^{\delta}\}$
do not exceed $C\delta h^{-1}|\log \mu|$ respectively. So, only zone
$\{x: r_1\mu^{-\delta}\le |x|\le \mu^{\delta}\}$ should be treated. After rescaling magnetic field in this zone is small.

Let stabilization conditions $\textup{(\ref*{23-2-39})}_{1,3}$ be fulfilled.
Then the standard arguments imply that under the standard non-periodicity assumption for Hamiltonian trajectories\footref{foot-23-12}\footref{foot-23-13}  with the Hamiltonian
\begin{gather}
a^0(x,\eta )=\sum_{j,k} g^{jk0}\eta _j\eta _k+V_0
\label{23-2-42}\\
\intertext{the improved asymptotics}
\N^-(\mu ,h)=\cN (\mu ,h)+ o(h^{-1}|\log \mu|)
\label{23-2-43}
\end{gather}
holds as $h\to +0$, $\mu\to 0$.

\item\label{example-23-2-7-vi}
Similarly in the framework of Example~\ref{example-23-2-6}\ref{example-23-2-6-ii} with $m=-1$ let stabilization conditions $\textup{(\ref{23-2-39})}^\#_{1,3}$ be fulfilled. Then  under the standard non-periodicity assumption for Hamiltonian trajectories\footref{foot-23-12}\footref{foot-23-13}  with the Hamiltonian (\ref{23-2-42})  asymptotics (\ref{23-2-43}) holds as $h\to +0$, $\mu\to 0$.
\end{enumerate}
\end{example}

Consider now case of $\mu \to \infty$; we would like to improve estimates
(\ref{23-2-32}) for $m_1 \ge 2m$ and \ref{23-2-32-*} for $m_1\le 2m$. Recall that these estimates hold provided $\mu \lesssim h^{-(m_1+1-m)/(m+1)}$.

\begin{example}\label{example-23-2-8}
\begin{enumerate}[label=(\roman*), wide, labelindent=0pt]
\item\label{example-23-2-8-i}
In the framework of Example~\ref{example-23-2-4}\ref{example-23-2-4-i} with $m_1>2m$ let stabilization conditions $\textup{(\ref{23-2-39})}_{1-3}$ be fulfilled. Then using arguments of Example~\ref{example-23-2-7}\ref{example-23-2-7-iii} one can prove easily that under the the standard non-periodicity assumption for Hamiltonian trajectories\footref{foot-23-12}\footref{foot-23-13}  with the Hamiltonian (\ref{23-2-40}) asymptotics (\ref{23-2-41}) holds as $h\to +0$, $\mu\to\infty$  provided  $\mu =o( h^{-(m_1+1-m)/(m+1)})$.

\item\label{example-23-2-8-ii}
Similarly in the framework of Example~\ref{example-23-2-4}\ref{example-23-2-4-ii} with $m_1<2m$ let stabilization conditions $\textup{(\ref{23-2-39})}^\#_{1-3}$ be fulfilled. Then under the the standard non-periodicity assumption for Hamiltonian trajectories\footref{foot-23-12}\footref{foot-23-13}  with the Hamiltonian (\ref{23-2-40}) asymptotics (\ref{23-2-41}) holds as $h\to +0$, $\mu\to\infty$  provided  $\mu =o( h^{-(m_1+1-m)/(m+1)})$.
\end{enumerate}
\end{example}

\subsection{Degenerations}
\label{sect-23-2-1-5}

Consider magnetic field with degeneration.

\begin{example}\label{example-23-2-9}
\begin{enumerate}[label=(\roman*), wide, labelindent=0pt]
\item\label{example-23-2-9-i}
Let $0$ be an inner singular point of $X$ and all assumptions of Example~\ref{example-23-2-3}~\ref{example-23-2-3-i} be fulfilled, except (\ref{23-2-9}) $F\asymp \rho_1$ which is replaced now by
\begin{equation}
|F|+|\nabla F|\gamma \asymp \rho_1.
\label{23-2-44}
\end{equation}
let $\Sigma = \{x\colon F_{12}=0\}$ be the manifold of degeneration,
$\cZ=\{x\colon  |F|\le \epsilon \rho_1\}$ be the vicinity of the degeneration.
Then we need to refer to Chapter~\ref{book_new-sect-14} of \cite{futurebook}. To have  $\mu_\eff \gg 1$,  and also $\mu_\eff \ge C h_\eff^{-2}$ near singularity with as before
$\mu_\eff= \mu |x|^{m_1+1-m}$, $h_\eff = h|x|^{-m-1}$ we assume that
$m_1< 3m +1$ if $m\le -1$. Assume that $m_1< \min (m-1, 3m+1)$.

We preserve the non-degeneracy assumption (\ref{23-2-27}) in $X\setminus \cZ$ and replace it by\footnote{\label{foot-23-14} We need it mainly to avoid correction terms.}
\begin{equation}
|\nabla_\Sigma W |  \asymp \rho^2\rho_1^{-\frac{2}{3}}\gamma^{-\frac{1}{3}}
\label{23-2-45}
\end{equation}
where $W:|W|\lesssim \rho^2\rho_1^{-\frac{2}{3}}\gamma^{\frac{2}{3}}$ is introduced according to Chapter~\ref{book_new-sect-14} of \cite{futurebook}. Then in addition to $m_1\ne 2m$ we assume also that $m_1\ne 3m+1$.

Recall that for the Schr\"odinger operator instead of $C\mu_\eff^{-1}h_\eff^{-1}$ the local remainder estimate now is $C\mu_\eff^{-\frac{1}{2}}h_\eff^{-1}\asymp
C\mu^{-\frac{1}{2}}h^{-1}\gamma ^{\frac{1}{2}(-m_1+3m+1)}$ and summation with respect to partition of unity returns $O(\mu^{-\frac{1}{2}}h^{-1})$ as
$h\to +0$, $\mu$ is disjoint from $0$ and $\mu^2 h$ is disjoint from infinity.

One can prove easily that $\cN^-(\mu,h)\asymp h^{-2}$ as $\mu \lesssim h^{-1}$ and $\cN^-(\mu,h)\asymp \mu^{-1} h^{-3}$ as $h^{-1}\lesssim \mu \lesssim h^{-2}$ exactly like in the regular case.

\item\label{example-23-2-9-ii}
Similar results hold in the case when infinity is an inner singular point,
$m_1> \max (m-1,\,3m+1)$.
\end{enumerate}
\end{example}

The following problems are not very challenging but interesting:

\begin{Problem}\label{Problem-23-2-10}
\begin{enumerate}[label=(\roman*), wide, labelindent=0pt]
\item\label{Problem-23-2-10-i}
Consider the case of $0$ being an inner singular point,  $m>-1$, $m_1\ge m-1$. The threshold value is $m_1=3m+1$. Assume that $h\to +0$.

\begin{enumerate}[label=(\alph*), wide, labelindent=0pt]
\item\label{Problem-23-2-10-ia}
Consider cases $1\lesssim \mu \lesssim h^{-1}$ and
$h^{-1}\lesssim \mu\lesssim h^{-2}$  (cf. Example~\ref{example-23-2-4}). Prove that in the case $m-1\le m_1<3m+1$ the remainder estimate is $O(-\mu^{\frac{1}{2}}h^{-1})$. Derive the remainder estimate in the case $m_1>3m+1$. Calculate the magnitude of $\cN^-$.

\item\label{Problem-23-2-10-ib}
Consider case $m_1>3m+1$ and $\mu \gtrsim h^{-2}$ (cf. Example~\ref{example-23-2-5}). Derive the remainder estimate. Calculate the magnitude of $\cN^-$.

\item\label{Problem-23-2-10-ic}
Consider $\mu \to 0$ (cf. Example~\ref{example-23-2-6} and possibly Example~\ref{example-23-2-7}).  Derive the remainder estimate. Calculate the magnitude of $\cN^-$.
\end{enumerate}

\item\label{Problem-23-2-10-ii}
Similarly consider the case of $0$ being an inner singular point,  $m>-1$, $m_1\le m-1$.
\end{enumerate}
\end{Problem}

\begin{example}\label{example-23-2-11}
Stronger asymptotic degenerations are not easily accessible. F.e. even if Chapter~\ref{book_new-sect-14} of \cite{futurebook} provides us with tools to consider
$F_{12}= x_1^n |x|^{m_1-n}$  with $n=2\ge $, it does not provide us with a tool to deal with the ``small perturbations'' like
$F_{12}= x_1^n |x|^{m_1-n}+  b x_1^{n-2} |x|^{m_2-n+2}$ with $m_2> m_1$
($m_2< m_1$), if we consider vicinity of $0$ (infinity respectively).

We want to recover whatever remainder estimates are possible.
\begin{enumerate}[label=(\roman*), wide, labelindent=0pt]
\item\label{example-23-2-11-i}
Case $n=2$ and $b>0$ should be the easiest as then the perturbation helps:
the first rescaling is the standard $x\mapsto 2(x-y) r^{-1}$, $r=|y|$ in $B(y,\frac{1}{2}r)$, transforming  it to $B(0,1)$, and as before
$h\mapsto h_\eff=hr^{-m-1}$, $\mu\mapsto \mu_\eff=\mu r^{m_1+1}$ and perturbing field has the strength $\varepsilon \mu$ with $\varepsilon= r^{m_2-m_1}$.

Then the second rescaling $x\to (x-y)\gamma^{-1}$ with $\gamma = |F|^{1/2}$ as long as  $|F|\ge \bar{\gamma}= C_0\max(\mu^{-1/3},\, \varepsilon^{\frac{1}{2}})$ and $\bar{\gamma}$ otherwise. Therefore the contribution of $B(y,\gamma(y))$ to the remainder does not exceed $C\mu^{-1}h^{-1}\gamma^{-2} $ and the summation over partition results in $C\mu^{-1}h^{-1}\bar{\gamma}^{-3} $, i.e.
$Ch^{-1}\min (1,\, \mu^{-1}\varepsilon^{-\frac{3}{2}})$.

We leave to the reader to plug $h_\eff, \mu_\eff$ instead of $\mu, h$ and take a summation over the primary partition.

\item\label{example-23-2-11-ii}
Cases $n=3$ and $b>0$ and $n=2,3$ and $b<0$ are harder but for $m_2$ close enough to $m_1$ we refer to Chapter~\ref{book_new-sect-14} of \cite{futurebook} after the second rescaling.

Then the second rescaling is the same $x\to (x-y)\gamma^{-1}$ with
$\gamma = |F|^{1/n}$ as long as
$|F|\ge \bar{\gamma}= C_0\max(\mu^{-1/(n+1)},\, \varepsilon^{\frac{1}{2}})$  and $\bar{\gamma}$ otherwise. Repeating the above arguments we conclude that the contribution of this zone is $Ch^{-1}$ for $\varepsilon \le \mu^{-2/(n+1)}$ and
$Ch^{-1}\mu^{-1}\varepsilon^{-(n+1)/2}$ for $\varepsilon \le \mu^{-2/(n+1)}$.

In the former case we are done, in the latter  we need to explore zone
$|F|\le \varepsilon^{n/2}$. Rescaling we get $\mu' = \mu \varepsilon^{(n+1)/2}$ and $h'= h\varepsilon^{-1/2}$. Then we can refer to Chapter~\ref{book_new-sect-14} of \cite{futurebook} rather than Chapter~\ref{book_new-sect-13} of \cite{futurebook} and the contribution of $B(y,\varepsilon^{1/2})$ to the remainder does not exceed $C\mu^{-1/2}h^{-1}\varepsilon^{-(n-1)/4}$ and summation over this zone results in $C\mu^{-1/2}h^{-1}\varepsilon^{-(n+1)/4}$, which is greater than the contribution of the previous zone.

Again, we leave to the reader to plug $h_\eff, \mu_\eff$ instead of $\mu, h$ and take a summation over the primary partition.
\end{enumerate}
\end{example}

\subsection{Power singularities. II}
\label{sect-23-2-1-6}

Let us modify  our arguments for the case $\rho_3<1$. Namely, in addition to \ref{23-2-3-*} we assume that
\begin{phantomequation}\label{23-2-46}\end{phantomequation}
\begin{align}
 &|D^\alpha g^{jk}|\le c\rho_2\gamma ^{-|\alpha |},\qquad
|D^\alpha F_{jk}|\le c\rho_2\rho _1 \gamma ^{-|\alpha |},
\tag*{$\textup{(\ref*{23-2-46})}_{1,2}$}\label{23-2-46-1}\\
&|D^\alpha {\frac{V}{F}}|\le
c\rho_3\rho ^2\rho _1^{-1}\gamma ^{-|\alpha |}\qquad \qquad
\forall \alpha :1\le |\alpha |\le K
\tag*{$\textup{(\ref*{23-2-46})}_3$}\label{23-2-46-3}
\end{align}
with $\rho_3\le \rho_2\le 1$ in the
corresponding regions where $\rho ,\rho _1,\gamma , \rho_3 $ are scaling functions.

Recall that in Chapter~\ref{book_new-sect-13} of \cite{futurebook} operator was reduced to the canonical form with the term, considered to be negligible, of magnitude $\rho_2\mu_\eff^{-2N}$. In this case impose non-degeneracy assumptions
\begin{phantomequation}\label{23-2-47}\end{phantomequation}
\begin{multline}
\bigl|V+(2n+1)\mu hF\bigr|\le \epsilon _0\rho_3\rho ^2,\quad
n\in \bZ^+ \implies\\
|\nabla v^*|\ge \epsilon _0\rho_3\rho ^2\rho _1^{-1}\gamma ^{-1}
\tag*{$\textup{(\ref*{23-2-47})}^*$}\label{23-2-47-*}
\end{multline}
and
\begin{equation}
\rho_3 \ge C_0\rho_2 (\mu \rho _1 \gamma \rho^{-1})^{-N}
\label{23-2-48}
\end{equation}
where $v^*$ is what this reduction transforms  $VF^{-1}$ to and (\ref{23-2-48})  means that ``negligible'' terms do not spoil \ref{23-2-47-*}.

Then according to Chapter~\ref{book_new-sect-13} of \cite{futurebook} the contribution of $B(x,\gamma)$ to the Tauberian remainder does not exceed
\begin{align}
&C\bigl(\rho_3 \mu_\eff ^{-1}h_\eff^{-1}+ 1\bigr)
\label{23-2-49}\\
\intertext{while the approximation error does not exceed}
&C\rho_2 \rho_3^{-1}\mu_\eff^{1-2N} h_\eff^{-1}
\underbracket{\bigl(\rho_3 \mu_\eff ^{-1}h_\eff^{-1}+ 1\bigr)}
\min \bigl((\mu_\eff h_\eff \rho_3^{-1})^s ,\, 1\bigr)
\label{23-2-50}
\end{align}
with arbitrarily large $s$ and therefore selected factor could be skipped.

\begin{example}\label{example-23-2-12}
\begin{enumerate}[label=(\roman*), wide, labelindent=0pt]
\item\label{example-23-2-12-i}
Let $0$ be an inner singular point\footref{foot-23-4}  and let assumptions $\textup{(\ref{23-2-46})}_{1-3}$  be fulfilled  with  $\gamma=\epsilon  |x|$,
$\rho =|x|^m\bigl(\bigl|\ln |x|\bigr|+1\bigr)^\alpha $,
$\rho_1 =|x|^{2m}\bigl(\bigl|\ln |x|\bigr|+1\bigr)^\beta $,
$\rho_2=1$ and $\rho_3 =\bigl(\bigl|\ln |x|\bigr|+1\bigr)^{-1}$.

Assume that
\begin{claim}\label{23-2-51}
\underline{Either} $m<-1$ and $\beta>2\alpha$ \underline{or} $m=-1$ and $\beta > \max (\alpha,2\alpha)$.
\end{claim}

Then (\ref{23-2-48}) is fulfilled with $N=1$ and we can replace \ref{23-2-47-*} by
\begin{multline}
\bigl|V+(2n+1)\mu hF\bigr|\le \epsilon _0\rho_3\rho ^2,\quad
n\in \bZ^+ \implies\\
|\nabla VF^{-1}|\ge \epsilon _0\rho_3\rho ^2\rho _1^{-1}\gamma ^{-1}
\tag{\ref*{23-2-47}}\label{23-2-47-**}
\end{multline}
Then (\ref{23-2-49}) becomes
$C\bigl(\mu^{-1}h^{-1} |\log r|^{2\alpha -\beta -1}+1\bigr)$ and for $s=2$ the summation over zone
$\{\mu_\eff h_\eff \lesssim 1\}=\{\mu h |\log r|^{\beta-2\alpha}\lesssim 1\}$ results in $C\bigl( (\mu h)^{-1} + (\mu h)^{-1/(\beta-2\alpha)}\bigr)$.

Meanwhile, (\ref{23-2-50}) with $N=1$  becomes
$C(\mu h)^s |\log r|^{s(\beta -2\alpha +1)+2}$ in the zone
$\{\mu h |\log r|^{\beta-2\alpha+1}\le 1\}$ and the summation over this zone results in $C(\mu h)^{-3/(\beta -2\alpha+1)}$.

On the other hand, (\ref{23-2-50}) with $N=1$  becomes
$C\mu^{-1}h^{-1}|\log r|^{1-\beta+2\alpha}$  in the zone
$\{\mu h |\log r|^{\beta-2\alpha+1}\ge 1,\,
 \mu h |\log r|^{\beta-2\alpha}\le 1\}$
and the summation over this zone results in $O(\mu^{-1}h^{-1})$
as $\beta > 2\alpha +2$, $C(\mu h)^{-2/(\beta-2\alpha)}$ as
$\beta <2\alpha +2$ and 	$C\mu^{-1}h^{-1}|\log (\mu h)|$ as $\beta=2\alpha+2$. Therefore, we conclude that the remainder is $O(R)$ with
\begin{equation}
R\coloneqq   \left\{\begin{aligned}
&(\mu h)^{-1} && \beta>2\alpha+2,\\
&(\mu h)^{-1}|\log (\mu h)|  &&\beta=2\alpha+2,\\
&(\mu h)^{-2/(\beta-2\alpha)} && 2\alpha <\beta <2\alpha+2.
\end{aligned}\right.
\label{23-2-52}
\end{equation}
In particular, for $\beta \ge 2\alpha +2$ asymptotics (\ref{23-2-25})--(\ref{23-2-26}) hold for $h\to+0$, $\mu $ disjoint from $0$ and $\mu h$ disjoint from infinity\footnote{\label{foot-23-15} Provided $V+\mu h F\ge \epsilon \rho^2$ on $\partial X$.}.

One can see easily that under condition (\ref{23-2-30}) for $\mu h\le t$ with small enough $t>0$
\begin{phantomequation}\label{23-2-53}\end{phantomequation}
\begin{align}
&\cN ^-(\mu ,h)\asymp h^{-2}\bigl|\log (\mu h) \bigr|
&&\text{for\ \ } m=-1,\alpha =- \frac{1}{2} ,
\tag*{$\textup{(\ref*{23-2-53})}_1$}\label{23-2-53-1}\\
&\cN ^-(\mu ,h)\asymp
h^{-2}(\mu h)^{(2\alpha +1)/(\beta -2\alpha )}
&&\text{for\ \ } m=-1,\alpha >-\frac{1}{2}
\tag*{$\textup{(\ref*{23-2-53})}_2$}\label{23-2-53-2}
\end{align}
and
\begin{multline}
\epsilon h^{-2}
\exp\bigl(\epsilon (\mu h)^{1/(2\alpha -\beta )}\bigr)
\le \cN ^-(\mu ,h)\le \\
C _1h^{-2}\exp\bigl(C_1(\mu h)^{1/(2\alpha -\beta )}\bigr)
\tag*{$\textup{(\ref*{23-2-53})}_3$}\label{23-2-53-3}
\end{multline}
for $m<-1$ where $C_1>\epsilon >0$ are constants.

On the other hand, for $m=-1$, $\alpha <-{\frac{1}{2}}$ the equivalence
$\cN ^-(\mu ,h)\asymp h^{-2}$ holds.

\item\label{example-23-2-12-ii}
Let infinity be an inner singular point\footref{foot-23-7} and let  $\textup{(\ref{23-2-46})}_{1-3}$  be fulfilled  with
$\gamma=\epsilon  \langle x\rangle $,
$\rho =\langle x\rangle ^m\bigl(\log\langle x\rangle +1\bigr) ^\alpha $,
${\rho _1= \langle x\rangle ^{2m}\bigl(\log\langle x\rangle +1\bigr)^\beta }$,
$\rho_2=1$, $\rho_3 =\bigl(\log\langle x\rangle +1\bigr) ^{-1}$.
Assume that

\medskip\noindent
\hypertarget{23-2-51-*}$\textup{(\ref*{23-2-51})}^\#$
\underline{Either} $m>-1$ and $\beta>2\alpha$ \underline{or} $m=-1$ and $\beta > \max (\alpha,2\alpha)$.
\medskip

Then all the statements of \ref{example-23-2-12-i} remain true with the obvious modification: condition (\ref{23-2-30}) should be replaced by \ref{23-2-30-*} and estimates \ref{23-2-53-3} hold for $m>-1$.
\end{enumerate}
\end{example}

Let us investigate further the case of $\beta \le 2\alpha +2$. Assume the following non-degeneracy assumption
\begin{equation}
-\langle x,\nabla \rangle VF ^{-1}\ge \epsilon |\log |x||^{2\alpha-\beta-1}
\qquad \text{for\ \ }|x|\le \epsilon.
\label{23-2-54}
\end{equation}
In fact, we need to check it for $V^*=V+\mu^{-2}\omega_1$, but this condition for $V$ and $V^*$ are equivalent as long as $\mu_\eff ^{-2}\ge C_0\rho_3$. This condition is fulfilled automatically for all $r$ provided \underline{either} $m<-1$ \underline{or} $\beta> \alpha +\frac{1}{2}$.

\medskip
Consider first the semiclassical error. Recall that the contribution of the partition element does not exceed (\ref{23-2-49}) and only summation over zone where $\rho_3 \lesssim \mu_\eff h_\eff\lesssim 1$  could bring an error, exceeding $C\mu^{-1}h^{-1}$.

Assume first that $m<-1$. Then in this zone $(\mu h)^{-1/(2m+2)}|\log \mu h|^a \lesssim r \lesssim (\mu h)^{-1/(2m+2)}|\log \mu h|^b$ and then its contribution is $C|\log (\mu h)|=O(\mu^{-1}h^{-1})$.

Consider the case $m=-1$.   In this case the problematic zone is
$\cZ\coloneqq   \{(\mu h)^{-1/(\beta -2\alpha +1)}\lesssim |\log r|\lesssim
(\mu h)^{-1/(\beta -2\alpha )}\}$. Observe that
$\int _\cZ\gamma^{-2}\,dx\asymp (\mu h)^{-1/(\beta -2\alpha )}$ which is $O(\mu^{-1}h^{-1})$ provided $\beta \ge 2\alpha+1$.

We can improve these arguments in the following way. Let us observe that the zone $\{|\log r|\le C_0(\mu h)^{-1/(\beta -2\alpha)} \}$ consists of
the \emph{spectral strips\/}
\begin{equation}
\Pi _n=\{x\colon  \bigl|v^*+(2n+1)\mu h\bigr|\le C_1|\log r|^{-1}\}
\label{23-2-55}
\end{equation}
with $v^*\coloneqq   (V+\omega_1 \mu^{-2})F^{-1}$ with
$n=0,\ldots,n_0= C_1(\mu h)^{-1/(\beta -2\alpha +1)}$, separated by the \emph{lacunary strips\/}.

Condition (\ref{23-2-54}) provides that for fixed $n$ and $x|x| ^{-1}$ (thus, only $|x|$ varies) on $\Pi _n$ the inequality
$\bigl| \mathfrak{d}\log |x|\bigr|\le C_1$ holds where $\mathfrak{d} f$ means the oscillation of $f$ (on $\Pi _n$) and the constant $C_1$ does not depend on $n,x|x| ^{-1},\mu ,h$.

Therefore one can easily see that the contribution of each strip $\Pi _n$ to the semiclassical remainder  does not exceed $C$ and their total contribution does not exceed  $Cn_0=O(\mu ^{-1}h^{-1})$.

Meanwhile, contribution of each lacunary strip does not exceed the contribution of the adjacent spectral strip, multiplied by
$C\rho_3^{-1}\times (\mu_\eff^{-1} h_\eff)^s$ and one can see easily that their total contribution is $O(\mu^{-1}h^{-1})$. Therefore

\begin{claim}\label{23-2-56}
If in the framework of Example~\ref{example-23-2-12}\ref{example-23-2-12-i} \underline{either} $m<-1$ \underline{or} $m=-1, \beta>\alpha+\frac{1}{2}$ then the semiclassical error is $O( \mu^{-1}h^{-1})$.
\end{claim}

The same arguments work for the approximation error (\ref{23-2-50}) with $N=1$: the contribution of the spectral strip gains factor $\rho_3$ and the contribution of the lacunary strip is $0$. Then the total approximation error does not exceed ``improved'' (\ref{23-2-52}):
\begin{equation}
R\coloneqq   \left\{\begin{aligned}
& (\mu h)^{-1} && \beta>2\alpha+1,\\
& (\mu h)^{-1}|\log (\mu h)| &&\beta=2\alpha+1,\\
&(\mu h)^{-1/(\beta-2\alpha)} && 2\alpha <\beta <2\alpha+1.
\end{aligned}\right.
\label{23-2-57}
\end{equation}

Then we arrive to

\begin{example}\label{example-23-2-13}
\begin{enumerate}[label=(\roman*), wide, labelindent=0pt]
\item\label{example-23-2-13-i}
Let all the assumptions of Example~\ref{example-23-2-12}\ref{example-23-2-12-i}
be fulfilled, in particular \textup{(\ref{23-2-51})}.  Moreover, let us assume that (\ref{23-2-54}) is fulfilled and \underline{either} $m<-1$ \underline{or} $\beta>\alpha+\frac{1}{2}$.
Then for $h\to+0$, $\mu$ disjoint from $0$ and $\mu h$ disjoint from infinity the asymptotics
\begin{equation}
\N ^-(\mu ,h)=\cN ^{-}(\mu ,h)+O(R)
\label{23-2-58}
\end{equation}
holds with $R$ defined by (\ref{23-2-57}).

\item\label{example-23-2-13-ii}
Let all the assumptions of Example~\ref{example-23-2-12}\ref{example-23-2-12-ii} be fulfilled.  Moreover, let us assume that condition
\begin{equation}
\langle x,\nabla \rangle VF ^{-1}\ge  \epsilon |\log x|^{2\alpha-\beta-1}
\qquad \text{for\ \ }|x|\ge c.
\tag*{$\textup{(\ref{23-2-54})}^\#$}\label{23-2-54-*}
\end{equation}
is fulfilled and  \underline{either} $m>-1$ \underline{or}
$\beta\ge \alpha+\frac{1}{2}$. Then for $h\to+0$, $\mu$ disjoint from $0$ and $\mu h$ disjoint from infinity the asymptotics (\ref{23-2-58}) holds. with $R$ defined by (\ref{23-2-57}).
\end{enumerate}
\end{example}

If we want to improve the remainder estimate in the case $\beta\le 2\alpha+1$ we should  take $N=2$, which would make our formulae more complicated but still ``computable''.  Then expression (\ref{23-2-50}) with $N=2$ acquires (in comparison with the same expression with $N=1$) the factor
$\mu^{-2}r^{-2(m+1)}|\log r|^{-2(\beta-\alpha)}$ and becomes
\begin{equation}
\mu^{-3}r^{-2(m+1)}|\log r|^{-3\beta +4\alpha +1}.
\label{23-2-59}
\end{equation}
Further, we can get gain a factor $|\log r|^{-1}$ dues to above arguments concerning  spectral and lacunary zones. Then after the summation we get $O(\mu^{-3}h^{-1})$  if \underline{either} $m<-1$ \underline{or} $m=-1$, $3\beta> 4\alpha+1$. However if $m=-1$ this is the case due to $\beta>2\alpha$ and $\beta\ge \alpha+\frac{1}{2}$. Then we arrive to

\begin{example}\label{example-23-2-14}
\begin{enumerate}[label=(\roman*), wide, labelindent=0pt]
\item\label{example-23-2-14-i}
Let all the assumptions of Example~\ref{example-23-2-12}\ref{example-23-2-12-i}
be fulfilled, in particular \textup{(\ref{23-2-51})}.  Moreover, let us assume that (\ref{23-2-54}) is fulfilled and \underline{either} $m<-1$ \underline{or} $\beta> \alpha+\frac{1}{2}$.

Then for $h\to+0$, $\mu$ disjoint from $0$ and $\mu h$ disjoint from infinity the asymptotics
\begin{gather}
\N ^-(\mu ,h)=\cN ^{-\prime}(\mu ,h)+O(\mu ^{-1}h^{-1})
\label{23-2-60}\\
\shortintertext{holds where}
\cN ^{-\prime}(\mu ,h)=\int h^{-2}\cN^{\MW\,\prime}(x,\mu h)\,dx
\label{23-2-61}
\end{gather}
and $\cN^{\MW\,\prime}$ is defined by \textup{(\ref{book_new-13-4-133}) of \cite{futurebook})} with $\psi =1$.

\item\label{example-23-2-14-ii}
Let all the assumptions of Example~\ref{example-23-2-12}\ref{example-23-2-12-ii} be fulfilled.  Moreover, let us assume that condition \ref{23-2-54-*}
is fulfilled and  \underline{either} $m>-1$ \underline{or}
$\beta> \alpha+\frac{1}{2}$. Then for $h\to+0$, $\mu$ disjoint from $0$ and $\mu h$ disjoint from infinity the asymptotics \ref{23-2-60}--\ref{23-2-61} holds.
\end{enumerate}
\end{example}

\begin{remark}\label{rem-23-2-15}
Asymptotics (\ref{23-2-60})--(\ref{23-2-61}) still holds if $m=-1$ and
$2\alpha <\beta\le \alpha+\frac{1}{2}$ (and therefore $\alpha<\frac{1}{2}$) under assumption  (\ref{23-2-62}) below.

Indeed, recall that we need $\beta>\alpha+\frac{1}{2}$ only to ensure that inequality (\ref{23-2-54}) for $V$ implies the same inequality for $V^*$. This conclusion  should be checked as $|\log r|\le C_0(\mu h)^{-1/(\beta-2\alpha)}$ only. Observe that this conclusion still holds in the case under consideration if $\mu^{-2}(\mu h) ^{(2\beta-2\alpha-1)/(\beta-2\alpha)}\le \epsilon$, i.e.
\begin{equation}
\mu\ge Ch^{(2\beta -2\alpha -1)/(1-2\alpha)}.
\label{23-2-62}
\end{equation}
Then the semiclassical error does not exceed $C\mu^{-1}h^{-1}$ and an approximation error does not exceed
$C\mu^{-3}h^{-1}(\mu h)^{-(4\alpha -3\beta+1)/(\beta-2\alpha)}$ (multiplied by $|\log (\mu h)|$ if $4\alpha -3\beta+1=0$). One can prove easily that it is less than $C\mu^{-1}h^{-1}$ under assumption (\ref{23-2-62}).
\end{remark}

We leave to the reader

\begin{Problem}\label{Problem-23-2-16}
\begin{enumerate}[label=(\roman*), wide, labelindent=0pt]
\item\label{Problem-23-2-16-i}
Consider  cases of $\mu \to \mu_0>0$ and $\mu \to \mu_0=0$ like in Example~\ref{example-23-2-6}.

\item\label{Problem-23-2-16-ii}
Consider cases of $m>-1$ and singularity at $0$, $m<-1$ and singularity at infinity, and $m=-1$, $\alpha <0$ and singularity either at $0$ or at infinity. In these three cases assumption $\beta >2\alpha$ is not necessary; therefore, one needs to consider also the case $\beta<2\alpha$.

\item\label{Problem-23-2-16-iii}
Furthermore, in the framework of \ref{Problem-23-2-16-ii} and $\beta<2\alpha$ consider
the case of $\mu h\to \infty$ like in Example~\ref{example-23-2-5}.
\end{enumerate}
\end{Problem}

\subsection{Exponential singularities}
\label{sect-23-2-1-7}

Consider now singularities of the exponential type. The following example follows immediately from Theorem~\ref{thm-23-2-2}:

\begin{example}\label{example-23-2-17}
\begin{enumerate}[label=(\roman*), wide, labelindent=0pt]
\item\label{example-23-2-17-i}
Let $0$ be an inner singular point\footref{foot-23-4} and let conditions of Theorem~\ref{thm-23-2-2} be fulfilled  with
$\gamma=\epsilon  |x|^{1-\beta}$, $\rho =\exp (a|x|^\alpha )$,
$\rho _1=\exp (b|x| ^\beta )$ with \underline{either} $0>\alpha >\beta $, $a>0$, $b>0$ \underline{or} $0>\alpha =\beta $, $b>2a>0$.

Let $V+\mu h F\ge \epsilon \rho^2$ on
$(\partial X\setminus 0)\cup \{x\colon |x|\ge c\}$. Then conditions (\ref{23-2-23}) and (\ref{23-2-24}) are fulfilled automatically and for $\mu \ge 1$, $h\to+0$ asymptotics (\ref{23-2-25})--(\ref{23-2-26}) hold.

Moreover, under condition (\ref{23-2-30}) for $\alpha =\beta $, $\mu h <t$ with a small enough constant $t>0$
\begin{equation}
\cN ^- (\mu ,h)\asymp
h^{-2}(\mu h)^{-2a/(b-2a) }|\log \mu h|^{{\frac{2}{\alpha }}-1};
\label{23-2-63}
\end{equation}
otherwise the left-hand expression is ``$O$'' only.

\item\label{example-23-2-17-ii}
Let infinity be an inner singular point\footref{foot-23-7} and let conditions of Theorem~\ref{thm-23-2-2} be fulfilled  with
$\gamma=\epsilon  \langle x\rangle ^{1-\beta }$,
$\rho =\exp (a\langle x\rangle ^\alpha )$,
$\rho_1 =\exp (b\langle x\rangle ^\beta )$ where either $0<\alpha <\beta$,
$a>0$, $b>0$ or $0<\alpha=\beta $, $b>2a>0$.

Let $V+\mu h F\ge \epsilon \rho^2$ on $\partial X$. Then conditions (\ref{23-2-23})--(\ref{23-2-24}) are fulfilled automatically and asymptotics (\ref{23-2-25})--(\ref{23-2-26}) hold.

Moreover, under condition \ref{23-2-30-*} for
$\alpha =\beta $, $\mu h <t$ with a small enough constant $t>0$
(\ref{23-2-45}) holds; otherwise the left-hand expression is ``$O$'' only.
\end{enumerate}
\end{example}

We leave to the reader the following

\begin{problem}\label{problem-23-2-18}
Calculate magnitude of $\cN^-(\mu,h)$ as $\alpha<\beta$ in Example~\ref{example-23-2-17}.
\end{problem}

\begin{example}\label{example-23-2-19}
\begin{enumerate}[label=(\roman*), wide, labelindent=0pt]
\item\label{example-23-2-19-i}
Let $0$ be an inner singular point\footref{foot-23-4}. Assume that conditions of Theorem~\ref{thm-23-2-2} are fulfilled  with
$\gamma=\epsilon  |x|^{1-\beta }$, $\rho =|x|^m\exp (|x| ^\beta )$,
$\rho_1 =|x|^{m_1}\exp (2|x| ^\beta )$, $\rho _2=1$, $\rho_3=|x|^{-\beta }$ and let $\beta <0$, $m_1<2m$.

Let $V+\mu h F\ge \epsilon \rho^2$ on
$(\partial X\setminus 0)\cup \{x\colon |x|\ge c\}$. Then conditions (\ref{23-2-24}), (\ref{23-2-48}) are fulfilled automatically.

Moreover,  condition
\begin{equation}
\int \rho_3^{-1}\rho_1^{-1}\rho^2\,dx<\infty
\label{23-2-64}
\end{equation}
is fulfilled provided $2m+3\beta>m_1$; then asymptotics (\ref{23-2-25})--(\ref{23-2-26}) holds.  Moreover,  under condition (\ref{23-2-30}) while we  cannot calculate magnitude of $\cN^-(\mu,h)$ itself, we conclude that
\begin{equation}
 \log(\cN ^-(\mu ,h))\asymp (\mu h)^{\beta /(2m-m_1)}).
\label{23-2-65}
\end{equation}
\vskip-5pt
\item\label{example-23-2-19-ii}
Let infinity be an inner singular point\footref{foot-23-7}. Assume that conditions of Theorem~\ref{thm-23-2-2} are fulfilled  with $\gamma =\epsilon _0 \langle x \rangle ^{1-\beta }$,
$\rho =\langle x\rangle ^m\exp (\langle x\rangle ^\beta )$,
$\rho_1 =\langle x\rangle ^{m_1}\exp (2\langle x\rangle ^\beta) $,
$\rho _2=1$, $\rho_3=\langle x\rangle ^{-\beta }$ where
$\beta >0$, $m_1>2m$.

Let $V+\mu h F\ge \epsilon \rho^2$ on $\partial X$. Then condition (\ref{23-2-64}) is fulfilled provided $2m+3\beta<m_1$ ; then asymptotics (\ref{23-2-25})--(\ref{23-2-26}) holds.  Moreover,  under condition \ref{23-2-30-*} estimates (\ref{23-2-65}) hold.
\end{enumerate}
\end{example}

Consider case of $2m+3\beta \le m_1$, $2m+3\beta \ge m_1$ in the frameworks of Statements~\ref{example-23-2-19-i} and~\ref{example-23-2-19-ii} respectively. One can get some remainder estimates albeit less precise. We want to improve them using the same technique as in Subsubsection~\emph{\ref{sect-23-2-1-6}.6.~\nameref{sect-23-2-1-6}}. However now things are a bit simpler. First of all, condition $\mu_\eff^{-2}\le\rho_3$ is fulfilled automatically.

\medskip
Contribution of the zone  $\cZ_1=\{r\ge  (\mu h)^{1/(-\beta-m_1+2m)}\}$ where
$\rho_3 \ge \mu _\eff h_\eff$ to the semiclassical error does not exceed
$(\mu h)^{-1}\int_{\cZ_1}  \rho_3 \rho_1^{-1}\rho^2\gamma^{-2}\, dx$. Contribution of the zone
$\cZ_2=\{(\mu h)^{1/(-m_1+2m)}\le r\le  (\mu h)^{1/(-\beta-m_1+2m)}\}$
where $\rho_3 \le \mu _\eff h_\eff \lesssim 1$ to the semiclassical error does not exceed  $\int_{\cZ_2}   \rho_3\gamma^{-2}\, dx$ due to the same ``spectal and lacunary strips'' arguments. Then the semiclassical error is does not exceed $C(\mu h)^{-1}$ provided $m_1<2m+\beta$.

\medskip
Further, contributions of these zones $\cZ_j$ to the approximation error with $N=1$ do not exceed
$C\mu^{-1}h^{-1}\int_{\cZ_j}\rho_1^{-1}\rho^2 \gamma^{-2}\, dx$ and the approximation error is $O(\mu^{-1}h^{-1})$ provided $m_1<2m+2\beta$.

On the other hand, one can see easily that the approximation error with $N=2$ is $O(\mu^{-1}h^{-1})$ for sure. We arrive to the following

\begin{example}\label{example-23-2-20}
\begin{enumerate}[label=(\roman*), wide, labelindent=0pt]
\item\label{example-23-2-20-i}
Let all the assumptions of Example~\ref{example-23-2-19}\ref{example-23-2-19-i}
be fulfilled.  Moreover, let us assume that condition
\begin{equation}
-\langle x,\nabla\rangle VF^{-1}\ge \epsilon |x|^{2m-m_1}\qquad
\text{for\ \ }|x|\le \epsilon.
\label{23-2-66}
\end{equation}
is fulfilled.

Then for $h\to+0$, $\mu$ disjoint from $0$ and $\mu h$ disjoint from infinity the asymptotics (\ref{23-2-25})--(\ref{23-2-26}) holds provided $m_1<2m+2\beta$ and  asymptotics (\ref{23-2-60})--(\ref{23-2-61}) holds with
$\cN^{\MW\,\prime}$  defined by \textup{(\ref{book_new-13-4-133}) of \cite{futurebook})} with $\psi =1$ provided $m_1<2m+\beta$.

\item\label{example-23-2-20-ii}
Let all the assumptions of Example~\ref{example-23-2-12}\ref{example-23-2-12-ii} be fulfilled.  Moreover, .  Moreover, let us assume that condition
\begin{equation}
\langle x,\nabla\rangle VF^{-1}\ge \epsilon |x|^{2m-m_1}\qquad
\text{for\ \ }|x|\ge \epsilon.
\tag*{$\textup{(\ref*{23-2-66})}^\#$}\label{23-2-66-*}
\end{equation}
is fulfilled.

Then for $h\to+0$, $\mu$ disjoint from $0$ and $\mu h$ disjoint from infinity the asymptotics (\ref{23-2-25})--(\ref{23-2-26}) holds provided $m_1>2m+2\beta$ and  asymptotics (\ref{23-2-60})--(\ref{23-2-61}) holds with
$\cN^{\MW\,\prime}$  defined by \textup{(\ref{book_new-13-4-133}) of \cite{futurebook})} with $\psi =1$
provided $m_1>2m+\beta$.
\end{enumerate}
\end{example}

\begin{example}\label{example-23-2-21}
\begin{enumerate}[label=(\roman*), wide, labelindent=0pt]
\item\label{example-23-2-21-i}
In the framework of Example~\ref{example-23-2-17}\ref{example-23-2-17-i}
assume that (\ref{23-2-66}) is fulfilled. Moreover, let
\begin{multline}
|D^\sigma g^{jk}|\le c|x|^{-|\sigma |},\qquad
|D ^\sigma \log F|\le c|x|^{\beta -|\sigma |},\\[3pt]
|D^\sigma V/F|\le c|x|^{2m-m_1-|\sigma |}\qquad
\forall \sigma :1\le |\sigma |\le 2.
\label{23-2-67}
\end{multline}
Then for $h\to +0$, $1\le \mu =O(h^{-1})$ asymptotics
\textup{(\ref{23-2-25})}--\textup{(\ref{23-2-26})} holds provided $m_1<2m+\beta$.%

Indeed,  one can easily see that under condition (\ref{23-2-67})
\begin{equation}
|\omega _1|\le C_1F^{-1}|x| ^{\beta +4m-2m_1-2}=\omega^* _1
\label{23-2-68}
\end{equation}
while the general theory yields only that
\begin{equation*}
|\omega _1|\le C_1F^{-1}|x| ^{2\beta +4m-2m_1-2}.
\end{equation*}
Then the error $|\cN ^-(\mu ,h)-\cN ^{-\prime}(\mu ,h)|$ does not exceed
\begin{multline*}
C\mu ^{-1}h^{-1}\int_{\{\rho _1\mu h\le \rho_3\rho ^2\}}
\rho _1^2\rho ^{-2}\rho_3^{-1}\omega _1^*\,dx+\\
Ch^{-2}\int_{\{\rho_3\rho ^2\le \rho _1\mu h
\le C_0\rho ^2\}}\rho _1(\omega _1^*\mu ^{-2}+\chi _\Pi )\,dx
\end{multline*}
where $\chi _\Pi $ is the characteristic function of the set
\begin{equation*}
\Pi =
\bigcup _{n\in \bZ^+}
\bigl\{x\colon \bigl|v+(2n+1)\mu h\bigr|\le C_1\mu ^{-2}\omega _1^*\bigr\}.
\end{equation*}
Applying estimate (\ref{23-2-68}) and condition (\ref{23-2-66}) one can prove that this error is $O(\mu ^{-1}h^{-1})$.

\item\label{example-23-2-21-ii}
Similarly, in the framework of Example~\ref{example-23-2-17}\ref{example-23-2-17-ii}
assume that \ref{23-2-66-*}  is fulfilled. Moreover, let
\begin{multline}
|D ^\sigma g^{jk}|\le c\langle x\rangle ^{-|\sigma |},
\qquad |D ^\sigma \log F|\le c\langle x\rangle ^{\beta -|\sigma |}\\[3pt]
|D^\sigma V/F|\le c\langle x\rangle^{2m-m_1-|\sigma |}\qquad
\forall \sigma :1\le |\sigma |\le 2.
\tag*{$\textup{(\ref*{23-2-67})}^\#$}\label{23-2-67-*}
\end{multline}
Then for $h\to +0$, $1\le \mu =O(h^{-1})$ asymptotics
\textup{(\ref{23-2-25})}--\textup{(\ref{23-2-26})} holds provided
$m_1>2m+\beta$.
\end{enumerate}
\end{example}

We leave to the reader:

\begin{Problem}\label{Problem-23-2-22}
Consider  cases of $\mu \to \mu_0>0$ and $\mu \to \mu_0=0$ like in Example~\ref{example-23-2-6}.
\end{Problem}

\section{Schr\"odinger-Pauli operators}
\label{sect-23-2-2}

Consider now {Schr\"odinger-Pauli operators, either genuine (\ref{book_new-0-34}) or generalized (\ref{book_new-13-5-3}) of \cite{futurebook}). The principal difference is that now  $F$ does not ``tame'' singularities of $V$, on the contrary, it needs to be ``tamed'' by itself. As a result there are fewer examples than for the Schr\"odinger. Also we do not have a restriction $\mu_\eff h_\eff =O(1)$ which we had in the most of the previous Subsection~\ref{sect-23-2-1}. Then we need to add $1$ and $\mu_\eff h_\eff^{-1}=\mu h^{-1}\rho_1\gamma^2$ to the contributions of this element to the remainder estimate and $\cN^-(\mu, h)$ respectively.

Because of this, here we do not consider an abstract theorem like Theorem~\ref{thm-23-2-2}, but go directly to the examples.

\begin{example-foot}\footnotetext{\label{foot-23-16} Cf. Example~\ref{example-23-2-3}.}\label{example-23-2-23}
\begin{enumerate}[label=(\roman*), wide, labelindent=0pt]
\item\label{example-23-2-23-i}
Let $0$ be an inner singular point\footref{foot-23-4}  and $\gamma=\epsilon  |x|$, $\rho =|x|^m$, $\rho _1=|x|^{m_1}$ and let $ m> -1$, $2m\ne m_1>-2$. Let our usual non-degeneracy assumptions be fulfilled. Then, in comparison with the theory of the previous Subsection~\ref{sect-23-2-1}, we need to consider also the  the zone $\mu _\eff h_\eff \gtrsim 1$.

Its contribution to the remainder does not exceed
$C\int \gamma^{-2}\,dx$ while its contribution to $\cN^-(\mu, h)$ does not exceed $C\mu h^{-1}\int \rho_1\,dx$. Indeed, contributions of each $\gamma$-element do not exceed $C$ and $C\mu_\eff h_\eff^{-1}\asymp C \mu h^{-1}\rho_1\gamma^2$.

\begin{enumerate}[label=(\alph*), wide, labelindent=0pt]
\item\label{example-23-2-23-ia}
Let $m_1>2m$; then this zone is disjoint from $0$ and this expression is $O(1)$. Furthermore, contribution of the zone $\{r\le \mu^{-1/(m_1+1-m)}\}$ where $\mu_\eff\lesssim 1$ to the remainder does not exceed
$Ch^{-1}\int \rho \gamma^{-1}\,dx$, taken over this zone, and it is $\asymp h^{-1}\mu^{-(m+1)/(m_1+1-m)}$ and we arrive to (\ref{23-2-70})--(\ref{23-2-71}) below; cf. (\ref{23-2-32}).

\item\label{example-23-2-23-ib}
Consider now the case $m_1<2m$. In this case zone where
$\mu_\eff h_\eff\gtrsim 1$ is not disjoint from $0$ and we need to consider also a singular zone $\{r\le \bar{r}\coloneqq   (\mu^{-1}h)^{-1/(m_1-2m)}\}$ where
$\mu_\eff ^{-1}h_\eff \gtrsim 1$.%

However due to the variational estimates as before
\begin{claim}\label{23-2-69}
For the disk $\{x\colon  |x|\le 2\bar{r}\}$ with the Dirichlet boundary conditions
\begin{equation*}
\N^-(\mu,h)\le
C h^{-2}\bar{r}^{2m+2}+ C (\mu h^{-1}) ^2 \bar{r}^{2(m_1+2)}
\end{equation*}
\end{claim}
which is $O(1)$ due to the choice of $\bar{r}$.
\end{enumerate}

We leave to the reader to prove (\ref{23-2-69}) and \hyperlink{23-2-69-*}{$\textup{(\ref*{23-2-69})}^\#$} below and to justify the final result using methods of Section~\ref{book_new-sect-10-1} of \cite{futurebook}: %
\begin{multline}
\N^- (\mu,h)= \\
\cN^- (\mu,h)+ \left\{\begin{aligned}
& O(h^{-1}\mu^{-1}+1) && m_1<2m,\\
& O(h^{-1}\mu^{-(m+1)/(m_1+1-m)})+1 && m_1 >2m.
\end{aligned}\right.
\label{23-2-70}
\end{multline}
with
\begin{equation}
\cN^- (\mu,h)\asymp h^{-2}+ \mu h^{-1}.
\label{23-2-71}
\end{equation}

\item\label{example-23-2-23-ii}
Let infinity be an inner singular point\footref{foot-23-7} and
$\gamma=\epsilon  \langle x \rangle$, $\rho =\langle x \rangle^m$,
$\rho _1=\langle x \rangle^{m_1}$ and let $ m< -1$, $2m\ne m_1<-2$. Let Let our usual non-degeneracy assumptions be fulfilled. Then again we need to consider cases (a) $m_1<2m$ and (b) $m_1>2m$ and in the latter case we need to consider contribution of zone $\{r\le \bar{r}\coloneqq   (\mu^{-1}h)^{-1/(m_1-2m)}\}$.

However due to variational estimates

\medskip\noindent
\hypertarget{23-2-69-*}$\textup{(\ref*{23-2-69})}^\#$
For the zone $\{x\colon  |x|\ge \frac{1}{2}\bar{r}\}$ with the Dirichlet boundary conditions
\begin{equation*}
\N^-(\mu, h) \le
C h^{-2}\bar{r}^{2m+2}+ C (\mu h^{-1}) ^2 \bar{r}^{2(m_1+2)}
\end{equation*}
which is again $O(1)$ due to the choice of $\bar{r}$.
\medskip
Then
\begin{multline}
\N^- (\mu,h)= \\
\cN^- (\mu,h)+ \left\{\begin{aligned}
& O(h^{-1}\mu^{-1}+1) && m_1>2m,\\
& O(h^{-1}\mu^{-(m+1)/(m_1+1-m)})+1 && m_1 <2m.
\end{aligned}\right.
\tag*{$\textup{(\ref*{23-2-70})}^\#$}\label{23-2-70-*}
\end{multline}
and (\ref{23-2-71}) hold.
\end{enumerate}
\end{example-foot}

We leave to the reader the following problems:

\begin{problem}\label{problem-23-2-24}
\begin{enumerate}[label=(\roman*), wide, labelindent=0pt]
\item\label{problem-23-2-24-i}
As $0$ is an inner point consider both Schr\"odinger and Schr\"odinger-Pauli operators as
\begin{enumerate}[label=(\alph*), wide, labelindent=0pt]
\item\label{problem-23-2-24-ia}
$\gamma=|x|$, $\rho= |x|^{m}$, $m>-1$ and $\rho_1=|x|^{-2}|\log |x||^\beta$\,\footnote{\label{foot-23-17} One should take $\beta<-1$ for the Schr\"odinger-Pauli operator.}.
\item\label{problem-23-2-24-ib}
$\gamma=|x|$, $\rho= |x|^{-1}|\log |x||^\alpha$, $\alpha<-\frac{1}{2}$\,\footnote{\label{foot-23-18} One needs to consider cases $\alpha<-1$, $\alpha=-1$, $-1<\alpha <-\frac{1}{2}$ separately.}  and $\rho_1=|x|^{m_1}|$, $m_1>-2$.
\end{enumerate}
\item\label{problem-23-2-24-ii}
As infinity is an inner point consider both Schr\"odinger and Schr\"odinger Pauli operators as
\begin{enumerate}[label=(\alph*), wide, labelindent=0pt]
\item\label{problem-23-2-24-iia}
$\gamma=|x|$, $\rho= |x|^{m}$, $m<-1$ and $\rho_1=|x|^{-2}|\log |x||^\beta$\,\footref{foot-23-17}.
\item\label{problem-23-2-24-iib}
$\gamma=|x|$, $\rho= |x|^{-1}|\log |x||^\alpha$, $\alpha<-\frac{1}{2}$\,\footref{foot-23-18}  and $\rho_1=|x|^{m_1}$, $m_1>-2$.
\end{enumerate}
\end{enumerate}

For the Schr\"odinger operator in cases (a) non-trivial results could be also obtained even as $\mu h\to +\infty$.
\end{problem}

\begin{problem}\label{problem-23-2-25}
Let either $0$ or infinity be an inner singular point.

Using the same arguments and combining them with the arguments of Subsubsection~\emph{\ref{sect-23-2-1-6}.6. \nameref{sect-23-2-1-6}\/} consider both Schr\"odinger-Pauli and Schr\"odinger operators with $\gamma=|x|$, $\rho=|x|^{-1}|\log |x||^\alpha$,
$\rho_1=|x|^{-2}|\log |x||^\beta$\,\footref{foot-23-17}.

Take into account whether  $\beta >2\alpha $ or $\beta< 2\alpha$.
For the Schr\"odinger operator in case $\beta<2\alpha$  non-trivial results could be obtained even as $\mu h\to +\infty$.
\end{problem}

\section{Dirac operator}
\label{sect-23-2-3}

\subsection{Preliminaries}
\label{sect-23-2-3-1}
Let us now consider the generalized magnetic Dirac operator (\ref{book_new-17-1-1}) of \cite{futurebook})
\begin{equation}
A=
\frac{1}{2}\sum_{l,j} \upsigma _l\bigl(\omega ^{jl}P_j+P_j\omega ^{jl}\bigr)
+\upsigma _0M+I\cdot V, \qquad P_j=hD_j-\mu V_j
\label{23-2-72}
\end{equation}
where $\upsigma_0,\upsigma_1,\upsigma_2$ are $2\times 2$-matrices.

We are interested in $\N(\tau_1,\tau_2)$, the number of eigenvalues in $(\tau_1,\tau_2)$\,\footnote{\label{foot-23-19} Assuming that this interval does not contain essential spectrum; otherwise $\N(\tau_1,\tau_2)\coloneqq   \infty$.  It is more convenient for us to exclude both ends of the segment.} with $\tau_1<\tau_2$, fixed in this subsection.

The theory of the Dirac operator is more complicated than the
theory of the Schr\"{o}dinger operator because it is different in the
cases when $V\pm M$\,\footnote{\label{foot-23-20} If $V\pm M \in [\tau_1+\epsilon,\tau_2-\epsilon]$ for both signs $\pm$, then infinity is not a singular point.} and $F$ tend to $0$ and (or) $\infty $ and thus singularities at $0$ and at infinity should be treated differently.

Furthermore, the theory of the magnetic Dirac operator is even more complicated. Indeed, the ``pointwise Landau levels" are
$V\pm \bigl(M^2+2j\mu hF\bigr)^{\frac{1}{2}}$ with $j=0,1,2,\ldots$ but one of those is, in fact, excepted: with the same sign as $\varsigma F_{12}$ and $j=0$. Recall that $\varsigma=\pm 1$ is defined by (\ref{book_new-17-1-14}) of \cite{futurebook})
\begin{equation}
\upsigma _0\upsigma _1 \upsigma _2=\varsigma i.
\label{23-2-73}
\end{equation}
Without any loss of generality one can assume that
\begin{equation}
\varsigma=1,\qquad  F_{12} >0.
\label{23-2-74}
\end{equation}
Then the Landau levels (at the point $x$) are
$V+\bigl(M^2+2j\mu hF\bigr)^{\frac{1}{2}}$ with $j=1,2,3,\ldots $ and
$V-\bigl(M ^2+2j\mu hF\bigr)^{\frac{1}{2}}$ with $j=0,1,2,\ldots $. But then the negative $V$ could be ``tamed'' by a larger $F$. This allows us to get  meaningful results in the situations impossible for non-magnetic Dirac operator: $V$ as singular at $0$ as $|x|^{m}$ with $m\le -1$ or $V+M$  as singular at infinity as $|x|^{2m}$ with $m\ge -1$ (as $M>0$) provided it is negative  there. Furthermore, if $F\to \infty$  as $|x|\to\infty$, we can get meaningful results even if $M=0$.

On the other hand, we often \emph{should prove\/} that the Dirac operator is
essentially self-adjoint while the Schr\"{o}dinger operator is
\emph{obviously\/} semibounded and therefore essentially self-adjoint.
We do this in Appendix~\ref{sect-23-A-3}.

Therefore, we treat the operator given by (\ref{23-2-72}) under the following assumptions\begin{phantomequation}\label{23-2-75}\end{phantomequation}
\begin{align}
&|D ^\alpha \omega ^{jk}|\le c\gamma ^{-|\alpha |},\qquad
|D ^\alpha F|\le c\rho_1 \gamma ^{-|\alpha |},
\tag*{$\textup{(\ref*{23-2-75})}_{1-2}$}\label{23-2-75-1}\\
&|D ^\alpha V|\le
c\min\bigl(\rho , \frac{1}{M} \rho ^2\bigr)\gamma ^{-|\alpha |}
\qquad (\alpha\ne 0)\quad \forall \alpha :|\alpha |\le K,
\tag*{$\textup{(\ref*{23-2-75})}_{3}$}\label{23-2-75-3}
\end{align}
\vskip-25pt
\begin{align}
&(V-\tau _2-M)_+\le c\min\bigl(\rho ,{\frac{1}{M}}\rho ^2\bigr),
\tag*{$\textup{(\ref*{23-2-75})}_{4+}$}\label{23-2-75-4+}\\
&(V-\tau _1+M)_-\le c\min\bigl(\rho , \frac{1}{M} \rho ^2\bigr)
\tag*{$\textup{(\ref*{23-2-75})}_{4-}$}\label{23-2-75-4-}
\end{align}
and also \ref{23-2-3-*}, (\ref{23-2-2}) for
$g^{jk}=\sum_{l,r}\omega ^{jl}\omega ^{kr}\updelta _{lr}$
and (\ref{23-2-9}) (with $F>0$). In what follows $\textup{(\ref{23-2-75})}_4$  means the pair of conditions $\textup{(\ref{23-2-75})}_{4\pm }$.

Moreover, let condition (\ref{23-2-13}) be fulfilled and
\begin{gather}
\bar{X}''\cap\partial X=\emptyset,
\label{23-2-76}\\
|V_j|\le c\rho ,\quad |D_j\omega ^{kl}|\le c\rho \qquad \text{in \ \ } X''.
\label{23-2-77}
\end{gather}

Finally, we assume that
\begin{claim}\label{23-2-78}
\underline{Either} $\partial X=\emptyset$
\underline{or} $\mu =O(1)$ and $\partial X\cap X'_2=\emptyset$ (in
what follows).
\end{claim}

\subsection{Asymptotics. I}
\label{sect-23-2-3-2}
\begin{example-foot}\label{example-23-2-26}\footnotetext{\label{foot-23-21} Cf. Example~\ref{example-23-2-3}.} Let condition (\ref{23-2-74})  be fulfilled.

\begin{enumerate}[label=(\roman*),wide, labelindent=0pt]
\item\label{example-23-2-26-i}
Let $0$ be an inner singular point and let all the above conditions,   be fulfilled with $\gamma =\epsilon _0 |x|$, $\rho =|x|^m$,
$\rho _1=|x|^{m_1}$, $m_1<\min(m-1,\,2m)$.

Further, let
\begin{equation}
-V\ge \epsilon \min(\rho,\,\frac{\rho^2 }{M}) \qquad \forall x:|x|\le \epsilon
\label{23-2-79}
\end{equation}
and let non-degeneracy condition
\begin{multline}
\bigl|V+\varsigma \bigl(M^2+2j\mu hF\bigr)^{\frac{1}{2}}-\tau_\iota\bigr|\le
\epsilon \rho \implies\\
\begin{aligned}
&|\nabla v_\iota|\ge \epsilon \rho _1\rho ^2\gamma ^{-1}\qquad \text{or \ \ } \quad
\det\Hess v_\iota\ge \epsilon \rho _1^{-2}\rho ^4\gamma ^{-4}\qquad \text{or \ \ } \\
&|\nabla V|\ge \epsilon \rho \gamma ^{-1}\qquad \text{and\ \ }
M\le \epsilon '\rho \quad \forall (j,\varsigma )\ne (0,1)
\end{aligned}
\label{23-2-80}
\end{multline}
be fulfilled with a small enough constant
$\epsilon '=\varepsilon (c,\epsilon,\epsilon _0)>0$; this assumption is a part of condition \textup{(\ref{23-2-80})})\,\footnote{\label{foot-23-22} Recall that $v_\iota=F ^{-1}\bigl((V-\tau _\iota )^2-M^2\bigr)$
and the conditions should be fulfilled for both $\iota =1,2$.}.

\begin{enumerate}[label=(\alph*),wide, labelindent=0pt]
\item\label{example-23-2-26-ia}
Let $m<0$, $\tau_1<\tau_2$. Then for $h\to+0$, $1\le \mu =O(h^{-1})$ asymptotics (\ref{23-2-25})--(\ref{23-2-26}) holds with $\cN$ defined by $\textup{(\ref{book_new-17-1-12})}_2$ of \cite{futurebook} with $d=2$. Moreover,
\begin{equation}
\cN  (\tau_1,\tau_2, \mu ,h)\asymp\left\{\begin{aligned}
&h^{-2}  &&m\ge  -1,\\[2pt]
&h^{-2}(\mu h)^{2(m+1)/(2m-m_1)}  &&m< -1.
\end{aligned}\right.
\label{23-2-81}
\end{equation}

\item\label{example-23-2-26-ib}
Let $m>0$, $M>0$, $\tau_1\in (-M,M)$, $\tau_2=M$. Then for $h\to+0$, $1\le \mu =O(h^{-1})$ asymptotics (\ref{23-2-25})--(\ref{23-2-26}) holds with $\cN$ defined by $\textup{(\ref{book_new-17-1-12})}_2$ of \cite{futurebook} with $d=2$. Moreover, equivalence (\ref{23-2-28}) holds.
\end{enumerate}

\item\label{example-23-2-26-ii}
Let infinity be an inner singular point  and let all the above conditions  be
fulfilled with $\gamma =\epsilon _0 \langle x\rangle $,
$\rho =\langle x\rangle ^m$, $\rho _1=\langle x\rangle ^{m_1}$,
$m_1>\max (m-1,\,2m)$.

Further, let
\begin{equation}
-V\ge \epsilon \min(\rho,\,\frac{\rho^2 }{M}) \qquad \forall x:|x|\ge c
\tag*{$\textup{(\ref*{23-2-79})}^\#$}\label{23-2-79-*}
\end{equation}
and let  non-degeneracy condition (\ref{23-2-80}) be fulfilled.

\begin{enumerate}[label=(\alph*),wide, labelindent=0pt]
\item\label{example-23-2-26-iia}
Let $m>0$, $\tau_1<\tau_2$. Then for $h\to+0$, $1\le \mu =O(h^{-1})$ asymptotics (\ref{23-2-25})--(\ref{23-2-26}) holds with $\cN$ defined by $\textup{(\ref{book_new-17-1-12})}_2$ of \cite{futurebook} with $d=2$. Moreover,
\begin{equation}
\cN ^- (\tau_1,\tau_2, \mu ,h)\asymp h^{-2}(\mu h)^{2(m+1)/(2m-m_1) }. \tag*{$\textup{(\ref*{23-2-81})}^\#$}\label{23-2-81-*}
\end{equation}

\item\label{example-23-2-26-iib}
Let $m<0$, $M>0$, $\tau_1\in (-M,M)$, $\tau_2=M$. Then for $h\to+0$, $1\le \mu =O(h^{-1})$ asymptotics (\ref{23-2-25})--(\ref{23-2-26}) holds with $\cN$ defined by $\textup{(\ref{book_new-17-1-12})}_2$ of \cite{futurebook} with $d=2$. Moreover, equivalence \ref{23-2-28-*} holds.
\end{enumerate}
\end{enumerate}
\end{example-foot}

\begin{example-foot}\label{example-23-2-27}\footnotetext{\label{foot-23-23} Cf. Examples~\ref{example-23-2-4} and~\ref{example-23-2-5}.} Let condition (\ref{23-2-74})  be fulfilled.

\begin{enumerate}[label=(\roman*),wide, labelindent=0pt]
\item\label{example-23-2-27-i}
Let $0$ be an inner singular point and let all the above conditions,   be fulfilled with $\gamma =\epsilon _0 |x|$, $\rho =|x|^m$,
$\rho _1=|x|^{m_1}$, $m_1>2m$, $m>-1$.

Further, let assumption (\ref{23-2-79}) and non-degeneracy condition
(\ref{23-2-80}) be fulfilled. Let \underline{either}
\begin{enumerate}[label=(\alph*)]
\item\label{example-23-2-27-ia}
$m<0$, $\tau_1<\tau_2$ \underline{or}
\item\label{example-23-2-27-ib}
$m>0$, $M>0$, $\tau_1\in (-M,M)$, $\tau_2=M$.
\end{enumerate}
Then for $h\to+0$, $1\le \mu $ asymptotics (\ref{23-2-32}), (\ref{23-2-26}) holds with $\cN$ defined by $\textup{(\ref{book_new-17-1-12})}_2$ of \cite{futurebook} with $d=2$.

Moreover, $\cN  (\tau_1,\tau_2, \mu ,h)\asymp h^{-2}$ as $\mu h\lesssim 1$ and (\ref{23-2-34}) holds as $\mu h\gtrsim 1$.

\item\label{example-23-2-27-ii}
Let infinity be an inner singular point  and let all the above conditions  be
fulfilled with $\gamma =\epsilon _0 \langle x\rangle $,
$\rho =\langle x\rangle ^m$, $\rho _1=\langle x\rangle ^{m_1}$,
$m_1<2m$, $m<-1$.

Further, let assumption \ref{23-2-79-*} and non-degeneracy condition
(\ref{23-2-80}) be fulfilled. Let $M>0$, $\tau_1\in (-M,M)$, $\tau_2=M$.
Then for $h\to+0$, $1\le \mu $ asymptotics \ref{23-2-32-*}, (\ref{23-2-26}) holds with $\cN$ defined by $\textup{(\ref{book_new-17-1-12})}_2$ of \cite{futurebook} with $d=2$.

Moreover, $\cN  (\tau_1,\tau_2, \mu ,h)\asymp h^{-2}$ as $\mu h\lesssim 1$ and (\ref{23-2-34}) holds as $\mu h\gtrsim 1$.
\end{enumerate}
\end{example-foot}

We leave  to the reader the following problems:

\begin{Problem}\label{Problem-23-2-28}
Modify for the Dirac operator under assumption (\ref{23-2-79}) or \ref{23-2-79-*} (as $0$ or infinity is an inner singular point, respectively) in the frameworks of
\begin{enumerate}[label=(\alph*),wide, labelindent=0pt]
\item\label{Problem-23-2-28-a}
 Examples~\ref{example-23-2-6},  \ref{example-23-2-7}, \ref{example-23-2-8} , \ref{example-23-2-9}, \ref{example-23-2-11} (power singularities) and Problem~\ref{Problem-23-2-10},
\item\label{Problem-23-2-28-b}
Examples~\ref{example-23-2-12}, \ref{example-23-2-13},
\ref{example-23-2-14}  and Problem~\ref{Problem-23-2-16}  (power-logarithmic singularities),
\item\label{Problem-23-2-28-bc}
Examples~\ref{example-23-2-17}, \ref{example-23-2-19}, \ref{example-23-2-20}, \ref{example-23-2-21}  and
Problems~\ref{problem-23-2-18}, \ref{Problem-23-2-22} (exponential singularities).
\end{enumerate}
\end{Problem}

\subsection{Asymptotics. II}
\label{sect-23-2-3-3}

While under assumptions (\ref{23-2-79}), \ref{23-2-79-*} the Dirac operator behaves as the Schr\"odinger operator, under the same assumptions albeit with an opposite sign the the Dirac operator behaves as the Schr\"odinger-Pauli operator.

\begin{example-foot}\label{example-23-2-29}\footnotetext{\label{foot-23-24} Cf. Example~\ref{example-23-2-23}.} Let condition (\ref{23-2-74})  be fulfilled.

\begin{enumerate}[label=(\roman*),wide, labelindent=0pt]
\item\label{example-23-2-29-i}
Let $0$ be an inner singular point and let all the above conditions,   be fulfilled with $\gamma =\epsilon _0 |x|$, $\rho =|x|^m$,
$\rho _1=|x|^{m_1}$, $m_1>2m$, $m>-1$.

Further, let assumption
\begin{equation}
V\ge \epsilon \min(\rho,\,\frac{\rho^2 }{M}) \qquad \forall x:|x|\le \epsilon
\label{23-2-82}
\end{equation}
and non-degeneracy condition (\ref{23-2-80}) be fulfilled. \begin{enumerate}[label=(\alph*)]
\item\label{example-23-2-29-ia}
$m<0$, $\tau_1<\tau_2$ \underline{or}
\item\label{example-23-2-29-ib}
$m>0$, $M>0$, $\tau_1=-M$, $\tau_2\in (-M,M)$.
\end{enumerate}

Then for $h\to+0$, $1\le \mu $ asymptotics (\ref{23-2-32}), (\ref{23-2-26}) holds with $\cN$ defined by $\textup{(\ref{book_new-17-1-12})}_2$ of \cite{futurebook} with $d=2$.

Moreover, $\cN  (\tau_1,\tau_2, \mu ,h)\asymp \mu h^{-1}+h^{-2}$.

\item\label{example-23-2-29-ii}
Let infinity be an inner singular point  and let all the above conditions  be
fulfilled with $\gamma =\epsilon _0 \langle x\rangle $,
$\rho =\langle x\rangle ^m$, $\rho _1=\langle x\rangle ^{m_1}$,
$m_1<2m$, $m<-1$.

Further, let assumption
\begin{equation}
V\ge \epsilon \min(\rho,\,\frac{\rho^2 }{M}) \qquad \forall x:|x|\ge c
\tag*{$\textup{(\ref*{23-2-82})}^\#$}\label{23-2-82-*}
\end{equation}
and non-degeneracy condition
(\ref{23-2-80}) be fulfilled. Let $M>0$, $\tau_1=-M$, $\tau_2\in (-M,M)$.
Then for $h\to+0$, $1\le \mu $ asymptotics \ref{23-2-32-*}, (\ref{23-2-26}) holds with $\cN$ defined by $\textup{(\ref{book_new-17-1-12})}_2$ of \cite{futurebook} with $d=2$.

Moreover, $\cN  (\tau_1,\tau_2, \mu ,h)\asymp \mu h^{-1}+h^{-2}$.
\end{enumerate}
\end{example-foot}

\begin{Problem}\label{Problem-23-2-30}
\begin{enumerate}[label=(\alph*),wide, labelindent=0pt]
\item\label{Problem-23-2-30-a}
Modify Problem~\ref{23-2-27} under assumptions (\ref{23-2-82}) or \ref{23-2-82-*} (if $0$ or infinity is an inner singular point, respectively).

\item\label{Problem-23-2-30-b}
Consider degenerations like in Subsubsection~\emph{\ref{sect-23-2-1-5}.5. \nameref{sect-23-2-1-5}\/}.
\end{enumerate}
\end{Problem}

Finally, we leave to the reader

\begin{Problem}\label{Problem-23-2-31}
Generalize results of this section to the even-dimensional full-rank case. In particular, consider power singularities\footnote{\label{foot-23-25} Note that for $m_1\ne -2$ one can construct $V_j$ positively homogeneous of degrees $m_1+1$ such that $F_{jk}$ is non-degenerate. F.e. one can take $V_{2j-1}=\frac{1}{2}x_{2j}|x|^{m_1}$, $V_{2j}=-\frac{1}{2}x_{2j-1}|x|^{m_1}$, $j=1,2,\ldots,d/2$, in which case $f_1=|(m_1+2)/2| \cdot |x|^{m_1}$, $f_2=\ldots=f_{d/2}=|x|^{m_1}$. See for details Appendix~\ref{sect-24-A-3} of \cite{futurebook}.}.
\end{Problem}

\chapter{$2\D$-case. Asymptotics of large eigenvalues}
\label{sect-23-3}

In this section we consider the case when $\mu=h=1$ are fixed and we consider the asymptotics of the eigenvalues, tending to $+ \infty$ and for Dirac operator also to $-\infty$.

\index{Schrodinger operator@Schr\"{o}dinger operator!magnetic with singularities!asymptotics of large eigenvalues}%
\index{operator!magnetic Schrodinger with singularities@magnetic Schr\"{o}dinger with singularities!asymptotics of large eigenvalues}%
\index{large eigenvalue asymptotics!magnetic Schrodinger operator with singularities@magnetic Schr\"{o}dinger with singularities}%
\index{asymptotics!large eigenvalue!magnetic Schrodinger operator with singularities@magnetic Schr\"{o}dinger with singularities}

Here we consider the case of the spectral parameter tending to $+\infty$ (and for the Dirac operator we consider $\tau\to -\infty$ as well).

\section{Singularities at the point}%
\label{sect-23-3-1}

We consider series of example with singularities at the point.

\subsection{Schr\"odinger operator}%
\label{sect-23-3-1-1}

\begin{example}\label{example-23-3-1}
\begin{enumerate}[label=(\roman*), wide, labelindent=0pt]
\item\label{example-23-3-1-i}
Let $X$ be a compact domain and conditions \ref{23-2-3-*} be fulfilled with  $\gamma=\epsilon |x|$, $\rho=|x|^m$, $\rho_1=|x|^{m_1}$ with $m_1<2m$. Let
\begin{phantomequation}\label{23-3-1}\end{phantomequation}
\begin{equation}
|F|\ge \epsilon_0 \rho_1,\qquad |\nabla F|\ge \epsilon_0 \rho_1\gamma^{-1}\qquad
\text{for\ \ } |x|\le \epsilon.
\tag*{$\textup{(\ref*{23-3-1})}_{1,2}$}\label{23-3-1-*}
\end{equation}
Then for the Schr\"odinger operator as $\tau\to +\infty$
\begin{equation}
\N^-(\tau)= \cN^-(\tau)+O(\tau^{(d-1)/2})
\label{23-3-2}
\end{equation}
while $\cN^-(\tau)\asymp \tau^{d/2}$.

Indeed, we need to consider only case $m_1\le -2$ (otherwise it is covered by Section~\ref{book_new-sect-11-2} of \cite{futurebook}). Assume for simplicity, that $V=0$ (modification in the general case is trivial). Then we need to consider zones
$\cX_1=\{x\colon |x|\ge \tau^{1/2(m_1+1)}\}$ where
$\mu_\eff= |x|^{m_1+1}\tau^{-1/2}\le 1$ and
$\cX_2=\{x\colon |x|\le \tau^{1/2(m_1+1)}\}$ where $\mu_\eff\ge 1$. Meanwhile $h_\eff= \tau^{-1/2}|x|^{-1}$.

Contribution to the remainder of the $\gamma$-element from $\cX_1$ does not exceed  $Ch_\eff^{1-d} =C \tau^{(d-1)/2}\gamma^{d-1}$ while contribution to the remainder of the $\gamma$-element from $\cX_2$ does not exceed
$C\mu_\eff^{-1}h_\eff^{1-d}=C\tau^{d/2} \gamma^{d-2-m_1}\le C \tau^{(d-1)/2}\gamma^{d-1}$ and the rest is easy.

\item\label{example-23-3-1-ii}
Under proper assumptions the same proof is valid in the full-rank even-dimensional case.

\item\label{example-23-3-1-iii}
The similar proof is valid for $d=3$ and under proper assumptions it is valid in the maximal-rank odd-dimensional case (while contribution to the remainder of $\gamma$-element from $\cX_2$ is $O(\tau^{(d-1)/2}\gamma^{d-1})$).

\item\label{example-23-3-1-iv}
On the other hand, without assumption $\textup{(\ref*{23-3-1})}_2$ contribution to the remainder $\gamma$-element from $\cX_2$ is  $O(\tau^{(d-2)/2}\gamma^{d+m_1})$ but we need to consider only $\gamma \gtrsim \tau^{1/m_1}$ (otherwise $\mu_\eff h_\eff\ge C_0$) and we arrive to the remainder estimates $O(\tau^{(d-1)/2})$ as $m_1\ge -2d$ and $O(\tau^{d/2+d/m_1})$ as $m_1\le -2d$.

Definitely these arguments are far from optimal for $d=3$ or in the maximal-rank odd-dimensional case; we leave this case to the Sections~\ref{sect-24-1}--\ref{sect-24-6}.
\end{enumerate}
\end{example}

Let us note that the case $m<-1$, $m_1\ge m-1$ is covered by Chapter~\ref{book_new-sect-11} of \cite{futurebook}; we need to assume that
\begin{equation}
V\ge \epsilon_0 \rho^2 \qquad \text{as\ \ } |x|\le \epsilon.
\label{23-3-3}
\end{equation}

\begin{example}\label{example-23-3-2}
\begin{enumerate}[label=(\roman*), wide, labelindent=0pt]
\item\label{example-23-3-2-i}
Let $X$ be a compact domain and conditions \ref{23-2-3-*}, $\textup{(\ref*{23-3-1})}_{1}$ be fulfilled with  $\gamma=\epsilon |x|$, $\rho=|x|^m$, $\rho_1=|x|^{m_1}$ and with $m<-1$, $2m\le m_1 <m-1$. Then we need to assume that
\begin{equation}
V+F\ge \epsilon_0 \rho^2\qquad \text{as\ \ } |x|\le \epsilon
\tag*{$\textup{(\ref*{23-3-3})}'$}\label{23-3-3-'}
\end{equation}
which for $m_1>2m$ is equivalent to (\ref{23-3-3}).

We need also to have some non-degeneracy assumption. Assume that\footnote{\label{foot-23-26} One can check easily that this condition or a similar condition in the multidimensional case is fulfilled provided $F$ and $V$ stabilize as $|x|\to 0$ to $V^0$ and $F^0$, positively homogeneous of degrees $2m$ and $m_1$ respectively.}
\begin{multline}
\tau \ge V+F,\qquad |\nabla (\tau - V)F^{-1}|\gamma \le \epsilon_0 \tau F^{-1}\\
\implies |\nabla ^2 (\tau - V)F^{-1}|\gamma^2 \ge \epsilon_0 \tau F^{-1}\qquad
\text{as\ \ } |x|\le \epsilon.
\label{23-3-4}
\end{multline}
Then asymptotics (\ref{23-3-2}) holds while $\cN^-(\tau)\asymp \tau^{d/2}$.

Indeed, it follows from Chapter~\ref{book_new-sect-13} of \cite{futurebook} (see condition (\ref{book_new-13-3-54})) that in this case the contribution to the remainder of the $\gamma$-element from $\cX_2$ is also $O(h_\eff^{-d+1})$.

\item\label{example-23-3-2-ii}
Under proper assumptions the same proof is valid in the maximal-rank multidimensional case.

\item\label{example-23-3-2-iii}
Without assumption (\ref{23-3-4}) we can apply arguments of Example~\ref{example-23-3-1}\ref{example-23-3-1-iv}, however cutting off as $|x|\le \tau^{1/(2m)}$ and we arrive to the remainder estimate $O(\tau^{(d-1)/2})$ as $d+m_1\ge m$ and
$O(\tau^{d(1/2+1/(2m)) +(m_1-2m)/(2m)})$ as $d+m_1\le m$.

Again, these arguments are far from optimal for $d=3$ or in the maximal-rank odd-dimensional case.; we leave this case to the Sections~\ref{sect-24-1}--\ref{sect-24-6}.
\end{enumerate}
\end{example}

\subsection{Schr\"odinger-Pauli operator}%
\label{sect-23-3-1-2}

Next, consider Schr\"odinger-Pauli operators. We will need to impose (\ref{23-3-3}) and the related non-degeneracy assumption
\begin{equation}
|\nabla V|\ge \epsilon_0 \rho^2\gamma^{-1}\qquad\text{for\ \ } |x|\le \epsilon.
\label{23-3-5}
\end{equation}

\begin{example}\label{example-23-3-3}
\begin{enumerate}[label=(\roman*), wide, labelindent=0pt]
\item\label{example-23-3-3-i}
Let $X$ be a compact domain, $d=2$. Let conditions \ref{23-2-3-*}, \ref{23-3-1-*}, (\ref{23-3-3}) and (\ref{23-3-5}) be fulfilled with
$\gamma=\epsilon $, $\rho=|x|^m$, $\rho_1=|x|^{m_1}$.

Let  \underline{either} $m_1<\min(2m,\,-2)$ \underline{or} $2m\le m_1<m-1$ and condition (\ref{23-3-4}) be fulfilled.

Then for the Schr\"odinger-Pauli  operator as  $\tau\to +\infty$ asymptotics (\ref{23-3-2}) holds while
\begin{equation}
\cN^-(\tau)\asymp \tau^{(m_1+2)/(2m)}+\tau.
\label{23-3-6}
\end{equation}

Indeed, if $m_1 \ge 2m$ then no modification to the arguments of Examples~\ref{example-23-3-1} and~\ref{example-23-3-2} is needed; if $m_1<2m$ we also need to consider the zone
$\{\tau^{1/(2m)}\lesssim |x|\lesssim \tau^{1/m_1}\}$. The contribution of the corresponding partition element to the principal part of the asymptotics is
$O(r^{m_1}\gamma^2)$ ($r=|x|\asymp \gamma$) while its contribution to the remainder is $O(\tau^{1/2}\gamma + 1)$.

\item\label{example-23-3-3-ii}
One can generalize this example to the even-dimensional full-rank case; then \begin{align}
&\cN^-(\tau)\asymp  \tau^{(m_1+2)d/(4m)}+\tau^{d/2}
\label{23-3-7}\\
\intertext{and the remainder is $O(R)$ with}
&R=\tau^{(m_1+2)(d-2)/(4m)}+\tau^{(d-1)/2}.
\label{23-3-8}
\end{align}
\end{enumerate}
\end{example}

\subsection{Dirac operator}%
\label{sect-23-3-1-3}

Finally, consider Dirac operator. We want to consider either $\N(0,\tau)$  with $\tau\to +\infty$ and  $\N(\tau,0)$ as $\tau\to -\infty$.%

\begin{example}\label{example-23-3-4}
Let $X$ be a compact domain, $d=2$ and conditions \ref{23-2-75-1}, \ref{23-3-1-*} and
\begin{equation}
|D^\alpha V|\le c\rho \gamma^{-|\alpha|}
\label{23-3-9}
\end{equation}
be fulfilled with $\gamma=\epsilon|x|$, $\rho=|x|^m$, $\rho_1=|x|^{m_1}$ and $m_1<\min (2m,-2)$. Let assumption (\ref{23-2-74}) be fulfilled as $|x|\le \epsilon$.
\begin{enumerate}[label=(\roman*), wide, labelindent=0pt]
\item\label{example-23-3-4-i}
Let  $V< M$ in the vicinity of $0$. Then for the Dirac operator asymptotics
\begin{equation}
\N (0,\tau) =\cN(0,\tau) +O(\tau )
\label{23-3-10}
\end{equation}
holds as $\tau\to +\infty$ and $\cN (0,\tau)\asymp \tau^2$.

Indeed, assumption $V<M$ guarantees that $V-(M^2+2j F)^{\frac{1}{2}}$ with $j=0,1,\ldots$ do not contribute.

\item\label{example-23-3-4-ii}
Let $V^2< M^2+F$ in the vicinity of $0$ and
\begin{phantomequation}\label{23-3-11}\end{phantomequation}
\begin{equation}
V\le -\epsilon_0\rho,\qquad |\nabla V|\ge \epsilon_0\rho\gamma^{-1}
\qquad \text{as\ \ } |x|\le \epsilon.
\tag*{$\textup{(\ref*{23-3-11})}_2$}\label{23-3-11-*}
\end{equation}
Then for the Dirac operator asymptotics
\begin{gather}
\N (\tau,0) =\cN(\tau,0) +O(|\tau|^{\frac{1}{2}})
\label{23-3-12}\\
\shortintertext{and}
\cN (\tau,0)\asymp \tau^2 +|\tau|^{(m_1+2)/m}
\label{23-3-13}
\end{gather}
holds as $\tau\to -\infty$.

Indeed, assumption $V^2< M^2+F$ guarantees that $V+(M^2+2j F)^{\frac{1}{2}}$ with $j=1,\ldots$ do not contribute.
\end{enumerate}
\end{example}

We leave to the reader

\begin{problem}\label{problem-23-3-5}
\begin{enumerate}[label=(\roman*), wide, labelindent=0pt]
\item\label{problem-23-3-5-i}
Extend results of Example~\ref{example-23-3-4}\ref{example-23-3-4-i} to the case $2m \le m_1<m-1$.
\item\label{problem-23-3-5-ii}
Expand results of Example~\ref{example-23-3-4}\ref{example-23-3-4-ii} to the case $m_1=2m<-2$.

\medskip
In both Statements~\ref{problem-23-3-5-i} and \ref{problem-23-3-5-ii} one needs to formulate the analogue of the non-degeneracy assumption (\ref{23-3-4}).

\item\label{problem-23-3-5-iii}
Consider the full-rank even-dimensional case.
\end{enumerate}
\end{problem}

\subsection{Miscellaneous singularities}%
\label{sect-23-3-1-4}

Consider now miscellaneous singularities in the point, restricting ourselves to $d=2$:

\begin{example}\label{example-23-3-6}
Let $X$ be a compact domain, $d=2$ and conditions \ref{23-2-3-*}, \ref{23-3-1-*} and (\ref{23-3-3}) be fulfilled with
$\gamma=\epsilon |x|$, $\rho=|\log |x||^\sigma$, $\rho_1=|x|^{m_1}$, $\sigma>0$, $m_1<-2$.

Then for the Schr\"odinger-Pauli operator asymptotics (\ref{23-3-2}) holds and \begin{equation}
\log (\cN^-(\tau) )\asymp \tau^{-(m_1+2)/2\sigma}.
\label{23-3-14}
\end{equation}
\end{example}

\begin{example}\label{example-23-3-7}
Let $X$ be a compact domain, $d=2$ and conditions \ref{23-2-3-*}, \ref{23-3-1-*}  be fulfilled with
$\gamma=\epsilon |x|^{1-\beta}$,  $\rho_1=\exp (b|x|^\beta)$, $\beta<0$, $b>0$ and with $\rho= \exp (a|x|^\alpha)$ where \underline{either} $\beta<\alpha<0$ \underline{or} $\beta=\alpha$ and $b>2a$.

\begin{enumerate}[label=(\roman*), wide, labelindent=0pt]
\item\label{example-23-3-7-i}
Then for Schr\"odinger operator asymptotics (\ref{23-3-2}) holds  and $\cN^-(\tau)\asymp \tau$ as $\tau\to +\infty$.

Indeed, as $\beta>-1$ it is easy, and as $\beta\le -1$ one can apply the same arguments as in Example~\ref{example-23-3-23} below (in which case Remark~\ref{rem-23-3-11}\ref{rem-23-3-11-iii} does not apply.

\item\label{example-23-3-7-ii}
Let also conditions $\textup{(\ref{23-2-3})}_{1,3}$,  (\ref{23-3-3}) and (\ref{23-3-5}) be fulfilled with $\gamma=\epsilon  |x|^{1-\alpha}$. Then for the Schr\"odinger-Pauli operator asymptotics (\ref{23-3-2}) holds with
\begin{align}
&\cN^-(\tau)\asymp \tau^{b/2a} |\log \tau|^{(2-\beta)/\beta}
&&\text{as\ \ } \alpha=\beta,
\label{23-3-15}\\
&\log (\cN^-(\tau))\asymp |\log \tau|^{\beta/\alpha}
&&\text{as\ \ }  \beta<\alpha<0.
\label{23-3-16}
\end{align}
\end{enumerate}
Indeed, in this case the forbidden zone is $\{x\colon |x|\le r^*\}$ with $r_*=\epsilon |\log \tau|^{1/\alpha}$ and contribution of the zone $\{x\colon |x|\ge r_*\}$ to the ``extra'' remainder does not exceed $\int r^{-1-2\beta}r\,dr$, taken over this zone, which is $r_*^{-2\beta}\asymp |\log \tau |^{2\beta /\alpha}$.
\end{example}

\begin{example}\label{example-23-3-8}
and conditions \ref{23-2-3-*}, \ref{23-3-1-*}  be fulfilled with
$\gamma=\epsilon |x|^{1-\beta}$,  $\rho_1=\exp (b|x|^\beta)$, $\beta<0$, $b>0$ and with $\rho=|x|^{2m}$.

\begin{enumerate}[label=(\roman*), wide, labelindent=0pt]
\item\label{example-23-3-8-i}
Then  Schr\"odinger operator is covered by the previous Example~\ref{example-23-3-7}..

\item\label{example-23-3-8-ii}
Let conditions $\textup{(\ref{23-2-3})}_{1,3}$, (\ref{23-3-3}), (\ref{23-2-23})
be fulfilled with $\gamma=\epsilon |x|$, $\rho=|x|^m$, $m<0$.

Then for the Schr\"odinger-Pauli operator the following asymptotics holds:
\begin{gather}
\N^-(\tau)= \cN^-(\tau)+ O(\tau^{1/2}+ \tau^{\beta/(2m)})
\label{23-3-17}\\
\shortintertext{holds with}
\log (\cN^-(\tau))\asymp \tau^{\beta/(2m)}.
\label{23-3-18}
\end{gather}
Indeed, in this case $r_*=\tau^{1/(2m)}$ (cf. Example~\ref{example-23-3-7}\ref{example-23-3-7-i}).
\end{enumerate}
Indeed, the contributions of the zone where $\mu_\eff h_\eff \le C$ to the main term of the asymptotics and to the remainder estimates are $\int_\cX  \rho_1 r\,dr$ and $\int_\cZ \gamma^{-2}r\,dr $ where $\cX=\{x\colon: V(x)\le \tau\}$ and  $\cZ $ is a $\gamma$-vicinity of $\partial X$; so we get (\ref{23-3-18}) (we cannot get magnitude of $\cN^-(\tau)$ itself precisely) and $\bar{r}$ with $\bar{r}=\tau^{1/(2m)}$.
\end{example}

The following problem seems to be very challenging:

\begin{Problem}\label{Problem-23-3-9}
Using the fact that singularities propagate along the drift lines, and the length of the drift line is $\asymp \bar{r}$ rather than
$\asymp \bar{\gamma}=\bar{r}^{1-\beta}$ prove that the contribution of $\cZ$ to the remainder is in fact $O(1)$ and thus improve the remainder estimate (\ref{23-3-17}) to $O(\tau^{1/2})$.
\end{Problem}

\begin{problem}\label{problem-23-3-10}
Extend results of Examples~\ref{example-23-3-6},~\ref{example-23-3-7} and~\ref{example-23-3-8} to the Dirac operator.
\end{problem}

\begin{remark}\label{rem-23-3-11}
\begin{enumerate}[label=(\roman*), wide, labelindent=0pt]
\item\label{rem-23-3-11-i}
Observe that the contribution to the remainder of the zone
$\{x\colon |x|\le \varepsilon\}$ does not exceed $\varepsilon^\sigma\tau^{(d-1)/2}$ with $\sigma>0$ in the frameworks of Example~\ref{example-23-3-1}\ref{example-23-3-1-i},
Example~\ref{example-23-3-1}\ref{example-23-3-1-iv} with $m_1>-2d$,
Example~\ref{example-23-3-2}\ref{example-23-3-2-i},
Example~\ref{example-23-3-2}\ref{example-23-3-2-iii} with $d+m_1>m$ and
Example~\ref{example-23-3-3}.

The same is correct in Example~\ref{example-23-3-6} and Examples~\ref{example-23-3-7} and~\ref{example-23-3-8} with $0>\beta>-1$.

Therefore, in these cases under the standard non-periodicity condition to the geodesic flow with reflections from $\partial X$ the asymptotics
\begin{equation}
\N (\tau )=\cN (\tau )+
\varkappa _1\tau ^{\frac{1}{2}}+ o\bigl(\tau ^{\frac{1}{2}}\bigr)
\label{23-3-19}
\end{equation}
holds with the standard coefficient $\varkappa _1$.

\item\label{rem-23-3-11-ii}
The similar statement (with $\tau$ replaced by $\tau^2$) is  true in the framework of Example~\ref{example-23-3-4}.

\item\label{rem-23-3-11-iii}
Since we used local estimates $O(h_\eff^{-1})$ rather than
$O(\mu_\eff^{-1}h_\eff^{-1})$ as $\mu_\eff\ge 1$ (the latter gave us no advantage) we do not need $0$ to be an inner singular point; the same results hold for $0\in \partial X$ under Dirichlet or Neumann boundary condition.
\end{enumerate}
\end{remark}

\begin{problem}\label{problem-23-3-12}
In the frameworks of Examples~\ref{example-23-3-1}\ref{example-23-3-1-i},  \ref{example-23-3-1}\ref{example-23-3-1-iv},
\ref{example-23-3-2}\ref{example-23-3-2-i},
\ref{example-23-3-2}\ref{example-23-3-2-iii} and  \ref{example-23-3-3}
estimate $|\cN^-(\tau)-\kappa_0 \tau^{d/2}|$.

Furthermore, in the frameworks of Examples~\ref{example-23-3-4}~\ref{example-23-3-4-i},  \ref{example-23-3-4-ii} estimate $|\cN(0,\tau)-\kappa_0 \tau^{2}|$ and
$|\cN(\tau,0)-\kappa_0 \tau^{2}|$ respectively.
\end{problem}

Finally, consider the case when the singularity is located on the curve.

\begin{example}\label{example-23-3-13}
Let $X$ be a compact domain, $d=2$ and conditions \ref{23-2-3-*}, \ref{23-3-1-*}  be fulfilled with
$\gamma=\epsilon \delta(x)$,  $\rho_1=\delta(x)^m$,
$\rho _1=\delta (x)^{m_1}$ with $m_1<\min (2m,\, -2)$  where
$\delta (x)=\dist (x,L)$, $m<  0$, $L$ is \underline{either} a closed curve ($q=1$) \underline{or} a closed set of Minkowski dimension $q<1$.

\begin{enumerate}[label=(\roman*), wide, labelindent=0pt]
\item\label{example-23-3-13-i}
Then for the Schr\"odinger operator asymptotics (\ref{23-3-2}) holds for
$\tau\to +\infty$ and $\cN^-(\tau)\asymp \tau$.

Indeed, using the same arguments as before we can get a remainder estimate $O(\tau^{1/2})$ if $q<1$ and $O(\tau^{\frac{1}{2}}|\log \tau|)$ if $q=1$ but in the latter case we can get rid off logarithm using standard propagation arguments.

\item\label{example-23-3-13-ii}
Let also conditions (\ref{23-3-3}) and (\ref{23-3-5}) be fulfilled with $m>-1$. Then for the Schr\"odinger-Pauli operator asymptotics holds
\begin{align}
&\N^- (\tau)=\cN^-(\tau)+ O(\tau^{1/2}+\tau^{(q-2)/(2m)})
\label{23-3-20}\\
\shortintertext{holds while}
&\cN^-(\tau)\asymp \tau+\tau^{(m_1+q)/(2m)}.
\label{23-3-21}
\end{align}
\end{enumerate}
Indeed, in this case a forbidden zone is $\{x\colon \delta(x)\le \delta_*= \epsilon\tau^{1/(2m)}\}$ and contributions of the zone $\{x\colon \delta(x)\ge \delta_*\}$ to the main term of the asymptotics and the remainder are $\asymp \delta_*^{m_1+q}$ and $\delta_*^{q-2}$ respectively.
\end{example}

\begin{problem}\label{problem-23-3-14}
\begin{enumerate}[label=(\roman*), wide, labelindent=0pt]
\item\label{problem-23-3-14-i}
Explore, if we can using propagation arguments 
improve remainder estimate (\ref{23-3-20}).

\item\label{problem-23-3-14-ii}
Extend results of Example~\ref{example-23-3-13} to different types of the singularities along $L$ and/or Dirac operator.
\end{enumerate}
\end{problem}

\section{Singularities at infinity}
\label{sect-23-3-2}

Let us consider \emph{unbounded domains:\/}

\subsection{Power singularities: Schr\"odinger operator}
\label{sect-23-3-2-1}
Let us  start from the power singularities.

\begin{example-foot}\footnotetext{\label{foot-23-27} Cf. Example~\ref{example-23-3-1}.}\label{example-23-3-15}
\begin{enumerate}[label=(\roman*), wide, labelindent=0pt]
\item\label{example-23-3-15-i}
Let $X$ be a connected exterior domain\footnote{\label{foot-23-28} I.e. a domain with compact complement $\complement X$ in $\bR^d$.}%
\index{exterior domain}%
\index{domain!exterior} with $\sC^K$ boundary, $d=2$. Let  conditions \textup{(\ref{23-2-2})} and  \ref{23-2-3-*} be fulfilled with
$\gamma =\epsilon _0 \langle x\rangle $, $\rho =\langle x\rangle ^m$,
$\rho _1=\langle x\rangle ^{m_1}$, $m_1>2m$. Further, let
\begin{equation}
|F|\ge \epsilon_0 \rho_1,\qquad |\nabla F|\ge \epsilon_0 \rho_1\gamma^{-1}\qquad
\text{for\ \ } |x|\ge c.
\tag*{$\textup{(\ref*{23-3-1})}^\#_{1,2}$}\label{23-3-1-**}
\end{equation}
Then for the Schr\"odinger operator the following asymptotics holds:
\begin{gather}
\N^-(\tau)=\cN^- (\tau)+ O(\tau^{(m_1+2)/2(m_1+1)})
\label{23-3-22}
\shortintertext{with}
\cN^- (\tau)\asymp \tau^{(m_1+2)/m_1}.
\label{23-3-23}
\end{gather}
Indeed, there will be zone $X'_1=\{|x|\le \tau^{1/2(m_1+1)}\}$ with
$\mu_\eff =|x|^{m_1+1}\tau^{-1/2}\lesssim 1$ and a zone
$X'_2=\{|x|\ge \tau^{1/2(m_1+1)}\}$ with $\mu_\eff \ge 1$. Contribution of the partition element in $X'_1$ to the remainder is $O(h_\eff^{-1})=O(\tau^{1/2}\gamma)$ and the total contribution of $X'_1$ is
$O(\tau^{(m_1+2)/2(m_1+1)})$. On the other hand, contribution of the partition element in $X'_2$ to the remainder is
$O(\mu_\eff^{-1} h_\eff^{-1})=O(\tau\gamma^{-m_1})$ and the total contribution of this zone is $O(\tau^{(m_1+2)/2(m_1+1)})$ again.

\item\label{example-23-3-15-ii}
Under proper assumptions the similar asymptotics holds in the full-rank even-dimensional case:
\begin{gather}
\N^-(\tau)=\cN^- (\tau)+ O(\tau^{(d-1)(m_1+2)/2(m_1+1)})
\label{23-3-24}
\shortintertext{with}
\cN^- (\tau)\asymp \tau^{d(m_1+2)/(2m_1)}.
\label{23-3-25}
\end{gather}
\end{enumerate}
\end{example-foot}

\begin{example-foot}\footnotetext{\label{foot-23-29} Cf. Example~\ref{example-23-3-2}.}\label{example-23-3-16}
\begin{enumerate}[label=(\roman*), wide, labelindent=0pt]
\item\label{example-23-3-16-i}
Let $X$ be a connected exterior domain\footref{foot-23-28} with $\sC^K$ boundary. Let  conditions  \textup{(\ref{23-2-2})},  \ref{23-2-3-*} and $\textup{(\ref{23-3-1})}^\#_{1}$ be fulfilled with $\gamma =\epsilon _0 \langle x\rangle $,
$\rho =\langle x\rangle ^m$, $\rho _1=\langle x\rangle ^{m_1}$, $m>0$,
$m-1< m_1\le 2m$.

Then we need to assume that
\begin{gather}
V+F\ge \epsilon_0 \rho^2\qquad \text{as\ \ } |x|\ge c
\tag*{$\textup{(\ref*{23-3-3})}^{\#\prime}$}\label{23-3-3-*'}\\
\intertext{which for $m_1<2m$ is equivalent to}
V \ge \epsilon_0 \rho^2\qquad \text{as\ \ } |x|\ge c
\tag*{$\textup{(\ref*{23-3-3})}^{\#}$}\label{23-3-3-*}
\end{gather}
We need also to have some non-degeneracy assumption. Assume that (cf. (\ref{23-2-16}))\footnote{\label{foot-23-30} One can see easily, that if $F$, $V$ stabilize at infinity to functions $F^0$, $V^0$, positively homogeneous of degrees $m_1\ge 0$, $2m$ respectively, then even stronger non-degeneracy assumption (cf. (\ref{23-2-14})) holds:
\begin{equation}
\tau \ge V+F\implies  |\nabla (\tau - V)F^{-1}|\gamma \ge \epsilon_0 \tau \qquad
\text{as\ \ } |x|\ge c.
\label{23-3-26}\end{equation}
On the other hand, for $m_1<0$ condition (\ref{23-3-26}) is fulfilled if $w(\theta)= F^{0\,-2m}V^{0\,m_1}$ has only nodegenerate critical points on $\bS^1$.}
\begin{multline}
\tau \ge V+F,\qquad |\nabla (\tau - V)F^{-1}|\gamma \le \epsilon_0 \tau F^{-1}\\
\implies |\det \Hess (\tau - V)F^{-1}| \gamma^2 \ge \epsilon_0 \tau F^{-1}\qquad
\text{as\ \ } |x|\ge c.
\label{23-3-27}\end{multline}
Then for the Schr\"odinger operator asymptotics
\begin{gather}
\N^-(\tau)=\cN^-(\tau) + O(R),
\label{23-3-28}\\
\shortintertext{holds with}
R=\left\{\begin{aligned}
&\tau^{(m_1+2)/2(m_1+1)} &&m_1>0,\\
&\tau|\log \tau|^2 &&m_1=0,\\
&\tau^{1-m_1/(2m)}|\log \tau| && m-1<m_1<0,
\end{aligned}\right.
\label{23-3-29}\\
\shortintertext{and}
\cN^- (\tau)\asymp \tau^{(m+1)/m}.
\label{23-3-30}
\end{gather}

Indeed, the contribution of $\gamma$-element in $X'_2$
(see Example~\ref{example-23-3-15}) to the remainder does not exceed $C\mu_\eff^{-1}h_\eff^{-1}|\log \mu_\eff|= C\tau r^{-m_1}|\log rr_*^{-1}|$,
$r_*=\tau^{1/(2(m_1+1)}$. Then summation with respect to partition returns $R$; we need to take into account that $\{|x|\ge C\tau^{1/(2m)}\}$ is a forbidden zone. Meanwhile, contribution of $X'_1$ does not exceed $C\tau^{(d-1)/2}r_*^{d-1}$.

\item\label{example-23-3-16-ii}
As $m_1<0$ we can get rid off the logarithmic factor in the remainder estimate, if we define $\cN^-(\tau)$ by the corrected magnetic Weyl formula; similarly, for $m_1=0$ we can then replace $|\log \tau|^2$ by $|\log \tau|$.

Also for $m_1=0$ we can  $|\log \tau|^2$ by $|\log \tau|$ under assumption
\begin{multline}
\tau \ge V+F,\qquad |\nabla (\tau - V)F^{-1}|\gamma \le \epsilon_0 \tau F^{-1}\\
\implies \det \Hess (\tau - V)F^{-1}  \gamma^2 \ge \epsilon_0 \tau F^{-1}\qquad
\text{as\ \ } |x|\ge c.
\label{23-3-31}
\end{multline}

\item\label{example-23-3-16-iii}
Under proper assumptions the similar asymptotics holds in the full-rank even-dimensional case:
\begin{gather}
R=\left\{\begin{aligned}
&\tau^{(d-1)(m_1+2)/2(m_1+1)} &&m_1>d-2,\\
&\tau^{d/2}|\log \tau|^2 &&m_1=d-2,\\
&\tau^{d/2+(d-2-m_1)/(2m)} |\log \tau|&& m-1<m_1<d-2,
\end{aligned}\right.
\label{23-3-32}
\shortintertext{with}
\cN^- (\tau)\asymp \tau^{d(m_1+2)/(2m_1)}.
\label{23-3-33}
\end{gather}
Moreover, we can get rid off one logarithmic factor under assumption similar to (\ref{23-3-31}).
\end{enumerate}
\end{example-foot}

Let is improve (\ref{23-3-29}) using arguments associated with long-range dynamics:

\begin{example}\label{example-23-3-17}
Let $m_1>0$ in the framework of Examples~\ref{example-23-3-15}\ref{example-23-3-15-i} or~\ref{example-23-3-16}\ref{example-23-3-16-i}.  Moreover, let the stabilization conditions
\begin{phantomequation}\label{23-3-34}\end{phantomequation}
\begin{align}
&D^\sigma (g^{jk}-g^{jk0})=o(|x|^{-|\sigma |}),
\tag*{$\textup{(\ref*{23-3-34})}_1$}\label{23-3-34-1}\\
&D ^\sigma (V_j-V^0_j)=o(|x|^{m_1+1-|\sigma |})\qquad
\forall \sigma :|\sigma |\le 1
\tag*{$\textup{(\ref*{23-3-34})}_2$}\label{23-3-34-2}
\end{align}
be fulfilled for $|x|\to \infty $ with the positively homogeneous
functions $g^{jk0},V^0_j\in \sC^K(\bR^2\setminus 0)$ of degrees $0$
and $m_1+1$ respectively.  Let the standard non-periodicity condition be fulfilled for the Hamiltonian
\begin{equation}
H^0=|x|\Bigl( \sum_{j,k} g^{jk0}(\xi _j-V^0_j)(\xi _k-V^0_k)-1\Bigr)
\label{23-3-35}
\end{equation}
on the energy level  $0$.  Then for the Schr\"odinger operator as
$\tau \to +\infty $ the following asymptotics  holds:
\begin{gather}
\N^- (\tau)=\cN^- (\tau )+
o\bigl(\tau ^{{\frac{1}{2}}(m_1+2)/(m_1+1) }\bigr).
\label{23-3-36}
\end{gather}

Indeed, let us observe that for $m_1>0$ the main contribution to the remainder estimate (\ref{23-3-29}) is given by the zone
$\{\varepsilon\le|x|\tau^{-1/2(m_1+1)}\le\varepsilon^{-1}\}$
with an arbitrarily small constant $\varepsilon>0$.  In this zone the magnetic field is normal, $\mu_\eff \lesssim 1$ (for every fixed $\varepsilon>0$) and $V\ll \tau$, $|\nabla V|\ll \tau |x|^{-1}$ Applying the improved Weyl asymptotics here we obtain (\ref{23-3-36}).
\end{example}

\begin{remark}\label{rem-23-3-18}
Similar improvements are possible in the full-rank even-dimensional case.
\end{remark}

\subsection{Power singularities: Schr\"odinger-Pauli operator}
\label{sect-23-3-2-2}

Next, consider Schr\"odinger-Pauli operators. We will need to impose \ref{23-3-3-*} and the related non-degeneracy assumption
\begin{equation}
|\nabla V|\ge \epsilon_0 \rho^2\gamma^{-1}\qquad\text{for\ \ } |x|\ge c.
\tag*{$\textup{(\ref*{23-3-5})}^\#$}\label{23-3-5-*}
\end{equation}

\begin{example-foot}\footnotetext{\label{foot-23-31} Cf. Example~\ref{example-23-3-3}.}\label{example-23-3-19}
Let \ref{23-3-3-*} and \ref{23-3-5-*} be fulfilled. Then for the Schr\"odinger-Pauli operator

\begin{enumerate}[label=(\roman*), wide, labelindent=0pt]
\item\label{example-23-3-19-i}
In the framework of Examples~\ref{example-23-3-15} and \ref{example-23-3-16} asymptotics (\ref{23-3-22}) and (\ref{23-3-29})--(\ref{23-3-29}) hold, respectively.

\item\label{example-23-3-19-ii}
In the framework of Example~\ref{example-23-3-19}\ref{example-23-3-19-i}, asymptotics (\ref{23-3-36}) holds.

\item\label{example-23-3-19-iii}
Further,
\begin{equation}
\cN^- (\tau )\asymp \tau^{(m+1)/m} + \tau ^{(m_1+2)/(2m)}.
\label{23-3-37}
\end{equation}
\item\label{example-23-3-19-iv}
Finally, under proper assumptions one can consider the full-rank even-dimensional case and prove asymptotics with the remainder estimate $O(R)$, with $R\coloneqq   R_1+R_2$ where $R_1$ is the remainder estimate for the Schr\"odinger operator,
\begin{align}
&R_2=\tau^{(m_1+2)(d-2)/(4m)}
\label{23-3-38}\\
\shortintertext{and}
&\cN^-(\tau)\asymp \tau^{d(m+1)/(2m_1)}+\tau^{d(m_1+2)/(4m)}.
\label{23-3-39}
\end{align}
\end{enumerate}
\end{example-foot}

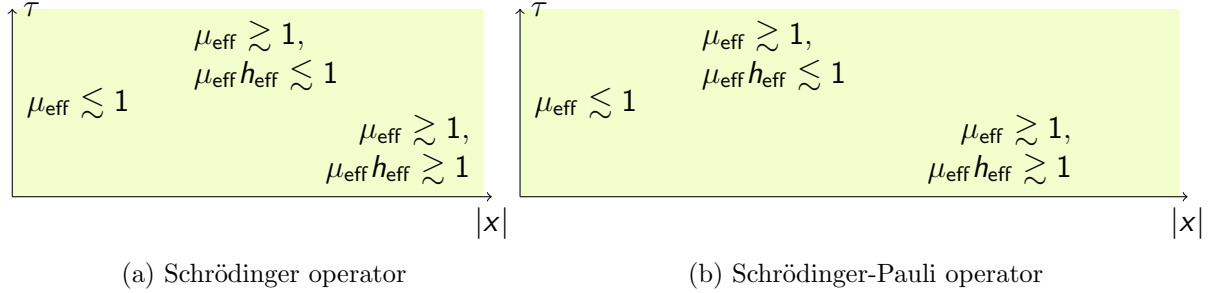
\begin{figure}
\subfloat[Schr\"odinger operator]{%
\begin{tikzpicture}[scale=1.25]
\begin{scope}
\clip(0,2) rectangle (5,4);
\fill[lime!20] (0,2) rectangle (5,4);
\fill[cyan!20, domain=0:5] plot function {1.2*x**2}--(5,2);
\fill[blue!20, domain=0:5] plot function {.45*x**1.5}--(5,4)--(5,2);
\fill[pattern=dots, pattern color=red, domain=0:5] plot function {.45*x**1.5}--(5,4)--(5,2);

\end{scope}
\draw[thin,->] (0,2)--(0,4) node[right]{$\tau$};
\draw[thin,->] (0,2)--(5.1, 2) node[below]{$|x|$};
\node at (.7,3) {$\mu_{\eff}\lesssim 1$};

\node at (2.7,3.5) {\begin{tabular}{ll}$\mu_{\eff}\gtrsim 1$,\\ $\mu_{\eff}h_{\eff}\lesssim 1$\end{tabular}};

\node at (4.1,2.5) {\begin{tabular}{rr}$\mu_{\eff}\gtrsim 1$,\\ $\mu_{\eff}h_{\eff}\gtrsim 1$\end{tabular}};

\end{tikzpicture}}
\subfloat[Schr\"odinger-Pauli operator]{%
\begin{tikzpicture}[scale=1.25]
\begin{scope}
\clip(0,2) rectangle (7,4);
\fill[lime!20] (0,2) rectangle (7,4);
\fill[cyan!20, domain=0:7] plot function {1.2*x**2}--(7,2);
\fill[blue!20, domain=0:7] plot function {.45*x**1.5}--(7,4)--(7,2);
\fill[pattern=dots, pattern color=red] plot function {.45*x}--(7,4)--(7,2);
\end{scope}
\draw[thin,->] (0,2)--(0,4) node[right]{$\tau$};
\draw[thin,->] (0,2)--(7.1, 2) node[below]{$|x|$};
\node at (.7,3) {$\mu_{\eff}\lesssim 1$};

\node at (2.7,3.5) {\begin{tabular}{ll}$\mu_{\eff}\gtrsim 1$,\\ $\mu_{\eff}h_{\eff}\lesssim 1$\end{tabular}};

\node at (5.1,2.5) {\begin{tabular}{rr}$\mu_{\eff}\gtrsim 1$,\\ $\mu_{\eff}h_{\eff}\gtrsim 1$\end{tabular}};

\end{tikzpicture}}

\caption{\label{fig-23-1} $m_1>2m$; dots show the forbidden zone}
\end{figure}

\subsection{Power singularities: Dirac operator}
\label{sect-23-3-2-3}

Finally, consider the Dirac operators. We want to explore either $\N(0,\tau)$  with $\tau\to +\infty$ and  $\N(\tau,0)$ as $\tau\to -\infty$.

\begin{example}\label{example-23-3-20}
Let $X$ be a connected exterior domain with $\sC^K$ boundary, $d=2$. Let  conditions (\ref{23-2-2}), \ref{23-2-75-1}, \ref{23-3-1-**} and (\ref{23-3-9})
be fulfilled with $\gamma =\epsilon _0 \langle x\rangle $,
$\rho =\langle x\rangle ^m$, $\rho _1=\langle x\rangle ^{m_1}$,
$m_1>\max(2m, -2)$. Further,  let assumption (\ref{23-2-74}) be fulfilled as $|x|\le \epsilon$.
\begin{enumerate}[label=(\roman*), wide, labelindent=0pt]
\item\label{example-23-3-20-i}
Let  $V< M$ in the vicinity of infinity. Then as $\tau\to +\infty$ for the Dirac operator asymptotics
\begin{gather}
\N^-(\tau)=\cN^- (\tau)+ O(\tau^{(m_1+2)/(m_1+1)})
\label{23-3-40}
\shortintertext{holds with}
\cN^- (\tau)\asymp \tau^{2(m_1+2)/m_1}.
\label{23-3-41}
\end{gather}

\item\label{example-23-3-20-ii}
Let $V^2< M^2+F$ in the vicinity of infinity and
\begin{equation}
V\le -\epsilon_0\rho,\qquad |\nabla V|\ge \epsilon_0\rho\gamma^{-1}
\qquad \text{as\ \ } |x|\ge c.
\tag*{$\textup{(\ref*{23-3-11})}^\#_2$}\label{23-3-11-**}
\end{equation}
Then as $\tau\to-\infty$ for the Dirac operator asymptotics (\ref{23-3-40}) holds with
\begin{equation}
\cN^- (\tau )\asymp \tau^{2(m+1)/m} + \tau ^{2(m_1+2)/(2m)}.
\label{23-3-42}
\end{equation}
\end{enumerate}
\end{example}

We leave to the reader

\begin{problem}\label{problem-23-3-21}
\begin{enumerate}[label=(\roman*), wide, labelindent=0pt]
\item\label{problem-23-3-21-i}
Using arguments of Example~\ref{example-23-3-16} extend results of Example~\ref{example-23-3-20}\ref{example-23-3-20-i} to the case $2m \ge m_1>m-1$.
\item\label{problem-23-3-21-ii}
Using arguments of Examples~\ref{example-23-3-16} and~\ref{example-23-3-19} expand results of Example~\ref{example-23-3-20}\ref{example-23-3-4-ii} to the case $m_1=2m<-2$.

\medskip
In both Statements formulate the analogue of the non-degeneracy assumption $\textup{(\ref{23-3-4})}^\#$.

\item\label{problem-23-3-21-iii}
Consider the full-rank even-dimensional case.
\end{enumerate}
\end{problem}

\begin{problem}\label{problem-23-3-22}
Extend  to the Dirac operator results of Example~\ref{example-23-3-17}; one still defines Hamiltonian $H^0$  by (\ref{23-3-35}).
\end{problem}

\subsection{Exponential singularities}
\label{sect-23-3-2-4}

Consider now an exponential growth at infinity.

\begin{example}\label{example-23-3-23}
Let $X$ be a connected exterior domain with $\sC^K$ boundary. Let  conditions  \textup{(\ref{23-2-2})},  \ref{23-2-3-*}, \ref{23-3-1-**} be fulfilled with
$\gamma =\epsilon _0 \langle x\rangle^{1-\beta }$,
$\rho =\exp (a \langle x\rangle ^\alpha )$,
$\rho _1=\exp (b \langle x\rangle ^\beta) $, $\beta > 0$ and \underline{either} $\beta>\alpha$ \underline{or} $\beta=\alpha$ and $b>2a>0$.

\begin{enumerate}[label=(\roman*), wide, labelindent=0pt]
\item\label{example-23-3-23-i}
Then for the Schr\"odinger operator the following asymptotics holds:
\begin{gather}
\N^- (\tau )=\cN^-(\tau )+
O\bigl(\tau ^{\frac{1}{2}}|\log \tau|^{1/\beta} \bigr)
\label{23-3-43}\\
\shortintertext{with}
\cN^- (\tau) \asymp \tau|\log\tau|^{2/\beta}.
\label{23-3-44}
\end{gather}

Indeed, using described $\gamma$, consider zone
$\{\bar{r}\le |x| \le C|\log \tau|^{1/\beta}\}$, where
$\mu_\eff= F \gamma \tau^{-1/2}\gtrsim 1$ and $F\le c\tau$; here $\bar{r}$ is defined by $\bar{r}^{1-\beta}\exp(b\bar{r}^\beta)=\tau^{1/2}$. Then contribution of $\gamma$-element to the remainder does not exceed $C\mu_\eff^{-1}h_\eff^{-1}=\tau \exp(- br^\beta)$ with $r=|x|$ and the total contribution of this zone does not exceed $CR$ with
\begin{equation*}
R=\int_{\bar{r}} \tau \exp(- br^\beta) r^{-2+2\beta} r\,dr,
\end{equation*}
which is equal to the integrand, multiplied by $r^{1-\beta}$ and calculated as $r=\bar{r}$:
$R\asymp \tau \exp(- b\bar{r}^\beta) \bar{r}^{\beta} =\tau^{1/2}\bar{r}$ with $\bar{r}\asymp |\tau|^{1/\beta}$.

On the other hand, consider zone $\{|x|\le \bar{r}\}$, where we can redefine $\gamma= \bigl(\tau^{1/2}r^{(1-\beta)\delta } \exp(-br^\beta)\bigr)^{1/(1+\delta)}$ with $\delta>0$; then its contribution to the remainder does not exceed $CR$ with
\begin{equation*}
R=\int^{\bar{r}} \tau^{1/2}
\bigl(\tau^{1/2}r^{(1-\beta)\delta } \exp(-br^\beta)\bigr)^{-1/(1+\delta)}r\,dr,
\end{equation*}
which is also equal to the integrand, multiplied by $r^{1-\beta}$ and calculated as $r=\bar{r}$:
$R\asymp \tau^{1/2} \bigl(\tau^{1/2}\bar{r}^{(1-\beta)\delta } \exp(-b\bar{r}^\beta)\bigr)^{-1/(1+\delta)}\bar{r}^{2-\beta}=\tau^{1/2}\bar{r} $ again.

\item\label{example-23-3-23-ii}
Let also conditions $\textup{(\ref{23-2-3})}_{1,3}$,  \ref{23-3-3-*} and \ref{23-3-5-*} be fulfilled with $\gamma=\epsilon  |x|^{1-\alpha}$. Then for the Schr\"odinger-Pauli operator asymptotics (\ref{23-3-43}) holds with
\begin{align}
&\cN^- (\tau) \asymp \tau^{b/(2a)}|\log\tau|^{(2-\beta)/\beta} &&\beta=\alpha,
\label{23-3-45}\\
&\log (\cN^-(\tau))\asymp |\log \tau|^{\beta/\alpha} && \beta>\alpha.
\label{23-3-46}
\end{align}
Indeed, in this case the forbidden zone is $\{x\colon |x|\ge r_*=c |\log \tau|^{1/\alpha}$ (cf. Example~\ref{example-23-3-7}~\ref{example-23-3-7-ii}).
\end{enumerate}
\end{example}

\begin{example}\label{example-23-3-24}
Let $X$ be a connected exterior domain with $\sC^K$ boundary. Let  conditions  \textup{(\ref{23-2-2})},  \ref{23-2-3-*}, \ref{23-3-1-**} be fulfilled with
$\gamma =\epsilon _0 \langle x\rangle^{1-\beta }$,
$\rho =\exp (a \langle x\rangle ^\alpha )$,
$\rho _1=\exp (b \langle x\rangle ^\beta) $, $\beta > 0$ and $\beta=\alpha$  and $2a>b >a >0$. Let conditions \ref{23-3-3-*} and \ref{23-3-5-*} be also fulfilled.

Then for both Schr\"odinger and Schr\"odinger-Pauli operators
(\ref{23-3-43}) and (\ref{23-3-44}) hold. Indeed, the forbidden zone is the same as in the previous example.
\end{example}

\begin{example}\label{example-23-3-25}
Let $X$ be a connected exterior domain with $\sC^K$ boundary. Let  conditions  \textup{(\ref{23-2-2})},  \ref{23-2-3-*}, \ref{23-3-1-**} be fulfilled with
$\gamma =\epsilon _0 \langle x\rangle^{1-\beta }$, $\rho _1=\exp (b \langle x\rangle ^\beta) $, $\beta > 0$ and  $\rho =\langle x\rangle ^m$, $m>0$.

\begin{enumerate}[label=(\roman*), wide, labelindent=0pt]
\item\label{example-23-3-25-i}
Then the Schr\"odinger operator is covered by Example~\ref{example-23-3-23-i}.

\item\label{example-23-3-25-ii}
Let also conditions $\textup{(\ref{23-2-3})}_{1,3}$,  \ref{23-3-3-*} and \ref{23-3-5-*} be fulfilled with $\gamma=\epsilon  |x|$. Then for the Schr\"odinger-Pauli operator the following asymptotics holds holds:
\begin{align}
&\N^- (\tau) = \cN^-(\tau)+ O(\tau^{1/2}|\log \tau|^{1/\beta}+\tau^{\beta/(2m)})
\label{23-3-47}
\shortintertext{and}
&\log (\cN^-(\tau))\asymp | \tau|^{\beta/(2m)}.
\label{23-3-48}
\end{align}
Indeed, in this case the forbidden zone is $\{x\colon |x|\ge r_*=c | \tau|^{1/(2m)}$ (cf. Example~\ref{example-23-3-8}~\ref{example-23-3-8-ii}).
\end{enumerate}
\end{example}

The following problem seems to be very challenging:

\begin{Problem-foot}\label{Problem-23-3-26}\footnotetext{\label{foot-23-32} Cf. Problem~\ref{Problem-23-3-9}}
Using the fact that singularities propagate along the drift lines, and the length of the drift line is $\asymp \bar{r}$ rather than
$\asymp \bar{\gamma}=\bar{r}^{1-\beta}$ prove that the contribution of $\cZ$ to the remainder is in fact $O(1)$ and thus improve the remainder estimate (\ref{23-3-17}) to $O(\tau^{1/2})$.
\end{Problem-foot}

\begin{problem}\label{problem-23-3-27}
\begin{enumerate}[label=(\roman*), wide, labelindent=0pt]
\item\label{problem-23-3-27-i}
Consider Dirac operator in the same settings as in Example~\ref{example-23-3-23}.

\item\label{problem-23-3-27-ii}
Since for Schr\"odinger, Schr\"odinger-Pauli and Dirac operators the main contribution  to the remainder is delivered by the zone where $\varepsilon \le \mu_\eff \le \varepsilon^{-1}$ (and $V\ll \tau$), while the contribution of the zones of the zones where $\mu_\eff \le \varepsilon $ and
$\mu _\eff \ge \varepsilon^{-1}$ do not exceed
$\sigma(\varepsilon)\tau |\log \tau|^{1/\beta}$  with $\sigma=o(1)$ as $\varepsilon\to 0$, derive remainder estimate $o(\tau |\log \tau|^{1/\beta})$ under non-periodicity condition for the Hamiltonian similar to (\ref{23-3-35}).

\item\label{problem-23-3-27-iii}
Consider Schr\"odinger-Pauli and Dirac operators in the same settings as in Example~\ref{23-3-11} albeit with $V$ of the logarithmic growth at infinity (i.e. with $\rho=|\log |x||^{\alpha}$, $\gamma=\epsilon |x|$).
\end{enumerate}
\end{problem}

More challenging is the following

\begin{problem}\label{problem-23-3-28}
\begin{enumerate}[label=(\roman*), wide, labelindent=0pt]
\item\label{problem-23-3-28-i}
In the frameworks of Examples~\ref{example-23-3-1},~\ref{example-23-3-2},~\ref{example-23-3-3} and~\ref{example-23-3-4} allow degenerations of $F$.
\item\label{problem-23-3-28-ii}
In the frameworks of Examples~\ref{example-23-3-2},~\ref{example-23-3-3} and~\ref{example-23-3-4} allow degenerations of $V$.
\end{enumerate}
\end{problem}

\chapter{$2\D$-case. Asymptotics of small eigenvalues}
\label{sect-23-4}

Now we consider external domains and asymptotics of eigenvalues tending to some finite limit.

\section{Operators stabilizing at infinity}
\label{sect-23-4-1}

We begin with the analysis of the Schr\"{o}dinger operator $A$ defined by (\ref{23-2-1}) under assumption (\ref{23-2-2}) assuming that
\begin{phantomequation}\label{23-4-1}\end{phantomequation}
\begin{equation}
\g \to \g_\infty, \quad \F\to \F_\infty ,\quad V\to 0
\qquad \text{as\ \ } |x|\to \infty.
\tag*{$\textup{(\ref*{23-4-1})}_{1-3}$}\label{23-4-1-*}
\end{equation}
Recall that $\F\coloneqq   (F_{jk})$ with $F_{jk}=\partial _k V_j- \partial_j V_k$,
$\g \coloneqq   (g^{jk})$.

We start from the theorem, describing the essential spectrum of $A$:

\begin{theorem}\label{thm-23-4-1}
\index{Schrodinger operator@Schr\"{o}dinger operator!magnetic!asymptotics of eigenvalues tending to point of essential spectrum}%
\index{operator!magnetic Schrodinger@magnetic Schr\"{o}dinger!asymptotics of eigenvalues tending to point of essential spectrum}%
\index{eigenvalues tending to a point of essential spectrum!asymptotics of!magnetic Schrodinger operator@magnetic Schr\"{o}dinger operator}%
\index{asymptotics!eigenvalues tending to point of essential spectrum!magnetic Schrodinger operator@magnetic Schr\"{o}dinger operator}
Let $X$ be an exterior domain\footnote{\label{foot-23-33} I.e. with a compact complement. If $X\ne \bR^d$, then the appropriate boundary condition are given on $\partial X$ such that operator is self-adjoint. In other words, infinity is an isolated singular point; see \footref{foot-23-7}.}.  with $\sC^K$ boundary. Let the Schr\"{o}dinger operator $A$ satisfy conditions \textup{(\ref{23-2-1})}, \textup{(\ref{23-2-2})}, and \ref{23-4-1-*}. Then
\begin{enumerate}[label=(\roman*), wide, labelindent=0pt]
\item\label{thm-23-4-1-i}
If $\rank \F_\infty =2r =d$ then
\begin{equation}
\Specess (A) = \{\sum_j \fz_j f_{\infty,j}\colon  \fz=(\fz_1,\ldots,\fz_r)\in (2\bZ^+ +1)^r\}
\label{23-4-2}
\end{equation}
where $\pm if_{\infty,j}$ are eigenvalues of $g_\infty\F_\infty$, $f_{\infty,j}>0$, $j=1,\ldots, r$.

\item\label{thm-23-4-1-ii}
If $\rank \F_\infty =2r <d$ then $\Specess (A)=[f^*,\infty)$ with
$f _*=f_{\infty,1}+\ldots+f_{\infty,r}$.
\end{enumerate}
\end{theorem}

\begin{proof}
Indeed, one can see easily that $\Specess (A)$ coincides with $\Spec(A_\infty)$
where $A_\infty$ is a toy-model operator in $\bR^d$ with $\g=\g_\infty$, $\F=\F_\infty$ and $V=0$. For such operator we calculated spectrum in Theorem~\ref{book_new-thm-13-1-1} of \cite{futurebook}.
\end{proof}

\begin{remark}\label{rem-23-4-2}
\begin{enumerate}[label=(\roman*), wide, labelindent=0pt]
\item\label{rem-23-4-2-i}
Similarly, for Schr\"odinger-Pauli operator $\Specess (A)$ is defined by (\ref{23-4-2}) albeit with $\fz$ running $(2\bZ^+ )^r$ if
$\rank \F_\infty =2r =d$ and $\Specess (A)=[0,\infty)$ if
$\rank \F_\infty =2r <d$.
\item\label{rem-23-4-2-ii}
Further, for the Dirac operator $\Specess (A)$ also coincides with $\Spec(A_\infty)$, calculated in Theorem~\ref{book_new-thm-17-1-2} of \cite{futurebook}.
\end{enumerate}
\end{remark}

In this section we assume that
\begin{equation}
\rank F_\infty=d;
\label{23-4-3}
\end{equation}
very different and a more complicated case of $\rank F_\infty<d$ is left for the next Sections~\ref{sect-24-1}--\ref{sect-24-6}.

According to Theorem~\ref{thm-23-4-1}\ref{thm-23-4-1-i} under assumption (\ref{23-4-3}) the essential spectrum consists of separate points, which are points of the pure point spectrum (of infinite multiplicity) of the limiting operator $A_\infty$. We are interested in the asymptotics of eigenvalues of $A$ tending to some fixed $\tau^*\in \Specess(A)$. Namely, let us introduce
\begin{phantomequation}\label{23-4-4-}\end{phantomequation}
\begin{align}
&\N^-(\eta)= \N (\tau^*-\epsilon, \tau^*- \eta)
\tag*{$\textup{(\ref*{23-4-4-})}_-$}\label{23-4-4--}\\
\shortintertext{and}
&\N^+(\eta)= \N ( \tau^*+ \eta, \tau^*+\epsilon)
\tag*{$\textup{(\ref*{23-4-4-})}_+$}\label{23-4-4-+}
\end{align}
with a small constant $\epsilon >0$ and a small parameter $\eta\to +0$. We also introduce
\begin{equation}
\fW\coloneqq   \{\fz\in (2\bZ^+ +1)^r\colon  \sum_j \fz_j f_{\infty,j} =\tau^*\}.
\label{23-4-5}
\end{equation}
To characterize the rate of the decay at infinity we assume that
\begin{phantomequation}\label{23-4-6}\end{phantomequation}
\begin{multline}
|\nabla^\alpha (\g-\g_\infty)|= o(\rho^2 \gamma^{-|\alpha|}),
\qquad
|\nabla^\alpha (\F-\F_\infty)|= o(\rho^2 \gamma^{-|\alpha|}),\\[3pt]
|\nabla^\alpha V|= O(\rho^2 \gamma^{-|\alpha|})\qquad\text{as\ \ } |x|\to \infty\quad \forall\alpha .
\tag*{$\textup{(\ref*{23-4-6})}_{1-3}$}\label{23-4-6-*}
\end{multline}

\begin{theorem}\label{thm-23-4-3}
Let $X$ be a connected exterior domain with $\sC^K$ boundary. Let the Schr\"{o}dinger operator $A$ satisfy conditions \textup{(\ref{23-2-1})}, \textup{(\ref{23-2-2})} and \ref{23-4-6-*} with scaling functions\footnote{\label{foot-23-34} Recall that this means that
$|\nabla \gamma|\le c $ and $|\nabla \rho \le c\rho\gamma^{-1}$.}
such that $\gamma \to \infty$ and $\rho\to 0$ as $|x|\to \infty$.

Let $\rank \F_\infty=2r=d$. Moreover let
\begin{phantomequation}\label{23-4-7}\end{phantomequation}
\begin{equation}
\mp V \ge -\epsilon \rho^2\implies
|\nabla V|\ge \epsilon_0 \rho^2\gamma^{-1} \qquad \text{as\ \ } |x|\ge c.
\tag*{$\textup{(\ref*{23-4-7})}_\mp$}\label{23-4-7-mp}
\end{equation}
\begin{enumerate}[label=(\roman*), wide, labelindent=0pt]
\item\label{thm-23-4-3-i}

Then%
\begin{gather}
|\N ^\mp (\eta) -\cN^\mp (\eta)|\le
C\int _{ \cZ (\eta)} \gamma^{-2}\,dx
+ C\int \gamma^{-s}\,dx
\label{23-4-8}\\
\shortintertext{where}
\cN ^\mp (\eta) \coloneqq
(2\pi)^{-r} \sum_{\fz \in \fW} \int _{\{ x\colon  \mp V_\fz(x) \ge \eta \}}
f_1f_2\cdots f_r \sqrt{g}\,dx
\label{23-4-9}\\
\intertext{$g=\det{\g}^{-1}$, $\pm i f_j$ are eigenvalues of $\g \F$, $f_j>0$, $j=1,\ldots,r$, and}
V_\fz(x) \coloneqq   V(x) + \sum_j \fz_j (f_j(x)-f_{\infty,j}),
\label{23-4-10}
\end{gather}
$\cZ(\eta)$ is $\epsilon \gamma$-vicinity\footnote{\label{foot-23-35} I.e. $\cZ(\eta)=\bigcup_{x\in \Sigma(\eta)} B(x,\epsilon\gamma(x)))$.}
of $\Sigma(\eta)=\{ x\colon \mp V_\fz(x) =\eta\}$.

\item\label{thm-23-4-3-ii}
Further, under assumption
\begin{phantomequation}\label{23-4-11}\end{phantomequation}
\begin{equation}
\mp V \ge \epsilon_0\rho^2
\tag*{$\textup{(\ref*{23-4-11})}_\mp$}\label{23-4-11-mp}
\end{equation}
$\tau^*\pm 0$ is not a limit point of the discrete spectrum.
\end{enumerate}
\end{theorem}

\begin{proof}
Indeed, in the zones $\cZ(\eta)$ and
\begin{equation}
\Omega (\eta)\coloneqq   \{x\colon  |\mp V(x)-\eta |\ge \epsilon (\rho^2 +\eta) \},
\label{23-4-12}
\end{equation}
it suffices to make $\gamma$-admissible partition of unity and observe that after rescaling $B(x,\gamma(x))\mapsto B(0,1)$ we have
$\mu\mapsto \mu_\new = \mu \gamma \rho^{-1}$,
$h\mapsto h_\new = h \gamma^{-1}\rho^{-1}$ and therefore
$\mu h \mapsto \mu h /\rho^2$,
$\mu^{-1}h \mapsto \mu^{-1}h\gamma^{-2}$ and before rescaling $\mu =h =1$.

Applying Theorem~\ref{book_new-thm-13-4-32}  for $d=2$ and Theorem~\ref{book_new-thm-19-6-25} of \cite{futurebook}
for $d\ge 4$ we estimate contribution of $\cZ(\eta)$ to the remainder by the first term in the right-hand expression of (\ref{23-4-8}).

Further, applying Theorem~\ref{book_new-thm-13-5-6}  for $d=2$ and similar results of Section~\ref{book_new-sect-19-6} of \cite{futurebook} for $d\ge 4$ case we estimate contribution of
$\Omega (\eta)\cap \{\rho^2 \ge \eta\}$ to the remainder by the second term in the right-hand expression of (\ref{23-4-8}).

In the same way we estimate contribution of
$\Omega(\eta)\cap \{\rho^2 \le \eta\}$ to the remainder by the second term in the right-hand expression of (\ref{23-4-8}) albeit now we use scale
$\mu\mapsto \mu_\new = \mu \gamma _\eta \eta^{-\frac{1}{2}}$,
$h\mapsto h_\new = h \gamma ^{-1}\eta^{-\frac{1}{2}}$.
\end{proof}

We discuss possible generalizations later; right now we want just get two simple corollaries which follow immediately from Theorem~\ref{thm-23-4-3}:

\begin{example}\label{example-23-4-4}
\begin{enumerate}[label=(\roman*), wide, labelindent=0pt]
\item\label{example-23-4-4-i}
In the framework of Theorem~\ref{thm-23-4-3} with $\gamma=\langle x\rangle$,
$\rho = \langle x\rangle^m$, $m<0$
\begin{equation}
|\N ^\mp (\eta) -\cN^\mp (\eta)|\le
C\left\{\begin{aligned}
& |\log \eta| && \text{for\ \ } d=2,\\
& \eta ^{(d-2)/(2m)} && \text{for\ \ } d\ge 4
\end{aligned}\right.
\label{23-4-13}
\end{equation}
with $\cN^\mp (\eta)=O(\eta^{d/(2m)})$. Further, $\cN^\mp (\eta)\asymp\eta^{d/(2m)}$ if condition \ref{23-4-11-mp} is fulfilled in some non-empty cone.

\item\label{example-23-4-4-ii}
Furthermore, if condition \ref{23-4-11-mp} is fulfilled, then  for $d=2$
\begin{equation}
\N ^\mp (\eta) =\cN^\mp (\eta) + O(1).
\ \label{23-4-14}
\end{equation}
\end{enumerate}
\end{example}

\begin{example}\label{example-23-4-5}
\begin{enumerate}[label=(\roman*), wide, labelindent=0pt]
\item\label{example-23-4-5-i}
In the framework of Theorem~\ref{thm-23-4-3} with
$\gamma=\langle x\rangle^{1-\sigma}$,
$\rho \le \exp(-\epsilon\langle x\rangle^\sigma)$, $0<\sigma<1$
\begin{equation}
\N ^\mp (\eta) =\cN^\mp (\eta) +O(|\log \eta|^{2+(d-2)/\sigma})
\label{23-4-15}
\end{equation}
with $\cN^\mp (\eta)=O(|\log \eta|^{d/\sigma})$. Further, $\cN^\mp (\eta)\asymp|\log \eta|^{d/\sigma}$ if condition \ref{23-4-11-mp} is fulfilled in some non-empty cone and $\rho \ge \exp(-c\langle x\rangle^\sigma)$.

\item\label{example-23-4-5-ii}
Furthermore, if condition \ref{23-4-11-mp} is fulfilled then the remainder estimate (\ref{23-4-15}) could be improved to
\begin{equation}
\N ^\mp (\eta) =\cN^\mp (\eta) +O(|\log \eta|^{2+(d-2)/\sigma}).
\label{23-4-16}
\end{equation}
\end{enumerate}
\end{example}

The following problem seems to be very challenging:

\begin{Problem-foot}\label{Problem-23-4-6}\footnotetext{\label{foot-23-36} Cf. Problem~\ref{Problem-23-3-9}.}
Let $d=2$, $\log(\pm V(x))=|x|^\beta \phi (x)$ where
$|\nabla \phi|\le C$. Then the drift line is of the length asymp $r$. Improve the remainder estimate (\ref{23-4-15})  to $O(1)$.
\end{Problem-foot}

\begin{remark}\label{rem-23-4-7}
\begin{enumerate}[label=(\roman*), wide, labelindent=0pt]
\item\label{rem-23-4-7-i}
We need conditions \ref{23-4-6-*} only for $|\alpha|\le 3$ due to Section~\ref{book_new-sect-19-6} of \cite{futurebook} and we need ``$o$" in this condition only for $|\alpha|\le 1$. Further, if $\epsilon_0$ in conditions \ref{23-4-7-mp} and \ref{23-4-11-mp} is fixed we can replace ``$=o(\rho^2\gamma^{-|\alpha|})$" by
``$\le \epsilon_1\rho^2\gamma^{-|\alpha|}$" with $\epsilon_1=\epsilon_1(\epsilon_0)$.

\item\label{rem-23-4-7-ii}
If $\#\fW=1$ we can have ``$O$'' but replace $V$ in \ref{23-4-7-mp} and \ref{23-4-11-mp} by $V_\fz$.
\end{enumerate}
\end{remark}

We leave to the reader the series of the following not challenging but interesting problems:%

\begin{Problem}\label{Problem-23-4-8}
\begin{enumerate}[label=(\roman*), wide, labelindent=0pt]
\item\label{Problem-23-4-8-i}
Consider even faster decaying
$\rho \le \exp(- |x|\gamma^{-1}(|x|))$ with monotone increasing $\gamma(t)$ such that $\gamma '(t)=o(\gamma (t)t^{-1})$ and $\gamma(t)\to \infty $ as
$t\to \infty$ and prove the remainder estimate
\begin{enumerate}[label=(\alph*), wide, labelindent=0pt]
\item\label{Problem-23-4-8-i-a}
$O(t^d \gamma(t)^{-2})$ in the general case and
\item\label{Problem-23-4-8-i-b}
$O(t^{d-1} \gamma(t)^{-1})$ under assumption \ref{23-4-11-mp}
\end{enumerate}
while $\cN^\mp (\eta)= O(t^d)$ in the general case and
$\cN^\mp (\eta)\asymp t^d$ under assumption \ref{23-4-11-mp} fulfilled some non-empty cone $\Gamma$ as $|x|\ge c$. Here $t=t(\eta)$ recovered from $t\gamma(t)^{-1}\asymp|\log \eta|$.

While proof of Theorem~\ref{thm-23-4-3} provides proper estimates of the contributions to the remainder of the zones $\cZ(\eta)$ and
$\Omega^+(\eta)\setminus \cZ(\eta)$ it fails in the zone $\Omega^-(\eta)\setminus \cZ(\eta)$ where
$\Omega^\pm (\eta)\coloneqq   \{x\colon  |V_\fz(x) |\gtrless \eta\}$. However one can use here $\gamma_\eta =\frac{1}{2}(r-r(\eta))$ instead of $\gamma$.

\item\label{Problem-23-4-8-ii}
For example, consider $\gamma(t)= (\log_{(n)} t)^\sigma$, where $\log_{(n)} t$ is $n$-tuple logarithm\footnote{\label{foot-23-37} I.e. $\log_{(1)} t=\log t$ and $\log_{(n)} t=\log (\log_{(n-1)} t)$.} with $\sigma>0$. Then
$t(\eta)= |\log \eta| |\log_{(n+1)}\eta|^\sigma$.

\item\label{Problem-23-4-8-iii}
Consider even $\exp(- c\varepsilon |x|) \le \rho\le \exp(- \varepsilon |x|)$, $\gamma=\varepsilon^{-1}$ with sufficiently small
$\varepsilon\le \varepsilon (c, \epsilon_0,\Gamma)$ and condition \ref{23-4-11-mp} fulfilled in some non-empty cone $\Gamma$. Then
\begin{equation}
\N^\mp (\eta)\asymp \varepsilon ^{-d}|\log \eta|^d.
\label{23-4-17}
\end{equation}
\end{enumerate}
\end{Problem}

\begin{remark}\label{rem-23-4-9}
Asymptotics in the case of $\rho \le \exp (-\epsilon_0|x|)$ or even compactly supported $V$ is out of reach of our methods.

Amazingly such asymptotics (without remainder estimate) were obtained in papers
M.~Melgaard and G.~Rozenblum~\cite{rozenblum:melgaard},
G.~Rozenblum and G.~Tashchiyan~\cite{rozenblioum:tash:1, rozenblioum:tash:2}
G.~Raikov and S.~Warzel~\cite{raikov:warzel} by completely different methods.
\end{remark}

\begin{Problem}\label{Problem-23-4-10}
Consider slowly decreasing potentials with $\gamma\asymp |x|$ and
$\rho = |\log_{(n)} |x||^{-\sigma}$ with $\sigma>0$.

In this case we need to replace assumptions \ref{23-4-6-*} with $|\alpha|\ge 1$ and \ref{23-4-7-mp} by
\begin{multline}
|\nabla^\alpha (\g-\g_\infty)|= o(\varrho\rho^2 \gamma^{-|\alpha|}),
\qquad
|\nabla^\alpha (\F-\F_\infty)|= o(\varrho\rho^2 \gamma^{-|\alpha|}),\\[3pt]
|\nabla^\alpha V|= O(\varrho\rho^2 \gamma^{-|\alpha|})\qquad\text{as\ \ } |x|\to \infty\quad \forall\alpha:|\alpha|\ge 1 .
\tag*{$\textup{(\ref*{23-4-6})}'_{1-3}$}\label{23-4-6-*'}
\end{multline}
and
\begin{equation}
\mp V \ge -\epsilon \rho^2\implies |\nabla V|\ge
\epsilon_0 \varrho \rho^2\gamma^{-1} \qquad \text{as\ \ } |x|\ge c.
\tag*{$\textup{(\ref*{23-4-7})}'_\mp$}\label{23-4-7-mp-'}
\end{equation}
respectively where $\varrho$ is another $\gamma$-admissible scaling function; $\varrho\le 1$.

Here again we apply Theorems~\ref{book_new-thm-13-4-32}  and~\ref{book_new-thm-13-6-6}  for $d=2$ and results of Section~\ref{book_new-sect-19-6} of \cite{futurebook} for $d\ge 4$.
\end{Problem}

The first of the following problems seems to be challenging enough while the second one is rather easy:

\begin{Problem}\label{Problem-23-4-11}
Using results of Chapter~15 of \cite{futurebook} consider $2$-dimensional domains $X$ with are $\gamma$-admissible boundaries, f.e. domains which are conical outside of the ball $B(0,c)$. Neumann boundary conditions would be especially interesting and challenging.
\end{Problem}

\begin{Problem}\label{Problem-23-4-12}
Generalize results of this subsection to genuine Schr\"odinger-Pauli and Dirac operators. While in the former case no modifications is needed (except the Landau levels), in the latter case we need to consider two cases
\begin{enumerate}[label=(\alph*),  labelindent=0pt]
\item\label{Problem-23-4-12-a}
$M^2+2j F_\infty>0$ and potential $V\sim \rho^2$ at infinity.
\item\label{Problem-23-4-12-b}
$M^2+2j F_\infty=0$ and potential $V\sim \rho$ at infinity.
\end{enumerate}
It is so because $M^2+2j F_\infty$ plays a role of the mass.
\end{Problem}

\section{Operators stabilizing at infinity. II}
\label{sect-23-4-2}

Assume now that $\g$ and $\F$ stabilize at infinity to
$\g_\infty=\g_\infty (\theta)$, $\F_\infty=\F_\infty (\theta)$, positively homogeneous of degree $0$, and $V\to 0$. Then one can see easily that the for the Schr\"odinger and Schr\"odinger-Pauli operators essential spectrum of $A$ consists of (possibly overlapping) \emph{spectral bands\/} $\Pi_\fz$
\begin{equation}
\Specess(A)=\bigcup_{\fz\in \fZ} \Pi_\fz,\qquad
\Pi_\fz \coloneqq   \bigl\{\sum_j \fz_j f_{\infty,j}(\theta),\,:
\theta \in [0,2\pi]\bigr\}.
\label{23-4-18}
\end{equation}
with the \emph{spectral gaps\/} between them.

In particular, for the Schr\"odinger operator all spectral bands in the generic case have non-zero width while for the Schr\"odinger-Pauli operator $\Pi_0=\{0\}$. Then under proper assumptions for the eigenvalues tending to $+0$ or $-0$ the results of the previous Subsection~\ref{sect-23-4-1} hold. In this subsection we are interested in the asymptotics of the eigenvalues tending to the border between a spectral gap and a spectral bund of non-zero width.

Further, for $d=2$ in the generic case there could be an infinite number of spectral gaps, but for $d\ge 4$ there is only finite number of them.

Similarly, for the Dirac operator essential spectrum consists of the spectral bunds, one of them consisting of a single point $M$ or $-M$.

\begin{theorem}\label{thm-23-4-13}
Let $S$ be a Schr\"odinger or Schr\"odinger-Pauli operator. Let conditions \ref{23-4-6-*} are fulfilled with $\g_\infty$, $\F_\infty$ positively homogeneous of degree $0$, $\gamma=\epsilon |x|$, $\rho=|x|^m$, $m<0$.
Let $x=r\theta$ with $r=|x|$, and $\theta\in \bS^{d-1}$.

Assume for simplicity\footnote{\label{foot-23-38} Otherwise we will get the sums of asymptotics.} that
\begin{claim}\label{23-4-19}
$\tau^*=\sum_j \fz_j f_{\infty,j}(\theta)$ if and only if $\fz=\bar{\fz}$ and
$\theta=\bar{\theta}$,
\end{claim}
\begin{claim}\label{23-4-20}
In the vicinity of $\bar{\theta}$ \ $f_{\infty,1}(\theta),\ldots , f_{\infty,r}(\theta)$ are disjoint,
\end{claim}
\begin{phantomequation}\label{23-4-21}\end{phantomequation}
\begin{align}
&\pm \sum_j \bar{\fz}_j f_{\infty,j} (\theta)\ge \epsilon |\theta-\bar{\theta}|^{2n}
\tag*{$\textup{(\ref*{23-4-21})}_\pm$}\label{23-4-21-pm}\\
\shortintertext{and}
&|\partial^\alpha \sum_j \bar{\fz}_j f_{\infty,j} (\theta)|\le
c_\alpha |\theta-\bar{\theta}|^{2n-|\alpha|} \qquad \forall\alpha:|\alpha|\le 2n.
\label{23-4-22}
\end{align}

Under assumption $\textup{(\ref*{23-4-21})}_+$ let $\N^-(\eta)\coloneqq  \N(\tau^*-\epsilon, \tau^*-\eta)$ and under assumption $\textup{(\ref*{23-4-21})}_-$ let $\N^+(\eta)\coloneqq  \N(\tau^*+\eta, \tau^*+\epsilon)$.
\begin{enumerate}[label=(\roman*), wide, labelindent=0pt]
\item\label{thm-23-4-13-i}
Let \begin{phantomequation}\label{23-4-23}\end{phantomequation}
\begin{equation}
\mp V(r \bar{\theta}) \ge \epsilon r^{2m},\quad
\mp \partial_r V(r \bar{\theta}) \ge \epsilon r^{2m-1}
\qquad \text{as\ \ }  r\ge c
\tag*{$\textup{(\ref*{23-4-23})}_{\mp,1,2}$}\label{23-4-23-mp}
\end{equation}
and $m+n>0$. Then as $\eta\to +0$
\begin{gather}
\N^\mp (\eta) =\cN^\mp (\eta)+ O(\eta ^{((m+n)(d-3)+n)/(2mn)}
\label{23-4-24}\\
\shortintertext{with}
\cN^\mp (\eta)=   (2\pi )^{-r}\int_{\{x\colon  \mp V_{\bar{\fz}}\ge \eta\}}f_1\cdots f_r\,dx \asymp \eta^{((m+n)(d-1)+n)/(2mn)},
\label{23-4-25}\\[3pt]
V_\fz(x) \coloneqq   V(x) + \sum_j \bar{\fz}_j (f_j(x)-f_{\infty,j}(\bar{\theta})).
\label{23-4-26}
\end{gather}

\item\label{thm-23-4-13-ii}
On the other hand, under assumption $\textup{(\ref*{23-4-23})}_\pm$
$\N^\mp (\eta)=O(1)$.
\end{enumerate}
\end{theorem}

\begin{proof}
Assume that $\bar{\theta}=(1,0,\ldots,0)$, $x=\{x_1;x')=(x_1;x_2,\ldots,x_d)$. Observe that outside of
$\cX(\eta)=\{x\colon |x'|\le c r^{1+m/n},\ 0<x_1=r\le c\eta^{1/(2m)} \}$ is a forbidden zone and one can prove easily that its contribution to the remainder is $O(1)$. On the other hand, contribution to the remainder of $\gamma'(r)$-partition element in $\cX'(\tau)$ is $O(\gamma^{d-2})$ and the total remainder does not exceed
$\int_{\{r\le c\eta^{1/(2m)} \}}\gamma^{\prime\,d-3}(r)\,dr$ which results in (\ref{23-4-24}).
\end{proof}

The following problem is rather easy:

\begin{problem}\label{problem-23-4-14}
Derive similar results for the Dirac operator.
\end{problem}

The following problem looks challenging:

\begin{Problem}\label{Problem-23-4-15}
Investigate what happens if $m+n\le 0$. Our methods provide only $\N^\mp(\eta)=O(\eta^{-1/(2n)})$. Probably methods of Section~\ref{book_new-sect-12-2} of \cite{futurebook} could provide an answer.
\end{Problem}

\section{Case \texorpdfstring{$F\to \infty$ as $|x|\to \infty$}{F\textrightarrow\textinfty\ as  |x|\textrightarrow\textinfty}}
\label{sect-23-4-3}

In this subsection we consider cases of $F\to \infty$ and  $V\to 0$ as $|x|\to \infty$.  In this case the Schr\"odinger operator does not have any essential spectrum at all and thus is not the subject of our analysis, while for the Schr\"odinger-Pauli and Dirac operators essential spectrum consists of just one point: $0$ and $\pm M$ respectively (see Theorem~\ref{book_new-thm-17-1-2} of \cite{futurebook} to find out which; if $d=2$ it is determined by signs of $F_{12}$ and $\varsigma$).
Again due to the specifics of the problem we can consider the multidimensional case without any modifications.

It turns out that for $d=2$ the remainder estimate is as in the previous Subsection~\ref{sect-23-4-1}, while the magnitude of the principal part is larger but it is still given by the same formula).

Let for the Schr\"odinger-Pauli operator $\N^- (\eta)$ be a number of eigenvalues in $(-\epsilon, -\eta)$ and $\N^+ (\eta)$ be a number of eigenvalues in $(\eta, \epsilon)$.

\begin{theorem-foot}\footnotetext{\label{foot-23-39} Cf. Theorem~\ref{thm-23-4-3}.}\label{thm-23-4-16}
\begin{enumerate}[label=(\roman*), wide, labelindent=0pt]
\item\label{thm-23-4-16-i}
Let $X$ be a connected exterior domain with $\sC^K$ boundary. Let  conditions \textup{(\ref{23-2-2})}, \ref{23-2-3-*}  and \ref{23-4-7-mp} be fulfilled with scaling functions $\gamma$, $\rho$ and $\rho_1$, $\rho\to 0$, $\rho_1\to \infty$ and $\rho_1\gamma^2\to \infty$ as $|x|\to \infty$. Assume that
\begin{equation}
|\F ^{-1}|\le c\rho_1^{-1}
\qquad \text{for\ \ }|x|\ge c
\label{23-4-27}
\end{equation}
and
\begin{claim}\label{23-4-28}
For each $j\ne k$ \underline{either} $f_j=f_k$ \underline{or}
$|f_j-f_k|\ge \epsilon \rho_1$  for all $|x|\ge c$.
\end{claim}

Then for the Schr\"odinger-Pauli operator
\begin{gather}
\N^\mp (\eta)=\cN^\mp (\eta)+ O(R)
\label{23-4-29}\\
\shortintertext{where}
\cN ^\mp (\eta) \coloneqq
(2\pi)^{-r} \int _{\{ x\colon  \mp V (x) \ge \eta \}}
f_1f_2\cdots f_r \sqrt{g}\,dx
\label{23-4-30}\\
R=C\int _{ \cZ (\eta)} \rho_1^{r-1} \gamma^{-2}\,dx
+ C\int \rho_1^{r-s} \gamma^{-2s}\,dx
\label{23-4-31}
\end{gather}
holds, $r=d/2$, $\cZ(\eta)$ is $\epsilon \gamma$-vicinity\footref{foot-23-35}
of $\Sigma(\eta)=\{ x\colon \mp V (x) =\eta\}$.

\item\label{thm-23-4-16-ii}
Further, under assumption \ref{23-4-11-mp} $0\pm 0$ is not a limit point of the discrete spectrum.
\end{enumerate}
\end{theorem-foot}

\begin{example-foot}\footnotetext{\label{foot-23-40} Cf. Example~\ref{example-23-4-4}.}\label{example-23-4-17}
\begin{enumerate}[label=(\roman*), wide, labelindent=0pt]
\item\label{example-23-4-17-i}
In the framework of Theorem~\ref{thm-23-4-16} with $\gamma=\langle x\rangle$,
$\rho = \langle x\rangle^m$, $\rho_1 = \langle x\rangle^{m_1}$, $m<0<m_1$ estimate (\ref{23-4-29}) holds with
\begin{equation}
R= \left\{\begin{aligned}
& |\log \eta| &&&& \text{for\ \ } d=2,\\
& \eta ^{(d-2)k/(2m)} \quad &&k=(m_1+2)/2 && \text{for\ \ } d\ge 4
\end{aligned}\right.
\label{23-4-32}
\end{equation}
and with $\cN^\mp (\eta)=O(\eta^{(dk/(2m)})$. Further, $\cN^\mp (\eta)\asymp\eta^{dk/(2m)}$ if condition \ref{23-4-11-mp} is fulfilled in some non-empty cone.

\item\label{example-23-4-17-ii}
Furthermore, if condition \ref{23-4-11-mp} is fulfilled, then even for $d=2$
\ $R=1$.
\end{enumerate}
\end{example-foot}

\begin{example}\label{example-23-4-18}
\begin{enumerate}[label=(\roman*), wide, labelindent=0pt]
\item\label{example-23-4-18-i}
Let  conditions \textup{(\ref{23-2-2})}, $\textup{(\ref{23-2-3})}_{2}$, (\ref{23-4-24}), (\ref{23-4-28}) be fulfilled with
$\gamma=\langle x\rangle^{1-\beta}$, $\rho = \langle x\rangle^m$,
$\rho_1 = \exp( b\langle x\rangle^{\beta})$, $m<0$, $\beta>0$.

Further, let  conditions $\textup{(\ref{23-2-3})}_{1,3}$ and \ref{23-4-7-mp} be fulfilled with $\rho=\langle x\rangle^m$, $\gamma=\epsilon \langle x\rangle$.
Then the following asymptotics holds:
\begin{align}
&\N^\mp (\eta)=\cN^\mp (\eta)+O(\eta ^{(d-2+2\beta)/(2m)})
\label{23-4-33}\\
\shortintertext{and}
&\log (\cN^\pm (\eta))=O(\eta^{\beta/(2m)}).
\label{23-4-34}
\end{align}

\item\label{example-23-4-18-ii}
Let conditions \textup{(\ref{23-2-2})}, \ref{23-2-3-*} and \ref{23-4-7-mp} be fulfilled with $\gamma=\langle x\rangle^{1-\alpha}$,
$\rho =   \exp( a\langle x\rangle^{\alpha})$,
$\rho_1 =\langle x\rangle^{m_1}$, $a<0$, $\alpha>0$, $m_1>0$, $m_1+2(1-\alpha)>0$. Then following asymptotics holds:
\begin{align}
&\N^\mp (\eta)=\cN^\mp(\eta)+O(|\log \eta |^{(d-2+2\alpha)/\alpha})
\label{23-4-35}\\
\shortintertext{and}
&\cN^\pm (\eta)=O(|\log \eta|^{(d+m_1)/\alpha}).
\label{23-4-36}
\end{align}

\item\label{example-23-4-18-iii}
Moreover, if condition \ref{23-4-11-mp} is fulfilled, then the remainder estimate (\ref{23-4-35}) could be improved to
\begin{align}
&\N^\mp (\eta)=\cN^\mp(\eta)+O(|\log \eta |^{(d-2+\alpha)/\alpha})
\label{23-4B-36}
\end{align}

\item\label{example-23-4-18-iv}
Let conditions \textup{(\ref{23-2-2})} and \ref{23-2-3-*}  be fulfilled with $\gamma=\langle x\rangle^{1-\sigma}$,
$\rho =   \exp( a\langle x\rangle^{\alpha})$,
$\rho_1 = \exp( b\langle x\rangle^{\beta})$, $a<0<b$, $\alpha>0$, $\beta>0$, $\sigma=\max(\alpha,\,\beta)$.

Further, let  conditions $\textup{(\ref{23-2-3})}_{1,3}$ and \ref{23-4-7-mp} be fulfilled with $\gamma=\langle x\rangle^{1-\alpha}$,
$\rho =   \exp( a\langle x\rangle^{\alpha})$. Then the following asymptotics holds:
\begin{align}
&\N^\pm(\eta)=\cN^\mp(\eta)+O(|\log \eta |^{(d-2+2\sigma)/\alpha})
\label{23-4-38}\\
\shortintertext{and}
&\log (\cN^\pm (\eta))=O( |\log \eta|^{\beta/\alpha}).
\label{23-4-39}
\end{align}

\item\label{example-23-4-18-v}
Moreover, if condition \ref{23-4-11-mp} is fulfilled, then the remainder estimate (\ref{23-4-38}) could be improved to
\begin{align}
&\N^\mp (\eta)=\cN^\mp(\eta)+O(|\log \eta |^{(d-2+2\sigma-\alpha)/\alpha})
\label{23-4-40}
\end{align}

\item\label{example-23-4-18-vi}
 Furthermore, if condition \ref{23-4-11-mp} is fulfilled in some non-empty cone then there is ``$\asymp \cdot$'' rather than ``$=O(.)$'' in (\ref{23-4-34}), (\ref{23-4-36}) and (\ref{23-4-39}).
 \end{enumerate}
\end{example}

\begin{Problem}\label{Problem-23-4-19}
Again, one can hope to improve estimates (\ref{23-4B-36}) and (\ref{23-4-40}) in the same way as specified in Problem~\ref{Problem-23-4-6}.
\end{Problem}

We leave to the reader

\begin{problem-foot}\footnotetext{\label{foot-23-41} Cf. Problem~\ref{Problem-23-4-12}.} \label{problem-23-4-20}
Consider the Dirac operator. In this case $\N^-(\eta)$ is a number of eigenvalues in
$(\pm M-\epsilon, \pm M-\eta)$ and $\N^+(\eta)$ is a number of eigenvalues in
$(\pm M+\eta, \pm M+\epsilon)$, $0<\eta<\epsilon$ and $\pm M$ is a point of the essential spectrum.

We need to distinguish two cases
\begin{enumerate}[label=(\alph*), labelindent=0pt]
\item\label{problem-23-4-20-a}
$M>0$ and potential $V\sim \rho^2$ at infinity.
\item\label{problem-23-4-20-b}
$M=0$ and potential $V\sim \rho$ at infinity.
\end{enumerate}
\end{problem-foot}

\section{Case \texorpdfstring{$F\to 0$ as $|x|\to \infty$}{F\textrightarrow0\ as  |x|\textrightarrow\textinfty}}
\label{sect-23-4-4}

In this subsection we consider cases of $F\to 0$ and  $V\to 0$ as $|x|\to \infty$. In this case the essential spectra of the Schr\"odinger and Schr\"odinger-Pauli operators are $[0,\infty)$; however, as $V=o(F)$ as $|x|\to \infty$ the Schr\"odinger operator has only a finite number of the negative eigenvalues and thus is not a subject of our analysis while the Schr\"odinger-Pauli operator is.

Further, the Dirac operator has its essential spectrum $(-\infty,-M]\cup[M,\infty)$ and we need to assume that $M>0$.

It turns out that  the remainder estimate is as in Subsection~\ref{sect-23-4-1}, while the magnitude of the principal part is smaller but it is still given by the same formula).

\begin{theorem-foot}\footnotetext{\label{foot-23-42} Cf. Theorems~\ref{thm-23-4-3} and~\ref{thm-23-4-16}.}\label{thm-23-4-21}
\begin{enumerate}[label=(\roman*), wide, labelindent=0pt]
\item\label{thm-23-4-21-i}
Let $X$ be a connected exterior domain with $\sC^K$ boundary. Let  conditions \textup{(\ref{23-2-2})}, \ref{23-2-3-*}, \textup{(\ref{23-2-19})}, \textup{(\ref{23-4-24})} and  \textup{(\ref{23-4-28})}  and \ref{23-4-7-mp} (with sign ``$-$'') be fulfilled with scaling functions $\gamma$, $\rho$ and $\rho_1$, $\rho\to 0$, $\rho_1\to 0$ and $\rho_1\gamma^2\to \infty$, $\rho_1\rho^{-2}\to \infty$ as
$|x|\to \infty$.
Then for the Schr\"odinger-Pauli operator \textup{(\ref{23-4-29})}--\textup{(\ref{23-4-32})} holds\footnote{\label{foot-23-43} With the sign ``$-$''.}.

\item\label{thm-23-4-21-ii}
Further, under assumption \ref{23-4-11-mp} (with sign ``$-$'')  $0- 0$ is not a limit point of the discrete spectrum.
\end{enumerate}
\end{theorem-foot}

\begin{example-foot}\footnotetext{\label{foot-23-44} Cf. Examples~\ref{example-23-4-4} and \ref{example-23-4-17}.}\label{example-23-4-22}
\begin{enumerate}[label=(\roman*), wide, labelindent=0pt]
\item\label{example-23-4-22-i}
Let conditions of Theorem~\ref{thm-23-4-21} be fulfilled with
$\gamma=\langle x\rangle$, $\rho = \langle x\rangle^m$, $\rho_1 = \langle x\rangle^{m_1}$, $m<0$, $\max(2m,\,-2)<m_1<0$.

Then estimate (\ref{23-4-29})\,\footref{foot-23-43} holds with $R$ defined by (\ref{23-4-32}) and  with $\cN^-(\eta)=O(\eta^{dk/(2m)})$. Further,
$\cN^- (\eta)\asymp\eta^{dk/(2m)}$ if condition $\textup{(\ref{23-4-11})}_-$ is fulfilled in some non-empty cone.

\item\label{example-23-4-22-ii}
Furthermore, if condition $\textup{(\ref{23-4-11})}_-$ is fulfilled, then even for $d=2$ \ $R=1$.
\end{enumerate}
\end{example-foot}

\begin{example}\label{example-23-4-23}
\begin{enumerate}[label=(\roman*),  wide, labelindent=0pt]
\item\label{example-23-4-23-i}
Let conditions of Theorem~\ref{thm-23-4-21} be fulfilled with
$\gamma=\epsilon \langle x\rangle$,
$\rho_1=\langle x\rangle^{-2}|\log \langle x\rangle|^\beta$,
$\beta>0$. Let either $\rho=\langle x\rangle^m$  with $m<-1$ or
$\rho=\langle x\rangle^{-1} |\log \langle x\rangle|^\alpha$ with $2\alpha <\beta$. Then the remainder estimate is $O(R)$ with $R$ defined by (\ref{23-4-32}) and
$\cN^-(\eta)= O(S)$ with
\begin{equation}
S=\left\{\begin{aligned}
&|\log \eta|^{\beta+1} &&d=2,\\
&|\eta^{(d-2)/(2m)}|\log \eta|^\beta &&d\ge 4.
\end{aligned}\right.
\label{23-4-41}
\end{equation}
Further,
$\cN^- (\eta)\asymp S$ if condition $\textup{(\ref{23-4-11})}_-$ is fulfilled in some non-empty cone.

\item\label{example-23-4-23-ii}
Furthermore, if condition $\textup{(\ref{23-4-11})}_-$ is fulfilled, then even for $d=2$ \ $R=1$.
\end{enumerate}
\end{example}

\begin{example-foot}\footnotetext{\label{foot-23-45} Cf. Example~\ref{example-23-4-18}.}\label{example-23-4-24}
\begin{enumerate}[label=(\roman*),  wide, labelindent=0pt]
\item\label{example-23-4-24-i}
Let conditions of Theorem~\ref{thm-23-4-21} be fulfilled with
 $\gamma=\langle x\rangle^{1-\alpha}$, $\rho =   \exp( a\langle x\rangle^{\alpha})$,
$\rho_1 =\langle x\rangle^{m_1}$, $a<0$, $\alpha>0$, $m_1<0$, $m_1+2-2\alpha>0$.
Then the remainder estimate is $O(R)$ with $R$ defined by (\ref{23-4-35}) and
(\ref{23-4-36}) holds.

\item\label{example-23-4-24-ii}
Further, if condition \ref{23-4-11-mp} is fulfilled, then the remainder estimate (\ref{23-4-35}) could be improved to (\ref{23-4B-36}).

\item\label{example-23-4-24-iii}
Furthermore, if condition $\textup{(\ref{23-4-11})}_-$ is fulfilled in some non-empty cone then there is $\asymp$ in (\ref{23-4-36}).
\end{enumerate}
\end{example-foot}

\begin{Problem}\label{Problem-23-4-25}
Again, one can hope to improve estimates (\ref{23-4B-36}) and (\ref{23-4-40}) in the same way as specified in Problem~\ref{Problem-23-4-6}.

\end{Problem}

We also leave to the reader

\begin{problem-foot}\footnotetext{\label{foot-23-46} Cf. Problem~\ref{Problem-23-4-12}  and \ref{problem-23-4-20}.} \label{problem-23-4-26}
Consider in this framework the Dirac operator. In this case $\N^-(\eta)$ is a number of eigenvalues in $(\pm M-\epsilon, \pm M-\eta)$ and $\N^+(\eta)$ is a number of eigenvalues in
$(\pm M+\eta, \pm M+\epsilon)$, $0<\eta<\epsilon$ and $\pm M$ is a point of the accumulation of the discrete spectrum.

We need to assume that
$M>0$ and potential $V\sim \rho^2$ at infinity.
\end{problem-foot}

Consider now the case when condition $\rho^2=o(\rho_1)$ as $|x|\to \infty$ is not fulfilled. Then the results will be similar to those of Section~\ref{sect-23-3}.

\begin{example-foot}\footnotetext{\label{foot-23-47} Cf. Example~\ref{example-23-3-16}.}\label{example-23-4-27}
\begin{enumerate}[label=(\roman*), wide, labelindent=0pt]
\item\label{example-23-4-27-i}
Let $X$ be a connected exterior domain\footref{foot-23-28} with $\sC^K$ boundary and $d=2$. Let  conditions  \textup{(\ref{23-2-2})},  \ref{23-2-3-*} and $\textup{(\ref{23-3-1})}^\#_{1}$ be fulfilled with
$\gamma =\epsilon _0 \langle x\rangle $,
$\rho =\langle x\rangle ^m$, $\rho _1=\langle x\rangle ^{m_1}$, $-1<m<0$,
$m-1< m_1\le 2m$.

 Assume that conditions (\ref{23-3-27}) and (\ref{23-3-28}) are fulfilled. Then for the Schr\"odinger operator asymptotics
\begin{gather}
\N^-(\eta)=\cN^-(\eta) + O(\eta^{1-m_1/(2m)})
\label{23-4-42}\\
\intertext{holds as $\eta\to +0$ with}
\cN^- (\eta)=O( \eta^{(m+1)/m}).
\label{23-4-43}
\end{gather}

Indeed, it follows from the arguments of Example~\ref{example-23-3-15}; we need to take into account that $\{|x|\ge C\tau^{1/(2m)}\}$ is a forbidden zone.

\item\label{example-23-4-27-ii}
Similar results hold in the full-rank even-dimensional case:
\begin{gather}
\N^-(\eta)=\cN^-(\eta) + O(\eta^{1-m_1/(2m)+(d-2)(m+1)/(2m)})
\label{23-4B-43}\\
\intertext{holds as $\eta\to +0$ with}
\cN^- (\eta)=O( \eta^{d(m+1)/(2m)}).
\label{23-4B-44}
\end{gather}
\end{enumerate}
\end{example-foot}

We leave to the reader:

\begin{problem}\label{problem-23-4-28}
Consider the case of $\gamma =\epsilon _0 \langle x\rangle $,
$\rho =\langle x\rangle ^{-2}|\log x|^\alpha$,\\
$\rho _1=\langle x\rangle ^{-2}|\log x|^\beta$, $2\alpha \ge \beta>\alpha$.
\end{problem}

We also leave to the reader
\begin{problem-foot}\footnotetext{\label{foot-23-48} Cf. Problems~\ref{Problem-23-4-12}, \ref{problem-23-4-20} and \ref{problem-23-4-26}.} \label{problem-23-4-29}
Consider in this framework the Dirac operator with $M>0$. In this case both points $M-0$ and $-M+0$ could be limits of the discrete spectrum simultaneously.
\end{problem-foot}

\chapter{$2\D$-case. Multiparameter asymptotics}
\label{sect-23-5}

In this section we consider asymptotics with respect to three parameters $\mu$, $h$ and $\tau$; here spectral parameter $\tau$ tends \underline{either} to $\pm \infty$ \underline{or} to the border of the essential spectrum \underline{or} to $-\infty$ (for Schr\"odinger and Schr\"odinger-Pauli operators) \underline{or} to the border of the spectrum. In two last cases presence of  $h\to+0$ is crucial. We consider here only $d=2$ and $h\ll 1$.

\section{Asymptotics of large eigenvalues}
\label{sect-23-5-1}

In this subsection $\tau\to +\infty$ for the Schr\"odinger and Schr\"odinger-Pauli operators and $\tau\to \pm \infty$ for the Dirac operator. We consider  the  Schr\"odinger and Schr\"odinger-Pauli operators, leaving the Dirac operator to the reader.

\begin{example}\label{example-23-5-1}
Assume first that $\psi\in \sC_0^\infty$  and there are no singularities on $\supp(\psi)$. We consider
\begin{equation}
\N_\psi^-(\tau)=\int e(x,x,\tau)\psi(x)\,dx.
\label{23-5-1}
\end{equation}
Then for scaling $A\mapsto \tau^{-1}A$ leads to
$h\mapsto h\tau^{-1/2}$ and $\mu\to \mu\tau^{-1/2}$.

\begin{enumerate}[label=(\roman*), wide, labelindent=0pt]
\item\label{example-23-5-1-i}
If $\mu \lesssim \tau^{1/2}$ then  we can apply the standard theory with the ``normal" magnetic field; we need to assume that $h\ll\tau^{1/2}$ and we need neither condition $d=2$, nor
$F\ge \epsilon_0$, nor $\partial X=\emptyset$; the principal part of the asymptotics has magnitude $h^{-d}\tau^{d/2}$ and the remainder estimate is $O(h^{1-d}\tau^{(d-1)/2})$ which one can even improve to $o(h^{1-d}\tau^{(d-1)/2})$ under proper non-periodicity assumption.

\item\label{example-23-5-1-ii}
Let $\mu \gtrsim \tau^{1/2}$, $\mu h\lesssim \tau$. Then   we can apply the standard theory with the ``strong" magnetic field; we assume that $d=2$,
$\partial X=\emptyset$ and  $F\ge \epsilon_0$. Then the principal part of the asymptotics has magnitude $h^{-2}\tau$ and under non-degeneracy assumptions
\begin{align}
&\nabla F=0\implies \det \Hess F\ge \epsilon
\label{23-5-2}\\
\shortintertext{and}
&\nabla F=0\implies |\det \Hess F|\ge \epsilon
\label{23-5-3}
\end{align}
fulfilled on $\supp(\psi)$ the remainder estimate is
$O(\mu^{-1}h^{-1}\tau )$ and \\ $O(\mu^{-1}h^{-1}\tau (|\log (\mu\tau^{-1/2})|+1)$
respectively\footnote{\label{foot-23-49} In the letter case logarithmic factor could be removed by adding a correction term.}. Without non-degeneracy assumption the remainder estimate is $O(\mu h^{-1})$.

\item\label{example-23-5-1-iii}
If $\mu \gtrsim \tau^{1/2}$, $\mu h \ge c\tau$ than $\N^-(\tau)=0$ for the Schr\"odinger operator; for the Schr\"odinger-Pauli operator  the principal part of the asymptotics has magnitude $\mu h^{-1}$ and under  non-degeneracy assumptions (\ref{23-5-2}) and (\ref{23-5-3}) the remainder estimate is $O(1)$ and $O(\log \mu)$ respectively (or better for $\tau$ belonging to the spectral gap).
\end{enumerate}
\end{example}

\begin{example-foot}\label{example-23-5-2}\footnotetext{\label{foot-23-50} Cf. Example~\ref{example-23-3-15}.}
Let $X$ be a connected exterior domain with $\sC^K$ boundary.  Let conditions \textup{(\ref{23-2-2})}, \ref{23-2-3-*}, \ref{23-3-1-**},  be fulfilled with
$\gamma =\epsilon _0 \langle x\rangle $, $\rho =\langle x\rangle ^m$,
$\rho _1=\langle x\rangle ^{m_1}$, $m_1>2m\ge 0$. Consider the Schr\"odinger operator and assume that
\begin{phantomequation}\label{23-5-4}\end{phantomequation}
\begin{gather}
\tau\ge \mu h, \qquad\tau^{2m-m_1}\le \epsilon (\mu h)^{2m}.
\tag*{$\textup{(\ref*{23-5-4})}_{1,2}$}\label{23-5-4-*}\\
\shortintertext{Then}
\cN^-(\tau,\mu,h)\asymp \tau^{(2+m_1)/m_1}h^{-2(1+m_1)/m_1}\mu^{-2/m_1}.
\label{23-5-5}
\end{gather}
\begin{enumerate}[label=(\roman*), wide, labelindent=0pt]
\item\label{example-23-5-2-i}
Further, if  $\tau \gtrsim \mu^2$, then the zone of the strong magnetic field
$\mu _\eff =\mu \langle x\rangle^{m_1+1}\tau^{-1/2}\ge C$ is contained in $\{x\colon |x|\ge c\}$
and here we have non-degeneracy condition fulfilled. Then the remainder estimate is $O(R)$ with
\begin{equation}
R=\tau^{(m_1+2)/2(m_1+1)}\mu^{-1/(m_1+1)}h^{-1},
\label{23-5-6}
\end{equation}
which could be improved under non-periodicity assumption; see Example~\ref{23-3-18}.

\item\label{example-23-5-2-ii}
On the other hand, if $\mu ^2\gg \tau$, then the contribution of the zone $\{x\colon |x|\ge c\}$ to the remainder is $O(\mu^{-1}h^{-1}\tau)$. The contribution of the zone $\{x\colon |x|\le c\}$ to the remainder is $O(\mu^{-1}h^{-1}\tau)$ provided $X=\bR^2$ and non-degeneracy assumption (\ref{23-5-2}) is fulfilled (etc) and $O(\mu h^{-1})$ in the general case.

\item\label{example-23-5-2-iii}
Let us replace $\textup{(\ref{23-5-4})}_{2}$ by the opposite inequality, and assume \ref{23-3-3-*}. Then (\ref{23-5-5}) is replaced by
$\cN^-(\tau,\mu,h)\asymp h^{-2}\tau ^{(m+1)/m}$. Let us discuss $R$.
\begin{enumerate}[label=(\alph*),  labelindent=0pt]
\item\label{example-23-5-2-iiia}
If $\mu \tau^{(m_1+1-m)/(2m)}\lesssim 1$, then $\mu_\eff \lesssim 1$ as
$|x|\lesssim \tau^{1/(2m)}$  and $R= h^{-1}\tau ^{(m+1)/(2m)}$.
\item\label{example-23-5-2-iiib}
If $\mu \tau^{(m_1+1-m)/(2m)}\gtrsim 1$, but $\mu^2\tau \lesssim 1$,  then $R$ is given by (\ref{23-5-6}).
\item\label{example-23-5-2-iiic}
If $\mu ^2\gg \tau$, then we are in the framework of \ref{example-23-5-2-ii}.
\end{enumerate}
\end{enumerate}
\end{example-foot}

\begin{example-foot}\label{example-23-5-3}\footnotetext{\label{foot-23-51} Cf. Example~\ref{example-23-3-19}.}
In the framework of Example~\ref{example-23-5-2} for the Schr\"odinger-Pauli operator under assumption \ref{23-3-3-*} the remainder estimate is the same as in Statement~\ref{example-23-5-2-iii} while
\begin{equation}
\cN^-(\tau,\mu,h)\asymp h^{-2}\tau ^{(m+1)/m} + \mu h^{-1}\tau^{(m_1+2)/(2m)}.
\label{23-5-7}
\end{equation}
\end{example-foot}

We leave to the reader

\begin{problem}\label{problem-23-5-4}
Consider the Schr\"odinger and Schr\"odinger-Pauli operators
\begin{enumerate}[label=(\roman*), wide, labelindent=0pt]
\item\label{problem-23-5-4-i}
In the same framework albeit with condition $m_1>2m$ replaced by
 $2m\ge m_1 \ge 0$. Assume that \ref{23-3-3-*} is fulfilled.

Then magnitude of $\cN^-(\tau,\mu,h)$ is described in  Examples~\ref{example-23-5-2} and~\ref{example-23-5-3}. Under proper non-degeneracy assumption (which we leave to the reader to formulate) derive the remainder estimate.

\item\label{problem-23-5-4-ii}
In the same framework as in \ref{problem-23-5-4-i} albeit  in with  $m_1<0$ (magnetic field is stronger in the center but there is no singularity), in which case the center can become a classically forbidden zone.

\item\label{problem-23-5-4-iii}
With other types of the behaviour at infinity.
\end{enumerate}
\end{problem}

\begin{problem}\label{problem-23-5-5}
For the Dirac operators derive similar results as $\tau\to \pm \infty$.
\end{problem}

\section{Asymptotics of the small eigenvalues}
\label{sect-23-5-2}

In this subsection for the Schr\"odinger and Schr\"odinger-Pauli operators we consider asymptotics of eigenvalues tending to $-0$.

\begin{example-foot}\label{example-23-5-6}\footnotetext{\label{foot-23-52} Cf. Example~\ref{example-23-4-22}.}
Let $X$ be a connected exterior domain with $\sC^K$ boundary.  Let conditions \textup{(\ref{23-2-2})}, \ref{23-2-3-*}, $\textup{(\ref{23-3-1})}^\#_1$  be fulfilled with $\gamma =\epsilon _0 \langle x\rangle $,
$\rho =\langle x\rangle ^m$, $\rho _1=\langle x\rangle ^{m_1}$, $-1< m< 0$,
$m_1 > m-1$.%

Consider the Schr\"odinger operator and assume that
\begin{phantomequation}\label{23-5-8}\end{phantomequation}
\begin{equation}
1\ge \mu h, \qquad|\tau|^{2m-m_1}\le \epsilon (\mu h)^{2m}.
\tag*{$\textup{(\ref*{23-5-8})}_{1,2}$}\label{23-5-8-*}
\end{equation}
Then $\cN^-(\tau,\mu,h)=O( h^{-2}|\tau|^{(m+1)/m})$ as $\tau\to -0$ with ``$\asymp $'' instead of ``$=O$'' if condition \ref{23-4-11-mp} (with the sign ``$-$'')  fulfilled in some non-empty cone.

Further, under non-degeneracy assumption (\ref{23-3-27}) the contribution to the remainder of zone $\{x\colon |x|\ge c\}$ is $O(R)$ with

\begin{enumerate}[label=(\roman*), wide, labelindent=0pt]
\item\label{example-23-5-6-i}
If $\mu |\tau|^{(m_1+1-m)/(2m)}\lesssim 1$ then $R= h^{-1}|\tau|^{(m+1)/(2m)}$.

\item\label{example-23-5-6-ii}
Let $\mu |\tau|^{(m_1+1-m)/(2m)}\gtrsim 1$. Then
\begin{enumerate}[label=(\alph*),  labelindent=0pt]
\item\label{example-23-5-6-iia}
If $m_1< 2m$ then $R=\mu^{-1}h^{-1}|\tau|^{(2m-m_1)/(2m)}$.
\item\label{example-23-5-6-iib}
If $m_1=2m$ then $R=\mu^{-1}h^{-1}|\log \mu|$.
\item\label{example-23-5-6-iic}
If $m_1>2m$  then $R=h^{-1}\mu^{(m+1)/(m_1+1-m)}$ for $\mu \lesssim 1$ and
$R=\mu^{-1}h^{-1}$ for $\mu \gtrsim 1$.
\end{enumerate}
\end{enumerate}
\end{example-foot}

\begin{example-foot}\label{example-23-5-7}\footnotetext{\label{foot-23-53} Cf. Example~\ref{example-23-4-22}.}
In the framework of Example~\ref{example-23-5-6} for the Schr\"odinger-Pauli operator under assumption \ref{23-3-3-*} the contribution to the remainder of the zone $\{x\colon |x|\ge c\}$ the same as in Example~\ref{example-23-5-6}\ref{example-23-5-6-ii} while
\begin{equation}
\cN^-(\tau,\mu,h)=O( h^{-2}|\tau| ^{(m+1)/m} + \mu h^{-1}|\tau|^{(m_1+2)/(2m)})
\label{23-5-9}
\end{equation}
$\asymp $ if condition \ref{23-4-11-mp} (with the sign ``$-$'')  fulfilled in some non-empty cone.
\end{example-foot}

\begin{problem}\label{problem-23-5-8}
Consider the Schr\"odinger and Schr\"odinger-Pauli operators if
\begin{enumerate}[label=(\roman*), wide, labelindent=0pt]
\item\label{problem-23-5-8-i}
If condition $\textup{(\ref{23-5-8})}_1$ is violated (then there could be a forbidden zone in the center).
\item\label{problem-23-5-8-ii}
$m_1\le m-1$.
\item\label{problem-23-5-8-iii}
With other types of the behaviour at infinity.
\end{enumerate}
\end{problem}

\begin{problem}\label{problem-23-5-9}
Consider the Schr\"odinger and Schr\"odinger-Pauli operators in the framework of Subsections~\ref{sect-23-4-1} and~\ref{sect-23-4-2} if
\begin{enumerate}[label=(\roman*), wide, labelindent=0pt]
\item\label{problem-23-5-9-i}
$\mu h=1$; then the essential spectrum  does not change.

\item\label{problem-23-5-9-ii}
$\mu h\to \infty$; then only  point $0$ of the essential spectrum is preserved for the Schr\"odinger-Pauli operators, while others go to $+\infty$. Consider  $\N^\pm(\eta)$ with $\eta\to 0$.

\item\label{problem-23-5-9-iii}
$\mu h\to 0$; then only  point $0$ of the essential spectrum is preserved for the Schr\"odinger-Pauli operators, while others move towards it. Consider $\N^-(\eta)$ with $\eta\to 0$ for both Schr\"odinger and Schr\"odinger-Pauli operators.
\end{enumerate}
\end{problem}

\begin{problem}\label{problem-23-5-10}
\begin{enumerate}[label=(\roman*), wide, labelindent=0pt]
\item\label{problem-23-5-10-i}
Consider the Schr\"odinger-Pauli operators in the framework of Subsection~\ref{sect-23-4-3}.

\item\label{problem-23-5-10-ii}
Consider the Schr\"odinger and Schr\"odinger-Pauli operators in the framework of Subsection~\ref{sect-23-4-4}.
\end{enumerate}
\end{problem}

\begin{problem}\label{problem-23-5-11}
For the Dirac operators derive similar results as $M\ne 0$ and
$\tau\to M-0$ and $-M+0$; or as $M=0$ albeit $m_1\ge 0$.
\end{problem}

\section{Case of $\tau\to+0$}
\label{sect-23-5-3}

In this subsection $\tau\to +0$ for the Schr\"odinger and Schr\"odinger-Pauli ope\-rators and $\tau\to \pm M\pm 0$ for the Dirac operator. Consider  the  Schr\"odinger and Schr\"odinger-Pauli operators first.

\begin{example-foot}\label{example-23-5-12}\footnotetext{\label{foot-23-54} Cf. Example~\ref{example-23-5-2}.}
Let $V>0$ everywhere except $V(0)=0$.  Let conditions \textup{(\ref{23-2-2})}, \ref{23-2-3-*}, \ref{23-3-1-**},  be fulfilled with
$\gamma =\epsilon _0 |x|$, $\rho =|x| ^m$,
$\rho _1=| x| ^{m_1}$, $m_1> 2m \ge 0$. Consider the Schr\"odinger operator and assume that $\tau\to +0$.
\begin{enumerate}[label=(\roman*), wide, labelindent=0pt]
\item\label{example-23-5-12-i}
Let
\begin{phantomequation}\label{23-5-10}\end{phantomequation}
\begin{equation}
\mu \ll \tau^{(m_1+2)/2}h^{-(m_1+1)}, \qquad
\tau^{2m-m_1}\le \epsilon (\mu h)^{2m}.
\tag*{$\textup{(\ref*{23-5-10})}_{1,2}$}\label{23-5-10-*}
\end{equation}
Then (\ref{23-5-5}) holds.

Then the remainder estimate is $O(R)$ with defined by (\ref{23-5-6})
which could be improved under non-periodicity assumption; see Example~\ref{23-3-18}.

\item\label{example-23-5-12-ii}
Let us replace $\textup{(\ref*{23-5-10})}_{2}$ by the opposite inequality, and assume \ref{23-3-3-*}. Then (\ref{23-5-5}) is replaced by
$\cN^-(\tau,\mu,h)\asymp h^{-2}\tau ^{(m+1)/m}$. Let us discuss $R$.
\begin{enumerate}[label=(\alph*),  labelindent=0pt]
\item\label{example-23-5-12-iia}
If $\mu \tau^{(m_1+1-m)/(2m)}\lesssim 1$, then $\mu_\eff \lesssim 1$ as
$|x|\lesssim \tau^{1/(2m)}$  and $R= h^{-1}\tau ^{(m+1)/(2m)}$.
\item\label{example-23-5-12-iib}
If $\mu \tau^{(m_1+1-m)/(2m)}\gtrsim 1$, but $\mu^2\tau \lesssim 1$,  then $R$ is given by (\ref{23-5-6}).
\end{enumerate}
\end{enumerate}
\end{example-foot}

\begin{example-foot}\label{example-23-5-13}\footnotetext{\label{foot-23-55} Cf. Example~\ref{example-23-5-3}.}
In the framework of Example~\ref{example-23-5-12} for the Schr\"odinger-Pauli operator under assumption \ref{23-3-3-*} the remainder estimate is the same as in Statement~\ref{example-23-5-12-ii} while (\ref{23-5-7}) holds.
\end{example-foot}

We leave to the reader

\begin{problem-foot}\label{problem-23-5-14}\footnotetext{\label{foot-23-56} Cf. Problem~\ref{problem-23-5-4}.}
Consider the Schr\"odinger and Schr\"odinger-Pauli operators
\begin{enumerate}[label=(\roman*), wide, labelindent=0pt]
\item\label{problem-23-5-14-i}
In the same framework albeit with condition $m_1>2m$ replaced by
 $2m\ge m_1 \ge 0$. Assume that \ref{23-3-3-*} is fulfilled.

Then magnitude of $\cN^-(\tau,\mu,h)$ is described in  Examples~\ref{example-23-5-12} and~\ref{example-23-5-13}. Under proper non-degeneracy assumption (which we leave to the reader to formulate) derive the remainder estimate.

\item\label{problem-23-5-14-ii}
In the same framework as in \ref{problem-23-5-14-i} albeit  in with  $m_1<0$ (magnetic field is stronger in the center but there is no singularity), in which case the center can become a classically forbidden zone.

\item\label{problem-23-5-14-iii}
With other types of the behaviour at infinity.
\end{enumerate}
\end{problem-foot}

\begin{problem}\label{problem-23-5-15}
For the Dirac operators derive similar results as $\tau\to \pm (M+0)$.
\end{problem}

\section{Case of $\tau\to-\infty$}
\label{sect-23-5-4}

In this subsection for the Schr\"odinger and Schr\"odinger-Pauli operators we consider asymptotics with $\tau\to -\infty$.

In this subsection for the Schr\"odinger and Schr\"odinger-Pauli operators we consider asymptotics of eigenvalues tending to $-0$.

\begin{example-foot}\label{example-23-5-16}\footnotetext{\label{foot-23-57} Cf. Example~\ref{example-23-4-22}.}
Let $X\ni 0$   and let conditions \textup{(\ref{23-2-2})}, \ref{23-2-3-*}, $\textup{(\ref{23-3-1})}^\#_1$  be fulfilled with $\gamma =\epsilon _0 \langle x\rangle $,
$\rho =\langle x\rangle ^m$, $\rho _1=\langle x\rangle ^{m_1}$, $-1< m< 0$,
$m_1 > m-1$.%

Consider the Schr\"odinger operator and assume that
\begin{phantomequation}\label{23-5-11}\end{phantomequation}
\begin{equation}
h\ll |\tau|^{(m+1)/(2m)}, \qquad|\tau|^{m_1-2m}\le \epsilon (\mu h)^{2m}.
\tag*{$\textup{(\ref*{23-5-11})}_{1,2}$}\label{23-5-11-*}
\end{equation}
Then $\cN^-(\tau,\mu,h)=O( h^{-2}|\tau|^{(m+1)/m})$ as $\tau\to -0$ with $\asymp $ if condition \ref{23-4-11-mp} (with the sign ``$-$'')  fulfilled in some non-empty cone.

\begin{enumerate}[label=(\roman*), wide, labelindent=0pt]
\item\label{example-23-5-16-i}
If $\mu |\tau|^{(m_1+1-m)/(2m)}\lesssim 1$ then $R= h^{-1}|\tau|^{(m+1)/(2m)}$.

\item\label{example-23-5-16-ii}
Let $\mu |\tau|^{(m_1+1-m)/(2m)}\gtrsim 1$. Then
\begin{enumerate}[label=(\alph*),  labelindent=0pt]
\item\label{example-23-5-16-iia}
If $m_1< 2m$ then $R=\mu^{-1}h^{-1}|\tau|^{(2m-m_1)/(2m)}$.
\item\label{example-23-5-16-iib}
If $m_1=2m$ then $R=\mu^{-1}h^{-1}|\log \mu|$.
\item\label{example-23-5-16-iic}
If $m_1>2m$  then $R=h^{-1}\mu^{(m+1)/(m_1+1-m)}$ for $\mu \lesssim 1$ and
$R=\mu^{-1}h^{-1}$ for $\mu \gtrsim 1$.
\end{enumerate}
\end{enumerate}
\end{example-foot}

\begin{example-foot}\label{example-23-5-17}\footnotetext{\label{foot-23-58} Cf. Example~\ref{example-23-4-22}.}
In the framework of Example~\ref{example-23-5-16} for the Schr\"odinger-Pauli operator under assumption \ref{23-3-3-*} the contribution to the remainder of the zone $\{x\colon |x|\ge c\}$ the same as in Example~\ref{example-23-5-16}\ref{example-23-5-16-ii} while
\begin{equation}
\cN^-(\tau,\mu,h)=O( h^{-2}|\tau| ^{(m+1)/m} + \mu h^{-1}|\tau|^{(m_1+2)/(2m)})
\label{23-5-12}
\end{equation}
$\asymp $ if condition \ref{23-4-11-mp} (with the sign ``$-$'')  fulfilled in some non-empty cone.
\end{example-foot}

\begin{problem}\label{problem-23-5-18}
Consider the Schr\"odinger and Schr\"odinger-Pauli operators if
\begin{enumerate}[label=(\roman*), wide, labelindent=0pt]
\item\label{problem-23-5-18-ii}
$m_1\le m-1$.
\item\label{problem-23-5-18-iii}
With other types of the behaviour at $0$.
\end{enumerate}
\end{problem}

\begin{subappendices}

\chapter{Appendices}

\section{On the self-adjointness of the Dirac operator}
\label{sect-23-A-3}

The Dirac operators treated in this Chapter are surely self-adjoint in the case of an exterior domain with the singularity at infinity.  However, the same fact should be proven for an interior domain with singular points. We consider a single singular point at $0$.

\begin{theorem}\label{thm-23-A-1}
Let $X \subset \bR^2/(\bar{\gamma}_1\bZ\times \bar{\gamma}_2\bZ)$ ($0<\bar{\gamma}_j\le \infty$) be an open domain. Let conditions \textup{(\ref{23-2-2})}, \ref{23-2-3-*}, \textup{(\ref{23-2-13})} be fulfilled. Further, let
\begin{phantomequation}\label{23-A-1}\end{phantomequation}
\begin{equation}
F\ge \epsilon \rho_1 \qquad |V|\ge \epsilon \rho \text{as\ \ }|x|\le \epsilon.
\tag*{$\textup{(\ref*{23-A-1})}_{1,2}$}\label{23-A-1-*}
\end{equation}
Let us assume that there exists a neighborhood of $\partial X$, denoted by $Y$, such that for every $r>0$ the inequalities
\begin{phantomequation}\label{23-A-2}\end{phantomequation}
\begin{align}
&\rho _1\ge \varepsilon \rho ^2,\quad
(\rho \gamma )^s\ge \varepsilon \rho _1,\quad \rho \ge 1,
\tag*{$\textup{(\ref*{23-A-2})}_{1-3}$}\label{23-A-2-*}\\[2pt]
&\bigl|V+\sqrt {2j\mu hF}\bigr|\ge \varepsilon \sqrt \rho _1-\varepsilon ^{-1}
\qquad \forall j\in \bZ^+\setminus 0,
\label{23-A-3}\\[2pt]
|&V|\ge \varepsilon \rho -\varepsilon ^{-1}
\label{23-A-4}
\end{align}
are fulfilled on $Y\cap X\cap\{|x|\le r\}$ with appropriate
$\varepsilon =\varepsilon (r)>0$.

Then for $\mu >0$, $h>0$ the operator $A$ with domain $\fD(A)=\sC_0^1(X,\bH)$ is essentially self-adjoint in $\sL^2(X,\bH)$.
\end{theorem}

\begin{proof}
Let us consider the adjoint operator $A^*$.  This operator is defined by the same formula with $\fD(A^*)=\{u\in \sL^2,Au\in \sL^2\}$ with $Au$ calculated as a distribution.  We should prove that $\Ker (A^*\pm iI)=0$ for both signs.
So, here and below let $u\in \sL^2$ and $(A^*\pm iI)u=0$ for some sign.

The microlocal canonical form of Section~\ref{book_new-sect-17-2} of \cite{futurebook} yields the inequality
\begin{multline}
\| \rho v\|\le M\bigl(\| Av\| +\| v\| +\|\gamma ^{-1}v\|\bigr)\\
\forall v\in \sC_0^1\bigl(B(y,{\frac{1}{2}}\gamma (y)\bigr)\qquad
\forall y\in X_r=X\cap \{|x|\le r\}
\label{23-A-5}
\end{multline}
with a constant $M=M(r)$; all the constants now depend on $\mu$ and $h$.

Let us prove that $(\rho \gamma )^nu\in \sL^2(X_r)$ for every $r>0$ by induction on $n$. This is true for $n=0$ by the assumption $u\in \sL^2$.

Let $(\rho \gamma )^nu\in \sL^2(X_r)$ for some $n$; then
$v=\gamma (\rho \gamma )^nu$ also belongs to $\sL^2(X_r)$ and
\begin{equation*}
Av=[A,\gamma (\rho \gamma )^n]u\mp iv.
\end{equation*}
One can easily see that $[A,\gamma (\rho \gamma )^n]$ is a matrix-valued function the matrix norm of which does not exceed $M(\rho \gamma )^n$.  Therefore, taking a $\gamma $-admissible partition of unity in a neighborhood of $X_r$ we see that $\sum_\nu \| A\psi _\nu v\|^2<\infty $ and therefore (\ref{23-A-5}) yields that $\sum_\nu \|\rho \psi _\nu v\|^2<\infty $.  Therefore
$\rho v=(\rho \gamma )^{n+1}u\in \sL^2(X_r)$ and the induction step is complete.

Thus we have proven that $(\rho \gamma )^nu\in \sL^2(X_r)$ for every $r>0$.
Then $\rho _1u\in \sL^2(X_r)$ by $\textup{(\ref{23-A-2})}_2$.  The ellipticity of $A$ yields that $D_ju\in \sL^2(X_r)$.  Therefore, if $\psi \in \sC_0^1
\bigl(\bR^2/(\bar{\gamma}_1\bZ\times \bar{\gamma}_2\bZ)\bigr)$ then
$\psi u\in \fD (\bar{A})$ where $\bar{A}$ is the closure of $A$.  Moreover,
\begin{equation*}
(\bar{A}\pm iI)\psi u=[\bar{A},\psi ]u.
\end{equation*}
Calculating the real part of the inner product with $i\psi u$ we obtain
the inequality
\begin{equation*}
\|\psi u\|\le \max |\nabla \psi^2 |\cdot \|u\|^2.
\end{equation*}
Let us take $\psi =\psi ^0(x/r)$ where $\psi ^0\in \sC_0^1$ is a fixed
function equal to $1$ in a neighborhood of 0.  Then for $r\to+\infty $
we see that the left-hand expression of this inequality tends to $\|u\|^2$
and the right-hand expression tends to $0$.  Therefore $u=0$.
\end{proof}

\begin{theorem}\label{thm-23-A-2}
Let all of the conditions of Theorem~\ref{thm-23-A-1} excluding
condition \ref{23-A-2-*} be fulfilled with $\gamma =\epsilon _0 |x|$,
$\rho =|x|^m$, $\rho _1=|x|^{m_1}$, $m_1\le 2m$, $m_1<-2$.  Then the
Dirac operator has a self-adjoint extension for $\mu >0$, $h>0$.
\end{theorem}

\begin{proof}
For $m<-1$ condition \ref{23-A-2-*} is also fulfilled and therefore the
operator $A$ is essentially self-adjoint.  So, let us treat the case
$m\ge -1$.  Let us consider the Dirac operator $A_\D$ with the potential
$V_t =V+t W$ where $t >0$, $W$ is a potential which is regular away from $\{x=0\}$, and $W=\pm |x|^{m'}$ in a neighborhood of $x=0$ with $-1>m'>\frac{1}{2}m_1$.  This operator is essentially self-adjoint
by Theorem~\ref{thm-23-A-1}.
For $m<0$ let us choose the sign of $W$ coinciding with the sign of $V$ on a neighborhood of $0$ (condition $\textup{(\ref{23-A-1})}_2$ yields that this is possible).

For $m\ge 0$ let us choose an appropriate interval $[\tau_1,\tau_2]$ with $\tau_1<\tau_2$.  Then applying the results of Section~\ref{sect-23-3} we see that the number of eigenvalues of the operator $\bar{A}_t$ lying in the interval $[\tau _1,\tau _2]$ is bounded uniformly with respect to $t >0$.  Then there exists a sequence $t _k\to +0$ such that there exist $\tau '_1<\tau '_2$ which do not depend on $k$ and such that
$[\tau '_1,\tau '_2]\cap\Spec (\bar{A}_t )=\emptyset$ for $t =t _k$.

Let $\tau = \frac{1}{2} (\tau '+\tau '_2)$; then all the operators
$(\bar{A}_t -\tau )^{-1}$ are uniformly bounded and
$\| A_t u\|\ge \epsilon \|u\|\;\forall u\in \fD(A_t )$ for
$t =t _k$.  Then the same is true for $A$.  Therefore $A$ has a self-adjoint extension $\tilde{A}$ satisfying the same estimate.%
\end{proof}
\end{subappendices}

\chapter{$3\D$-case. Introduction}
\label{sect-24-1}

In this chapter we obtain eigenvalue asymptotics for $3\D$-Schr\"{o}dinger, Schr\"{o}\-dinger-Pauli  and Dirac operators in the situations in which the role of the magnetic field is important.  We have seen in Chapters~\ref{book_new-sect-13}  and~\ref{book_new-sect-17} of \cite{futurebook} that these operators are essentially different and they also differ significantly from the corresponding $2\D$-operators which we considered in the Sections~\ref{sect-23-1}--\ref{sect-23-5}.

Now we \emph{usually\/} find ourselves in the situation much closer to Chapter~\ref{book_new-sect-11}  than Sections~\ref{sect-23-1}--\ref{sect-23-5} was. Indeed, our local asymptotics now are the same as without magnetic field, under very week non-degeneracy assumptions. We also allow boundaries and a singular points, finite or infinite, belonging to the boundary.

We start from Section~\ref{sect-24-2} in which we consider the case when the spectral parameter is fixed ($\tau=\const$) and study asymptotics with respect to $\mu,h$ exactly like in Section~\ref{book_new-sect-11-1} of \cite{futurebook} we considered asymptotics with respect to $h$. However, since now we have two parameters, we need to consider an interplay between them: while always $h\to +0$, we cover
$\mu\to +0$, $\mu $ remains disjoint from $0$ and $\infty$ and $\mu\to \infty$, which in turn splits into subcases $\mu h\to 0$, $\mu h$  remains disjoint from $0$ and $\infty$ and $\mu h\to \infty$.

In Section~\ref{sect-24-3} we consider asymptotics with $\mu=h=1$ and with $\tau$ tending to $+\infty$ for the Schr\"odinger and Schr\"odinger-Pauli operators and to $\pm\infty$ for the Dirac operator. We consider bounded domains with the singularity at some point and unbounded domains with the singularity at infinity.

In Section~\ref{sect-24-4} we consider asymptotics with the singularity at infinity and $\mu=h=1$ and with $\tau$ tending to $+0$ for the Schr\"odinger and Schr\"odinger-Pauli operators and to $\pm (M-0)$ for the Dirac operator.

It  includes the most interesting case (see Subsection~\ref{sect-23-4-2} of \cite{futurebook}) when magnetic field is either constant or stabilizes fast at infinity and potential fast decays at infinity in the direction of magnetic field.  In this case we consider a reduced one-dimensional operator which has just one negative eigenvalue $\Lambda (x')$ and it turns out that the asymptotics of the eigenvalues tending to the bottom of the continuous spectrum for $3\D$-operator coinsides with the asymptotics obtained in Subsection~\ref{sect-23-4-1}  for $2\D$-operator with the potential $\Lambda(x')$.
In contrast to the rest of the section we consider multidimensional case as well.

In Section~\ref{sect-24-5} we consider asymptotics with respect to
$\mu, h, \tau$, like in Section~\ref{book_new-sect-11-7} of \cite{futurebook}  and~\ref{sect-23-5}  again with significant differences mentioned above.

Further, in Section~\ref{sect-24-6} we consider Riesz means for the 3-dimensional Schr\"{o}\-dinger operator with the strong magnetic field, which will be useful in the Part~8  of \cite{futurebook}.

Finally, in Appendices~\ref{sect-24-A-1} and~\ref{sect-24-A-2} we investigate $1\D$-Schr\"odinger operators and in Appendix~\ref{sect-24-A-3} we construct examples of vector potentials with different rates of growth of the magnetic field at infinity.

\chapter{$3\D$-case.  Asymptotics with fixed spectral parameter}
\label{sect-24-2}

In this section we consider asymptotics with a fixed spectral parameter for $3$-dimensional magnetic Schr\"odinger, Schr\"odinger-Pauli and Dirac operators and discuss some of the generalizations\footnote{\label{foot-24-1} Mainly to higher dimensions with maximal-rank magnetic field.}.

As in Chapters \ref{book_new-sect-9}--\ref{book_new-sect-12} of \cite{futurebook} and Sections~\ref{sect-23-1}--\ref{sect-23-5} we will introduce a semiclassical zone and a singular zone, where $\rho\gamma \ge h$ and $\rho\gamma\le h$ respectively. In the semiclassical zone we apply asymptotics of Chapters~\ref{book_new-sect-13}, \ref{book_new-sect-18}  and \ref{book_new-sect-20} of \cite{futurebook}--in the multidimensional case. In the singular zone we need to apply estimates for a number of eigenvalues; usually it would be sufficient to use non-magnetic estimate\footnote{\label{foot-24-2} With $V$ modified accordingly; for example, for the Schr\"odinger and Schr\"odinger-Pauli operators $V_-$ is replaced by
$C\bigl((1-\epsilon) V - C_\epsilon \mu^2|\vec{V}|^2\bigl)_-$.}
for number of eigenvalues which trivially follows from standard one but if needed one can use more delicate estimates.

\section{Schr\"odinger operator}
\label{sect-24-2-1}

\subsection{Estimates of the spectrum}
\label{sect-24-2-1-1}

Consider first the Schr\"{o}dinger operator (\ref{book_new-13-1-1} of \cite{futurebook}) where $g^{jk},V_j,V$ satisfy (\ref{book_new-13-1-2})  and (\ref{book_new-13-1-4}) of \cite{futurebook} i.e.
\begin{equation}
\epsilon |\xi |^2\le \sum_{j,k} g^{jk}\xi _j\xi _k\le c|\xi |^2 \qquad
\forall \xi \in\bR^d.
\label{24-2-1}
\end{equation}
Without any loss of the generality we can fix $\tau =0$ and then in the important function $V_\eff F_\eff ^{-1}$ the parameters $\mu $ and $h$ enter as factors. Thus, we treat the operator (\ref{book_new-13-1-1}) of \cite{futurebook} assuming that it is self-adjoint.

We make assumptions the same assumptions \ref{23-2-3} of \cite{futurebook} i.e.
\begin{phantomequation}\label{24-2-2}\end{phantomequation}
\begin{multline}
|D^\alpha g^{jk}|\le c\gamma ^{-|\alpha |},\quad
|D^\alpha F_{jk}|\le c\rho_1 \gamma ^{-|\alpha |},\quad
|D^\alpha V|\le c\rho ^2\gamma ^{-|\alpha |}
\tag*{$\textup{(\ref*{24-2-2})}_{1-3}$}\label{24-2-2-*}
\end{multline}
where scaling function $\gamma(x)$ and weight functions $\rho(x),\rho_1(x)$ satisfy the standard assumptions $\textup{(\ref{book_new-9-1-6})}_{1,2}$ of \cite{futurebook}. Then
\begin{phantomequation}\label{24-2-3}\end{phantomequation}
\begin{equation}
\mu_\eff= \mu\rho_1\gamma\rho^{-1}, \qquad  h_\eff = h\rho^{-1}\gamma^{-1}.
\tag*{$\textup{(\ref*{24-2-3})}_{1,2}$}\label{24-2-3-*}
\end{equation}

Let us introduce a \emph{semiclassical zone\/}
$X'=\{ x\colon  \rho \gamma \ge h\}$ and a \emph{singular zone\/}
$X''=\{ x\colon  \rho \gamma \le 2h\}$ by (\ref{23-2-5}) and (\ref{23-2-6}) of \cite{futurebook}respectively.

Further, let us introduce two other overlapping zones
$X'_1=\{x\in X_\scl\colon \mu \rho _1 \le 2c\rho \gamma ^{-1}\}$ and
$X'_2=\{x\in X'\colon \mu \rho _1\ge c\rho \gamma ^{-1}\}$
where the magnetic field $\mu_\eff=\mu\rho_1\rho^{-1}\gamma$  is \emph{normal\/} ($\mu_\eff \le 2c$) and where it is \emph{strong\/} ($\mu_\eff \ge c$) respectively (see (\ref{23-2-7})  and (\ref{23-2-8}) of \cite{futurebook}. We also assume that
\begin{equation}
|F|\ge \epsilon \rho _1\qquad \text{in \ \ } X'_2
\label{24-2-4}
\end{equation}
where $F_{jk}$, $F^j$ and $F$ are the tensor, vector (as $d=3$) and scalar intensities of the magnetic field respectively.
 Moreover, let us assume that
 \begin{equation}
u|_{\partial X\cap B\bigl(x,\gamma (x)\bigr)}=0\qquad
\forall x\in X'_2\quad \forall u\in \fD(A);
\label{24-2-5}
\end{equation}
we do not need $\textup{(\ref{23-2-10})}_1$ since in $3\D$ the boundary does not lead to the deterioration of the remainder estimate.
We define  $X'_{-}=\{x\in X'_2\colon  V+\mu h F \ge \epsilon \rho^2\}$ and
$X'_{2+}=\{x\in X'\colon  V+\mu h F \le 2\epsilon \rho^2\}$ by (\ref{23-2-10}) and (\ref{23-2-11}) of \cite{futurebook} respectively.

Finally, let the standard boundary regularity condition  be fulfilled:
\begin{claim}\label{24-2-6}
For every $y\in X$,
$\partial X\cap B(y,\gamma (y))=\{x_k=\phi _k(x_{\hat{k}})\}$ with
\begin{equation*}
|D ^\alpha \phi _k|\le c\gamma ^{-|\alpha |}
\end{equation*}
and $k=k(y)$.
\end{claim}

Recall that according to Chapter~\ref{book_new-sect-13} of \cite{futurebook} the contribution of the partition element $\psi\in \sC_0^K(B(y,\frac{1}{2}\gamma(y))$ to the principal part of asymptotics is
\begin{equation}
\cN^-(\mu, h)= \cN^{\MW\,-}(\mu, h)\coloneqq
 h^{-3}\int \cN^\MW (x,\mu h)\psi(x)\,dx
\label{24-2-7}
\end{equation}
with $\cN^\MW (x,\mu h)$ given by \textup{(\ref{book_new-13-1-9}) of \cite{futurebook}} with $d=3$.

On the other hand, its contribution to the remainder  does not exceed
$Ch^{-2} \rho^2\gamma^2$ if $\mu\rho_1 \gamma\le c\rho $ \underline{or}
$\mu\rho_1 \gamma\ge c \rho $ but   $\mu h\rho_1\le \rho^2$, $y\in X'_+$
and non-degeneracy assumption
\begin{phantomequation}\label{24-2-8}\end{phantomequation}
\begin{equation}
\sum_{\alpha:|\alpha|\le k}
|\nabla^\alpha (VF^{-1}+(2n+1)\mu h) |\gamma^{|\alpha|}\ge
\epsilon \rho^2\rho_1^{-1}
\qquad\forall n\in \bZ^+
\tag*{$\textup{(\ref*{24-2-8})}_k$}\label{24-2-8-*}
\end{equation}
is fulfilled\footnote{\label{foot-24-3} It does not exceed the same expression
plus $\mu h^{-1-\delta} \rho_1\rho^\delta\gamma^{2+\delta}$ in the general case; here $\delta>0$ is arbitrarily small but $K$ in \ref{24-2-2-*} depends on it.}, \underline{and}
it does not exceed $ C(\mu ^{-s}\rho_1^{-s}\gamma^{-2s})$ if
$C\mu\rho_1 \gamma\ge \rho $, $\mu h\rho_1\le \rho^2$, $y\in X'_{2-}$.

Then we get estimate of $\N^- $ from below by the magnetic Weyl approximation $\cN^-(\mu, h)$ minus corresponding remainder, and also from above  by magnetic Weyl approximation plus corresponding remainder, provided $X=X'$ (so, there is no singular zone $X''=\emptyset$):
\begin{multline}
h^{-d} \int _{X'} \cN^\MW (x,\mu h)\,dx - CR_1  \le \N^- (0) \le\\
h^{-d} \int _{X'} \cN^\MW (x,\mu h)\,dx + CR_1 +C'R_2
\label{24-2-9}
\end{multline}
with
\begin{gather}
R_1= \mu^{-1}h^{1-d} \int_{X'_+} \rho^{d-1}\gamma^{-1}\,dx , \label{24-2-10}\\
R_2= \mu h^{s-d}\int_{X'_-} \rho_1 \rho^{d-s-1}\gamma^{1-s}\,dx
\label{24-2-11}
\shortintertext{provided}
\mu \rho_1 \gamma \ge \rho
\label{24-2-12}
\end{gather}
where the latter condition could be assumed without any loss of the generality, $C'$ depens also on $s$ and $\epsilon$\,\footnote{\label{foot-24-4} Cf. (\ref{23-2-18})--(\ref{23-2-21}) of \cite{futurebook}.}.

We leave to the reader the following not very challenging set of problems:

\begin{Problem}\label{Problem-24-2-1}
\begin{enumerate}[label=(\roman*), wide, labelindent=0pt]
\item\label{Problem-24-2-1-i}
Consider in the current framework Problems~\ref{Problem-23-2-1} \ref{Problem-23-2-1-i}, \ref{Problem-23-2-1-ii}, \ref{Problem-23-2-1-iv} of \cite{futurebook}.
\item\label{Problem-24-2-1-ii}
Using results of Subsubsection~\emph{\ref{book_new-sect-13-7-2-1}.1 of \cite{futurebook}~\nameref{book_new-sect-13-7-2-1} of \cite{futurebook}\/} replace condition \ref{24-2-8-*} by $\textup{(\ref{book_new-13-7-19})}_{m}$.
\end{enumerate}
\end{Problem}

In what follows $h\to +0$ and the semiclassical zone $X'$ expands to $X$ while $\mu$ is either bounded (then we can assume that the zone of the strong magnetic field $X'_2$ is fixed) or tends to $\infty$ (then $X'_2$ expands to $X$). We assume that all conditions of the previous subsection are fulfilled with $\mu=h=1$ but we will assume them fulfilled in the corresponding zones.

The other important question is whether  $\mu h\to 0$, remains bounded and disjoint from $0$ or tends to $\infty$.

Finally, we should consider the singular zone $X''$.  In order to avoid this task we assume initially that
\begin{equation}
\rho _1\gamma ^2+\rho \gamma \ge \epsilon.
\label{24-2-13}
\end{equation}

\subsection{Power singularities}
\label{sect-24-2-1-2}

\begin{example-foot}\label{example-24-2-2}\footnotetext{\label{foot-24-5} Cf. Example~\ref{example-23-2-3} of \cite{futurebook}.}
\begin{enumerate}[label=(\roman*), wide, labelindent=0pt]
\item\label{example-24-2-2-i}
Let $X$ be a compact domain, $0\in \bar{X}$  and let conditions \ref{24-2-8-*}, (\ref{24-2-1}), \ref{24-2-2-*}, (\ref{24-2-4}) and (\ref{24-2-6}) be fulfilled
 with  $\gamma =\epsilon_0 |x|$, $\rho =|x|^m$, $\rho _1=|x|^{m_1}$,
 $m_1<2m\le -2$\,\footnote{\label{foot-24-6} Such potential $(V_1,V_2,V_3)$ exists, see Appendix~\ref{sect-24-A-3}.}\footnote{\label{foot-24-7} The non-degeneracy condition $\textup{(\ref{24-2-8})}_2$  is fulfilled in the vicinity of $0$ if $V,F$ stabilize as $x\to 0$ to $V^0,F^0$ positively homogeneous of degrees $2m, m_1$ respectively.}.

Then, for the Schr\"odinger operator the following asymptotic holds as $h\to +0$, $\mu h$ bounded:
\begin{gather}
\N^-(\mu,h)=\cN^-(\mu,h)+
\left\{\begin{aligned}
&O(h^{-2}(\mu h)^{2(m+1)/(2m-m_1)}) &&m<-1,\\
&O(h^{-2}(|\log \mu h|+1)) &&m=-1
\end{aligned}\right.
\label{24-2-14}
\shortintertext{with}
\cN ^- (\mu ,h)=
\left\{\begin{aligned}
&O(h^{-3}(\mu h)^{2(m+1)/(2m-m_1)})  &&m<-1,\\
&O(h^{-3}(|\log \mu h|+1)) &&m=-1.
\end{aligned}\right.
\label{24-2-15}
\end{gather}
Furthermore, one can replace in (\ref{24-2-15}) ``$=O$'' with ``$\asymp$" if
\begin{equation}
V\le -\epsilon \rho ^2\qquad
\text{in \ \ } \Gamma \cap\{|x|\le \epsilon \} \subset X
\label{24-2-16}
\end{equation}
where $\Gamma $ is an open non-empty sector (cone) with vertex at $0$, and $\mu h\le t$ with small enough $t>0$.

\item\label{example-24-2-2-ii}
Let $X$ be unbounded domain  and let conditions \ref{24-2-8-*}, (\ref{24-2-1}), \ref{24-2-2-*}, (\ref{24-2-4}) and (\ref{24-2-6})  be fulfilled
 with  $\gamma =\epsilon_0 \langle x\rangle$, $\rho \langle x\rangle^m$,
 $\rho _1=\langle x\rangle^{m_1}$, $m_1> 2m\ge -2$ \,\footref{foot-24-6}\footref{foot-24-6}.

Then for the Schr\"odinger operator asymptotics (\ref{24-2-14}) holds  as
$h\to +0$, $\mu h$ bounded.

Further, (\ref{24-2-15}) holds and one can replace ``$=O$'' by ``$\asymp$" if
\begin{equation}
V\le -\epsilon \rho ^2\qquad \text{in \ \ } \Gamma \cap\{|x|\ge c \} \subset X
\tag*{$\textup{(\ref*{24-2-16})}^\#$}\label{24-2-16-*}
\end{equation}
where $\Gamma $ is an open non-empty sector (cone) with vertex at $0$, and
$\mu h\le t$ with small enough $t>0$.
\end{enumerate}
\end{example-foot}

\begin{example-foot}\label{example-24-2-3}\footnotetext{\label{foot-24-8} Cf. Example~\ref{example-23-2-4} of \cite{futurebook}.}
\begin{enumerate}[label=(\roman*), wide, labelindent=0pt]
\item\label{example-24-2-3-i}
Assume now that  $m>-1$ while all other assumptions of Example~\ref{example-24-2-2}\ref{example-24-2-2-i} are fulfilled.  Then
$\cN^-(\mu,h)=O( h^{-3})$. Let us calculate the remainder estimate and prove that
\begin{equation}
\N^-(\mu,h)=\cN^-(\mu, h)+O(h^{-2}).
\label{24-2-17}
\end{equation}
Obviously, the contribution of the regular zone $X'=\{x\colon |x|\ge r^*=h^{1/(m+1)}\}$ is $O(h^{-2})$ and we need to consider the contribution of the singular zone
$X''=\{x\colon |x|\le r_2\}$.

\begin{enumerate}[label=(\alph*), wide, labelindent=0pt]
\item\label{example-24-2-3-ia}
Assume first that $m_1\ge m-1$. If $\mu r_1^{m_1+1-m}\le c$, then by virtue of LCR we estimate contribution of $X''$
\begin{equation}
Ch^{-3}\int _{X''} \bigl(r^{3m} +\mu ^{3}r^{3(m_1+1)}\bigr)\,dx \asymp
h^{-3}\bigl(r_1^{3m+3} + \mu ^3 r_1^{3m_1+6}\bigr) =O(1).
\label{24-2-18}
\end{equation}
Indeed, we can take $\vec{V}=O(r^{m_1+1})$.

On the other hand, if $\mu r_1^{m_1+1-m}\ge c$, the same estimate would work for
$X'''=\{x\colon |x|\le r_{**}=\mu^{-1/(m_1+1-m)}\}$ while contribution of $X''\setminus X'''$ does not exceed due to Chapter~\ref{book_new-sect-13} of \cite{futurebook}
\begin{equation}
C\int _{X''\setminus X'''}
(\mu r^{m_1+1-m})^{-s}r^{-3}\,dx =O(|\log h|)
\label{24-2-19}
\end{equation}
(actually it is $O(1)$ if $m_1>m-1$).

\item\label{example-24-2-3-ib}
Let now $m_1< m-1$. If $\mu r_1^{m_1+1-m}\ge c$, then we can apply estimate (\ref{24-2-19}) in the whole zone $X''$.

 On the other hand, if $\mu r_1^{m_1+1-m}\le c$, then we can apply estimate (\ref{24-2-19}) in the zone $X'''$ and estimate (\ref{24-2-18}) in $X''\setminus X''$.
\end{enumerate}

\item\label{example-24-2-3-ii}
Assume now that $m<-1$  while all other assumptions of Example~\ref{example-24-2-2}\ref{example-24-2-2-ii} are fulfilled. Again, considering cases
\begin{enumerate}[label=(\alph*), wide, labelindent=0pt]
\item\label{example-24-2-3-iia}
$m_1 \le  m-1$ and
\item\label{example-24-2-3-iib}
$m_1> m-1$
\end{enumerate}
we arrive to the asymptotics (\ref{24-2-17}) and $\cN^-(\mu,h)=O(h^{-3})$.
\end{enumerate}
\end{example-foot}

Consider now fast increasing $\mu $ so that $\mu h\to \infty$. We will get non-trivial results only when domain defined by $\mu_\eff h_\eff \le C_0$ shrinks but remains non-empty which happens only if $m_1>2m$,  $m_1<2m$ in the frameworks of Example~\ref{example-24-2-3}\ref{example-24-2-3-i} and \ref{example-24-2-3-ii} respectively.

\begin{example-foot}\label{example-24-2-4}\footnotetext{\label{foot-24-9} Cf. Example~\ref{example-23-2-5} of \cite{futurebook}.}
\begin{enumerate}[label=(\roman*), wide, labelindent=0pt]
\item\label{example-24-2-4-i}
In the framework of Example~\ref{example-24-2-3}\ref{example-24-2-3-i} with $m_1>2m$  consider  $\mu h\to \infty$. Then the  allowed domain is
\begin{equation}
\{x\colon  |x|\lesssim r_2 = (\mu h)^{-1/(m_1-2m)}\}
\label{24-2-20}
\end{equation}
and we have $r_1\le r_2$ if $\mu \lesssim h^{-(m_1+1-m)/(m+1)}$ while for
$\mu \gtrsim h^{-(m_1+1-m)/(m+1)}$ inequalities go in the opposite direction.

Therefore as $h\to +0$, $ch^{-1}\le \mu \le h^{-(m_1+1-m)/(m+1)}$ asymptotics
\begin{equation}
\N^- (\mu,h) = \cN^-(\mu,h) +O(\mu^{-2(m+1)/(m_1-2m)}h^{-2(m_1+ 1-m)/(m_1-2m)})
\label{24-2-21}
\end{equation}
holds and one can see easily that
$\cN^- (\mu,h)\asymp h^{-2}r_2^{2m+2}$:
\begin{equation}
\cN^- (\mu,h) \asymp \mu ^{-3(m+1)/(m_1-2m)}h^{-3(m_1+ 1-m)/(m_1-2m)}.
\label{24-2-22}
\end{equation}
\item\label{example-24-2-4-ii}
Similarly, in the framework of Example~\ref{example-24-2-3}\ref{example-24-2-3-ii} with $m_1<2m$  asymptotics (\ref{24-2-21}) and (\ref{24-2-22}) hold as $h\to +0$,
$ch^{-1}\le \mu \le h^{-(m_1+1-m)/(m+1)}$.
\end{enumerate}
\end{example-foot}

\subsection{Improved remainder estimates}
\label{sect-24-2-1-3}
Let us improve remainder estimates under certain non-periodicity-type assumptions.

\begin{example-foot}\label{example-24-2-5}\footnotetext{\label{foot-24-10} Cf. Example~\ref{example-23-2-7} of \cite{futurebook}.}
\begin{enumerate}[label=(\roman*), wide, labelindent=0pt]
\item\label{example-24-2-5-i}
In the case of the singularity at $0$ with $m>-1$ the contribution to the remainder of the zone $\{x\colon  |x|\le \varepsilon\}$ does not exceed
$\sigma h^{-1}$ with $\sigma=\sigma(\varepsilon)\to 0$ as $\varepsilon\to +0$. Then the standard arguments imply that under the standard non-periodicity assumption for Hamiltonian billiards\footnote{\label{foot-24-11} On the energy level $0$.} with the Hamiltonian
\begin{gather}
a (x,\xi,\mu_0)  =
\sum _{j,k} g^{jk}(\xi_j-\mu_0V_j)(\xi_k -\mu_0V_k)+V(x)
\label{24-2-23}\\
\intertext{the improved asymptotics}
\N^- (\mu,h) =\cN^- (\mu,h) + \kappa_1 h^{-2} +o(h^{-2})
\label{24-2-24}
\end{gather}
holds as $h\to +0$, $\mu\to \mu_0$ where $\kappa_1h^{-2}$ is the contribution of $\partial X$ calculated as $\mu=\mu_0$.

\item\label{example-24-2-5-ii}
Similarly in the case of the singularity at infinity with $m<-1$  under the standard non-periodicity assumption for Hamiltonian billiards\footref{foot-24-11}  with the Hamiltonian (\ref{24-2-23})  asymptotics (\ref{24-2-24}) holds as $h\to +0$, $\mu\to \mu_0$.

\item\label{example-24-2-5-iii}
In the case of the singularity at $0$ with $m< -1$ (and thus $m_1<2m$) and $h\to 0$, $\mu h\to 0$ the contributions to the remainder of the zone
$\{x\colon  |x|\ge \varepsilon^{-1} r _2\}$ with $r_2= (\mu h)^{-1/(2m-m_1)}$
do not exceed $\sigma h^{-2} r_2^{m+1}$ with $\sigma=\sigma(\varepsilon)\to 0$ as $\varepsilon\to +0$. After scaling $x\mapsto xr_2^{-1}$ etc the magnetic field in the zone  $\{x\colon   |x|\le \varepsilon ^{-1}r _2\}$ becomes strong.

Assume that  $g^{jk},V_j, V$ stabilize to positively homogeneous of degrees
$0, m_1+1, 2m$ functions $g^{jk0}, V_j^0$, $V^0$ as $x\to 0$: namely, assume that $\textup{(\ref{23-2-39})}_{1-3}$ of \cite{futurebook} are fulfilled.
Then the standard arguments imply that under the standard non-periodicity assumption for $1$-dimensional Hamiltonian movement along magnetic lines  (see Subsection~\ref{book_new-sect-13-6-2} of \cite{futurebook}) the improved asymptotics
\begin{equation}
\N^-(\mu ,h)=\cN (\mu ,h)+
o\bigl(h^{-2}(\mu h)^{-2(m+1)/(m_1-2m)}\bigr)
\label{24-2-25}
\end{equation}
holds as $h\to +0$, $\mu h\to 0$. Furthermore, if $0$ is not an inner point and domain stabilizes to the conical $X^0$ near it, we need to consider movement along magnetic lines with reflection at $\partial X^0$ and include into asymptotics the term $\kappa_1 h^{-2}(\mu h)^{-2(m+1)/(m_1-2m)}$ which comes out from the contribution of $\partial X^0$:
\begin{multline}
\N^-(\mu ,h)=\cN (\mu ,h)+ \kappa_1 h^{-2}(\mu h)^{-2(m+1)/(m_1-2m)}+\\
o\bigl(h^{-2}(\mu h)^{-2(m+1)/(m_1-2m)}\bigr)
\label{24-2-26}
\end{multline}

\item\label{example-24-2-5-v}
In the case of the singularity at $0\in X$  with $m_1<2m=-2$ the main contribution to the remainder comes from the zone
$\{x\colon \varepsilon ^{-1} r_2\le |x|\le \varepsilon \}$. After rescaling magnetic field in this zone is strong.

Let stabilization conditions $\textup{(\ref{23-2-39})}_{1,3}$ of \cite{futurebook} be fulfilled.
Then the standard arguments imply that under the same non-periodicity assumption as in \ref{example-24-2-5-iii} the improved asymptotics
\begin{equation}
\N^-(\mu ,h)=\cN (\mu ,h)+ o(h^{-2}|\log (\mu h)|)
\label{24-2-27}
\end{equation}
holds as $h\to +0$, $\mu h\to 0$. Furthermore, if $0$ is not an inner point and domain near it stabilizes to the conical $X^0$ near it,
\begin{equation}
\N^-(\mu ,h)=\cN (\mu ,h)+ \kappa_1 h^{-2} \log (\mu h) +
o(h^{-2}|\log (\mu h)|);
\label{24-2-28}
\end{equation}
again, an extra term  comes out from the contribution of $\partial X^0$.

\end{enumerate}
\end{example-foot}

\begin{remark}\label{rem-24-2-6}
Statements, similar to \ref{example-24-2-5-iii}, \ref{example-24-2-5-v} but with the singularity at infinity seem to have impossible conditions.
\end{remark}

\subsection{Power singularities. II}
\label{sect-24-2-1-4}

Let us modify  our arguments for the case $\rho_3<1$. Namely, in addition to \ref{24-2-2-*} we assume that
\begin{phantomequation}\label{24-2-29}\end{phantomequation}
\begin{align}
 &|D^\alpha g^{jk}|\le c\rho_2\gamma ^{-|\alpha |},\qquad
|D^\alpha F_{jk}|\le c\rho_2\rho _1 \gamma ^{-|\alpha |},
\tag*{$\textup{(\ref*{24-2-29})}_{1,2}$}\label{24-2-29-1}\\
&|D^\alpha {\frac{V}{F}}|\le
c\rho_3\rho ^2\rho _1^{-1}\gamma ^{-|\alpha |}\qquad \qquad
\forall \alpha \colon 1\le |\alpha |\le K
\tag*{$\textup{(\ref*{24-2-29})}_3$}\label{24-2-29-3}
\end{align}
with $\rho_3\le \rho_2\le 1$ in the
corresponding regions where $\rho ,\rho _1,\gamma , \rho_3 $ are scaling functions.

Recall that in Chapter~\ref{book_new-sect-13} of \cite{futurebook} operator was reduced to the canonical form with the term, considered to be negligible, of magnitude $\rho_2\mu_\eff^{-2N}$. In this case impose non-degeneracy assumptions
\begin{phantomequation}\label{24-2-30}\end{phantomequation}
\begin{gather}
\sum_{\alpha:  |\alpha|\le k}
|\nabla^\alpha (v^*+(2n+1)\mu h) |\gamma^{|\alpha|}\ge \epsilon  \rho_3\rho^2
\qquad\forall n\in \bZ^+
\tag*{$\textup{(\ref*{24-2-30})}^*_k$}\label{24-2-30-*}\\
\shortintertext{and}
\rho_3 \ge C_0\rho_2 (\mu \rho _1 \gamma \rho^{-1})^{-N}
\label{24-2-31}
\end{gather}
where $v^*$ is what this reduction transforms  $VF^{-1}$ to and (\ref{24-2-31})  means that ``negligible'' terms do not spoil \ref{24-2-30-*}.

Then according to Chapter~\ref{book_new-sect-13} of \cite{futurebook} the contributions of $B(x,\gamma)$ to the both to the Tauberian remainder and an approximation error do not exceed $C\rho ^2\gamma^{-1}h^{-2}$.

\begin{example-foot}\label{example-24-2-7}\footnotetext{\label{foot-24-12} Cf. Example~\ref{example-23-2-12} of \cite{futurebook}.}
\begin{enumerate}[label=(\roman*), wide, labelindent=0pt]
\item\label{example-24-2-7-i}
Let $0\in \bar{X}$ be a singular point and let assumptions $\textup{(\ref{24-2-29})}_{1-3}$  be fulfilled  with  $\gamma =\epsilon_0 |x|$,
$\rho =|x|^{-m}\bigl(\bigl|\ln |x|\bigr|+1\bigr)^\alpha $,
$\rho_1 =|x|^{-2m}\bigl(\bigl|\ln |x|\bigr|+1\bigr)^\beta $,
$\rho_2=1$ and $\rho_3 =\bigl(\bigl|\ln |x|\bigr|+1\bigr)^{-1}$.

Assume that
\begin{equation}
m=-1, \beta > \max (\alpha,2\alpha).
\label{24-2-32}
\end{equation}

Then (\ref{24-2-31}) is fulfilled with $N=1$ and we can replace \ref{24-2-30-*} by
\begin{equation}
\sum_{\alpha:  |\alpha|\le k}
|\nabla^\alpha (VF^{-1}+(2n+1)\mu h) |\gamma^{|\alpha|}\ge \epsilon \rho_3\rho^2
\qquad\forall n\in \bZ^+.
\tag*{$\textup{(\ref*{24-2-30})}_m$}\label{24-2-30-**}
\end{equation}
 Therefore, we conclude that the remainder is $O(R)$ with
\begin{equation}
R\coloneqq   \left\{\begin{aligned}
&h^{-2} &&  \alpha<-\frac{1}{2},\\
&h^{-2}|\log (\mu h)|  &&\alpha=-\frac{1}{2},\\
&h^{-2}(\mu h)^{-(2\alpha+1)/(\beta-2\alpha)} && 2\alpha <\beta <2\alpha+2.
\end{aligned}\right.
\label{24-2-33}
\end{equation}
One can see easily that under condition (\ref{24-2-16}) for $\mu h\le t$ with small enough $t>0$
\begin{equation}
\cN ^-(\mu ,h)\asymp
\left\{\begin{aligned}
&h^{-3} &&\alpha<-\frac{1}{3},\\
&h^{-3}(|\log (\mu h)| +1) &&\alpha=-\frac{1}{3},\\
&h^{-3}(\mu h)^{-(3\alpha+1)/(\beta-2\alpha)} &&\alpha>-\frac{1}{3}.
\end{aligned}\right.
\label{24-2-34}
\end{equation}

\item\label{example-24-2-7-ii}
Let infinity be a singular point and let  conditions $\textup{(\ref{24-2-29})}_{1-3}$  be fulfilled  with  $\gamma =\epsilon_0 \langle x\rangle $,
$\rho =\langle x\rangle ^m\bigl(\log\langle x\rangle +1\bigr) ^\alpha $,
${\rho _1= \langle x\rangle ^{2m}\bigl(\log\langle x\rangle +1\bigr)^\beta }$,
$\rho_2=1$, $\rho_3 =\bigl(\log\langle x\rangle +1\bigr) ^{-1}$.
Assume that (\ref{24-2-32}) is fulfilled.

Then all the statements of \ref{example-24-2-7-i} remain true with the obvious modification: condition (\ref{24-2-16}) should be replaced by \ref{24-2-16-*}.
\end{enumerate}
\end{example-foot}

\begin{problem}\label{problem-24-2-8}\footref{foot-24-12}
Consider
\begin{enumerate}[label=(\roman*), wide, labelindent=0pt]
\item\label{problem-24-2-8-i}
Case of a singular point at $0$, $m<-1$, $\beta>2\alpha$.
\item\label{problem-24-2-8-ii}
Case of a singular point at infinity, $m<-1$, $\beta>2\alpha$.
\end{enumerate}
In both cases $\log (\cN^-(\mu,h)h^3)\asymp (\mu h)^{-1/(\beta-2\alpha)}$ and
$\log (\cN^-(\mu,h)/R^{3/2})\sim 1$.
\end{problem}

\subsection{Exponential singularities}
\label{sect-24-2-1-6}

Consider now singularities of the exponential type.

Since we assume only that either $0\in \bar{X}$ or $X$ is a unbounded domain \underline{and} $\gamma (x)\ll |x|$, it can happen that
\begin{equation}
\mes (\{x\in X\colon |x|\le r\})\asymp r^n
\label{24-2-35}
\end{equation}
with $n \ne  d$ (more precisely, $n\ge d$ if $0$ is a singular point and
$n\le d$ is infinity is a singular point). See Figure~\ref{book_new-fig-11-Y} of \cite{futurebook} for unbounded domains; for singularity at $0$ it is a spike:

\begin{figure}[h]
\centering
\begin{tikzpicture}
\fill[cyan!20, domain=0:6,scale=1, smooth]  plot function{.05*x**2}
-- plot[domain=6:0] function{-.05*x**2};
\end{tikzpicture}
\caption{\label{fig-24-1} Domains of the type (\ref{24-2-35})}
\end{figure}

\begin{example-foot}\label{example-24-2-9}\footnotetext{\label{foot-24-13} Cf. Example~\ref{example-23-2-17} of \cite{futurebook}.}
\begin{enumerate}[label=(\roman*), wide, labelindent=0pt]
\item\label{example-24-2-9-i}
Let $0\in \bar{X}$ be a singular point and let our standard assumptions be fulfilled  with
$\gamma =\epsilon_0 |x|^{1-\beta}$, $\rho =\exp (a|x|^\beta )$,
$\rho _1=\exp (b|x| ^\beta )$ with  $\beta <0$, $b>2a>0$.

Then  for  $\mu h <t$ with a small enough constant $t>0$
\begin{align}
&\N^-(\mu,h)=\cN^-(\mu,h) +
O(h^{-2}(\mu h)^{-2a/(b-2a)}|\log (\mu h)|^{(n-1)/\beta}),
\label{24-2-36}\\
&\cN^-(\mu,h)= O (h^{-3}(\mu h)^{-3a/(b-2a)}|\log (\mu h)|^{n/\beta-1}).
\label{24-2-37}
\end{align}

\item\label{example-24-2-9-ii}
Let infinity be a singular point and let our standard assumptions be fulfilled  with $\gamma =\epsilon_0 \langle x\rangle ^{1-\beta }$,
$\rho =\exp (a\langle x\rangle ^\beta )$,
$\rho_1 =\exp (b\langle x\rangle ^\beta )$ where  $\beta >0$, $b>2a>0$.

Then  asymptotics (\ref{24-2-36})--(\ref{24-2-37}) holds.
\end{enumerate}
\end{example-foot}

We leave to the reader the following

\begin{problem}\label{problem-24-2-10}
\begin{enumerate}[label=(\roman*), wide, labelindent=0pt]
\item\label{problem-24-2-10-i}
Consider the case of a singular point at $0$ and
$\gamma =\epsilon_0 |x|^{1-\beta }$, $\rho =|x|^m\exp (|x| ^\beta )$,
$\rho_1 =|x|^{m_1}\exp (2|x| ^\beta )$, $\rho _2=1$, $\rho_3=|x|^{-\beta }$, $\beta <0$, $m_1<2m$.

\item\label{problem-24-2-10-ii}
Consider the case of a singular point at infinity and
$\gamma =\epsilon _0 \langle x \rangle ^{1-\beta }$,
$\rho =\langle x\rangle ^m\exp (\langle x\rangle ^\beta )$,
$\rho_1 =\langle x\rangle ^{m_1}\exp (2\langle x\rangle ^\beta) $,
$\rho _2=1$, $\rho_3=\langle x\rangle ^{-\beta }$,
$\beta >0$, $m_1>2m$.
\end{enumerate}
\end{problem}

We leave to the reader:

\begin{Problem}\label{Problem-24-2-11}
Consider  cases of $\mu \to \mu_0>0$ and $\mu \to \mu_0=0$.
\end{Problem}

\section{Schr\"odinger-Pauli operators}
\label{sect-24-2-2}

Consider now {Schr\"odinger-Pauli operators, either genuine (\ref{book_new-0-34})  or generalized (\ref{book_new-13-5-3}) of \cite{futurebook}. The principal difference is that now  $F$ does not ``tame'' singularities of $V$, on the contrary, it needs to be ``tamed'' by itself. As a result there are fewer examples than for the Schr\"odinger. Also we do not have a restriction $\mu_\eff h_\eff =O(1)$ which we had in the most of the previous Subsection~\ref{sect-24-2-1}.

\begin{example-foot}\footnotetext{\label{foot-24-14} Cf. Example~\ref{example-24-2-2}.}\label{example-24-2-12}
\begin{enumerate}[label=(\roman*), wide, labelindent=0pt]
\item\label{example-24-2-12-i}
Let $0$ be a singular point  and let our standard assumptions be fulfilled with $\gamma =\epsilon_0 |x|$, $\rho =|x|^m$,
$\rho _1=|x|^{m_1}$ and let $ m> -1$, $2m\ne m_1>-2$.  Then, in comparison with the theory of the previous Subsection~\ref{sect-24-2-1}, we need to consider also the  the zone $\mu _\eff h_\eff \gtrsim 1$, and also a singular zone where $h_\eff\gtrsim 1$.

The contribution to the remainder of the regular part  (defined by
$\mu _\eff h_\eff \gtrsim 1$, $h_\eff \lesssim 1$) does not exceed
$C\mu h^{-1}\int \rho_1\gamma^{-1}\,dx$ while its contribution to
$\cN^-(\mu, h)$ does not exceed $C\mu h^{-2}\int \rho_1\rho\,dx$. Indeed, contributions of each $\gamma$-element do not exceed
$C\mu_\eff h_\eff^{-1}\asymp C\mu h^{-1}\rho_1\gamma^2$ and
$C\mu_\eff h_\eff^{-2}\asymp C \mu h^{-2}\rho_1\rho\gamma^3$.

Consider now the singular zone $\{x\colon |x|\le h^{1/(m+1)}\}$. We claim that

\begin{claim}\label{24-2-38}
Contribution of the singular zone to the asymptotics does not exceed
$C(\mu h+1)$.
\end{claim}

Indeed, if $\mu h\le 1$ it suffices to apply LCR as in $2\D$-case. If
$\mu h\ge 1$, then LCR  returns $\mu^{3/2}h^{1/2}$ and we need to be more tricky. Without any loss of the generality we can assume that $V_1=0$ and then apply (variant of) LCR for $1\D$-operator $h^2D_1^2 -|x|^{2m}$ and for $2\D$-operator $h^2|D|^2 - \mu ^2|x|^{m_1} -|x|^{2m}$. We leave details to the reader.
\begin{align}
&N^-(\mu ,h) = \cN^-(\mu, h) + O(\mu h^{-1}+ h^{-2})
\label{24-2-39}\\
\shortintertext{with}
&\cN^-(\mu,h)= O(\mu h^{-2}+ h^{-3})
\label{24-2-40}
\end{align}
where under condition (\ref{24-2-16}) fulfilled in non-empty cone ``$=O(\cdot)$''could be replaced by ``$\asymp \cdot$''.

\item\label{example-24-2-12-ii}
Let infinity be a  singular point and let our standard assumptions be fulfilled with $\gamma =\epsilon_0 \langle x \rangle$, $\rho =\langle x \rangle^m$,
$\rho _1=\langle x \rangle^{m_1}$ and let $ m< -1$, $2m\ne m_1<-2$. Then the same arguments yield (\ref{24-2-39})--(\ref{24-2-40}) again.
\end{enumerate}
\end{example-foot}

We leave to the reader the following problems:\enlargethispage{2\baselineskip}

\begin{problem}\label{problem-24-2-13}
\begin{enumerate}[label=(\roman*), wide, labelindent=0pt]
\item\label{problem-24-2-13-i}
As $0$ is a singular point consider both Schr\"odinger and Schr\"odinger-Pauli operators as
\begin{enumerate}[label=(\alph*), wide, labelindent=0pt]
\item\label{problem-24-2-13-ia}
$\gamma=|x|$, $\rho= |x|^{m}$, $m>-1$ and $\rho_1=|x|^{-2}|\log |x||^\beta$\,\footnote{\label{foot-24-15} One should take $\beta<-1$ for the Schr\"odinger-Pauli operator.}.
\item\label{problem-24-2-13-ib}
$\gamma=|x|$, $\rho= |x|^{-1}|\log |x||^\alpha$, $\alpha<-\frac{1}{2}$  and $\rho_1=|x|^{m_1}|$, $m_1>-2$.
\end{enumerate}
\item\label{problem-24-2-13-ii}
As infinity is a point consider both Schr\"odinger and Schr\"odinger Pauli operators as
\begin{enumerate}[label=(\alph*), wide, labelindent=0pt]
\item\label{problem-24-2-13-iia}
$\gamma=|x|$, $\rho= |x|^{m}$, $m<-1$ and $\rho_1=|x|^{-2}|\log |x||^\beta$\,\footref{foot-24-15}.
\item\label{problem-24-2-13-iib}
$\gamma=|x|$, $\rho= |x|^{-1}|\log |x||^\alpha$, $\alpha<-\frac{1}{2}$\,\footref{foot-24-15}  and $\rho_1=|x|^{m_1}$, $m_1>-2$.
\end{enumerate}
\end{enumerate}

For the Schr\"odinger operator in cases (a) non-trivial results could be obtained even as $\mu h\to +\infty$.
\end{problem}

\begin{problem}\label{problem-24-2-14}
Let either $0$ or infinity be a singular point.

Using the same arguments and combining them with the arguments of Subsubsection~\emph{\ref{sect-24-2-1-4}.6. \nameref{sect-24-2-1-4}\/} consider both Schr\"odinger-Pauli and Schr\"odinger operators with $\gamma=|x|$, $\rho=|x|^{-1}|\log |x||^\alpha$, $\rho_1=|x|^{-2}|\log |x||^\beta$, $\alpha <-\frac{1}{2}$, $\beta <-1$.

For the Schr\"odinger operator in case $\beta<2\alpha$  non-trivial results could be obtained even as $\mu h\to +\infty$.
\end{problem}

The following problem seems to be rather challenging:

\begin{Problem}\label{problem-24-2-15}
Investigate $\rho_1=|x|^{-2}$.
\end{Problem}

\section{Dirac operator}
\label{sect-24-2-3}

\subsection{Preliminaries}
\label{sect-24-2-3-1}
Let us now consider the generalized magnetic Dirac operator (\ref{book_new-17-1-1}) of \cite{futurebook}
either
\begin{equation}
A=
\frac{1}{2}\sum_{l,j} \upsigma _l\bigl(\omega ^{jl}P_j+P_j\omega ^{jl}\bigr)
+\upsigma _0M+I\cdot V, \qquad P_j=hD_j-\mu V_j
\label{24-2-41}
\end{equation}
where  $\upsigma_0,\upsigma_1,\upsigma_2,\upsigma_3$ are  $4\times 4$-matrices
and $M>0$ or
\begin{equation}
A=
\frac{1}{2}\sum_{l,j} \upsigma _l\bigl(\omega ^{jl}P_j+P_j\omega ^{jl}\bigr)
+I\cdot V, \qquad P_j=hD_j-\mu V_j
\label{24-2-42}
\end{equation}
where  $\upsigma_1,\upsigma_2,\upsigma_3$ are $2\times 2$-matrices and $M=0$.

We are interested in $\N(\tau_1,\tau_2)$, the number of eigenvalues in $(\tau_1,\tau_2)$\,\footnote{\label{foot-24-16} Assuming that this interval does not contain essential spectrum; otherwise $\N(\tau_1,\tau_2)\coloneqq   \infty$.  It is more convenient for us to exclude both ends of the segment.} with $\tau_1<\tau_2$, fixed in this subsection.

In contrast to $2\D$-case relations similar to (\ref{23-2-73})--(\ref{23-2-74}) of \cite{futurebook} do not matter. Tthe Landau levels (at the point $x$) are
$V\pm \bigl(M^2+2j\mu hF\bigr)^{\frac{1}{2}}$ with $j=0,1,2,3,\ldots $ and magnetic Dirac operator \emph{always\/} behaves like Schr\"odinger-Pauli operator.

Therefore, we treat the operator given by (\ref{24-2-41}) under the following assumptions\begin{phantomequation}\label{24-2-43}\end{phantomequation}
\begin{align}
&|D ^\alpha \omega ^{jk}|\le c\gamma ^{-|\alpha |},\qquad
|D ^\alpha F|\le c\rho_1 \gamma ^{-|\alpha |},
\tag*{$\textup{(\ref*{24-2-43})}_{1-2}$}\label{24-2-43-1}\\
&|D ^\alpha V|\le
c\min\bigl(\rho , \frac{1}{M} \rho ^2\bigr)\gamma ^{-|\alpha |}
\qquad (\alpha\ne 0)\quad \forall \alpha \colon|\alpha |\le K,
\tag*{$\textup{(\ref*{24-2-43})}_{3}$}\label{24-2-43-3}
\end{align}
\vskip-25pt
\begin{align}
&(V-\tau _2-M)_+\le c\min\bigl(\rho ,{\frac{1}{M}}\rho ^2\bigr),
\tag*{$\textup{(\ref*{24-2-43})}_{4+}$}\label{24-2-43-4+}\\
&(V-\tau _1+M)_-\le c\min\bigl(\rho , \frac{1}{M} \rho ^2\bigr)
\tag*{$\textup{(\ref*{24-2-43})}_{4-}$}\label{24-2-43-4-}
\end{align}
and also \ref{24-2-2-*}, (\ref{24-2-1}) for
$g^{jk}=\sum_{l,r}\omega ^{jl}\omega ^{kr}\updelta _{lr}$
and (\ref{24-2-4}) (with $F>0$). In what follows $\textup{(\ref{24-2-43})}_4$  means the pair of conditions $\textup{(\ref{24-2-43})}_{4\pm }$.

Moreover, let condition (\ref{24-2-6}) be fulfilled and
\begin{gather}
\bar{X}''\cap\partial X=\emptyset,
\label{24-2-44}\\
|V_j|\le c\rho ,\quad |D_j\omega ^{kl}|\le c\rho \qquad \text{in \ \ } X''.
\label{24-2-45}
\end{gather}

Finally, we assume that
\begin{claim}\label{24-2-46}
\underline{Either} $\partial X=\emptyset$
\underline{or} $\mu =O(1)$ and $\partial X\cap X'_2=\emptyset$ (in
what follows).
\end{claim}

\subsection{Asymptotics}
\label{sect-24-2-3-2}

\begin{example-foot}\label{example-24-2-16}\footnotetext{\label{foot-24-17} Cf. Example~\ref{example-24-2-12}.}
\begin{enumerate}[label=(\roman*),wide, labelindent=0pt]
\item\label{example-24-2-16-i}
Let $0$ be an inner singular point and let all the above conditions,   be fulfilled with $\gamma =\epsilon _0 |x|$, $\rho =|x|^m$,
$\rho _1=|x|^{m_1}$, $m_1>2m$, $m>-1$.

Further, let assumption
\begin{equation}
|V|\ge \epsilon \min(\rho,\,\frac{\rho^2 }{M}) \qquad \forall x\colon |x|\le \epsilon
\label{24-2-47}
\end{equation}
and non-degeneracy condition \ref{24-2-8-*} be fulfilled. \begin{enumerate}[label=(\alph*)]
\item\label{example-24-2-16-ia}
$m<0$, $\tau_1<\tau_2$ \underline{or}
\item\label{example-24-2-16-ib}
$m>0$, $M>0$, $\tau_1=-M$, $\tau_2\in (-M,M)$.
\end{enumerate}

Then for $h\to+0$, $1\le \mu $ asymptotics
\begin{equation}
\N (\tau_1,\tau_2; \mu, h) =\cN (\tau_1,\tau_2; \mu, h) +O(\mu h^{-1}+h^{-2})
\label{24-2-48}
\end{equation}
holds with $\cN$ defined by $\textup{(\ref{book_new-17-1-12})}_{1}$ with $d=3$, $r=1$. Moreover,
$\cN  (\tau_1,\tau_2, \mu ,h)\asymp \mu h^{-2}+h^{-3}$.

\item\label{example-24-2-16-ii}
Let infinity be an inner singular point  and let all the above conditions  be
fulfilled with $\gamma =\epsilon _0 \langle x\rangle $,
$\rho =\langle x\rangle ^m$, $\rho _1=\langle x\rangle ^{m_1}$,
$m_1<2m$, $m<-1$.

Further, let assumption
\begin{equation}
|V|\ge \epsilon \min(\rho,\,\frac{\rho^2 }{M}) \qquad \forall x\colon |x|\ge c
\tag*{$\textup{(\ref*{24-2-47})}^\#$}\label{24-2-47-*}
\end{equation}
and non-degeneracy condition \ref{24-2-8-*} be fulfilled. Let $M>0$, $\tau_1=-M$, $\tau_2\in (-M,M)$.
Then for $h\to+0$, $1\le \mu $ asymptotics (\ref{24-2-48}) holds with $\cN$ defined by $\textup{(\ref{book_new-17-1-12})}_{1}$ with $d=3$, $r=1$. Moreover, $\cN  (\tau_1,\tau_2, \mu ,h)\asymp \mu h^{-2}+h^{-3}$.
\end{enumerate}
\end{example-foot}

\begin{Problem}\label{Problem-24-2-17}
Generalize results of this section to the odd-dimensional maximal-rank case. In particular, consider power singularities.
\end{Problem}

\chapter{$3\D$-case.  Asymptotics of large eigenvalues}
\label{sect-24-3}

In this section we consider the case when $\mu$ and $h$ are fixed and we consider the asymptotics of the eigenvalues, tending to $+ \infty$ and for Dirac operator also to $-\infty$.

\index{Schrodinger operator@Schr\"{o}dinger operator!magnetic with singularities!asymptotics of large eigenvalues}%
\index{operator!magnetic Schrodinger with singularities@magnetic Schr\"{o}dinger with singularities!asymptotics of large eigenvalues}%
\index{large eigenvalue asymptotics!magnetic Schrodinger operator with singularities@magnetic Schr\"{o}dinger with singularities}%
\index{asymptotics!large eigenvalue!magnetic Schrodinger operator with singularities@magnetic Schr\"{o}dinger with singularities}

Here we consider the case of the spectral parameter tending to $+\infty$ (and for the Dirac operator we consider $\tau\to -\infty$ as well).

\section{Singularities at the point}%
\label{sect-24-3-1}

We consider series of example with singularities at the point.

\subsection{Schr\"odinger operator}%
\label{sect-24-3-1-1}

\begin{example}\label{example-24-3-1}
\begin{enumerate}[label=(\roman*), wide, labelindent=0pt]
\item\label{example-24-3-1-i}
Let $X$ be a compact domain, $0\in \bar{X}$ and let conditions \ref{24-2-2-*} and (\ref{24-2-6}) be fulfilled with  $\gamma= \epsilon |x|$, $\rho=|x|^m$, $\rho_1=|x|^{m_1}$,  $m_1<2m$. Let
\begin{equation}
|F|\ge \epsilon_0 \rho_1 \qquad
\text{for\ \ } |x|\le \epsilon.
\label{24-3-1}
\end{equation}
Then for the Schr\"odinger operator as $\tau\to +\infty$
\begin{equation}
\N^-(\tau)= \cN^-(\tau)+O(\tau^{(d-1)/2})
\label{24-3-2}
\end{equation}
while $\cN^-(\tau)\asymp \tau^{d/2}$.

Indeed, we need to consider only case $m_1\le -2$ (otherwise it is covered by Section~\ref{book_new-sect-11-2} of \cite{futurebook}). Assume for simplicity, that $V=0$ (modification in the general case is trivial). Recall that $\mu_\eff= |x|^{m_1+1}\tau^{-1/2}$ and $h_\eff= \tau^{-1/2}|x|^{-1}$. Then under week non-degeneracy assumption\footnote{\label{foot-24-18} Which we do not, however, assume fulfilled at this example.} contribution to the remainder of any element with $\mu_\eff h_\eff \le c$ does not exceed   $Ch_\eff^{1-d} =C \tau^{(d-1)/2}\gamma^{d-1}$ while contribution to the remainder of the $\gamma$-element $\mu_\eff h_\eff \ge c$  does not exceed
$C\mu_\eff^{-s}h_\eff^{-d-s}$ and the rest is easy. Without non-degeneracy assumption we need to add to the contribution of the $\gamma$-element
$C\mu_\eff h_\eff^{1-d-\delta}$, which after summation results in $O(\tau^{(d-1-\delta')/2})$.

\item\label{example-24-3-1-ii}
Under proper assumptions the same proof is valid in the odd-dimensional maximal-rank  case.
\end{enumerate}
\end{example}

Let us note that the case $m<-1$, $m_1\ge m-1$ is covered by Chapter~\ref{book_new-sect-11} of \cite{futurebook}; we need to assume that
\begin{equation}
V\ge \epsilon_0 \rho^2 \qquad \text{as\ \ } |x|\le \epsilon.
\label{24-3-3}
\end{equation}

\begin{example}\label{example-24-3-2}
\begin{enumerate}[label=(\roman*), wide, labelindent=0pt]
\item\label{example-24-3-2-i}
Let $X$ be a compact domain, $0\in \bar{X}$  and let conditions \ref{24-2-2-*}, (\ref{24-2-6}) and $\textup{(\ref*{24-3-1})}_{1}$  be fulfilled with
$\gamma= \epsilon |x|$, $\rho=|x|^m$, $\rho_1=|x|^{m_1}$ and with $m<-1$, $2m\le m_1 <m-1$. Then we need to assume that
\begin{equation}
V+F\ge \epsilon_0 \rho^2\qquad \text{as\ \ } |x|\le \epsilon
\tag*{$\textup{(\ref*{24-3-3})}'$}\label{24-3-3-'}
\end{equation}
which for $m_1>2m$ is equivalent to (\ref{24-3-3}).

By the same reason as in the previous Example~\ref{example-24-3-1} we do not need any non-degeneracy assumption. Then asymptotics (\ref{24-3-2}) holds while $\cN^-(\tau)\asymp \tau^{d/2}$.

\item\label{example-24-3-2-ii}
Under proper assumptions the same proof is valid in the odd-dimensional maximal-rank  case.
\end{enumerate}
\end{example}

\subsection{Schr\"odinger-Pauli operator}%
\label{sect-24-3-1-2}

Next, consider Schr\"odinger-Pauli operators. We will need to impose (\ref{24-3-3}) and the related non-degeneracy assumption
\begin{phantomequation}\label{24-3-4}\end{phantomequation}
\begin{equation}
\sum_{\alpha:1\le |\alpha|\le k}
|\nabla^\alpha V|\gamma^{|\alpha|}\ge \epsilon \rho^2
\qquad \text{as\ \ } |x|\le \epsilon
\tag*{$\textup{(\ref*{24-3-4})}_k$}\label{24-3-4-*}
\end{equation}

\begin{example}\label{example-24-3-3}
\begin{enumerate}[label=(\roman*), wide, labelindent=0pt]
\item\label{example-24-3-3-i}
Let $X$ be a compact domain, $0\in \bar{X}$ and let conditions \ref{24-2-2-*}, (\ref{24-2-6}), (\ref{24-3-1}), (\ref{24-3-3}) and \ref{24-3-4-*} be fulfilled with $\gamma= \epsilon |x|$, $\rho=|x|^m$, $\rho_1=|x|^{m_1}$, $m<-1$.

Then for the Schr\"odinger-Pauli  operator as  $\tau\to +\infty$ asymptotics \begin{gather}
\N^-(\tau) = \cN^-(\tau)+ O(\tau^{(m_1+2)/(2m)}+\tau)
\label{24-3-5}\\
\shortintertext{holds while}
\cN^-(\tau)\asymp \tau^{(m_1+m+3)/(2m)}+\tau^{3/2}.
\label{24-3-6}
\end{gather}

Indeed, if $m_1 \ge 2m$ then no modification to the arguments of Examples~\ref{example-24-3-1} and~\ref{example-24-3-2} is needed; if $m_1<2m$ we also need to consider the zone where $\mu_\eff h_\eff \gtrsim 1$.

The contribution of the corresponding partition element to the principal part of the asymptotics is $\mu_\eff h_\eff^{-2}$  while its contribution to the remainder is $O(\mu_\eff h_\eff^{-1})$; also, zone
$\{x\colon |x|\le \epsilon \tau^{1/(2m)}\}$ is forbidden.

\item\label{example-24-3-3-ii}
Without condition \ref{24-3-4-*} one needs to add $\mu_\eff h_\eff^{-1-\delta}$ to the contribution of the partition element to the remainder; it does not affect the remainder estimate $O(\tau)$  if $m_1> 2m+2$; if $m_1\le 2m+2$ we arrive to the remainder estimate $O(\tau^{(m_1+2)/(2m) +\delta})$.

\item\label{example-24-3-3-iii}
One can generalize this example to the odd-dimensional maximal-rank case; then \begin{align}
&\cN^-(\tau)\asymp  \tau^{(m_1+m+3)(d-1)/(4m)}+\tau^{d/2}
\label{24-3-7}\\
\intertext{and the remainder is $O(R)$ with}
&R=\tau^{(m_1+2)(d-1)/(4m)}+\tau^{(d-1)/2}.
\label{24-3-8}
\end{align}
\end{enumerate}
\end{example}

\begin{problem}\label{problem-24-3-4}
\begin{enumerate}[label=(\roman*), wide, labelindent=0pt]
\item\label{problem-24-3-4-i}
Investigate the case of $\gamma=\epsilon |x|$, $\rho= |x|^{-1}|\log |x||^\alpha$, $\alpha >0$, $\rho_1=|x|^{m_1}$, $m_1<-2$.

\item\label{problem-24-3-4-ii}
Investigate the case of $\gamma=\epsilon |x|$, $\rho= |x|^{-1}|\log |x||^\alpha$,  $\rho_1=|x|^{-2}|\log |x||^\beta$, $\beta>\alpha >0$.
\end{enumerate}
\end{problem}

The following problem seems to be challenging; we don't know even if $\N^-(\tau)<\infty$ for  $\tau>0$.

\begin{Problem}\label{Problem-24-3-5}
\begin{enumerate}[label=(\roman*), wide, labelindent=0pt]
\item\label{Problem-24-3-5-i}
Investigate the case of $\gamma=\epsilon |x|$,
$\rho= |x|^{-1}|\log |x||^\alpha$, $\alpha \le 0$, $\rho_1=|x|^{m_1}$, $m_1<-2$.

\item\label{Problem-24-3-5-ii}
Investigate the case of $\gamma=\epsilon |x|$, $\rho= |x|^{m}$,
 $\rho_1=|x|^{m_1}|$, $m_1<-2$, $-1<m<0$.
\end{enumerate}
\end{Problem}

\subsection{Miscellaneous singularities}%
\label{sect-24-3-1-3}

Consider now miscellaneous singularities in the point.

\begin{example}\label{example-24-3-6}
Let $0\in \bar{X}$ be a compact domain,  and let conditions \ref{24-2-2-*}, (\ref{24-2-6}) and (\ref{24-3-1}) be fulfilled with $\gamma=\epsilon|x|^{1-\beta}$,
$\rho_1=\exp (b|x|^\beta)$, $\beta<0$, $b>0$ and with $\rho= \exp (a|x|^\beta)$ where $b>2a$. Assume also that
\begin{phantomequation}\label{24-3-9}\end{phantomequation}
\begin{equation}
\sum_{\alpha:1\le |\alpha|\le k}
|\nabla^\alpha F|\gamma^{|\alpha|}\ge \epsilon \rho_1
\qquad \text{as\ \ } |x|\le \epsilon.
\tag*{$\textup{(\ref*{24-3-9})}_k$}\label{24-3-9-*}
\end{equation}

\begin{enumerate}[label=(\roman*), wide, labelindent=0pt]
\item\label{example-24-3-6-i}
Then for the Schr\"odinger operator $\N^-(\tau)=\cN^-(\tau)+O(R)$ with $\cN^-(\tau)\asymp \tau^{3/2}$ and
\begin{equation}
R=\left\{\begin{aligned}
& \tau && \beta\ge -2,\\
&\tau |\log \tau|^{(2+\beta)/\beta} \qquad&& \beta<-2.
\end{aligned}\right.
\label{24-3-10}
\end{equation}
Indeed, this is trivial unless $\beta >-2$. If $\beta \le -2$ we need to take into account that the forbidden zone where $\mu_\eff h_\eff\ge C_0$ is
$\{z\colon |x|\le \epsilon _0 |\log \tau|^{1/\beta}\}$. Then we get the remainder $O(\tau r_*^{2+\beta})$ for $\beta<-2$ and $O(\tau |\log r_*|)$ for $\beta=-2$.\enlargethispage{\baselineskip}

Consider $\beta=-2$. Since
$\mu_\eff \lesssim 1$ in the zone
$\{x\colon |x|\ge r^*\coloneq C|\log \tau|^{1/\beta}\}$, we can take $\gamma\asymp |x|^{1-\beta '}$ with $\beta'>-2$ there and the contribution of this zone to the remainder is $O(\tau)$. Meanwhile, the contribution of the zone
$\{x\colon r_*\le |x|\le r^*\}$ to the remainder is
$O(\tau |\log (r_*/r^*)|)=O(\tau)$.

\item\label{example-24-3-6-ii}
Let   conditions (\ref{24-3-3}) and \ref{24-3-4-*} be fulfilled. Then for the Schr\"odinger-Pauli operator asymptotics $\N^-(\tau)=\cN^-(\tau)+O(R)$ holds with
\begin{align}
&R= \tau ^{b/2a}|\log \tau|^{2/\beta}
\label{24-3-11}\\
\shortintertext{and}
&\cN^-(\tau)\asymp \tau^{b/2a+1/2} |\log \tau|^{(3-\beta)/\beta}.
\label{24-3-12}
\end{align}

\item\label{example-24-3-6-iii}
On the other hand, let  $a<b\le 2a$ and  conditions (\ref{24-3-3}) and \ref{24-3-4-*} Be fulfilled. Then for both Schr\"odinger and Schr\"odinger-Pauli operators $\cN^-(\tau)\asymp \tau^{3/2}$ and $R$ is defined by (\ref{24-3-10}).

Moreover, for $b<2a$ we do not need the non-degeneracy assumption \ref{24-3-9-*} because in the allowed zone $\{x\colon V(x)\le C\tau\}$ we have
$\mu_\eff\le h_\eff^{\delta-1}$.
\end{enumerate}
\end{example}

\begin{problem}\label{problem-24-3-7}
Extend results of Example~\ref{example-24-3-6}  to  the Dirac operator.
\end{problem}

\begin{remark}\label{rem-24-3-8}
\begin{enumerate}[label=(\roman*), wide, labelindent=0pt]
\item\label{rem-24-3-8-i}
Observe that the contribution to the remainder of the zone
$\{x\colon |x|\le \varepsilon\}$ does not exceed $\varepsilon^\sigma\tau^{(d-1)/2}$ with $\sigma>0$ in the frameworks of
some above examples (sometimes under certain additional assumptions).

Therefore, in these cases under the standard non-periodicity condition to the geodesic flow with reflections from $\partial X$ the asymptotics
\begin{equation}
\N (\tau )=\cN (\tau )+
\varkappa _1\tau + o(\tau )
\label{24-3-13}
\end{equation}
holds with the standard coefficient $\varkappa _1$.

\item\label{rem-24-3-8-ii}
The similar statement (with $\tau$ replaced by $\tau^2$) is  true for the Dirac operator.
\end{enumerate}
\end{remark}

\begin{problem}\label{problem-24-3-9}
In the frameworks of the examples above
estimate\\ $|\cN^-(\tau)-\kappa_0 \tau^{d/2}|$.
\end{problem}

Finally, consider the case when the singularity is located on the curve or a surface (or a more general set).

\begin{example}\label{example-24-3-10}
Let $0\in \bar{X}$ be a compact domain, and let conditions \ref{24-2-2-*}, (\ref{24-2-6}) and (\ref{24-3-1})  be fulfilled with $\gamma=\epsilon_0\delta(x)$,  $\rho_1=\delta(x)^m$,
$\rho _1=\delta (x)^{m_1}$ with $m_1<\min (2m,\, -2)$  where
$\delta (x)=\dist (x,L)$, $m< 0$, $L$ is a set of Minkowski codimension $p>1$ or a smooth surface; in the latter case we assume also that \ref{24-3-9-*} is fulfilled.

\begin{enumerate}[label=(\roman*), wide, labelindent=0pt]
\item\label{example-24-3-10-i}
Then for the Schr\"odinger operator asymptotics (\ref{24-3-2}) holds for
$\tau\to +\infty$ and $\cN^-(\tau)\asymp \tau^{3/2}$.

Indeed, using the same arguments as before we can get a remainder estimate $O(\tau)$ if $p>1$ or $O(\tau|\log \tau|)$ if $p=1$ but in the latter case we can get rid off logarithm using standard propagation arguments.

\item\label{example-24-3-10-ii}
Let also conditions (\ref{24-3-3}) and \ref*{24-3-4-*}  be fulfilled. Then for the Schr\"odinger-Pauli operator asymptotics
\begin{align}
&\N^- (\tau)=\cN^-(\tau)+ O(\tau+\tau^{(m_1+2-p)/(2m)})
\label{24-3-14}\\
\shortintertext{holds while}
&\cN^-(\tau)\asymp \tau+\tau^{(m_1+m+3-p)/(2m)}.
\label{24-3-15}
\end{align}
\end{enumerate}
\end{example}

\begin{problem}\label{problem-24-3-11}
Extend results of Example~\ref{example-24-3-10} to different types of the singularities along $L$.
\end{problem}

\section{Singularities at infinity}
\label{sect-24-3-2}

Let us consider \emph{unbounded domains:\/}

\subsection{Power singularities: Schr\"odinger operator}
\label{sect-24-3-2-1}
Let us  start from the power singularities.

\begin{example-foot}\footnotetext{\label{foot-24-19} Cf. Example~\ref{example-24-3-1}.}\label{example-24-3-12}
\begin{enumerate}[label=(\roman*), wide, labelindent=0pt]
\item\label{example-24-3-12-i}
Let $X$ be an unbounded domain. Let  conditions (\ref{24-2-1}),  (\ref{24-2-6}), \ref{24-2-2-*}, (\ref{24-3-1}) and \ref{24-3-9-*}  be fulfilled with
$\gamma =\epsilon _0 \langle x\rangle $, $\rho =\langle x\rangle ^m$,
$\rho _1=\langle x\rangle ^{m_1}$, $m_1>2m$.

Then for the Schr\"odinger operator the following asymptotics holds:
\begin{gather}
\N^-(\tau)=\cN^- (\tau)+ O(\tau^{(m_1+2)/m_1})
\label{24-3-16}
\shortintertext{with}
\cN^- (\tau)\asymp \tau^{3(m_1+2)/(2m_1)}.
\label{24-3-17}
\end{gather}
The proof is standard.

\item\label{example-24-3-12-ii}
Under proper assumptions the similar asymptotics holds in the odd-dimensional maximal-rank case:
\begin{gather}
\N^-(\tau)=\cN^- (\tau)+ O(\tau^{(d-1)(m_1+2)/(2m_1)})
\label{24-3-18}
\shortintertext{with}
\cN^- (\tau)\asymp \tau^{d(m_1+2)/(2m_1)}.
\label{24-3-19}
\end{gather}
\end{enumerate}
\end{example-foot}

\begin{example-foot}\footnotetext{\label{foot-24-20} Cf. Example~\ref{example-24-3-2}.}\label{example-24-3-13}
\begin{enumerate}[label=(\roman*), wide, labelindent=0pt]
\item\label{example-24-3-13-i}
Let $X$ be an unbounded domain. Let  conditions  \textup{(\ref{24-2-1})},  (\ref{24-2-6}), \ref{24-2-2-*} and $\textup{(\ref{24-3-1})}^\#$ be fulfilled with $\gamma =\epsilon _0 \langle x\rangle $, $\rho =\langle x\rangle ^m$,
$\rho _1=\langle x\rangle ^{m_1}$, $m>0$, $m-1< m_1\le 2m$.

Let  us assume that \underline{either} $m_1=2m$ and
\begin{gather}
V+F\ge \epsilon_0 \rho^2\qquad \text{as\ \ } |x|\ge c
\tag*{$\textup{(\ref*{24-3-3})}^{\#\prime}$}\label{24-3-3-*'}\\
\shortintertext{\underline{or} $m_1<2m$ and}
V \ge \epsilon_0 \rho^2\qquad \text{as\ \ } |x|\ge c
\tag*{$\textup{(\ref*{24-3-3})}^{\#}$}\label{24-3-3-*}
\end{gather}
 Assume that if $m_1=2m$ then
 \begin{phantomequation}\label{24-3-20}\end{phantomequation}
\begin{multline}
\tau \ge V+F\implies
\sum_{\alpha:1\le |\alpha|\le k} |\nabla (\tau - V)F^{-1}|\gamma^{|\alpha|} \ge \epsilon_0 \tau \rho_1^{-1}\\
\text{as\ \ } |x|\ge c.
\tag*{$\textup{(\ref*{24-3-20})}_k$}\label{24-3-20-*}
\end{multline}
Then for the Schr\"odinger operator asymptotics
\begin{gather}
\N^-(\tau)=\cN^-(\tau) + O(\tau^{(d-1)(m+1)/(2m)}),
\label{24-3-21}\\
\shortintertext{holds with}
\cN^- (\tau)\asymp \tau^{d(m+1)/(2m)}.
\label{24-3-22}
\end{gather}

\item\label{example-24-3-13-ii}
Under proper assumptions the similar asymptotics holds in the  odd-dimensional maximal-rank case.
\end{enumerate}
\end{example-foot}

\subsection{Power singularities: Schr\"odinger-Pauli operator}
\label{sect-24-3-2-2}

Next, consider Schr\"odinger-Pauli operators. We will need to impose \ref{24-3-3-*} and the related non-degeneracy assumption
\begin{equation}
\sum_{\alpha:1\le |\alpha|\le k}
|\nabla^\alpha V|\gamma^{|\alpha|}\ge \epsilon \rho_1
\qquad \text{as\ \ } |x|\ge \epsilon.
\tag*{$\textup{(\ref*{24-3-4})}^\#_k$}\label{24-3-4-**}
\end{equation}

\begin{example-foot}\footnotetext{\label{foot-24-21} Cf. Example~\ref{example-24-3-3}.}\label{example-24-3-14}
Let \ref{24-3-3-*} and \ref{24-3-4-**} be fulfilled. Then for the Schr\"odinger-Pauli operator

\begin{enumerate}[label=(\roman*), wide, labelindent=0pt]
\item\label{example-24-3-14-i}
In the framework of Example~\ref{example-24-3-12}
\begin{gather}
\N^-(\tau)=\cN^- (\tau)+ O(\tau^{(m_1+2)/(2m)})
\label{24-3-23}
\shortintertext{with}
\cN^- (\tau)\asymp \tau^{(m_1+m+3)/(2m)}.
\label{24-3-24}
\end{gather}

\item\label{example-24-3-14-ii}
In the framework of Example~\ref{example-24-3-13}
\begin{gather}
\N^-(\tau)=\cN^- (\tau)+ O(\tau^{(m+1)/m})
\label{24-3-25}
\shortintertext{with}
\cN^- (\tau)\asymp \tau^{3(m+1)/(2m)}).
\label{24-3-26}
\end{gather}

\item\label{example-24-3-14-iv}
Finally, under proper assumptions one can consider the odd-dimensional maximal-rank  case and prove asymptotics with the remainder estimate $O(R)$, with
\begin{align}
R(\tau)=&\left\{\begin{aligned}
&\tau^{(d-1)(m_1+2)/(4m)} && m_1>2m,\\
&\tau ^{(d-1)(m+1)/(2m)}  && m_1\le 2m,
\end{aligned}\right.
\label{24-3-27}\\
\shortintertext{and}
\cN^-(\tau)\asymp &\left\{\begin{aligned}
&\tau^{(d-1)(m_1+2)/(4m) +(m+1)/(2m)} && m_1>2m,\\
&\tau^{d(m+1)/(2m)}&& m_1\le 2m,
\end{aligned}\right.
\label{24-3-28}
\end{align}
\end{enumerate}
\end{example-foot}

\subsection{Exponential singularities}
\label{sect-24-3-2-4}

Consider now an exponential growth at infinity.

\begin{example}\label{example-24-3-15}
Let $X$ be an unbopunded domain. Let  conditions  \textup{(\ref{24-2-1})},  \ref{24-2-2-*},  (\ref{24-2-4}), (\ref{24-2-6})  and \ref{24-3-9-*}  be fulfilled with
$\gamma =\epsilon _0 \langle x\rangle^{1-\beta }$,
$\rho =\exp (a \langle x\rangle ^\alpha )$,
$\rho _1=\exp (b \langle x\rangle ^\beta) $, $\beta > 0$ and \underline{either} $\beta>\alpha$ \underline{or} $\beta=\alpha$ and $b>2a>0$.

\begin{enumerate}[label=(\roman*), wide, labelindent=0pt]
\item\label{example-24-3-15-i}
Then for the Schr\"odinger operator the following asymptotics holds:
\begin{gather}
\N^- (\tau )=\cN^-(\tau )+
O\bigl(\tau |\log \tau|^{(2+\beta)/\beta} \bigr)
\label{24-3-29}\\
\shortintertext{with}
\cN^- (\tau) \asymp \tau^{3/2}|\log\tau|^{3/\beta}.
\label{24-3-30}
\end{gather}

\item\label{example-24-3-15-ii}
Let $\alpha=\beta$ and conditions  \ref{24-3-3-*} and \ref{24-3-4-**} be fulfilled. Then for the Schr\"odinger-Pauli operator asymptotics holds
\begin{align}
R(\tau)=&\tau^{b/2a}|\log \tau|^{2/\beta}
\label{24-3-31}\\
\shortintertext{and}
\cN^-(\tau)\asymp &\tau^{(b+a)/2a} |\log \tau|^{3/\beta}.
\label{24-3-32}
\end{align}

\item\label{example-24-3-15-iii}
On the other hand, let $\beta=\alpha$, $a<b\le 2a$ and  conditions \ref{24-3-3-*} and \ref{24-3-4-**} Be fulfilled. Then for both Schr\"odinger and Schr\"odinger-Pauli operators $\cN^-(\tau)\asymp \tau^{3/2}$ and $R$ is defined by (\ref{24-3-10}).

Moreover, for $b<2a$ we do not need the non-degeneracy assumption \ref{24-3-9-*} because in the allowed zone $\{x\colon V(x)\le C\tau\}$ we have
$\mu_\eff\le h_\eff^{\delta-1}$.
\end{enumerate}
\end{example}

We leave to the reader

\begin{problem}\label{problem-24-3-16}
\begin{enumerate}[label=(\roman*), wide, labelindent=0pt]

\item\label{problem-24-3-16-ii}
Consider the Schr\"odinger-Pauli  operators in the same settings as in Example~\ref{example-24-3-15}, albeit with $\rho$ of the power growth at infinity.

\item\label{problem-24-3-16-iii}
Consider the Schr\"odinger-Pauli  operators in the same settings as in Example~\ref{24-3-14} albeith with $V$ of the logarithmic growth at infinity (i.e. with $\rho=|\log |x||^{\alpha}$, $\gamma =\epsilon |x|$).
\end{enumerate}
\end{problem}

\chapter{$3\D$-case.  Asymptotics of small eigenvalues}
\label{sect-24-4}

In this section we need to consider first  miscellaneous asymptotics (cf. Subsections~\ref{sect-23-4-3} and~\ref{sect-23-4-4} of \cite{futurebook}) and only after case of $F$ stabilizing at infinity (cf. Subsections~\ref{sect-23-4-1} and~\ref{sect-23-4-2} of \cite{futurebook}).

\section{Miscellaneous asymptotics}
\label{sect-24-4-1}

\subsection{Case \texorpdfstring{$F\gtrsim 1$ as $|x|\to \infty$}{F\textgtrsim 1\ as  |x|\textrightarrow\textinfty}}
\label{sect-24-4-1-1}

In this subsection we consider cases of either $F\to \infty$ or $F\asymp 1$ \underline{and}  $V\to 0$ as $|x|\to \infty$.  If $F\to \infty$ the Schr\"odinger operator either does not have any essential spectrum ``at infinity''\footnote{\label{foot-24-22} More precisely, the lowest Landau level tends to $+\infty$ as $|x|\to \infty$.}; if $F\asymp 1$ we do not assume any stabilization conditions so far. Anyway, the Schr\"odinger operator  is not the subject of our analysis, while for the Schr\"odinger-Pauli and Dirac operators essential spectrum ``at infinity'' equals $[0,\infty)$ and $(-\infty, M]\cup [M,\infty)$ respectively\footnote{\label{foot-24-23} Again, understood in the sense of the Lanfdau levels.}.

Again due to the specifics of the problem we can consider the multidimensional case with minimal  modifications. In this case we assume that
\begin{claim}\label{24-4-1}
$\rank F=2p$ as $|x|\ge c$ and $f_1\asymp f_2 \asymp \ldots \asymp f_p
\asymp \rho_1$ as $|x|\ge c$
\end{claim}
\vspace{-7pt} and\vspace{-5pt}
\begin{claim}\label{24-4-2}
For each $j\ne k$ \underline{either} $f_j=f_k$ \underline{or}
$|f_j-f_k|\ge \epsilon \rho_1$  for all $|x|\ge c$.
\end{claim}

Let for the Schr\"odinger-Pauli operator $\N^- (\eta)$ be a number of eigenvalues in $(-\epsilon, -\eta)$ and $\N^+ (\eta)$ be a number of eigenvalues in $(\eta, \epsilon)$.

\begin{theorem-foot}\footnotetext{\label{foot-24-24} Cf. Theorem~\ref{thm-23-4-16} of \cite{futurebook}.}\label{thm-24-4-1}
Let $X$ be an unbounded domain. Let  conditions \textup{(\ref{24-2-1})}, \ref{24-2-2-*}, \textup{(\ref{24-4-1})}, \textup{(\ref{24-4-2})} and \ref{24-3-4-**} be fulfilled with scaling functions  $\gamma$, $\rho$ and $\rho_1$, $\rho\to 0$, $\rho_1\gtrsim 1$, $\rho_1\gamma/\rho\to \infty$ and $\rho\gamma\to \infty$ as $|x|\to \infty$. Then for the Schr\"odinger-Pauli operator the following asymptotics holds $\N^- (\eta)=\cN^- (\eta)+ O(R)$
where
\begin{multline}
\cN ^- (\eta) \coloneqq\\
(2\pi)^{-d+p} \varpi_{d-2p} \int _{\{ x\colon  -V (x) \ge \eta \}}
f_1f_2\cdots f_p (-V-\eta)_+^{(d-2p)/2}\sqrt{g}\,dx
\label{24-4-3}
\end{multline}
and
\begin{equation}
R=C\int _{\{ x\colon  -V (x) \ge \eta \}} \rho_1^{p}\rho^{d-2p-1} \gamma^{-1}\,dx
+ C\int \rho_1^{p-s} \gamma^{-2s}\,dx.
\label{24-4-4}
\end{equation}
\end{theorem-foot}

\begin{example}\label{example-24-4-2}
Let $X$ be an unbounded domain. Let  conditions \textup{(\ref{24-2-1})}, \ref{24-2-2-*}, (\ref{24-2-6}), \textup{(\ref{24-4-1})}, \textup{(\ref{24-4-2})}  and \ref{24-3-4-**} be fulfilled with scaling functions $\gamma=\epsilon\langle x\rangle$,
$\rho = \langle x\rangle^m$, $\rho_1 = \langle x\rangle^{m_1}$,
$-1< m<0\le m_1$.

Then for the Schr\"odinger-Pauli operator
\begin{align}
&R= \eta^{(m_1+2)(d-1)/(4m)}
\label{24-4-5}
\shortintertext{and}
&\cN^-(\eta)= O(\eta^{(m_1+2)(d-1)/(4m)+(m+1)/(2m)}).
\label{24-4-6}
\end{align}
 Further, we can replace ``$=O(\cdot)$'' by ``$\asymp\cdot$ if condition
 $V\le -\epsilon \rho^2$ is fulfilled in some non-empty cone.
\end{example}

We leave to the reader

\begin{problem} \label{problem-24-4-3}
\begin{enumerate}[label=(\roman*), wide, labelindent=0pt]
\item\label{problem-24-4-3-i}
Consider the case of $\gamma=\epsilon \langle x\rangle$,
$\rho=\langle x\rangle^{-1}|\log \langle x\rangle|^\alpha$,
$\rho_1= \langle x\rangle^{m_1}$, $m_1\ge 0$, $\alpha>0$.

\item\label{problem-24-4-3-ii}
Consider the case of $\gamma=\epsilon \langle x\rangle$,
$\rho=|\log \langle x\rangle|^\alpha$,
$\rho_1= \langle x\rangle^{m_1}$, $m_1\ge 0$, $\alpha<0$.

\item\label{problem-24-4-3-iii}
Consider the case of $\gamma=\langle x\rangle^{1-\beta}$,
$\rho_1 = \exp( b\langle x\rangle^{\beta})$,  $\beta>0$ while conditions to $V,g^{jk}$ are fulfilled with $\gamma=\epsilon \langle x\rangle$, and \underline{either}
$\rho=\langle x\rangle^{-1}|\log \langle x\rangle|^\alpha$, $\alpha>0$ \underline{or} $\rho=|\log \langle x\rangle|^\alpha$, $\alpha<0$.
\end{enumerate}
\end{problem}

\begin{problem} \label{problem-24-4-4}
Consider the Dirac operator. In this case $\N^-(\eta)$ is a number of eigenvalues in
$(M-\epsilon, M-\eta)$ and $\N^+(\eta)$ is a number of eigenvalues in
$(- M+\eta, - M+\epsilon)$, $0<\eta<\epsilon$, $M>0$.
\end{problem}

\subsection{Case \texorpdfstring{$F\to 0$ as $|x|\to \infty$}{F\textrightarrow0\ as  |x|\textrightarrow\textinfty}}
\label{sect-24-4-1-2}

In this subsection we consider cases of $F\to 0$ and  $V\to 0$ as $|x|\to \infty$. In this case the essential spectra of the Schr\"odinger and Schr\"odinger-Pauli operators are $[0,\infty)$; however, as $V=o(F)$ as
$|x|\to \infty$ the Schr\"odinger operator has only a finite number of the negative eigenvalues and thus is not a subject of our analysis while the Schr\"odinger-Pauli operator is.

Further, the Dirac operator has its essential spectrum $(-\infty,-M]\cup[M,\infty)$ and we need to assume that $M>0$.

\begin{theorem-foot}\footnotetext{\label{foot-24-25} Cf. Theorem~\ref{thm-24-4-1}.}\label{thm-24-4-5}
Let $X$ be an unbounded domain. Let  conditions \textup{(\ref{24-2-1})}, \ref{24-2-2-*}, \textup{(\ref{24-2-6})}, \textup{(\ref{24-4-1})}, \textup{(\ref{24-4-2})} and \ref{24-3-4-**} be fulfilled with scaling functions  $\gamma$, $\rho$ and $\rho_1$, $\rho\to 0$, $\rho_1\to 0$, $\rho_1\rho/\rho^2 \ge C_0$ and $\rho\gamma\to \infty$ as $|x|\to \infty$.

Then for the Schr\"odinger-Pauli operator  $\N^-(\eta)=\cN^-(\eta)+O(R)$ with $\cN^-(\eta)$ and $R$ defined by \textup{(\ref{24-4-3})}-- \textup{(\ref{24-4-4})}.
\end{theorem-foot}

\begin{example}\label{example-24-4-6}
Let conditions of Theorem~\ref{thm-24-4-5} be fulfilled with
$\gamma=\langle x\rangle$, $\rho = \langle x\rangle^m$, $\rho_1 = \langle x\rangle^{m_1}$, $m<0$, $\max(2m,\,-2)<m_1<0$.

Then all statements of the Example~\ref{example-24-4-2} remain true.
\end{example}

\begin{example}\label{example-24-4-7}
Let conditions of Theorem~\ref{thm-24-4-5} be fulfilled with
$\gamma=\epsilon \langle x\rangle$,
$\rho_1=\langle x\rangle^{-2}|\log \langle x\rangle|^\beta$,
$\rho=\langle x\rangle^{-1} |\log \langle x\rangle|^\alpha$ with
$0<2\alpha <\beta$. Then
\begin{align}
&R= |\log \eta|^{(\beta p+\alpha (d-2p-1)+1)/2}
\label{24-4-7}\\
\intertext{and $\cN^-(\eta)= O(S)$ with}
&S= |\log \eta|^{(\beta p+\alpha (d-2p)+1)/2}
\label{24-4-8}
\end{align}
Further,
$\cN^- (\eta)\asymp S$ if condition $V\le -\epsilon \rho^2$ is fulfilled in some non-empty cone.
\end{example}

We also leave to the reader
\begin{problem}\label{problem-24-4-8}
Consider in this framework the Dirac operator. In this case $\N^-(\eta)$ is a number of eigenvalues in $(M-\epsilon, M-\eta)$ and $\N^+(\eta)$ is a number of eigenvalues in $(-M+\eta, -M+\epsilon)$, $0<\eta<\epsilon$. We need to assume that $M>0$ and potential $V\sim \rho^2$ at infinity.
\end{problem}

Consider now the case when condition $\rho^2=o(\rho_1)$ as $|x|\to \infty$ is not fulfilled. Then the results will be similar to those of Section~\ref{sect-24-3}.

\begin{example}\label{example-24-4-9}
\begin{enumerate}[label=(\roman*), wide, labelindent=0pt]
\item\label{example-24-4-9-i}
Let $X$ be an unbounded domain with $\sC^K$ boundary.  Let  conditions \textup{(\ref{24-2-1})}, \ref{24-2-2-*}, \textup{(\ref{24-2-6})}, (\ref{24-4-1}), (2\ref{24-4-2})  and \ref{24-3-4-**} be fulfilled with scaling functions
$\gamma =\epsilon _0 \langle x\rangle $,
$\rho =\langle x\rangle ^m$, $\rho _1=\langle x\rangle ^{m_1}$, $-1<m<0$,
$m-1< m_1< 2m$.

Then for the Schr\"odinger and Schr\"odinger-Pauli operators asymptotics
\begin{align}
&R=\eta ^{(m+1)(d-1)/2m}
\label{24-4-9}\\
\shortintertext{and}
&S=\eta ^{(m+1)(d-1)/2m}
\label{24-4-10}
\end{align}

\item\label{example-24-4-9-ii}
Similar results hold if $m_1=2m$ but as $p=1$ one needs to assume a non-degeneracy assumption (we leave it to the reader) and for the Schr\"odinger operator $\cN^-(\eta)\asymp S$ provided $V+F\le -\epsilon \rho^2$ in some non-empty cone.
\end{enumerate}
\end{example}

We leave to the reader:\enlargethispage{\baselineskip}

\begin{problem}\label{problem-24-4-10}
Consider the case of $\gamma =\epsilon _0 \langle x\rangle $,
$\rho =\langle x\rangle ^{-2}|\log x|^\alpha$,\\
$\rho _1=\langle x\rangle ^{-2}|\log x|^\beta$, $2\alpha \ge \beta>\alpha$.
\end{problem}

\begin{problem}\label{problem-24-4-11}
Consider in this framework the Dirac operator with $M>0$.
\end{problem}

\section{Case \texorpdfstring{$\rank F_\infty=d-1$}{rank F\textunderscore\textinfty= d-1}. Fast decaying potential}
\label{sect-24-4-2}
Consider case of stabilization at infinity (cf. Subsection~\ref{sect-23-4-1} of \cite{futurebook})
\begin{phantomequation}\label{24-4-11}\end{phantomequation}
\begin{equation}
\g \to \g_\infty, \quad \F\to \F_\infty ,\quad V\to 0
\qquad \text{as\ \ } |x|\to \infty.
\tag*{$\textup{(\ref*{24-4-11})}_{1-3}$}\label{24-4-11-*}
\end{equation}
Recall that $\F\coloneqq   (F_{jk})$ with
$F_{jk}=\partial _k V_j- \partial_j V_k$, $\g \coloneqq   (g^{jk})$.

Assume now that
\begin{claim}\label{24-4-12}
$\rank \F_\infty =2p$ while $d=2p+1$
\end{claim}
and potential $V$ decays faster than in the previous subsection--at least in the direction of the magnetic field.

\subsection{Preliminary analysis}
\label{sect-24-4-2-1}

Observe first that
\begin{claim}\label{24-4-13}
If $\F$ and $\g$ are constant than without any loss of the generality we can assume that $g^{jk}=\updelta_{jk}$ and $\Ker \F = \bR^{d-2p}\times \{0\} $.
\end{claim}
Indeed we can achieve it by a linear change of the coordinates.
In the general case under assumption (\ref{24-4-12}) we can assume that
\begin{gather}
\Ker \F = \bR \times \{0\} \qquad \text{as\ \ }|x|\ge c
\label{24-4-14}\\
\shortintertext{and}
V_1=0\qquad \text{as\ \ }|x'|\ge c
\label{24-4-15}
\end{gather}
where $x=(x_1;x')$. Indeed, we can achieve (\ref{24-4-14}) by the change of the coordinate system which straightens magnetic lines\footnote{\label{foot-24-26} After this we are allowed only changes $x\mapsto y$ with $y'=y'(x')$ and $y_d=y_d(x)$.} and we can achieve \ref{24-4-15}) by the gauge transformation.

These two assumptions imply $V_j=v_j(x')$ for $j=2,\ldots,d$ and together with stabilization as $x_1\to \infty$ we conclude that $\F$ is constant. Without any loss of the generality we can assume that
\begin{claim}\label{24-4-16}
$\g_\infty=\updelta_{jk}$, $F_{jk}=0$ unless either $j=2l$, $k=2l+1$ when $F_{jk}=f_{\infty,l}$, or $j=2l+1$, $k=2l$ when $F_{jk}=-f_{\infty,l}$.
\end{claim}

\begin{claim}\label{24-4-17}
By means of the allowed change of the coordinates\footref{foot-24-26} on each magnetic line\footnote{\label{foot-24-27} But not necessarily on all of them in simultaneously.} $\{x\colon  x'=y'\}$ with $|y'|\ge c$ we can achieve
\begin{equation}
g^{jd}=0\qquad j=0,\ldots,d.
\label{24-4-18}
\end{equation}
\end{claim}
\begin{remark}\label{rem-24-4-12}
One can prove easily that in this reduction $V$ is perturbed by $O(\rho^4\gamma^{-2})$ which would not affect the principal part and an error estimate.
\end{remark}

As (\ref{24-4-14}), (\ref{24-4-15}) and (\ref{24-4-18})\,\footnote{\label{foot-24-28} Only as $x'=y'$.} are fulfilled consider $1\D$-operator on $\bR\ni z$
\begin{gather}
\cL(y')\coloneqq   D_z g^{11} (z;y') D_z + V^*(z; y'),
\label{24-4-19}\\
V^*(x)\coloneqq   V(x)+\sum_j (f_j(x)-f_\infty).
\label{24-4-20}
\end{gather}

Let us consider operator for which (\ref{24-4-14}) and (\ref{24-4-15}) are fulfilled allowing instead some anisotropy:
\begin{phantomequation}\label{24-4-21}\end{phantomequation}
\begin{align}
&|\nabla^\alpha (\g-\g_\infty)|=
o(\rho^2 \gamma^{-|\alpha'|}\gamma_1^{-\alpha_1}),
\tag*{$\textup{(\ref*{24-4-21})}_{1}$}\label{24-4-21-1}\\[3pt]
&|\nabla^\alpha V|=
O(\rho^2 \gamma^{-|\alpha'|}\gamma_1^{-\alpha_1}),
\tag*{$\textup{(\ref*{24-4-21})}_{2}$}\label{24-4-21-2}\\[3pt]
&|\nabla^\alpha F_{jk}|= O(\rho^2 \gamma^{-|\alpha'|}\gamma_1^{-\alpha_1})
\qquad\text{as\ \ } |x|\to \infty\quad \forall\alpha
\tag*{$\textup{(\ref*{24-4-21})}_{3}$}\label{24-4-21-3}
\end{align}
with scaling functions
\begin{phantomequation}\label{24-4-22}\end{phantomequation}
\begin{align}
&\gamma_1(x')\ge 1,\qquad
\gamma=\gamma(x')\to \infty\quad \text{as\ \ } |x'|\to\infty
\tag*{$\textup{(\ref*{24-4-22})}_{1-2}$}\label{24-4-22-*}\\[3pt]
&\rho (x)= \varrho(x')\varrho_1(x_1/\gamma_1)
\label{24-4-23}\\
\shortintertext{such that}
&|\nabla \gamma|\le \frac{1}{2},\qquad |\nabla \gamma_1|\le \frac{1}{2}\gamma_1\gamma^{-1},
\label{24-4-24}\\[3pt]
&|\varrho_1|\le 1,\qquad \int_{\bR} |t|\varrho^2_1(t)\,dt <\infty,
\label{24-4-25}\\[3pt]
&|\nabla \varrho|\le \varrho\gamma^{-1},\qquad
|\nabla \varrho_1|\le \varrho_1\gamma^{-1}
\label{24-4-26}\\[3pt]
&\zeta \coloneqq   \varrho ^2\gamma_1 \to 0\qquad \text{as\ \ }|x'|\to \infty.
\label{24-4-27}
\end{align}
In virtue of Proposition~\ref{prop-24-A-1} operator $\cL(x')$ has a finite number of negative eigenvalues for all $x'$ and no more than one negative eigenvalue as $|x'|\ge c$; further, under assumption
\begin{multline}
W(x')\coloneqq
-\frac{1}{2}\int _\bR V^*(x_1;x')\,dx_1 >0\quad\text{and}
\quad W(x')\asymp \zeta \\
\text{as\ \ }|x'|\ge c,
\label{24-4-28}
\end{multline}
there is exactly one negative eigenvalue $\lambda(x')$ and
\begin{align}
&\nabla^{\alpha'}\bigl(\lambda(x')+W(x')^2) = o(\zeta^2\gamma^{-|\alpha'|})
\label{24-4-29}\\
\shortintertext{while}
&\nabla^{\alpha'}W=O(\zeta \gamma^{-|\alpha'|})
\label{24-4-30}\\
\shortintertext{and}
&\1 \nabla ^\alpha v\1 =O(\gamma^{-|\alpha'}\gamma_1^{-\alpha_1}).
\label{24-4-31}
\end{align}
Here $v= v(x_1;x')$ is a corresponding eigenfunction.

\subsection{The main theorem}
\label{sect-24-4-2-2}

The principal result of this subsection is the following theorem:

\begin{theorem}\label{thm-24-4-13}
Let conditions \ref{24-4-11-*},
\textup{(\ref{24-4-14})}, \textup{(\ref{24-4-15})}, $\textup{(\ref{24-4-21})}_{1-3}$, $\textup{(\ref{24-4-22})}_{1-2}$, \textup{(\ref{24-4-24})}--\textup{(\ref{24-4-28})} be fulfilled.
Moreover, let
\begin{equation}
|\nabla W|\ge \epsilon_0 \zeta^2\gamma^{-1} \qquad \text{as\ \ } |x|\ge c.
\label{24-4-32}
\end{equation}
Then
\begin{gather}
|\N ^- (\eta) -\cN^- (\eta)|\le C \int _{ \cZ (\eta)} \gamma^{-2}\,dx',
+ C \int \gamma^{-s}\,dx'
\label{24-4-33}
\shortintertext{where}
\cN ^-(\eta) \coloneqq
(2\pi)^{-r} \int _{\{ x\colon   -\lambda \ge \eta \}}
f_{\infty,1}f_{\infty,2}\cdots f_{\infty,r} \,dx'
\label{24-4-34}
\end{gather}
$\cZ(\eta)$ is $\epsilon \gamma$-vicinity
of $\Sigma(\eta)=\{ x'\colon  -\lambda(x') =\eta\}$; cf. \footref{foot-23-34}.
\end{theorem}

\begin{proof}
\begin{enumerate}[label=(\alph*), wide, labelindent=0pt]
\item\label{pf-24-4-13-a}
We know that
\begin{gather}
\N^- (\eta )= \N ^-(A-f^*_\infty +\eta )=
\N^-(\A_\eta )=\Tr (E_\eta (0)),
\label{24-4-35}\\
\shortintertext{where}
\A_\eta \coloneqq   J^{-\frac{1}{2}}(A-f^*_\infty+\eta )J^{-\frac{1}{2}}
\label{24-4-36}
\end{gather}
is a self-adjoint operator in $\sL^2(\bR^d)$ and $E_\eta (\tau )$ is the spectral projector of this operator and $J\asymp \rho^2$ is such that
\begin{equation}
|\nabla^\alpha J|= O(J\gamma^{-|\alpha'|}\gamma_1^{-|\alpha_1|})\qquad \forall \alpha.
\label{24-4-37}
\end{equation}

We consider only $\eta >0$ and for any fixed $\eta >0$ and $\tau $ this projector is finite-dimensional and its Schwartz kernel belongs to $\sS(\bR^{2d})$ uniformly on $\tau \le \tau _0$. Let us note that
\begin{equation}
((A-f^*_\infty)v,v)\ge
(1-\epsilon)((A_\infty -f^*_\infty)v,v) -C \|\rho v\|^2
\qquad \forall v\in \sC_0^2(\bR^d)
\label{24-4-38}
\end{equation}
where $(A_\infty -f^*_\infty)$ is non-negative.

Indeed, without any loss of the generality one can assume that
\begin{gather}
A_\infty = D_1^2 + \sum_{1\le j\le p} \bigl( D_{2j}^2 + (D_{2j+1}+f_j x_{2j})^2\bigr).
\label{24-4-39}\\
\shortintertext{Then}
A_\infty-f_\infty = D_1^2 +\sum_{1\le j\le p} Z_j^* Z_j
\label{24-4-40}\\
\shortintertext{with}
Z_j= iD_{2j} + (D_{2j+1}+f_j x_{2j}).
\label{24-4-41}
\end{gather}
On the other hand, $A-A_\infty$ is a linear combination of $D_1^2$, $D_1Z_j$,
$D_1Z_j^*$, $Z^*_jZ_k$, $Z_jZ^*_k$, $Z_jZ_k$, $Z^*_jZ^*_k$, $Z_j$, $Z^*_j$ and $1$ with the coefficients $\beta_*$ satisfying
\begin{equation}
|\nabla ^\alpha \beta_*|=O(\rho^2 \gamma^{-|\alpha'|}\gamma_1^{-\alpha_1})\qquad\forall \alpha.
\label{24-4-42}
\end{equation}
Then extra terms in $(Av,v)$ do not exceed
\begin{gather*}
C\Bigl(\|\rho D_1 v\|^2+\sum_j \|\rho Z_jv\|^2 + \sum_j \|\rho Z^*_jv\|^2 +\|\rho v\|^2\Bigr),\\
\shortintertext{where obviously}
\|\rho Z_j^*v\|^2\le C\Bigl(\|\rho Z_jv\|^2+\|\rho v\|^2\Bigr).
\end{gather*}

This inequality (\ref{24-4-38}) immediately yields estimates
\begin{phantomequation}\label{24-4-43}\end{phantomequation}
\begin{align}
&\|D_1 J^{- \frac{1}{2} }E_\eta (\tau )\|\le C,
&&\|Z_jJ^{-{\frac{1}{2}}}E_\eta (\tau )\|\le C\quad j=1,\ldots, p,
\tag*{$\textup{(\ref*{24-4-43})}_{1,2}$}\label{24-4-43-1}\\[2pt]
&\|J^{-\frac{1}{2}} E_\eta (\tau )\|\le C\eta^{-\frac{1}{2}},
&&\|E_\eta(\tau)\|\le 1
\tag*{$\textup{(\ref*{24-4-43})}_{3,4}$}\label{24-4-43-3}\\
\intertext{and therefore}
&\|Z_j^*J^{-\frac{1}{2}} E_\eta (\tau )\|\le C\eta^{-\frac{1}{2}},
&&\|Z_j^*E_\eta(\tau)\|\le C
\tag*{$\textup{(\ref*{24-4-43})}_{5,6}$}\label{24-4-43-5}
\end{align}
for operator norms where here and below $\tau \le \tau _0$.

Then one can prove easily that
\begin{claim}\label{24-4-44}
Let $Q$ be a product of several factors $D_1$, $Z_\bullet$ and $Z_\bullet^*$. Then $\|Q J^{-{\frac{1}{2}}}E_\eta (\tau )\|\le C$ provided there more of factors $D_1$, $Z_\bullet$ than of $Z_\bullet^*$, and
$\|Q J^{-{\frac{1}{2}}}E_\eta (\tau )\|\le C\eta^{-\frac{1}{2}}$,
$\|Q E_\eta (\tau )\|\le C$ provided there as many of factors $D_1$, $Z_\bullet$ as of $Z_\bullet^*$.
\end{claim}
Then this claim remains true for $Q$ replaced by $Q'\coloneqq   D_{x'}^\alpha Q$ for any $\alpha$ and then in virtue of the embedding theorem this is also true for
the operator norm from $\sL^2(\bR^d)\mapsto \sL^2(\bR^d)$ replaced by the
operator norm from $\sL^2(\bR^d)\mapsto \sL^2(\bR)$ taken over any magnetic line $\{x\colon   x'=y'\}$ uniformly with respect to $y'$. In particular,
\begin{gather*}
\1 D_1 J^{-\frac{1}{2}} E_\eta (\tau) v\|_{\sL^2(\bR)}\le C\|v\|\qquad
\text{and}\qquad \1 E_\eta (\tau) v\|_{\sL^2(\bR)}\le C\|v\| \\
\shortintertext{and therefore}
| \bigl(E_\eta (\tau) v\bigr)(x)|\le
\underbracket{CJ^{\frac{1}{2}}(x)
\Bigl( J_0^{-\frac{1}{2}}(x')\gamma_1^{-1} + \langle x_1\rangle^{\frac{1}{2}}\Bigr)} \|v\|
\end{gather*}
with $J_0(x')= \max _{x_1} J(x_1,x')$. So, we estimated the operator norm of $w\to (E_\eta(\tau)w)(x)$ from $\sL^2(\bR^d)$ to $\bC$; therefore
\begin{equation*}
|e_\eta(x,x,\tau)|\le CJ(x)\Bigl( J_0^{-1}(x')\gamma_1^{-1} +
\langle x_1\rangle\gamma_1^2\Bigr)
\end{equation*}
and therefore
$|\int e_\eta(x,x,\tau)\,dx_1|\le C(1 + \varrho^2(x')\gamma_1^2)\le C$ due to properties of $J$, $\varrho$ and $\varrho_1$.
Therefore we have proven:

\begin{proposition}\label{prop-24-4-14}
In the framework of Theorem~\ref{thm-24-4-13} and the definitions of $J$ and $e_\eta(x,y,\tau)$,
\begin{equation*}
\int_{|x'|\le r} e_\eta (x,x,\tau )\,dx=O(1)
\end{equation*}
for all $\eta>0$, $\tau\le\tau_0$ and for any fixed $r$ and $\tau_0$.
\end{proposition}
Recall that $e_\eta (x,y,\tau )$ is the Schwartz kernel of $E_\eta (\tau )$. So, we only need to treat the contribution of the zone $\{x\colon   |x'|\ge r\}$.
\item\label{pf-24-4-13-b}
Let us fix $\y'\in \bR^d$ and consider $\psi(x')$,
$\psi \in \sC_0^K\bigl(B(\y',\frac{1}{2}\gamma)\bigr)$
with $\gamma =\gamma(\y')$ such that
$|D^\alpha\psi |\le c\gamma ^{-|\alpha'|}$
$\forall \alpha': |\alpha'|\le K$. We want to derive asymptotics of
\begin{equation}
\int \psi (x') e_\eta (x,x,0 )\,dx=
\int \psi \Tr_{\bH} \bigl(\mathsf{e}_\eta (x',x',0)\bigr)\, dx'=
\Tr (\psi E_\eta (0)),
\label{24-4-45}
\end{equation}
where $\mathsf{e}_\eta (x',y',\tau)$ is the family of operators in $\bH$ with Schwartz kernel $e_\eta (x,y,\tau)$.
Let us rescale $x'_\new=(x'-\y') \gamma ^{-1}$, $x_{1\new}=x_1\gamma_1 ^{-1}$.
Then we obtain the standard LSSA problem for an operator with an operator-valued symbol, with the semiclassical parameters $h=\gamma ^{-1}$ and
$h_1=\gamma_1 ^{-1}$ and with magnetic field intensity parameter $\mu =\gamma $. Recall that the rescaled operator is
\begin{gather}
h_1^2D_1^2+ \sum_{1\le j\le p }
\Bigl(h^2 D_{2j}^2 + \bigl(h D_{2j+1}-h^{-1}f_j x_{2j}\bigr)^2\Bigr)+
\rho^2 A',
\label{24-4-46}
\end{gather}
where $A'=a'(x,h_1D_1, hD',h)$ is an operators with uniformly smooth symbol $a'$ (we consider it more carefully later).
Let $U(x,y,t)$ be the Schwartz kernel of the operator
$\exp (ih ^{-2}t \A_\eta )$. Later we rescale $t$.
Then
\begin{equation}
(h^2 D_t-\A_\eta )U=0, \qquad U|_{t=0}= \gamma^{2p} \updelta(x-y)I
\label{24-4-47}
\end{equation}
So let $\psi$ be a $\gamma$-admissible partition element.
It follows from (\ref{24-4-44}) that the operator norm (from $\sL^2(\bR^d)$ to $\sL^2(\bR^d)$) of $Q \psi E_\eta (\tau )Q^*$ does not exceed $C$ for the operators $Q$ which are products of several factors $h_1D_1$, $Z_\bullet$ and $Z^*_\bullet$\,\footnote{\label{foot-24-29} Due to (\ref{24-4-41}) now
$Z_j= ihD_{2j} + (hD_{2j+1}+h^{-1}f_j x_{2j})$.}
and there are more factors of $h_1D_1$, $Z_\bullet$ and than of $Z^*_\bullet$.
Then the operator norm of $F_{t\to h^{-2}\tau }\chi _T(t) Q\psi UQ^*$ does not exceed $CT$ for the operators $Z$ listed above where $\chi \in \sC_0^K(\bR)$ is fixed and $T\ge T_0$ with constant $T_0>0$.

Let us apply the transformation $\cT =\cT _0^{-1}\cT _1\cT _0$
where $\cT _0=F_{x '''\to h ^{-2}\xi''' }$, $x'=(x'',x''')$, $x''=(x_2,x_4,\ldots, x_{d-1})$,
$x'''=(x_3,x_5,\ldots, x_d)$ and the same partition for $\xi'=(\xi'',\xi''')$, and
\begin{equation}
\cT _1v(x_2,\xi _3, x_4,\ldots, \xi_d)=v(x_2-f_1^{-1}\xi _3,\xi _3,\ldots, x_{d-1}-f_p^{-1}\xi_d,\xi_d).
\label{24-4-48}
\end{equation}
Then instead of $hD_{2j}$ and $(hD_{2j+1}-h^{-1}f_j x_{2j})$ we obtain $hD_{2j}$ and $-h^{-1}x_{2j+1}$ respectively. Let $\Psi $ be the corresponding linear symplectic transformation. Let $\bar{U}=\cT_x\psi' U^t\cT_y$ where $\psi'$ is supported in $B(0,1-\epsilon)$\,\footnote{\label{foot-24-30} We shifted the coordinate system so that our partition element is supported there.} and equals $1$ in $B(0,1-2\epsilon )$.

Let us decompose $U(x,y,t)$ in terms of the functions
\begin{gather}
\Upsilon _\varsigma(x'')=
h^{- \frac{1}{2} }\upsilon _{\varsigma_1}(x_2h ^{-1})
h^{- \frac{1}{2} }\upsilon _{\varsigma_2}(x_4h ^{-1})\cdots
h^{- \frac{1}{2} }\upsilon _{\varsigma_p}(x_{2p}h ^{-1})
\label{24-4-49}\\
\shortintertext{and $\Upsilon _{\nu} (y'')$:}
\bar{U}(x,y,t)=\sum_{\varsigma,\nu\in \bZ^+}\Upsilon _{\varsigma}(x'')
\Upsilon _\nu (y'')U_{\varsigma\nu}(x_1, x''';y_1,y''';t).
\label{24-4-50}
\end{gather}
We make the same decomposition for $E(x,y,\tau)$.
Then the above estimates yield that
\begin{claim}\label{24-4-51}
The operator norm of $F_{t\to h^{-2}\tau}\chi_T(t)U_{\varsigma\nu}$ does not exceed $CT$.
\end{claim}
Next, the standard ellipticity arguments show that
\begin{claim}\label{24-4-52}
The operator norm of\footnote{\label{foot-24-31} In the obvious situations we do not distinguish operators and their Schwartz kernels.}
$F_{t\to h^{-2}\tau}\chi_T(t)J^{-\frac{1}{2}}(x_1) U_{\varsigma\nu}$ does not exceed $C\varrho^{|\varsigma|-1} T$ for $\varsigma \ne 0$, and also the operator norm of $F_{t\to h^{-2}\tau}\chi_T(t)J^{-\frac{1}{2}}(y_1) U_{\varsigma\nu}$ does not exceed $C\varrho^{|\nu|-1} T$ for $\nu \ne 0$, and, finally, the operator norm of $F_{t\to h^{-2}\tau}\chi_T(t)J^{-\frac{1}{2}}(x_1)J^{-\frac{1}{2}}(y_1) U_{\varsigma\nu}$ does not exceed $C\varrho^{|\nu|+|\varsigma|-2} T$ for
$\nu \ne 0$ and $\varsigma\ne 0$ and the same is true if we apply $D_{x_1}^k$ and $D_{y_1}^l$.
\end{claim}
Moreover, for $\varsigma=\nu=0$ we have
\begin{claim}\label{24-4-53}
The operator norm of
$F_{t\to h^{-2}\tau}\chi_T(t)D_{x_1}^kJ^{-\frac{1}{2}}(x_1) U_{00}$ does not exceed $C T$ for $k\ge 1$, and also the operator norm of
$F_{t\to h^{-2}\tau}\chi_T(t)D_{y_1}^lJ^{-\frac{1}{2}}(y_1) U_{00}$ does not exceed $CT$ for $l\ge 1$, and, finally, the operator norm of
\begin{equation*}
F_{t\to h^{-2}\tau}\chi_T(t)D_{x_1}^kD_{y_1}^l J^{-\frac{1}{2}}(x_1)J^{-\frac{1}{2}}(y_1) U_{00}
\end{equation*} does not exceed $CT$ for
$k\ge 1$ and $l\ge 1$.
\end{claim}

Then\enlargethispage{\baselineskip}
\begin{equation}
\Tr(\psi E) = \sum_{\varsigma,\nu} h^{|\varsigma-\nu|} \Tr(\psi_{\varsigma\nu}E_{\varsigma,\nu}) + O(h^s)
\label{24-4-54}
\end{equation}
where we have the original expression on the left-hand side,
$\psi_{\varsigma\nu}=\psi_{\varsigma\nu}(x''', h^2D''',h^2)$,
$\psi_{\varsigma\varsigma}(x''',\xi''',0)=\psi(x''',\xi''')$. Moreover,
$\supp(\psi_{\varsigma\nu})\subset\supp(\psi )$
and one can replace $\psi_{\varsigma\nu}-\updelta_{\varsigma\nu}\psi$ by a linear combination of the derivatives of $\psi$ of non-zero order.
\item\label{pf-24-4-13-c}
It follows from Proposition~\ref{prop-24-A-1} that operator $(A-f^*_\infty)J^{-1}$ is elliptic outside of $\cZ(\eta)$ and then one can prove easily that the total contribution of $\bR^d\setminus \cZ(\eta)$ to the remainder does not exceed
\begin{equation}
C\int \gamma^{-s}\,dx',
\label{24-4-55}
\end{equation}
while its contribution to the principal part of asymptotics is given by the Tauberian expression.
\item\label{pf-24-4-13-d}
From now on $\psi'$ is a partition element in $\cZ(\eta)$. Recall that
\begin{equation}
(h^2D_t -J^{-\frac{1}{2}}A J^{-\frac{1}{2}})=0,\qquad U|_{t=0}= \updelta (x-y).
\label{24-4-56}
\end{equation}
and the dual equation with respect to $y$. Then using ellipticity arguments we can express $U_{\varsigma\nu}$ with $|\varsigma|+|\nu|\ge 1$ via $U_{00}$ via some $(h_1, h^2)$-pseudodifferential operators (with respect to $(x_1,x''')$) and $h_1^2$ and then
plugging back into equation we get
\begin{equation}
(h^2D_t -J_0^{-\frac{1}{2}}A_0 J_0^{-\frac{1}{2}})=0,\qquad
U_{00}|_{t=0}= M \updelta (x-y).
\label{24-4-57}
\end{equation}
where $A_0$ differs from $h_1^2 D_1 + V^*+\eta$, with $V^*=V+f^*-f^*_\infty$ by $o(\rho^2)$; from the beginning we could assume that $g^{11}=1$. Here $J_0$ and $A_0$ are $(h_1, h^2)$-pseudodifferential operators (with respect to $(x_1,x''')$)-pseudodifferential operators. Let $\hslash\coloneq h^2$.
Let us observe that in virtue of Proposition~\ref{prop-24-A-1} the operator $J_0^{-\frac{1}{2}}A_0J_0^{-\frac{1}{2}}$
has discrete spectrum in $\bH$ and all the eigenvalues of this operator excluding at most one are positive and uniformly disjoint from $0$ and there is one (the lowest) eigenvalue
$\Lambda =\Lambda (x''',\xi''',\eta)$ which is $O(1)$; moreover, due to (\ref{24-4-32}) it satisfies the microhyperbolicity condition
\begin{equation}
|\Lambda|+|\nabla \Lambda|\asymp 1.
\label{24-4-58}
\end{equation}
Then there exists a symbol $q(x''',\xi ''',\hslash):\bH\to \bC\oplus \bH$ such that for the operator $Q=q(x''',\hslash D''',\hslash )$ and for $U'=QU_{00}Q^*$ we obtain separate equations for all four blocks of
$U'=\begin{pmatrix} U'_{00} & U'_{01}\\U''_{10}&U'_{11}\end{pmatrix}$.
Moreover, for the blocks $U'_{10}$ and $U'_{11}$ the equations are elliptic for
$\tau \le \epsilon _1\hslash $ and for $U'_{10}$ and $U'_{11}$ this is
true for the dual equations.
Therefore $U'\equiv \begin{pmatrix} u & 0\\0&0\end{pmatrix}$ with
$u=U'_{00}$ and
\begin{equation}
F_{t\to \hslash ^{-1}\tau }\chi _T(t)\Tr (\psi U)\equiv
F_{t\to \hslash ^{-1}\tau }\chi _T(t)\Tr (\psi'' u)
\label{24-4-59}
\end{equation}
for $\tau \le \epsilon _0h^2$, $T\in (h^{-\delta },\epsilon_0 )$ where
$\psi ''=\psi ''(x''', h''' D''',\hslash)$.
We have an equation for $u$:
\begin{equation}
(\hslash D_t-\Lambda )u=0
\label{24-4-60}
\end{equation}
where $\Lambda $ is an $\hslash$-pseudodifferential operator. More precisely:
due to the microhyperbolicity we conclude that
\begin{claim}\label{24-4-61}
The contribution of the partition element to the final answer is given by Tauberian expression with $T=\hslash^{1-\delta}$ with an error $O(1)$.
\end{claim}

Therefore, the total contribution of $\cZ(\eta)$ to the remainder does not exceed $C\int_{\cZ(\eta)}\gamma^{-2}\,dx$ (in the original coordinates).

\item\label{pf-24-4-13-e}
Employing the method of the successful approximations and picking $\psi=1$, and we conclude that the final answer is given by (\ref{24-4-34}) since since $\Lambda <0 \iff \lambda <-\eta $.
We leave easy details to the reader.
\end{enumerate}
\end{proof}

\begin{remark}\label{rem-24-4-15}
If $V^*(x)\asymp |x|^{-2}$ then (formally) $W(x') \asymp |x'|^{-1}$ and
$\lambda(x')\asymp |x'|^{-2}$ and $\cN(\eta)\asymp \eta^{-r}$ as follows from the results of Subsection~\ref{sect-24-4-1}.
\end{remark}

\subsection{Generalizations}
\label{sect-24-4-2-3}

\begin{remark}\label{rem-24-4-16}
\begin{enumerate}[label=(\roman*), wide, labelindent=0pt]
\item\label{rem-24-4-16-i}
For the Schr\"odinger-Pauli operator Theorem~\ref{thm-24-4-13} obviously holds albeit with $f^*=f^*_\infty=0$.

\item\label{rem-24-4-16-ii}
The same is true for the Dirac operator. The proof is essentially the same. We need to assume that the mass $M\ne 0$, otherwise the spectral gap $(-M,M)$ is empty. Then we consider $\N^-(\eta)= \N (0, M-\eta)$ and
$\N^+(\eta)= \N (-M+\eta,0)$. Instead of $0$ we can take any $\bar{\tau}\in (-M,M)$ which preserves the result modulo $O(1)$.

Let us consider $\N^-(\eta)$. Modulo $O(1)$ it equals to
$\tilde{\N} (\eta ;-\epsilon _2, 0)$, where $\tilde{\N} (\eta ;\tau _1,\tau _2)$ is the number of eigenvalues of the problem
\begin{equation}
(A-M+\eta )v+\tau \J v=0
\label{24-4-62}
\end{equation}
belonging to the interval $(\tau _1,\tau _2)$ and
$\J =\frac{1}{2}(I+\upsigma _0)J$ where $J$ was introduced in the proof of Theorem~\ref{thm-24-4-13}. This problem is equivalent to the problem
\begin{gather}
(\cA _\eta -\tau J)w=0,\qquad
\cA _\eta =\cL ^*(2M-\eta -V)^{-1}\cL +V + \eta ,
\label{24-4-63}\\
\intertext{where we assume that}
\upsigma _0=\begin{pmatrix} I &0\\0&-I\end{pmatrix},\quad
\upsigma _j=\begin{pmatrix} 0 &\upsigma '_j\\ \upsigma'_j &0\end{pmatrix},
\quad \upsigma '_j{}^*=\upsigma '_j \;(j=1,\ldots, d)
\label{24-4-64}\\
\shortintertext{and}
\cL=\sum_{1\le j\le d}\frac{1}{2}(P_k\omega^{jk}+\omega^{jk}P_k)\upsigma'_j.
\label{24-4-65}
\end{gather}
One can easily transform $\cA_\eta$ to the form of the Schr\"odinger-Pauli operator with the metric $\tilde{g}^{jk}=(2M-\eta-V)^{-1}g^{jk}$.
\end{enumerate}
\end{remark}

\begin{example}\label{example-24-4-17}
\begin{enumerate}[label=(\roman*), wide, labelindent=0pt]
\item\label{example-24-4-17-i}
In the standard isotropic case $\gamma_1=\gamma = \langle x'\rangle$ and as
$\rho (x)= \langle x\rangle ^l$ with $l<-2$; then for $W$ defined by (\ref{24-4-28}) satisfies similar conditions with $m=2l+1$ and thus
we are in the framework of Example~\ref{example-23-4-4} of \cite{futurebook}.

\item\label{example-24-4-17-ii}
However, we can also consider $\rho (x)= \langle x\rangle ^l\langle x'\rangle^p$ with $l<-2$; then $m=2l+2p+1$.

\item\label{example-24-4-17-iii}
We can also consider faster and slower decaying potentials, as soon as $W^2$ satisfies conditions imposed on $V$ in Subsection~\ref{sect-24-4-1}. See Example~\ref{example-23-4-5}, Problems~\ref{Problem-23-4-8}  and~\ref{Problem-23-4-10} of \cite{futurebook}.
\end{enumerate}
\end{example}

\begin{example}\label{example-24-4-18}
Let us consider the case when at infinity $f_j$ stabilize not to $f_\infty=\const$ but to $f_\infty(\theta)$, $\theta=x'/|x'|\in \bS^{d-2}$.

Then as long as other assumptions are fulfilled, we arrive to asymptotics described in Theorem~\ref{thm-23-4-13} of \cite{futurebook}; again, instead of $V(x')$ we have $\Lambda(x')$ in the conditions and in the expression for $\cN^-(\eta)$.
\end{example}

\begin{remark}\label{rem-24-4-19}
Let us consider an auxiliary operator with potential $V$ which
is $\asymp |x_1|^{-2}$ as $x_1 \to \infty$. One can easily prove (see
Proposition~\ref{prop-24-A-3} below) that if
\begin{equation}
V^*\ge - \frac{1}{4}|x| ^{-2}\qquad \forall x\colon |x|\ge C,
\label{24-4-66}
\end{equation}
then the number of negative eigenvalues is finite and there is
no more than one negative eigenvalue if this inequality holds for
all $x$. Moreover, under the conditions
\begin{equation}
V^*\ge (\epsilon - {\frac{1}{4}})|x| ^{-2}\qquad \forall x\colon |x|\ge C
\tag*{$\textup{(\ref*{24-4-66})}^*$}\label{24-4-66-*}
\end{equation}
and (\ref{24-A-12}) with arbitrarily small $\epsilon >0$ all the
statements of Proposition~\ref{prop-24-A-1} remain true. Furthermore, under
condition \ref{24-4-66-*}
\begin{equation}
\blangle \boldsymbol{a}v,v\brangle \ge
{\frac{\epsilon }{2}}\1\langle x\rangle ^{-1}v\1 ^2-
C\1\langle x\rangle ^{-s}v\1 ^2\qquad \forall v
\label{24-4-67}
\end{equation}
with arbitrarily large $s$. Therefore we can cover the case
\begin{gather}
\rho (x)=\langle x\rangle ^{-2}\langle x'\rangle ^{p+2}, \qquad p< -1
\label{24-4-68}\\
\shortintertext{provided}
V^*\ge (\epsilon -{\frac{1}{4}})|x| ^{-2}\qquad
\forall x\colon |x_1|\ge c|x'|
\label{24-4-69}
\end{gather}
with arbitrarily small $\epsilon >0$. The remainder estimate is
the same $O(1)$ as above. The details are left to the reader.
\end{remark}

\begin{remark}\label{rem-24-4-20}
Let $\rank \F(x) =2r\le d-2$ (as $|x|\ge c$). Then the auxiliary operator is $(d-2r)$-dimensional and does not have negative eigenvalues at all in the assumptions of this subsection.

Then one can prove easily that $\N^-(\eta)=O(1)$. In particular, if $\gamma_1=\gamma=\langle x'\rangle$ and $\rho = \langle x'\rangle^{m}$, $\N^-(\eta)=O(1)$ for $m<-1$ (and even for $m=-1$ under assumption (\ref{24-4-69}) but there is a non-trivial asymptotics for $m>-1$; see Subsection~\ref{sect-24-4-3} below.\enlargethispage{\baselineskip}
\end{remark}

\subsection{Possible generalizations}
\label{sect-24-4-2-4}

Consider the case when condition (\ref{24-4-28}) is not fulfilled. We believe that while the Part~\ref{Problem-24-4-21-i} is not extremely challenging, the Part~\ref{Problem-24-4-21-ii} is:
\begin{Problem}\label{Problem-24-4-21}
\begin{enumerate}[label=(\roman*), wide, labelindent=0pt]
\item\label{Problem-24-4-21-i}
Prove that the main part of the asymptotics is still given by (\ref{24-4-34}).
\item\label{Problem-24-4-21-ii}
Prove that (\ref{24-4-33}) still holds.
\end{enumerate}
\end{Problem}

\section{Case \texorpdfstring{$\rank F_\infty=d-1$}{rank F\textunderscore\textinfty= d-1}. Slow decaying potential}
\label{sect-24-4-3}

Now we consider the case as in the previous Subsection~\ref{sect-24-4-2} but we assume that the potential $V$ which either decays slower than $x_1^{-2}$ or as $x_1^{-2}$ but fails condition \ref{24-4-66-*}. We only sketch the main arguments.

\subsection{Main theorem (statement)}
\label{sect-24-4-3-1}

We assume that
\begin{phantomequation}\label{24-4-70}\end{phantomequation}
\begin{align}
&|\nabla^\alpha (\g-\g_\infty)|=
o(\rho^2 \langle x'\rangle^{-|\alpha'|}\langle x\rangle^{-\alpha_1}),
\tag*{$\textup{(\ref*{24-4-70})}_{1}$}\label{24-4-70-1}\\[3pt]
&|\nabla^\alpha V|=
O(\rho^2 \langle x'\rangle^{-|\alpha'|}\langle x\rangle^{-\alpha_1}),
\tag*{$\textup{(\ref*{24-4-70})}_{2}$}\label{24-4-70-2}\\[3pt]
&|\nabla^\alpha (F_{jk}-F_{\infty,jk})|=
O(\rho^2 \langle x'\rangle^{-|\alpha'|}\langle x\rangle^{-\alpha_1})
\qquad\text{as\ \ } |x|\to \infty\quad \forall\alpha
\tag*{$\textup{(\ref*{24-4-70})}_{3}$}\label{24-4-70-3}
\end{align}
where\enlargethispage{3\baselineskip}
\begin{equation}
\rho =\langle x\rangle ^{-q}\langle x'\rangle ^{m+q}
\label{24-4-71}
\end{equation}
$q>0$, $m< 0$. Further, as $m+q=0$ we assume in addition that
\begin{phantomequation}\label{24-4-72}\end{phantomequation}
\begin{align}
&|\nabla^\alpha (\g-\g_\infty)|=
O(\rho^2 \langle x'\rangle^{2-|\alpha'|}\langle x\rangle^{-2-\alpha_1}),
\tag*{$\textup{(\ref*{24-4-72})}_{1}$}\label{24-4-72-1}\\[3pt]
&|\nabla^\alpha V|=
O(\rho^2 \langle x'\rangle^{2-|\alpha'|}\langle x\rangle^{-2-\alpha_1}),
\tag*{$\textup{(\ref*{24-4-72})}_{2}$}\label{24-4-72-2}\\[3pt]
&|\nabla^\alpha (F_{jk}-F_{\infty,jk})|=
O(\rho^2 \langle x'\rangle^{2-|\alpha'|}\langle x\rangle^{-2-\alpha_1})
\tag*{$\textup{(\ref*{24-4-72})}_{3}$}\label{24-4-72-3}\\
&\hskip150pt\text{as\ \ } |x|\to \infty\quad \forall\alpha':|\alpha'|\ge 1.
\notag
\end{align}

\begin{theorem}\label{thm-24-4-22}
Let conditions  $\textup{(\ref{24-4-70})}_{1-3}$ be fulfilled. Let one of two assumptions  be fulfilled:
\begin{enumerate}[label=(\roman*), wide, labelindent=0pt]
\item\label{thm-24-4-22-i}
$m+q<0$ and
\begin{equation}
-\langle x' , \nabla' V^*\rangle \ge \epsilon \rho^2\qquad
\forall x\colon |x'|\ge C_0.
\label{24-4-73}
\end{equation}
\item\label{thm-24-4-22-ii}
$m+q=0$, conditions  $\textup{(\ref{24-4-72})}_{1-3}$ be also fulfilled
\begin{equation}
-\langle x' , \nabla' V^*\rangle \ge \epsilon \rho^2\langle x'\rangle^2
\langle x\rangle^{-2}\qquad \forall x\colon |x'|\ge C_0.
\label{24-4-74}
\end{equation}
\end{enumerate}
Then
\begin{align}
&\N^-(\eta)=\cN^-(\eta)+O(R(\eta))
\label{24-4-75}\\
\shortintertext{with}
&\cN^-(\eta)=(2\pi)^{-r} \int\boldsymbol{n}(x',\eta )
f_{\infty,1}f_{\infty,2}\cdots  f_{\infty,r}\,dx'
\label{24-4-76}\\
\shortintertext{and}
&R(\eta)= \int_{\Lambda (\eta)}
\bigl(\boldsymbol{n}(x',\eta )+1\bigr) \langle x'\rangle ^{-2}\,dx' +
\int \gamma^{-s}\,dx'
\label{24-4-77}
\end{align}
where $\boldsymbol{n}(x',\eta )$ is the number of eigenvalues of the operator $\cL(x')$ which are less than $-\eta$ and $\Lambda(\eta)$ is
$\epsilon\gamma$-vicinity of $\Sigma (\eta)=\{x'\colon   \boldsymbol{n}(x',\eta)>0\}$, $\gamma=\langle x'\rangle$ and $\cL(x')$ is  defined by \textup{(\ref{24-4-19})}--\textup{(\ref{24-4-20})}.
\end{theorem}

\subsection{Proof of Theorem~\ref{thm-24-4-22}: Propagation of singularities}
\label{sect-24-4-3-2}

Again let us consider the number of negative eigenvalues of operator $A_\eta$, defined by (\ref{24-4-36}) with
\begin{equation}
J(x)= j(x') \langle x\rangle^{-2p} \langle x'\rangle ^{2p} \qquad
p=\left\{\begin{aligned}&q&&\text{if\ \ } m+q<0,\\
&q+1&&\text{if\ \ } m+q=0
\end{aligned}\right.
\label{24-4-78}
\end{equation}
and $\gamma$-admissible function $j(x')$. Let $\rho=\langle x'\rangle ^{2m}$.
As in the previous Subsection~\ref{sect-24-4-2} we consider $A_\eta$ as operator in $\sL^2 (\bR^{2r},\bH)$ with $\bH=\sL^2 (\bR,\bC)$. As usual after proper scaling $h=\rho^{-1}\gamma^{-1}$ and $\mu=\rho^{-1}\gamma$.

Again let consider the corresponding propagator. Our first goal is to estimate the propagation speed with respect to $x'$ from  above and then to estimate it under the microhyperbolicity condition also from below.

\begin{proposition}\label{prop-24-4-23}
Let assumptions $\textup{(\ref{24-4-70})}_{1-3}$ be fulfilled with $m<0$ and $0<q\le 1$. Further, let for $m+q=0$ assumptions $\textup{(\ref{24-4-72})}_{1-3}$ be fulfilled as well.

Then the propagation speed with respect to $x'$  does not exceed
$C_0 j^{-1}\rho^2 \gamma^{-1}$ (before scaling) and therefore singularity initially supported in $B(\y',\frac{1}{2}\bar{\gamma})$ is confined to $B(\y',\gamma)$ for $T\le T^*=\epsilon j\rho^{-2}\gamma^2$ calculated at $\y'$:
\begin{equation}
\3 F_{t\to h^{-1}\tau} \bar{\chi}_T(t) \psi (x') (1-\psi_0 (y'))u(x,y,t)\3\le
C T\gamma^{-s}
\label{24-4-79}
\end{equation}
provided $\psi \in \sC_0^\infty B(\y',\frac{1}{2}\gamma)$,
$\psi_0\in \sC_0^\infty B(\y',\gamma )$, $\psi_0=1$ in
$B(\y', \gamma)$, $\gamma=\gamma(\y')$, $|\tau|\le \epsilon$,
$\3\cdot\3$ is a standard operator norm from $\sL^2(\bR^d)$ to $\sL^2(\bR^d)$.
\end{proposition}

\begin{proposition}\label{prop-24-4-24}
In the framework of Proposition~\ref{prop-24-4-23} assume that the microhyperbolicity assumption \textup{(\ref{24-4-73})} is fulfilled for $m+q<0$ and the microhyperbolicity assumption \textup{(\ref{24-4-74})} is fulfilled for $m+q=0$. Then the propagation speed with respect to $x'$ in an appropriate direction is greater than $\epsilon_1 j^{-1}\rho^2 \gamma^{-1}$ (before scaling) and therefore
\begin{equation}
|F_{t \to h^{-1}\tau} \chi_T (t)\Gamma ' (u \psi )|\le C\gamma^{d-1}(T/T_*)^{-s}\qquad \text{for\ \ } |\tau|\le \epsilon
\label{24-4-80}
\end{equation}
where, as usual, $\chi\in \sC_0^\infty ([-1,-\frac{1}{2}]\cup [\frac{1}{2},1])$, $\psi \in \sC_0^\infty (B(x',\frac{1}{2}\gamma(x'))$ and
$T\in [T_*, T^*]$, $T_*=c j ^{-1}\rho^{-2}$ and $\gamma =\gamma (y')$.
\end{proposition}
\enlargethispage{2\baselineskip}

\begin{proof}[Proofs of Propositions~\ref{prop-24-4-23} and \ref{prop-24-4-24}]
Standard proofs are left to the reader. We just observe the following:

If $m+q<0$ then due to $\textup{(\ref{24-4-70})}_{1-3}$ and (\ref{24-4-78}) the drift speed with respect to $x'$  does not exceed $cj ^{-1}\rho^2\gamma^{-1}$ (Proposition~\ref{prop-24-4-23}) and is exactly of this magnitude due to (\ref{24-4-73}) (Proposition~\ref{prop-24-4-23}). This is also true for $m+q=0$ due to $\textup{(\ref{24-4-72})}_{1-3}$ and (\ref{24-4-78}), and (\ref{24-4-74}).
\end{proof}

\subsection{Proof of Theorem~\ref{thm-24-4-22}: Estimates}
\label{sect-24-4-3-3}

Now we can use the standard successive approximations method leading us to the estimate
\begin{align}
&|F_{t \to h^{-1}\tau} \bar{\chi}_T (t)\Gamma ' (u \psi )|\le C\gamma^{d-1}
\label{24-4-81}\\
\intertext{and then to estimate }
&|F_{t \to h^{-1}\tau} \bar{\chi}_T (t)\Gamma  (u \psi )|\le C\gamma^{d-1} N
\label{24-4-82}
\end{align}
as $T_*\le T\le T^*$ and $|\tau|\le \epsilon$, where
$N=\sup_{B(\y', (1+\epsilon)\gamma)}\boldsymbol{n}(x', \eta)$. The latter estimate leads us to the estimate (\ref{24-4-77}) to the Tauberian error and in the conjugation with successive approximation method to the estimate (\ref{24-4-77}) itself.

\subsection{Discussion}
\label{sect-24-4-3-4}

\begin{remark}\label{rem-24-4-25}
If we replace $\boldsymbol{n}(x',\eta)$ by the corresponding $1\D$-Weyl expression, we will get
\begin{align}
&\N^-(\eta)=\cN^{\MW\,-}(\eta)+O(R(\eta)+R_1(\eta))
\label{24-4-83}\\
\shortintertext{with}
&\cN^{\MW\,-}(\eta)=(2\pi)^{-r} \int (-V^*-\eta)_+^{\frac{1}{2}}/\sqrt{g^{11}}
f_{\infty,1}f_{\infty,2}\cdots  f_{\infty,r}\,dx
\label{24-4-84}\\
\shortintertext{and}
&R_1(\eta)= \int_{\Lambda (\eta)} dx'.
\label{24-4-85}
\end{align}
\end{remark}
\enlargethispage{2\baselineskip}

\begin{remark}\label{rem-24-4-26}
\begin{enumerate}[label=(\roman*), wide, labelindent=0pt]
\item\label{rem-24-4-26-i}
Obviously $V(x)= -\langle x\rangle^{-2q} U(x')$ with $U$ positively homogeneous of degree $2(m+q)$ satisfies $\textup{(\ref{24-4-70})}_{1-3}$, and for $q+m=0$ it also satisfies $\textup{(\ref{24-4-72})}_{1-3}$.

\item\label{rem-24-4-26-ii}
Furthermore, if  $U\asymp \langle x'\rangle^{2(m+q)}$ this $V$ satisfies (\ref{24-4-73}) if $q+m<0$ and (\ref{24-4-74}) if $q+m=0$.

\item\label{rem-24-4-26-iii}
On the other hand, it does not satisfy (\ref{24-4-73}) if $q+m>0$.
\end{enumerate}
\end{remark}

\begin{example}\label{example-24-4-27}
Let us evaluate  magnitudes of $\cN^-(\eta)$, $R(\eta)$ and $R_1(\eta)$, defined by (\ref{24-4-76}),   (\ref{24-4-77}) and (\ref{24-4-85}). To do this consider $\boldsymbol{n}(x',\eta)$. We are interested only in the case of $0<q\le 1$ since $q>1$ combined with $m+q\le 0$ would imply $m< -1$ and this is covered by previous Subsection~\ref{sect-24-4-2}. One can see easily that in this case Theorem~\ref{thm-24-4-22} is a special case of Theorem~\ref{thm-24-4-13} as $\boldsymbol{n}(x',\eta)\le 1$ for $|x'|\ge c$ and $\boldsymbol{n}(x',\eta)\le C$ as $|x'|\le c$.

Assume  that $0<q<1$. Then
$\boldsymbol{n} (x',\eta)\asymp \eta^{(q-1)/(2q)}\gamma^{(m+q)/q}$ with
$\gamma =\langle x'\rangle$ as
$\gamma \le  \epsilon \min (\bar{\gamma}_1,\bar{\gamma}_2)$ with
\begin{equation}
\bar{\gamma}_1=\eta^{(1-q)/(2(m+q))},\qquad \bar{\gamma}_2=\eta^{1/(2m)}
\label{24-4-86}
\end{equation}
 and relying upon Proposition~\ref{prop-24-A-4} we conclude that
\begin{equation}
\boldsymbol{n}(x',\eta)\asymp
\eta^{(q-1)/(2q)}\gamma^{(m+q)/q} \qquad
\text{as\ \ }
\gamma\le  \epsilon \min(\bar{\gamma}_1,\bar{\gamma}_2)
\label{24-4-87}
\end{equation}
provided $m+q<0$. On the other hand,  obviously $\boldsymbol{n}(x',\eta)=0$ as $\gamma\ge C\bar{\gamma}_2$. Recall that $m<0$. Observe that
$\bar{\gamma}_1\ge \bar{\gamma}_2$ for $\eta \le 1$ if and only if
$m\in [-1,0)$.

\begin{enumerate}[label=(\alph*), wide, labelindent=0pt]
\item\label{example-24-4-27-a}
Let $m\in [-1,0)$. Then  (\ref{24-4-87}) holds  as
$\gamma \le \epsilon \bar{\gamma}_2$ and we conclude that
\begin{phantomequation}\label{24-4-88}\end{phantomequation}
\begin{multline}
\cN^- (\eta)\asymp \eta^{(d+m)/(2m)},\quad R(\eta) \asymp \eta^{(d+m-2)/(2m)},\\
R_1(\eta)=\eta^{(d-1)/(2m)}.
\tag*{$\textup{(\ref*{24-4-88})}_{1-3}$}\label{24-4-88-*}
\end{multline}

\item\label{example-24-4-27-b}
Let $m<-1$.  Then (\ref{24-4-87}) holds  as $\gamma \le \epsilon \bar{\gamma}_1$
and contributions of the zone
$\cX_0=\{x\colon |x'|\lesssim \bar{\gamma}_1\}$ to $\cN^-(\eta)$ and $R(\eta)$  are respectively of the magnitudes $\eta^{(1-q)(d-1)/2(m+q)}$ and  $\eta^{(1-q)(d-3)/2(m+q)}$ .

We need to consider the zone
$\cX_1\coloneqq \{x\colon \bar{\gamma}_1 \lesssim |x'|\lesssim  \bar{\gamma}_2\}$ separately.

\item\label{example-24-4-27-c}
If $1>q> \frac{1}{2}$ we in virtue of Proposition~\ref{prop-24-A-6}
$\boldsymbol{n}(x',\eta)\asymp 1$ if $\gamma \le \epsilon \bar{\gamma}_3$ and
$\boldsymbol{n}(x',\eta)= 0$ if $\gamma \ge C \bar{\gamma}_3$ with
$\bar{\gamma}_3= \eta^{1/2(2m+1)}$; one can see easily that
$\bar{\gamma}_1\le \bar{\gamma}_3\le \bar{\gamma}_2$. Then contribution of $\cX_1$ to $\cN^-(\eta)$ and $R(\eta)$ are of magnitudes  $\bar{\gamma}_3^{d-1}$ and $\bar{\gamma}_3^{d-3}$ respectively  except the case $d=3$ when the contribution to $R(\eta)$ is of magnitude $|\log (\eta)|$.

Combining with Statement~\ref{example-24-4-27-b} we conclude that $\cX_1$ contributes more to $\cN^-(\eta)$ and $R(\eta)$ than $\cX_0$ and therefore
\begin{phantomequation}\label{24-4-89}\end{phantomequation}
\begin{equation}
\cN^- (\eta)\asymp \eta^{(d-1)/2(2m+1)},\quad
R(\eta) \asymp \eta^{(d-3)/2(2m+1)}
\tag*{$\textup{(\ref*{24-4-89})}_{1,2}$}\label{24-4-89-*}
\end{equation}
except the case $d=3$ when  $R(\eta)$ is of magnitude $|\log (\eta)|$. Meanwhile, $R_1(\eta)$ is of the same magnitude as $\cN^-(\eta)$.

\item\label{example-24-4-27-d}
If $0<q\le \frac{1}{2}$ then in virtue of Proposition~\ref{prop-24-A-8} with $\varepsilon=\gamma^{2(m+q)}$ we conclude that
$\boldsymbol{n}(x',\eta)\asymp 1$ if $\gamma \le \epsilon \bar{\gamma}_1$ and
$\boldsymbol{n}(x',\eta)= 0$ if $\gamma \ge C  \bar{\gamma}_1$.
Combining with Statement~\ref{example-24-4-27-b} we conclude that in this case
\begin{phantomequation}\label{24-4-90}\end{phantomequation}
\begin{equation}
\cN^- (\eta)\asymp \eta^{(1-q)(d-1)/2(m+q)},\quad
R(\eta) \asymp \eta^{-(q-1)(d-3)/2(m+q)}.
\tag*{$\textup{(\ref*{24-4-90})}_{1,2}$}\label{24-4-90-*}
\end{equation}
Meanwhile, $R_1(\eta)$ is of the same magnitude as $\cN^-(\eta)$.
\end{enumerate}
\end{example}

\begin{Problem}\label{Problem-24-4-28}
As $d=3$ derive remainder estimate $O(1)$ in the framework of Example~\ref{example-24-4-27}\ref{example-24-4-27-c}. This analysis should be done in the zone  $\cX_1$ and we need to consider the spectral stripes $\Lambda_k\coloneqq   \{x'\colon   \lambda_k(x')\asymp \eta\}$ and zone
$\cX_1 \setminus (\Lambda_1\cup \Lambda_2\cup\ldots\cup\Lambda_K)$ separately.
\end{Problem}

\begin{example}\label{example-24-4-29}
\begin{enumerate}[label=(\alph*), wide, labelindent=0pt]
\item\label{example-24-4-29-a}
If $q>1$ then $\boldsymbol{n}(x',\eta)\le C\gamma^{m+1}$ with $m+1 <0$ since we need to assume that $m+q\le 0$ in virtue Remark~\ref{rem-24-4-26}\ref{rem-24-4-26-iii} which puts us in the framework of Subsection~\ref{sect-24-4-2}.

\item\label{example-24-4-29-ba}
If $q=1$, $m<-1$ then $\boldsymbol{n}(x',\eta)=0$ for $|x'|\ge c$ which leads only to the estimate $\N^-(\eta) \le C|\log \eta|$ rather than asymptotics.

\item\label{example-24-4-29-c}
Thus, consider $q=1$, $m=-1$.  Then
$\boldsymbol{n}(x',\eta)\lesssim \gamma^{m+1}|\log (\gamma /\bar{\gamma})|$ and
$\boldsymbol{n}(x',\eta)=0$ if $\gamma \ge C\bar{\gamma}$, $\bar{\gamma}=\eta^{-1/2}$.

This leads to
\begin{phantomequation}\label{24-4-91}\end{phantomequation}
\begin{equation}
\cN^-(\eta)\lesssim \eta^{-(d-1)/2},\qquad R(\eta)\asymp \eta^{-(d-3)/2}
\tag*{$\textup{(\ref*{24-4-91})}_{1,2}$}\label{24-4-91-*}
\end{equation}
again except $d=3$, in which case $R(\eta)\asymp |\log \eta|^2$. However, under assumption $V\le -C_0\langle x\rangle ^{-2}$ with sufficiently large $C_0$, we conclude that
$\boldsymbol{n}(x',\eta)\asymp \gamma^{m+1}|\log (\gamma /\bar{\gamma})|$ if  $\gamma \le \epsilon \bar{\gamma}$ and then
$\cN^-(\eta)\asymp \eta^{-(d-1)/2}$.
\end{enumerate}
\end{example}

\begin{Problem}\label{Problem-24-4-30}
Let us consider the case $m+q>0$. In this case we do not have a microhyperbolicity condition and we can apply only more simple and less precise approach of Subsection~\ref{sect-24-4-1}. So, we leave to the reader to derive the remainder estimate in the following cases:

\begin{enumerate}[label=(\alph*), wide, labelindent=0pt]
\item\label{Problem-24-4-30-a}
Let $0<q<1$\,\footnote{\label{foot-24-32} And then $m\in (-1,0)$.}. Then, exactly as in Example~\ref{example-24-4-27}\ref{example-24-4-27-a} (\ref{24-4-87}) holds if
$\gamma\le \epsilon \bar{\gamma}_2=\eta^{1/(2m)}$ and $\boldsymbol{n}(x',\eta)=0$ if $\gamma\ge \bar{\gamma}_2$ and
therefore $\textup{(\ref{24-4-88})}_1$ holds.

\item\label{Problem-24-4-30-b}
Let $q=1$\,\footref{foot-24-32}.  Then we have the same magnitude of $\boldsymbol{n}(x',\eta)$ as in  Example~\ref{example-24-4-29}\ref{example-24-4-29-c} and
$\cN^-(\eta)\asymp \eta ^{(d+m)/(2m)}$.

\item\label{Problem-24-4-30-c}
Let $q>1$. Then for $m<-1$ we are in the framework of Subsection~\ref{sect-24-4-2}, and for $m=-1$ we are either in the framework of that Subsection or close to it.

So, let us assume that $m\in (-1,0)$. In this case $\boldsymbol{n}(x',\eta)\lesssim \gamma^{m+1}$,  $\boldsymbol{n}(x',\eta)\asymp \gamma^{m+1}$ for
$\gamma\le \epsilon\bar{\gamma}_2$ and $\boldsymbol{n}(x',\eta)=0$ for
$\gamma\ge C\bar{\gamma}_2$. Then $\textup{(\ref{24-4-88})}_1$ holds.
\end{enumerate}
\end{Problem}

Finally, we leave to the reader

\begin{Problem}\label{Problem-24-4-31}
Consider the case of $\rank \F =d-2r\le d-2$. To do so we need to modify Theorem~\ref{thm-24-4-22} in rather obvious way, in this case $x'=(x'';x')=
(x_1,\ldots,x_n; x_{n+1},\ldots, x_d)$ with $n=d-2r$ and $\cL(x')$ is $n$-dimensional Schr\"odinger operator.
\end{Problem}

\chapter{$3\D$-case.  Multiparameter asymptotics}
\label{sect-24-5}

In this section we consider asymptotics with respect to three parameters $\mu$, $h$ and $\tau$; here spectral parameter $\tau$ tends \underline{either} to $\pm \infty$ \underline{or} to the border of the essential spectrum \underline{or} to $-\infty$ (for Schr\"odinger and Schr\"odinger-Pauli operators) \underline{or} to the border of the spectrum. In two last cases presence of  $h\to+0$ is crucial. We consider here only $d=2$ and $h\ll 1$.

\section{Asymptotics of large eigenvalues}
\label{sect-24-5-1}

In this subsection $\tau\to +\infty$ for the Schr\"odinger and Schr\"odinger-Pauli operators and $\tau\to \pm \infty$ for the Dirac operator. We consider  the  Schr\"odinger and Schr\"odinger-Pauli operators, leaving the Dirac operator to the reader.

\begin{example}\label{example-24-5-1}
Assume first that $\psi\in \sC_0^\infty$  and there are no singularities on $\supp(\psi)$. We consider
\begin{equation}
\N_\psi^-(\tau)=\int e(x,x,\tau)\psi(x)\,dx.
\label{24-5-1}
\end{equation}
Then for scaling $A\mapsto \tau^{-1}A$ leads to
$h\mapsto h_\eff=h\tau^{-1/2}$ and $\mu\to \mu_\eff=\mu\tau^{-1/2}$.

\begin{enumerate}[label=(\roman*), wide, labelindent=0pt]
\item\label{example-24-5-1-i}
If $\mu \lesssim \tau^{1/2}$ then  we can apply the standard theory with the ``normal" magnetic field; we need to assume that $h\ll\tau^{1/2}$ and we need neither condition $d=3$, nor $F\ge \epsilon_0$, nor $\partial X=\emptyset$; the principal part of the asymptotics has magnitude $h^{-d}\tau^{d/2}$ and the remainder estimate is $O(h^{1-d}\tau^{(d-1)/2})$ which one can even improve to $o(h^{1-d}\tau^{(d-1)/2})$ under proper non-periodicity assumption.

\item\label{example-24-5-1-ii}
Let $\mu \gtrsim \tau^{1/2}$, $\mu h\lesssim \tau$. Then   we can apply the standard theory with the ``strong" magnetic field; we assume that $d=3$ and  $F\ge \epsilon_0$. Then the principal part of the asymptotics has magnitude $h^{-3}\tau^{3/2}$ and under weak non-degeneracy assumption (which is needed, only in the case $\mu_\eff\le h_\eff^{\delta-1}$)
fulfilled on $\supp(\psi)$ the remainder estimate is
$O(h^{-2}\tau )$ and marginally worse without non-degeneracy assumption

\item\label{example-24-5-1-iii}
If $\mu \gtrsim \tau^{1/2}$, $\mu h \ge c\tau$ than $\N^-(\tau)=0$ for the Schr\"odinger operator; for the Schr\"odinger-Pauli operator  the principal part of the asymptotics has magnitude $\mu h^{-1}$ and under  weak non-degeneracy assumptions the remainder estimate is $O(\mu h^{-1})$.
\end{enumerate}
\end{example}
\enlargethispage{2\baselineskip}

\begin{example-foot}\label{example-24-5-2}\footnotetext{\label{foot-24-33} Cf. Example~\ref{example-24-3-12}.}
Let $X$ be an unbounded domain.  Let  conditions (\ref{24-2-1}),  (\ref{24-2-6}), \ref{24-2-2-*}, (\ref{24-3-1}) and \ref{24-3-9-*}  be fulfilled with $\gamma =\epsilon _0 \langle x\rangle $, $\rho =\langle x\rangle ^m$,
$\rho _1=\langle x\rangle ^{m_1}$, $m_1>2m$. Consider the Schr\"odinger operator and assume that
\begin{phantomequation}\label{24-5-2}\end{phantomequation}
\begin{gather}
\tau\ge \mu h, \qquad\tau^{2m-m_1}\le \epsilon (\mu h)^{2m}.
\tag*{$\textup{(\ref*{24-5-2})}_{1,2}$}\label{24-5-2-*}\\
\shortintertext{Then}
\cN^-(\tau,\mu,h)\asymp \tau^{3(2+m_1)/(2m_1)}h^{-3(1+m_1)/m_1}\mu^{-3/m_1}.
\label{24-5-3}
\end{gather}
\begin{enumerate}[label=(\roman*), wide, labelindent=0pt]
\item\label{example-24-5-2-i}
Further, if  $\tau \gtrsim \mu^2$, then the zone of the strong magnetic field
$\mu _\eff =\mu \langle x\rangle^{m_1+1}\tau^{-1/2}\ge C$ is contained in $\{x\colon |x|\ge c\}$
and here we have non-degeneracy condition fulfilled. Then the remainder estimate is $O(R)$ with
\begin{equation}
R=\tau^{(2+m_1)/m_1}h^{-2(1+m_1)/m_1}\mu^{2/m_1}.
\label{24-5-4}
\end{equation}

\item\label{example-24-5-2-ii}
On the other hand, if $\mu ^2\gg \tau$, then the contribution of the zone $\{x\colon |x|\ge c\}$ to the remainder is $O(R)$ with $R$ defined by (\ref{24-5-4}). The contribution of the zone $\{x\colon |x|\le c\}$ to the remainder is $O(h^{-2}\tau)$ under weak non-degeneracy assumption  fulfilled there.

\item\label{example-24-5-2-iii}
Let us replace $\textup{(\ref{24-5-2})}_{2}$ by the opposite inequality, and assume \ref{24-3-3-*}. Then (\ref{24-5-3}) is replaced by
$\cN^-(\tau,\mu, h)\asymp h^{-3}\tau ^{3(m+1)/(2m)}$ while under non-degeneracy assumption
\begin{equation}
R=h^{-2}\tau ^{(m+1)/m}.
\label{24-5-5}
\end{equation}
\end{enumerate}
\end{example-foot}

\begin{example-foot}\label{example-24-5-3}\footnotetext{\label{foot-24-34} Cf. Example~\ref{example-24-3-14}.}
In the framework of Example~\ref{example-24-5-2} for the Schr\"odinger-Pauli operator under assumption $\textup{(\ref{23-3-3})}^{\#\prime}$ of \cite{futurebook}
\begin{align}
&\cN^-(\tau,\mu, h)\asymp h^{-3}\tau ^{3(m+1)/(2m)} + \mu h^{-2}\tau^{(m_1+m+3)/(2m)}.
\label{24-5-6}\\
&R=  h^{-2}\tau ^{(m+1)/m}+\mu h^{-1}\tau^{(m_1+2)/(2m)}
\label{24-5-7}
\end{align}
\end{example-foot}

We leave to the reader

\begin{problem}\label{problem-24-5-4}
Consider the Schr\"odinger and Schr\"odinger-Pauli operators
with other types of the behaviour at infinity.
\end{problem}

\begin{problem}\label{problem-24-5-5}
For the Dirac operators derive similar results as $\tau\to \pm \infty$.
\end{problem}

\section{Asymptotics of the small eigenvalues}
\label{sect-24-5-2}

In this subsection for the Schr\"odinger and Schr\"odinger-Pauli operators we consider asymptotics of eigenvalues tending to $-0$.

\begin{example-foot}\label{example-24-5-6}\footnotetext{\label{foot-24-35} Cf. Example~\ref{example-24-4-2}.}
Let $X$ be an unbounded domain.  Let conditions \textup{(\ref{24-2-1})}, \ref{24-2-3-*}, $\textup{(\ref{24-3-1})}^\#_1$  be fulfilled with
$\gamma =\epsilon _0 \langle x\rangle $,
$\rho =\langle x\rangle ^m$, $\rho _1=\langle x\rangle ^{m_1}$, $-1< m< 0$,
$m_1 > m-1$.

Consider the Schr\"odinger operator and assume that
\begin{phantomequation}\label{24-5-8}\end{phantomequation}
\begin{equation}
1\ge \mu h, \qquad|\tau|^{2m-m_1}\le \epsilon (\mu h)^{2m}.
\tag*{$\textup{(\ref*{24-5-8})}_{1,2}$}\label{24-5-8-*}
\end{equation}
Then $\cN^-(\tau)=O( h^{-3}|\tau|^{3(m+1)/(2m)})$ as $\tau\to -0$ with ``$\asymp$''  instead of ``$=O$'' if condition $V\le -\epsilon \rho^2$ fulfilled in some non-empty cone.

Further, under the  non-degeneracy assumption \ref{24-3-4-**} the contribution to the remainder of the zone $\{x\colon |x|\ge c\}$ is
$O( h^{-2}|\tau|^{(m+1)/m})$).
\end{example-foot}

\begin{example-foot}\label{example-24-5-7}\footnotetext{\label{foot-24-36} Cf. Example~\ref{example-24-4-2}.}
In the framework of Example~\ref{example-24-5-6} for the Schr\"odinger-Pauli operator under assumption \ref{24-3-3-*} the contribution to the remainder of the zone $\{x\colon |x|\ge c\}$ is $O(R)$ with
\begin{align}
&R=h^{-2}|\tau| ^{(m+1)/m} + \mu h^{-1}|\tau|^{(m_1+2)/(2m)},
\label{24-5-9}\\
\shortintertext{while}
&\cN^-(\tau)=O( h^{-3}|\tau| ^{3(m+1)/(2m)} + \mu h^{-2}|\tau|^{(m_1+m+3)/(2m)})
\label{24-5F-9}
\end{align}
 with ``$\asymp$''  instead of ``$=O$'' if condition $V\le -\epsilon \rho^2$  fulfilled in some non-empty cone.
\end{example-foot}

\begin{problem}\label{problem-24-5-8}
Consider the Schr\"odinger and Schr\"odinger-Pauli operators if
\begin{enumerate}[label=(\roman*), wide, labelindent=0pt]
\item\label{problem-24-5-8-i}
If condition $\textup{(\ref{24-5-8})}_1$ is violated (then there could be a forbidden zone in the center).
\item\label{problem-24-5-8-iii}
With other types of the behaviour at infinity.
\end{enumerate}
\end{problem}

\begin{problem}\label{problem-24-5-9}
Consider the Schr\"odinger and Schr\"odinger-Pauli operators in the framework of Subsections~\ref{sect-24-4-2} and~\ref{sect-24-4-3} if
\begin{enumerate}[label=(\roman*), wide, labelindent=0pt]
\item\label{problem-24-5-9-i}
$\mu h=1$; then the essential spectrum  does not change.

\item\label{problem-24-5-9-ii}
$\mu h\to \infty$; then the limit of the essential spectrum is $\emptyset$ (for the Schr]"odinger operator) and $[0,\infty)$ (for the Schr\"odinger-Pauli operator). Consider  $\N^-(\eta)$ with $\eta\to 0$.

\item\label{problem-24-5-9-iii}
$\mu h\to 0$;  then the limit of the essential spectrum is  $[0,\infty)$ (for both the Schr\"odinger and Schr\"odinger-Pauli operators). Consider  $\N^-(\eta)$ with $\eta\to 0$.
\end{enumerate}
\end{problem}

\begin{problem}\label{problem-24-5-10}
For the Dirac operators derive similar results as $M\ne 0$ and
$\tau\to M-0$ and $-M+0$.
\end{problem}

\section{Case of $\tau\to+0$}
\label{sect-24-5-3}

In this subsection $\tau\to +0$ for the Schr\"odinger and Schr\"odinger-Pauli ope\-rators and $\tau\to \pm M\pm 0$ for the Dirac operator. Consider  the  Schr\"odinger and Schr\"odinger-Pauli operators first.

\begin{example-foot}\label{example-24-5-11}\footnotetext{\label{foot-24-37} Cf. Example~\ref{example-24-5-2}.}
Let $V>0$ everywhere except $V(0)=0$.  Let conditions (\ref{24-2-1}), \ref{24-2-2-*} and (\ref{24-3-1}) be fulfilled with
$\gamma =\epsilon _0 |x|$, $\rho =|x| ^m$,
$\rho _1=| x| ^{m_1}$, $m_1> 2m \ge 0$. Consider the Schr\"odinger operator and assume that $\tau\to +0$.
\begin{enumerate}[label=(\roman*), wide, labelindent=0pt]
\item\label{example-24-5-11-i}
Let
\begin{phantomequation}\label{24-5-11}\end{phantomequation}
\begin{equation}
\mu \ll \tau^{(m_1+2)/2}h^{-(m_1+1)}, \qquad
\tau^{2m-m_1}\le \epsilon (\mu h)^{2m}.
\tag*{$\textup{(\ref*{24-5-11})}_{1,2}$}\label{24-5-11-*}
\end{equation}
Then (\ref{24-5-3}) holds while  the remainder estimate is $O(R)$ with defined by (\ref{24-5-4}).

\item\label{example-24-5-11-ii}
Let us replace $\textup{(\ref*{24-5-11})}_{2}$ by the opposite inequality, and assume \ref{24-3-3-*}. Then (\ref{24-5-3}) is replaced by
$\cN^-(\tau)\asymp h^{-3}\tau ^{3(m+1)/(2m)}$ while $R=h^{-2}\tau^{(m+1)/2m}$.
\end{enumerate}
\end{example-foot}

\begin{example-foot}\label{example-24-5-12}\footnotetext{\label{foot-24-38} Cf. Example~\ref{example-24-5-3}.}
In the framework of Example~\ref{example-24-5-11} for the Schr\"odinger-Pauli operator under assumption \ref{24-3-3-*} both  (\ref{24-5-6}) and (\ref{24-5-7}) hold.
\end{example-foot}

We leave to the reader

\begin{problem-foot}\label{problem-24-5-13}\footnotetext{\label{foot-24-39} Cf. Problem~\ref{problem-24-5-4}.}
Consider the Schr\"odinger and Schr\"odinger-Pauli operators
\begin{enumerate}[label=(\roman*), wide, labelindent=0pt]
\item\label{problem-24-5-13-i}
In the same framework albeit with condition $m_1>2m$ replaced by
 $2m\ge m_1 \ge 0$. Assume that \ref{24-3-3-*} is fulfilled.

Then magnitude of $\cN^-(\tau)$ is described in  Examples~\ref{example-24-5-11} and~\ref{example-24-5-12}. Under proper non-degeneracy assumption (which we leave to the reader to formulate) derive the remainder estimate.

\item\label{problem-24-5-13-ii}
In the same framework as in \ref{problem-24-5-13-i} albeit  in with  $m_1<0$ (magnetic field is stronger in the center but there is no singularity), in which case the center can become a classically forbidden zone.

\item\label{problem-24-5-13-iii}
With other types of the behaviour at infinity.
\end{enumerate}
\end{problem-foot}

\begin{problem}\label{problem-24-5-14}
For the Dirac operators derive similar results as $\tau\to \pm (M+0)$.
\end{problem}

\section{Case of $\tau\to-\infty$}
\label{sect-24-5-4}

In this subsection for the Schr\"odinger and Schr\"odinger-Pauli operators we consider asymptotics with $\tau\to -\infty$.

In this subsection for the Schr\"odinger and Schr\"odinger-Pauli operators we consider asymptotics of eigenvalues tending to $-0$.

\begin{example-foot}\label{example-24-5-15}\footnotetext{\label{foot-24-40} Cf. Example~\ref{example-24-4-2}.}
Let $X\ni 0$   and let conditions \textup{(\ref{24-2-1})}, \ref{24-2-2-*}, (\ref{24-3-1}) and a week non-degeneracy assumption be fulfilled with $\gamma =\epsilon _0 \langle x\rangle $,
$\rho =\langle x\rangle ^m$, $\rho _1=\langle x\rangle ^{m_1}$, $-1< m< 0$,
$m_1 > m-1$.

Consider the Schr\"odinger operator and assume that
\begin{phantomequation}\label{24-5-12}\end{phantomequation}
\begin{equation}
h\ll |\tau|^{(m+1)/(2m)}, \qquad|\tau|^{m_1-2m}\le \epsilon (\mu h)^{2m}.
\tag*{$\textup{(\ref*{24-5-12})}_{1,2}$}\label{24-5-12-*}
\end{equation}
Then $\cN^-(\tau)=O( h^{-3}|\tau|^{3(m+1)/(2m)})$ as $\tau\to -0$ with
``$\asymp $'' instead of ``$=O$'' if condition $V\le -\epsilon \rho^2$  as $|x|\le \epsilon$  fulfilled in some non-empty cone and $R=h^{-2}|\tau|^{(m+1)/m}$.
\end{example-foot}

\begin{example-foot}\label{example-24-5-16}\footnotetext{\label{foot-24-41} Cf. Example~\ref{example-24-4-2}.}
In the framework of Example~\ref{example-24-5-15} for the Schr\"odinger-Pauli operator under assumption \ref{24-3-3-*}
\begin{align}
&\cN^-(\tau)=O( h^{-3}|\tau| ^{3(m+1)/(2m)} + \mu h^{-2}|\tau|^{(m_1+m+3)/(2m)})
\label{24-5-13}\\
\shortintertext{and}
&R=h^{-2}|\tau| ^{(m+1)/m}+\mu h^{-1}|\tau|^{(m_1+2)/(2m)})
\end{align}
with ``$\asymp $'' instead of ``$=O$''  if condition $V\le -\epsilon \rho^2$  as $|x|\le \epsilon$  fulfilled in some non-empty cone.
\end{example-foot}

\begin{problem}\label{problem-24-5-17}
Consider the Schr\"odinger and Schr\"odinger-Pauli operators if
with other types of the behaviour at $0$.
\end{problem}

\chapter[$3\D$-case. Strong mgnetic field and singularities]{$3\D$-case.  Asymptotics of
Riesz means for Schr\"{o}dinger operators singular at a point with
strong magnetic field}
\label{sect-24-6}

In this section we consider Riesz means for the 3-dimensional Schr\"{o}\-dinger-Pauli operator.  We leave it to the reader to treat the easier case of the Schr\"odinger operator.

\section{The regular case}
\label{sect-24-6-1}
We would like to recall the results of Chapter~\ref{book_new-sect-13} of \cite{futurebook} for the regular case, under some more restrictive assumptions.  Namely, let us
assume that
\begin{equation}
B(0,1)\subset X,
\label{24-6-1}
\end{equation}
in $B(0,1)$ conditions \ref{24-2-3-*}  are fulfilled and
\begin{equation}
F\ge \epsilon _0.
\label{24-6-2}
\end{equation}

Then in Chapter~\ref{book_new-sect-13} of \cite{futurebook} the following statement was proven:

\begin{theorem}\label{thm-24-6-1}
Let $d=3$, let the Schr\"odinger-Pauli operator  be self-adjoint and let conditions \ref{24-2-3-*}, \textup{(\ref{24-6-1})} and \textup{(\ref{24-6-2})}be fulfilled.  Moreover, let us assume that in $B(0,1)$
\begin{equation}
V\le \epsilon _0\implies
\bigl|\nabla \frac{V}{F} \bigr|\ge \epsilon _0.
\label{24-6-3}
\end{equation}
Finally, let
$\psi \in \sC_0^K\bigl(B(0, \frac{1}{2} )\bigr)$.  Then
\begin{multline}
\bigl|\int \psi (x)\Bigl(e_\vartheta (x,x,\tau )-
h^{-3} S_\vartheta (x,\mu h,\tau )\Bigr)\,dx\bigr|\le
C h^{-2+\vartheta }(1+\mu h)\\
\forall \tau \le  \frac{1}{2} \epsilon _0
\label{24-6-4}
\end{multline}
where
\begin{equation}
S(x,\mu h,\tau )={\frac{1}{4\pi ^2}}\sum_{n\in \bZ^+}
\bigl(\tau -V-2n\mu hF\bigr)_+^{\frac{1}{2}}F\mu h\sqrt{g}
\label{24-6-5}
\end{equation}
and for $\vartheta \in [0,1]$
\begin{equation}
S_\vartheta (x,\mu h,\tau )=
\vartheta \tau_+ ^{\vartheta-1} *S=
\const\sum_{n\in \bZ^+}
\bigl(\tau -V-2n\mu hF\bigr)_+^{ \frac{1}{2} +\vartheta }
F\mu h\sqrt{g}.
\label{24-6-6}
\end{equation}
\end{theorem}

\begin{remark}\label{rem-24-6-2}
\begin{enumerate}[label=(\roman*), wide, labelindent=0pt]
\item\label{rem-24-6-2-i}
 This theorem also remains true without certain conditions for
$\mu\le 1$.  This follows from routine semiclassical
asymptotics.

\item\label{rem-24-6-2-ii}
Condition (\ref{24-6-3}) is too restrictive; a weaker condition
was used in Chapter~\ref{book_new-sect-13} of \cite{futurebook}.  In fact, no condition was necessary for
$\vartheta >0$.  However, that is not very important in this Section and it could lead to more complicated expression for $S_\vartheta (x,\mu, h,\tau )$.

\item\label{rem-24-6-2-iii}
For $\vartheta >1$ it is difficult to provide explicit formula for
 $S_\vartheta (x,\mu, h,\tau )$.
\end{enumerate}
\end{remark}

\section{The singular case. I}
\label{sect-24-6-2}

Now let us consider the Schr\"odinger-Pauli in the case when there
is a singularity at $0$.  Namely, we replace assumptions \ref{24-2-3-*}, (\ref{24-6-2}) and (\ref{24-6-3}) by
\begin{phantomequation}\label{24-6-7}\end{phantomequation}
\begin{align}
&|D ^\sigma g^{jk}|\le c|x|^{-|\sigma |},\qquad
&&|D^\sigma V_j|\le c|x|^{\beta -|\sigma |+1},
\tag*{$\textup{(\ref*{24-6-7})}_{1,2}$}\label{24-6-7-1}\\[3pt]
&|D ^\sigma V|\le c|x|^{2\alpha -|\sigma |}\qquad
&&\forall \sigma : |\sigma |\le K
\tag*{$\textup{(\ref*{24-6-7})}_{3}$}\label{24-6-7-3}
\end{align}
and
\begin{gather}
F\ge \epsilon _0|x|^\beta ,
\label{24-6-8}\\
V\le \epsilon _0|x|^{2\alpha }\implies
\bigl|\nabla {\frac{V}{F}}\bigr|\ge \epsilon _0|x|^{2\alpha -\beta -1}
\label{24-6-9}
\end{gather}
where we assume that $\alpha \in (-1,0]$, $\beta >2\alpha $.

Our standard goal is to derive the asymptotics of
\begin{equation}
\int \psi (x)e_\vartheta (x,x,0)\,dx
\label{24-6-10}
\end{equation}
but instead we derive the asymptotics of
\begin{equation}
\int \psi (x)\Bigl(e_\vartheta (x,x,0)-e^0_\vartheta (x,x,0)\Bigr)\,dx
\label{24-6-11}
\end{equation}
where $A^0$ is the same Schr\"odinger-Pauli operator albeit  with $\mu=0$ in $B(0,1)$.  The corresponding asymptotics for this operator are due to Section~\ref{book_new-sect-12-5} of \cite{futurebook} under appropriate conditions.  First of all we treat
the general situation without propagation of singularities arguments.

Let us consider the operator $A$ in $B(\bar{x},\frac{1}{2}r)$ with $r=|\bar{x}|$. Then the routine rescaling procedure results in the operator with $\hbar =hr^{-\alpha-1}$ and $\mu'=\mu r^{\beta+1-\alpha}$ and one should compare these quantities with $1$.  For $\hbar \le 1$ one can apply the semiclassical approach to $\int \bar{\psi} e_\vartheta(x,x,0)\,dx$ which provides the
following remainder estimate
\begin{multline}
R_1=
r^{2\alpha \vartheta }\bigl(\hbar^{-1}+\mu '\bigr)\hbar^{-1+\vartheta}=\\
 h^{-2+\vartheta }r^{\alpha \vartheta -\vartheta +2\alpha +2}+
\mu h ^{-1+\vartheta }r^{\alpha \vartheta -\vartheta +\beta +2}
\label{24-6-12}
\end{multline}
where $\bar{\psi}$ is an element of the partition of unity.

On the other hand, let us apply semiclassical asymptotics to
\begin{equation}
\int \bar{\psi }Be^\eta _{\vartheta -1}(x,x,0)\,dx
\label{24-6-13}
\end{equation}
where $B=A-A^0$ and $A^\eta =\eta A^0+(1-\eta)A$.  Then the remainder estimate
is $\mu' \hbar^{-1}R_1$ (provided $\mu'\le 1$).  This is justified
for $\vartheta\ge 1$; for $0<\vartheta<1$ one should apply
the arguments of Subsection~\ref{book_new-sect-12-5-4} of \cite{futurebook} to get the remainder estimate
$(\mu' \hbar^{-1})^\vartheta R_1$ (see the rigorous analysis  below).

Both of these remainder estimates are better than $R_1$ if $\mu '<h'$.  Therefore one should introduce $r_*$ with $\mu '=\hbar$ and apply the first and second approaches for $r>r_*$ and for
$r<r_*$ respectively.  Thus, $r_*=( h/\mu\bigr)^{1/(\beta  +2)}$.  However, one should remember the condition $\hbar\le  1$ which yields two different situations:
\begin{enumerate}[label=(\Alph*), wide, labelindent=0pt]
\item\label{sect-24-6-2-A}
The magnetic field is not very strong:
\begin{equation}
\mu ^{\alpha +1} h^{\beta +1-\alpha }\le 1.
\label{24-6-14}
\end{equation}
\item\label{sect-24-6-2-B}
 The magnetic field is rather strong:
\begin{equation}
\mu ^{\alpha +1} h^{\beta +1-\alpha }\ge 1.
\label{24-6-15}
\end{equation}
\end{enumerate}
In the former case \ref{sect-24-6-2-A} surely $\hbar\le 1$ for
$r\ge r_*$.

Now we want a better implementation of the above idea.  Let us introduce
$r_2=\mu ^{-1/(\beta +1-\alpha ) }$ from the condition $\mu '=1$.  Due to (\ref{24-6-14}) $r^*\ge r_*$.  Let us apply straightforward semiclassical approximation \emph{only} in the zone $\{r^*\le |x|\le \varepsilon \}$.  Then the contribution of this zone to the remainder estimate is $O(R^*_1)$ with
\begin{multline}
R^*_1=\left\{\begin{aligned}
&h^{-2+\vartheta }+\mu h ^{-1+\vartheta } \qquad
&&\text{for\ \ } \vartheta <\bar{\vartheta }_0,\\
&h^{-2+\vartheta }\log \mu +\mu h ^{-1+\vartheta } \qquad
&&\text{for\ \ } \vartheta =\bar{\vartheta }_0,\\
&h^{-2+\vartheta }
\mu^{(1-\alpha)(\vartheta- \bar{\vartheta }_0)/(\beta +1-\alpha )}+
\mu h ^{-1+\vartheta }
\quad
&&\text{for\ \ } \vartheta >\bar{\vartheta }_0
\end{aligned}\right.
\label{24-6-16}
\end{multline}
with the same $\bar{\vartheta }_0=(2\alpha +2)/(1-\alpha ) $
as in Section~\ref{book_new-sect-12-5} of \cite{futurebook}.

Let us now treat the contribution of $B(0,r^*)$.  Making the rescaling
$x_\new =r ^{*\,-1}x$ we are in exactly the framework of Section~\ref{book_new-sect-12-5} of \cite{futurebook}
with $\hbar=h \mu ^{(\alpha+1)/(\beta+1-\alpha) }\le 1$.  Let us apply the results of this section.  We should consider the general situation and the
situation of escape condition (see Definition~\ref{book_new-def-12-5-14} of \cite{futurebook}).

\emph{Let us first consider the general situation:}
\begin{enumerate}[label=(\roman*), wide, labelindent=0pt]
\item\label{sect-24-6-2-i}
First of all, one can easily see that the contribution of $B(0,r^*)$
to the remainder  is $O(h^{-2+\vartheta})$ for $\vartheta< \bar{\vartheta}_0$.

\item\label{sect-24-6-2-ii}
On the other hand, for
$\bar{\vartheta }\ge \vartheta \ge  \bar{\vartheta }_0$ (with
$\bar{\vartheta }= -(d-1)(\alpha +1)/(\alpha +\beta +1) $ for
$\alpha +\beta +1<0$ and $\bar{\vartheta }=\infty $ for
$\alpha +\beta +1\ge 0$) the contribution of
$B(0,r^*)$ to the remainder estimate is given by the first or
second line of (\ref{book_new-12-5-74}) of \cite{futurebook}\,\footnote{\label{foot-24-42} With $hr^{*\,-\alpha-1}$ instead of $h$ and with the additional factor
$r^{*\,2\alpha\vartheta}$.}
and is
$O\bigl(r_*^{2\alpha\vartheta} (hr_3^{-\alpha-1})^{-2+\vartheta}\bigr)$;
this expression should include also the factor $\log h$ if
$\vartheta -2(\alpha +1)/(\beta +2)\in \bZ^+$.

\item\label{sect-24-6-2-iii}
Finally, for $\vartheta >\bar{\vartheta }$ the contribution of
$B(0,r^*)$ to the remainder is given by the third line in
(\ref{book_new-12-5-74}) of \cite{futurebook}\,\footref{foot-24-42} and is
$O\bigl(h^{(\alpha+\beta+1)\vartheta/(\alpha+1)}\mu^\vartheta\bigr)$.
\end{enumerate}

On the other hand, \emph{under escape condition\/} (see Definition~\ref{book_new-def-12-5-14} of \cite{futurebook}))  there is no adjustment in the cases \ref{sect-24-6-2-i} and \ref{sect-24-6-2-iii} but in the case \ref{sect-24-6-2-ii} the contribution of $B(0,r^*)$ is
$O\bigl(r^{*\,2\alpha\vartheta} (hr^{*\,-\alpha-1})^{-2+\vartheta}\bigr)$.

However, to be able to enjoy these remainder estimates in
cases \ref{sect-24-6-2-ii} and \ref{sect-24-6-2-iii} one should assume that \underline{either} $\vartheta \le \bar{\vartheta}_1$ \underline{or} the homogeneity condition is fulfilled.  For the sake of simplicity in the second case we assume that

\begin{claim}\label{24-6-17}
$g^{jk}=g^{0jk}+g^{1jk}$, $V=V^0+V^1$ and $g^{0jk}$, $g^{1jk}$, $V^0$, $V^1$, $V_j$ are $\sC^K$ functions (outside of $0$), positively homogeneous of degrees
$0, \beta+1-\alpha, 2\alpha, \alpha+\beta+1, \beta+1$ respectively.
\end{claim}

In the first case we only assume that

\begin{claim}\label{24-6-18}
$g^{jk}=g^{0jk}+g^{1jk}$, $V=V^0+V^1$ and
$g^{0jk}$, $V^0$, are $\sC^K$ functions (outside of $0$), positively
homogeneous of degrees $0, 2\alpha$ respectively and
\begin{equation*}
\bigl|D ^\sigma g ^{1jk}\bigr|
\le c|x| ^{\beta +1-\alpha -|\sigma |},\qquad
\bigl|D ^\sigma V ^1\bigr|\le c|x| ^{\beta +1+\alpha -|\sigma |}
\qquad
 \forall \sigma :|\sigma |\le K
\end{equation*}
\end{claim}
in addition to (\ref{24-6-14}).

Then the terms generated by the singularity are
\begin{equation}
\sum_{j,k} \omega _{jk}
h ^{(2\alpha +(j+k)(\beta -\alpha +1))(\alpha +1) ^{-1}}\mu ^j
\label{24-6-19}
\end{equation}
while the other terms are semiclassical (but not necessary Weyl).
To avoid computational difficulties we will only consider the case
$\vartheta\le 1$ in the final statement.  The calculations
for $\vartheta>1$ are left to the reader.
\enlargethispage{2\baselineskip}

\begin{theorem}\label{thm-24-6-3}
\begin{enumerate}[label=(\roman*), wide, labelindent=0pt]
\item\label{thm-24-6-3-i}
Let the Schr\"odinger-Pauli operator be self-adjoint and let conditions \textup{(\ref{24-2-2})}, $\textup{(\ref{24-6-7})}_{1-3}$, \textup{(\ref{24-6-8})},
 \textup{(\ref{24-6-9})}  and \textup{(\ref{24-6-18})}
be fulfilled.  Let $1\ge \vartheta\ge \bar{\vartheta}_0$.  Then for
$\mu^{\alpha+1} h^{\beta+1-\alpha}\le 1$ the following estimate holds:
\begin{multline}
R_\vartheta \coloneqq
\bigl|\int \psi (x)\Bigl(e_\vartheta (x,x,0)-
e_\vartheta ^0(x,x,0 )- \\
\shoveright{h^{-3}S_\vartheta (x,\mu h,0)+
h^{-3}\varkappa _{1,0}(x)\Bigr)\,dx\bigr|\le} \\
Ch ^{-2+\vartheta }
\bigl(\frac{h}{\mu }\bigr)^{(\alpha\vartheta -\vartheta +2\alpha +2)/(\beta +2)} + C\mu h ^{-1+\vartheta }.
\label{24-6-20}
\end{multline}

\item\label{thm-24-6-3-ii}
Moreover, under escape condition (see Definition~\ref{book_new-def-12-5-14} of \cite{futurebook}) the following estimate holds
\begin{equation}
R_\vartheta \le Ch ^{-2+\vartheta }
\mu ^{-(\alpha \vartheta -\vartheta +2\alpha +2)/(\beta +1-\alpha )}+
C\mu h ^{-1+\vartheta }
\label{24-6-21}
\end{equation}
\end{enumerate}
\end{theorem}

\begin{remark}\label{rem-24-6-4}
Let us compare the remainder estimates with the Scott correction $\omega h^l$ with $l=2\alpha\vartheta/(\alpha+1)$. One can see easily that for
$\mu^{\alpha+1} h^{\beta+1-\alpha}\le 1$ the first terms in these remainder estimates are surely less than the Scott term. Therefore the remainder estimate is less than the Scott term if and only if $\mu \le h^l$ which means exactly that $\mu\le h^{(\alpha\vartheta+\alpha+1-\vartheta)/(\alpha+1)}$
and the idea to treat $(e_\vartheta-e^0_\vartheta)$
instead of $e_\vartheta$ is reasonable only in this case.
\end{remark}

\section{The singular case. II}
\label{sect-24-6-3}

Let us consider the case (\ref{24-6-15}).  In this case we should treat
\begin{equation}
R^*_\vartheta =
\bigl|\int \psi (x)\Bigl(e_\vartheta (x,x,0)-
h^{-3}S_\vartheta (x,\mu h,0)\Bigr)\,dx\bigr|.
\label{24-6-22}
\end{equation}
Let us pick $r_0=h^{1/(\alpha +1)}$.  In the zone $\{хэцолон |x|\ge r_0\}$ the routine semiclassical technique is applicable and the contribution of this zone to the remainder estimate does not exceed $CR$ with
\begin{equation}
R=\left\{\begin{aligned}
&\mu h^{-1+\vartheta }\qquad
&&\text{for\ \ } (3\alpha +1)\vartheta+\beta -2\alpha >0\\
&\mu h^{-1+\vartheta }\log\mu \qquad
&&\text{for\ \ } (3\alpha +1)\vartheta+\beta -2\alpha =0\\
&\mu h^{2\alpha \vartheta /(\alpha +1)}\qquad
&&\text{for\ \ } (3\alpha +1)\vartheta+\beta -2\alpha <0
\end{aligned}\right.
\label{24-6-23}
\end{equation}
So we need to estimate the contribution of $B(0,r_0)$.  The routine
estimate due to Chapter~8 and standard variational estimates gives
a very poor result (which the reader can derive if desired).  So
we want to improve the result under certain conditions.  Namely,
let us assume that conditions \ref{24-6-7-1} are fulfilled.

The rest of this subsection is devoted to the proof of

\begin{proposition}\label{prop-24-6-5}
Let  the Sch\"odinger-Pauli $A$  be self-adjoint and let conditions \textup{(\ref{24-2-2})}, \textup{(\ref{24-6-1})}
$\textup{(\ref{24-6-7})}_{1-3}$ and \textup{(\ref{24-6-8})} be fulfilled with $\beta=0$.  Let $h=1$, $\mu \ge 1$.  Then the operator $A$ in $B(0,1)$ (with the Dirichlet boundary conditions) is semibounded from below uniformly with respect to $\mu$ and the estimate
\begin{equation}
\N (\tau)\le C\mu\tau^{\frac{1}{2}}+C\tau^{\frac{3}{2}}
\label{24-6-24}
\end{equation}
holds for $\tau \ge 1$.
\end{proposition}

Then the standard arguments of Chapter~\ref{book_new-sect-9} of \cite{futurebook} yield the following theorem
(details are left to the reader):

\begin{theorem}\label{thm-24-6-6}
Let the Schr\"odinger-Pauli operator $A$  be self-adjoint and let conditions \textup{(\ref{24-2-2})}, \textup{(\ref{24-6-1})}
$\textup{(\ref{24-6-7})}_{1-3}$ and \textup{(\ref{24-6-8})} be fulfilled
with $\beta=0$.  Then for $\mu ^{\alpha +1}h^{1-\alpha }\ge 1$
the following estimate holds:
\begin{equation}
R_\vartheta \le C\mu h^{-1+\vartheta }.
\label{24-6-25}
\end{equation}
\end{theorem}
\begin{proof}[Proof of Proposition \ref{prop-24-6-5}]
Without any loss of the generality one can assume that $F^1=F^2=0$ and $V_2=V_3=0$, $g^{13}=g^{23}=0$.  Then $V_1=V_1(x')$, $x'=(x_1,x_2)$.  Surely in
this case the standard variational estimates are much better than in the general case but we also need to improve them.

Let $\bH=\sL^2\bigl(B(0,1)\bigr)$ and on $\bK\subset \sC_0^2\bigl(B(0,1)\bigr)$ let the operator $A-\tau I$ be negative definite.  We should prove that
\begin{claim}\label{24-6-26}
$\dim\bK =0$ for $\tau \le -C_0$ and
$\dim\bK\le C\mu \tau^{\frac{1}{2}}+C\tau ^{\frac{3}{2}}$ for
$\tau \ge 1$.
\end{claim}
Let us consider
\begin{equation}
A-\tau I= g^{33}D_3^2+
\sum_{1\le j,k\le 2}P_jg^{jk}P_k-\mu F+V -\tau I.
\label{24-6-27}
\end{equation}
Observe that the operator
\begin{equation*}
A'=\sum_{1\le j,k\le 2}P_jg^{jk}P_k-\mu F
\end{equation*}
is semibounded from below on $\bH$ and, moreover, that
$A'+C_0I\ge \epsilon _0A''$ where
\begin{equation*}
A''=P_1^2+P_2^2-\mu F''
\end{equation*}
and the scalar intensity $F''$ is calculated for the Euclidean metrics;
then the operator $A''$ does not contain $x_3$.  On the other hand, let us
note that $V\ge -c_0|x_3|^{2m}$ and the quadratic form
$\epsilon \|D_3 ^2u\|^2 -c\|Wu\|^2$ is semibounded from below on
$\sL^2([-1,1])$ for every $\epsilon >0$, $W=|x_3|^\alpha$ since $\alpha>-1$.

Therefore for appropriate $\varepsilon _0>0$ and $C_0$ the operator
$\epsilon _0D_3^2+A'' -\tau-C_0$
should be negative definite on $\bK$.  One can replace the ball by
the cylinder $\Omega\times[-1,1]$ where
$\Omega\subset\bR^2$ is a unit circle
and separate variables $x'$ and $x_3$.  So,
our original problem is reduced to a problem for the operator $A''$
in the circle (and for the operator $\epsilon_0 D_3^2$ on $[-1,1]$) and
estimate (\ref{24-6-25}) follows from the well known estimate
\begin{equation*}
\N (\tau )\le C\mu +C\tau
\label{24-6-28}
\end{equation*}
for this operator for $\tau >0$ (and the routine estimate for the
one-dimensional operator).
\end{proof}

\section{The case $d=2$}
\label{sect-24-6-4}

Let us discuss briefly the case $d=2$.  Not going into details of the regular case, recall that in the framework of Subsection~\ref{sect-24-6-1} the
remainder estimate is $O\bigl(\mu ^{-1-\vartheta }h ^{-1+\vartheta }\bigr)$.  Let us attack the singular case \emph{assuming that conditions of
Subsection~\ref{sect-24-6-2} are fulfilled (including condition\/} (\ref{24-6-14})).  Then $r^*\ge r_*$, where $r_0,r_*,r^*$ are introduced in Subsection~\ref{sect-24-6-2}.
There are again two possibilities:
\begin{enumerate}[label=(\roman*), wide, labelindent=0pt]
\item\label{sect-24-6-4-i}
For $0\le \vartheta \le \vartheta ^*=(\beta +2)(\beta +2-2\alpha ) ^{-1}$
the contribution of the zone $\{r^*\le |x|\le r_1\}$ to the remainder
is $O\bigl(\mu ^{-1-\vartheta }h ^{-1+\vartheta }\bigr)$ with the
factor $\log \mu $ for $\vartheta =\vartheta ^*$.  However, since
$\vartheta ^*< \bar{\vartheta }_0$ we can apply the standard semiclassical asymptotics and no Scott correction term appears. The final remainder estimate is $O\bigl(\mu ^{-1-\vartheta }h ^{-1+\vartheta }\bigr)$ with the
factor $\log \mu $ for $\vartheta =\vartheta ^*$.

\item\label{sect-24-6-4-ii}
Further, for $\vartheta ^* <\vartheta \le \bar{\vartheta }_0$ the best
possible remainder estimate is $O\bigl(h^{-1-\vartheta} \mu^{(\alpha\vartheta+\beta+1-\vartheta)(\beta+1-\alpha)^{-1}}\bigr)$
which is also appears as a contribution of $B(0,r_3)$ to the remainder and there is still no Scott correction term.

\item\label{sect-24-6-4-iii}
For $\vartheta >\bar{\vartheta }_0$ it is not important which
regular remainder estimate (namely,
$O\bigl(\mu ^{-1-\vartheta }h ^{-1+\vartheta }\bigr)$ or
$O\bigl(h ^{-1+\vartheta }\bigr)$) was applied in the zone
$\{r^*\le |x|\le r_1\}$.  So, the arguments of Subsection~\ref{sect-24-6-2}
remain valid and for $\vartheta \ge \bar{\vartheta }_0$ we can
obtain the same type of remainder estimate as before with the Scott
correction term(s).
\end{enumerate}

We have not succeeded in getting a variational estimate in the case
$h=1$, $\mu \ge 1$ and thus we have no extension of the results of
Subsection~\ref{sect-24-6-3}.

\begin{subappendices}
\chapter{Appendices}

\label{sect-24-A}

\section{$1\D$ Schr\"odinger operator}
\label{sect-24-A-1}

Operators of the type we consider here studied by many authors. Related statements could be found in many books, including Chapter XIII, Part~2 of N.~Danford and J.~T.~Schwarz \cite{danford:schwartz}, M.~S. Birman and M.~Z. Solomyak \cite{birman:function} and V.~Maz'ya and I.~Verbitsky~\cite{mazya:verb}.

\begin{proposition}\label{prop-24-A-1}
Let us consider the operator
\begin{equation}
\boldsymbol{a}_\varepsilon =
D_tg_\varepsilon (t )D_t+\varepsilon^{-1} V_\varepsilon (t)
\label{24-A-1}
\end{equation}
in $\bH=\sL^2(\bR)$ with $D=D_t$,
$V_\varepsilon (x)=V( \frac{x}{\varepsilon } )$, etc., $t\in\bR$,
\begin{equation}
\epsilon _0\le g\le c, \quad |V|\le \rho ^2,\quad 0\le \rho \le c,
\quad \|\rho \|_{\sL^1}+\|\rho ^2t\|_{\sL^1}\le c.
\label{24-A-2}
\end{equation}
Then
\begin{enumerate}[label=(\roman*), wide, labelindent=0pt]
\item\label{prop-24-A-1-i}
The number of negative eigenvalues of the operator
$\boldsymbol{a}$ does not exceed $C_0$ for $|\varepsilon |\le 1$.

\item\label{prop-24-A-1-ii}
The number of negative eigenvalues of the operator
$\boldsymbol{a}$ does not exceed $\,1$ for $|\varepsilon |\le \epsilon $
with a small enough constant $\epsilon >0$.

\item\label{prop-24-A-1-iii}
Further, let us assume that
\begin{equation}
W=-\frac{1}{2}\int_{-\infty } ^{+\infty }V(t)\,dt\ge \epsilon _0.
\label{24-A-3}
\end{equation}
Then for $\varepsilon\in(0,\epsilon]$ there is exactly one
negative eigenvalue $\lambda(\varepsilon)$ and
\begin{equation}
-\epsilon _2\ge \lambda \ge -c_1.
\label{24-A-4}
\end{equation}

\item\label{prop-24-A-1-iv}
Furthermore, let us assume that \textup{(\ref{24-A-3})} holds and
\begin{gather}
|g-1|\le c\rho.
\label{24-A-5}\\
\shortintertext{Then}
|\lambda + W^2|\le C_0\varepsilon.
\label{24-A-6}
\end{gather}

\item\label{prop-24-A-1-v}
Moreover, let us assume that \textup{(\ref{24-A-3})} holds and that $g$ and
$V$ depend on the parameter $z\in \Omega $ and
\begin{gather}
|D _z^\alpha g|\le c,\quad |D _z^\alpha V|\le c\rho ^2
\qquad \forall \alpha :|\alpha |\le K.
\label{24-A-7}\\
\shortintertext{Then}
|D_z ^\alpha \lambda |\le C_0;
\label{24-A-8}
\end{gather}
moreover, under the condition
\begin{gather}
|D_z ^\alpha g|\le c\rho \qquad \forall \alpha :|\alpha |\le K
\label{24-A-9}\\
\shortintertext{we obtain that}
|D_z^\alpha (\lambda+W^2)|\le C_0\varepsilon.
\label{24-A-10}
\end{gather}
\item\label{prop-24-A-1-vi}
Finally, let $v\in \bH$, $\1v\1=1$ be an appropriate
eigenfunction of $\boldsymbol{a}$ with eigenvalue $\lambda$. Then
\begin{phantomequation}\label{24-A-11}\end{phantomequation}
\begin{equation}
\1 D_z ^\alpha v\1 \le C_0, \qquad \1 D_z ^\alpha D_tv\1 \le C_0,
\qquad \1D_z ^\alpha v\1_\infty \le C_0,
\tag*{$\textup{(\ref*{24-A-11})}_{1-3}$}\label{24-A-11-*}
\end{equation}
where $\1.\1_p$ means the $\sL^p$-norm and we skip $p=2$ in this notation.
\end{enumerate}
\end{proposition}

\begin{proof}
Statement \ref{prop-24-A-1-i} follows from the fact that the operator
$\rho ^s(D_3^2 +1)^{-s}$ is compact in $\bH$ for any $s>0$.

In order to prove Statement \ref{prop-24-A-1-ii} let us consider the
quadratic form $Q(u)=\blangle \boldsymbol{a}u,u\brangle $ on the
subspace $\bH_1=\{u\in\bH,\int_{-\varepsilon}^\varepsilon u\, dt=0\}$ of
codimension $1$\,\footnote{\label{foot-24-43} One can consider the subspace
$\{u\in\fD(\boldsymbol{a}),u(0)=0\}$ as well.}
Obviously $|u(t)|\le 3(|t|+\varepsilon)^{\frac{1}{2}}\1D_tu\1$
for $u\in \bH_1$ and therefore
\begin{equation*}
|\blangle V_\varepsilon u,u\brangle|\le C_0\varepsilon^2 \1D_t^2u\1^2
\end{equation*}
in virtue of (\ref{24-A-2}). Then (\ref{24-A-2}) yields that the quadratic
form $Q(u)$ is positive definite on $\bH_1$ for
$|\varepsilon|\le \sigma_2$ and therefore $\boldsymbol{a}$ has no more
than one negative eigenvalue $\lambda$. Moreover, for arbitrary
$u\in\bH$ the inequality
$|u(t)|\le \sigma\1D_tu\1+C_\sigma\1u\1$ with arbitrarily
small $\epsilon>0$ yields that
\begin{equation*}
|\blangle V_\varepsilon u,u\brangle |\le
c_0\sigma \varepsilon \1D_tu\1^2+
C_\sigma \varepsilon \1u\1^2
\end{equation*}
and hence $Q(u)$ is uniformly semibounded from below and therefore
\begin{equation}
\lambda \ge -C_0.
\label{24-A-12}
\end{equation}

Let $v$ be the corresponding eigenfunction with $\1v\1=1$ (if there
exists a negative eigenvalue). Then obviously $|v(t)|\le C_0$ and
then $|D_tv(t)|\le C_0$ and hence $|v(t)-v(0)|\le C_0|t|$. Then
(\ref{24-A-2}), (\ref{24-A-5}) yield that
\begin{equation*}
|Q(v)- \bar{Q}(v)|\le C_1\varepsilon
\end{equation*}
for the quadratic form $\bar{Q}(u)=\1D_tu\1^2-W|u(0)|^2$. Therefore
$\lambda \ge \bar{\lambda }-C_1\varepsilon $ where
$\bar{\lambda }$ is the lower bound of $\bar{Q}(u)\1u\1^{-2}$ at
$\bH$.

One can apply the same arguments to the eigenvalues and eigenfunctions of $\bar{Q}$; as a result we obtain that
$\bar{\lambda }\ge \lambda -C_1\varepsilon $. On the other hand, one can see easily that $\bar{\lambda }=-{\frac{1}{4}}W^2$ if $W>0$ (otherwise $\bar{Q}$ is non-negative definite) and $\bar{v}(t)= \frac{W}{2} \exp(- \frac{1}{2} W|x|)$ and hence we obtain that for $W>0$ there is a negative eigenvalue and (\ref{24-A-6}) holds.

Moreover, if (\ref{24-A-5}) is violated then one can treat the quadratic
form $C_0\1D_t^2\1+\varepsilon ^{-1}(V_\varepsilon u,u)$
instead of the original one and since (\ref{24-A-4}) holds for this form it
remains true for the original one.

\emph{So all the statements excluding those associated with derivatives on $z$
are proven\/}\footnote{\label{foot-24-44} One can easily prove that for $W<0$ and small enough $\varepsilon $ the operator $\boldsymbol{a}$ is non-negative
definite. We think that it would be nice to treat the case $W=0$.
However we are not an expert here.}.

The proof of (\ref{24-A-8}), (\ref{24-A-11}) is standard, due to K.~O.~Friedrichs~\cite{friedrichs:hilbert}. Let these estimates be proven for $|\alpha|\le n$; then applying the operator $\partial_z^\alpha$ with $|\alpha|=n$ to the equation
\begin{gather}
(\boldsymbol{a}-\lambda )v=0
\label{24-A-13}\\
\intertext{we obtain}
\sum_{\beta \le \alpha } \frac{\alpha!}{\beta!(\alpha-\beta)!}
(\boldsymbol{a}-\lambda)^{(\alpha-\beta)} v^{(\beta)}=0
\label{24-A-14}
\end{gather}
with $u^{(\alpha )}=\partial _z^\alpha u$.

Let us multiply this equation by $v$. Then terms with $v^{(\alpha )}$ disappear and we obtain terms with $|\beta |<n$
\begin{equation}
\blangle g_\varepsilon ^{(\alpha -\beta )}D_tv^{(\beta )}, D_t v\brangle ,\quad
\varepsilon^{-1}
\blangle V_\varepsilon ^{(\alpha -\beta }v^{(\beta )}, v\brangle ,
\quad \lambda ^{(\alpha -\beta )}\blangle v^{(\beta )},v\brangle.
\label{24-A-15}
\end{equation}
Terms of the first and second types are bounded in virtue of
$\textup{(\ref{24-A-11})}_2$, $\textup{(\ref{24-A-11})}_3$ respectively for $|\beta |<n$. Terms of the third type are bounded for $\beta \ne 0$ by
(\ref{24-A-8}), $\textup{(\ref{24-A-11})}_1$. Therefore the remaining term
\begin{equation*}
-\lambda ^{(\alpha )}\1v\1^2
\end{equation*}
should also be bounded and (\ref{24-A-8}) holds for $|\alpha |=n$.
Let us consider equation (\ref{24-A-14}); we now multiply it by
$w=v^{(\alpha )}$. We obtain terms with $|\beta |\le n$
\begin{equation*}
\blangle g_\varepsilon ^{(\alpha -\beta )}v^{(\beta )}, w\brangle ,
\quad \varepsilon^{-1}
\blangle V_\varepsilon ^{(\alpha -\beta }v^{(\beta )},w\brangle ,
\quad \lambda^{(\alpha-\beta)}\blangle v^{(\beta)},w\brangle.
\end{equation*}
For $|\beta|<n$ terms of the first and second type do not
exceed $C\1D_tw\1$ and $C\1w\1_\infty$ due to
$\textup{(\ref{24-A-11})}_2$, $\textup{(\ref{24-A-11})}_3$. Finally, terms of the third type for $|\beta|<n$ do not exceed $C\|w\|$ due to
(\ref{24-A-8}) and $\textup{(\ref{24-A-11})}_1$. Thus
\begin{equation*}
|\blangle (\boldsymbol{a}-\lambda )w,w\brangle |\le
C\1D_tw\1+C\1w\1
\end{equation*}
because $\1w\1_\infty\le C\1D_tw\1+C\1w\1$. Taking into account that
$\lambda\le -\epsilon_0$ we obtain from this inequality that
\begin{equation}
\1D_tw\1^2+\1w\1^2\le
C\varepsilon ^{-1}|\blangle V_\varepsilon w,w\brangle |+C.
\label{24-A-16}
\end{equation}
Let us assume that $v(0)=1$. Surely, we should reject the condition
$\1v\1=1$ but our above arguments yield that $\1v\1\asymp |v(0)|$. Then $w(0)=0$ and $|w(t)|\le |t|^{\frac{1}{2}}\1D_tw\1$ and therefore
$|\blangle V_\varepsilon w,w\brangle |\le C\varepsilon ^2\1D_tw\1$
and therefore (\ref{24-A-16}) yields
$\textup{(\ref{24-A-11})}_{1,2}$ for $|\alpha |=n$; $\textup{(\ref{24-A-11})}_3$ follows from these estimates.

In order to prove (\ref{24-A-10}) let us note that equation (\ref{24-A-14}) and
\ref{24-A-11-*} yield that
\begin{equation}
\1D_z ^\alpha D_tv\1_\infty \le C_0
\tag*{$\textup{(\ref{24-A-11})}_4$}\label{24-A-11-4}
\end{equation}
and therefore under condition (\ref{24-A-9}) terms of the first type in
(\ref{24-A-15}) do not exceed $C\varepsilon $. Moreover,\ref{24-A-11-4}
yields that
\begin{equation*}
|v(t)-v(0)|\le C_0|t|,\qquad
|D_z ^\alpha v|\le C_0|t|\quad \forall \alpha \ne 0
\end{equation*}
(because of condition $v(0)=1$) and therefore terms of the second type in (\ref{24-A-15}) do not exceed $C\varepsilon$ for $\beta\ne 0$. Moreover, the error does not exceed $C\varepsilon$ if we replace $v(t)$ by $v(0)$ in
this term with $\beta =0$; we then obtain
$W^{(\alpha )}|v(0)|^2$ and under additional the restriction $W=\const$
this term vanishes. Then induction on $n$ yields that
$|\lambda ^{(\alpha )}|\le C_0\varepsilon $ under this restriction.
So under this restriction (\ref{24-A-10}) holds. However one can reduce
the general case to the case $W=1$ by introducing $t'=tW ^{-1}$ and
multiplying $\boldsymbol{a}$ by $W^2$.
\end{proof}

\begin{remark}\label{rem-24-A-2}
Applying the above results one can find $v$ in the form
\begin{equation}
v=\exp \bigl(\int_0^t\phi _\varepsilon (t')dt'\bigr)\cdot
\bigl(1+\varepsilon ^2\psi _\varepsilon +\ldots \bigr)
\label{24-A-17}
\end{equation}
where the number of terms depends on $m$ and
$\lambda = -W^2 + \mu\varepsilon + \cdots$ with $\partial_t\phi=V$ and one can obtain $\mu\ne 0$ in the generic case; so estimate (\ref{24-A-6}) is the best
possible estimates without this correction term. Therefore (\ref{24-4-33}) remains true with $\lambda(x')$ replaced by $-W(x')^2$ provided $m\le -2$ and $\rho(x)=\langle x\rangle^m$, $\gamma(x)=\gamma_1(x)= \langle x\rangle$.

For $m>-2$ this is correct with the remainder estimate $O\bigl(\eta^{(m+2)(2m+1)^{-1}}\bigr)$ coinciding
with the principal part for $m=-1$ (in the framework of Remark~\ref{rem-24-4-19}).
\end{remark}

\begin{proposition}\label{prop-24-A-3}
\begin{enumerate}[label=(\roman*), wide, labelindent=0pt]
\item\label{prop-24-A-3-i}
Under condition \textup{(\ref{24-4-66})}
the operator $\cL $ has a finite number of negative eigenvalues.

\item\label{prop-24-A-3-ii}
Moreover, if this condition is fulfilled for all $x$ then there is
at most one negative eigenvalue.

\item\label{prop-24-A-3-iii}
On the other hand, if
\begin{equation}
W\le - (\frac{1}{4}+\epsilon) |x|^{-2}\qquad \forall x\colon x\ge C
\label{24-A-18}
\end{equation}
then there is an infinite number of negative eigenvalues.
\end{enumerate}
\end{proposition}

\begin{proof}
To prove Statements \ref{prop-24-A-3-i} and \ref{prop-24-A-3-ii} one needs to prove the estimate
\begin{equation}
\1u'\1^2\ge \frac{1}{4}\1|x|^{-1}u\1^2\qquad \forall u:u(0)=0
\label{24-A-19}
\end{equation}
where $\1u\1$ is the $L^2(\bR^+)$-norm. However, the left side is
equal to
\begin{equation*}
\1x ^{\frac{1}{2}}\bigl(ux ^{-{\frac{1}{2}}}\bigr)'+
{\frac{1}{2}}x ^{-1}u\1 ^2=
\1x ^{\frac{1}{2}}\bigl(ux ^{-{\frac{1}{2}}}\bigr)'\1 ^2+
{\frac{1}{4}}\1|x| ^{-1}u\1 ^2
\end{equation*}
provided $u=o(x ^{\frac{1}{2}})$ as $x\to 0$.

To prove Statement~\ref{prop-24-A-3-iii} it is sufficient to prove that the inequality $\1u'\1^2\le (\frac{1}{4}+\epsilon)\1|x| ^{-1}u\1 ^2$ is fulfilled on
some subspace of $\sL^2([1,\infty))$ of infinite dimension. It is sufficient to prove that for any $n$ this inequality is fulfilled on some function supported in $[L^n,L^{n+1}]$ with sufficiently large $L$.

Further, due to homogeneity it is sufficient to consider only $n=0$.
Substituting $u=x^{\frac{1}{2}}v$, $x=e ^t$ we obtain that it is
sufficient to fulfill the inequality $\1v'\1 ^2\le \epsilon \1v\1 ^2$
with some $v$ such that $v(0)=v(\log L)=0$. But this is obvious
provided $L$ is large enough.
\end{proof}

\section{$1\D$ Schr\"odinger operator. II}
\label{sect-24-A-2}

We consider operator
\begin{gather}
\boldsymbol{b}_\varepsilon = D^2 +\varepsilon V(x)
\label{24-A-20}\\
\shortintertext{with}
|V|\le \rho^2=\langle x\rangle^{-2q}, \qquad V\le - \epsilon_0\rho^2 \qquad \text{for\ \ }|x|\ge c,
\label{24-A-21}
\end{gather}
$0<q\le 1$, and $\varepsilon>0$.

We are interested in $\boldsymbol{n}_\varepsilon (\eta)$, the number of eigenvalues of $\boldsymbol{b}_\varepsilon$  which are less than $-\eta$. Consider first  the corresponding Weyl's expression
\begin{equation}
\n^\W_\varepsilon(\eta)\coloneqq
(2\pi )^{-1}\int (\varepsilon V(x)-\eta)_+^{\frac{1}{2}}\,dx.
\label{24-A-22}
\end{equation}

\begin{proposition}\label{prop-24-A-4}
\begin{enumerate}[label=(\roman*), wide, labelindent=0pt]
\item\label{prop-24-A-4-i}
If $\n^W_\varepsilon (\eta)\ge C_0$ then
$\boldsymbol{n}_\varepsilon(\eta)\asymp \n^\W(\varepsilon,\eta)$.
\item\label{prop-24-A-4-ii}
If $\n^W(\varepsilon,\eta)\le C_0$ then $\boldsymbol{n}\varepsilon(\eta)\le C_1$.
\end{enumerate}
\end{proposition}

\begin{remark}\label{rem-24-A-5}
Obviously
$\n^W_\varepsilon(\eta)\asymp \varepsilon^{(2q-1)/2q}\eta^{-(1-q)/2q}$ and
$\n^W_\varepsilon(\eta)\le C_0$ if and only if
$\eta \ge c_0\varepsilon^{(2q-1)/(1-q)}$.
\end{remark}

\begin{proof}[Proof of Proposition~\ref{prop-24-A-3}]
One can easily prove Statement~\ref{prop-24-A-4-i} using our semiclassical theory.

On the other hand, one can easily prove Statement~\ref{prop-24-A-4-ii} using variational methods, and covering $\bR$ by a finite number of intervals $[L_k,L_{k+1}]$ and $[-L_{k+1},-L_k,]$ with $k=1,\ldots, n-1$, $[L_n,\infty]$ and $[-\infty,-L_n]$ and $[-L_0,L_0]$ such that
$L_{k+1} =\epsilon_0 L_k^q \varepsilon^{-1/2}$, $L_0=1$,
$L_n\ge c_0\eta^{-1/2q}\varepsilon^{1/2q} $.

We leave the easy details to the reader.
\end{proof}

Now we need to figure out when $\boldsymbol{n}_\varepsilon(\eta)\ge 1$. To do so we need to evaluate the lowest eigenvalue $\lambda (\varepsilon)<0$ of operator (\ref{24-A-20}).

\begin{proposition}\label{prop-24-A-6}
Let $V\in \sL^1(\bR)$ and $W >0$. Then
\begin{equation}
\lambda (\varepsilon)=-\varepsilon^2 (W^2 +o(1))\qquad \
text{as\ \ }
\varepsilon \to 0
\label{24-A-23}
\end{equation}
with $W$ defined by \textup{(\ref{24-A-3})}.
\end{proposition}

\begin{remark}\label{rem-24-A-7}
\begin{enumerate}[label=(\roman*), wide, labelindent=0pt]
\item\label{rem-24-A-7-i}
Since after scaling $x\mapsto x/\varepsilon$ and multiplication by $\varepsilon^{-1}$ with $\varepsilon$ operator (\ref{24-A-20}) becomes (\ref{24-A-1}), this is consistent with (\ref{24-A-6}).

\item\label{rem-24-A-7-ii}
For $V= -\langle x\rangle ^{-2q}$ we have $W<\infty$ if and only if $q>\frac{1}{2}$.
\end{enumerate}
\end{remark}

\begin{proof}[Proof of Proposition~\ref{prop-24-A-6}]
\begin{enumerate}[label=(\alph*), wide, labelindent=0pt]
\item\label{pf-24-A-6-a}
We can apply Proposition~\ref{prop-24-A-1} for $V\ge 0$, $V=0$ as
$|t|\ge \varepsilon^{-1}$. Therefore
$\boldsymbol{b}'_\varepsilon = D^2 +V' (x)\ge -C\varepsilon^2$ where
$V'=V(x)$ as $|x|\le t$, $V(x)=0$ as $|x|\ge t$. Indeed it is true for
$V$ replaced by $-C\rho(x)^2$.

Consider $t_n=2^n$ and consider $V_0(x)= V(x)$ as $|x|\le t_0$,
$V_n(x)=V(x)$ as $t_{n-1}\le |x|\le t_n$ and vanishing on all other segments. Consider
\begin{gather}
\sigma_n>0, \qquad \sum_n \sigma_n \le 1
\label{24-A-24}\\
\shortintertext{Then}
\boldsymbol{b}_\varepsilon \ge \sum_n \boldsymbol{b}_{n,\varepsilon},
\qquad
\boldsymbol{b}_{n,\varepsilon}=\sigma_n D^2  + \varepsilon V_n(x).\notag
\end{gather}
Scaling $x\mapsto x/t_n$ we have
$\boldsymbol{b}_{n,\varepsilon}\mapsto  \sigma_n t_n^{-2}
\bigl[D^2+ \sigma_n^{-1}\varepsilon t_n^2 \rho(t_n)^2 U_n(x)\bigr]$, with $U_n(x)=\rho(t_n)^{-2} V_n(x/t_n)$ and if
$\sigma_n^{-1}\varepsilon t_n^2 \rho(t_n)^2\le 1$ we can apply the above estimate to the operator  in the brackets. On the other hand, it is greater than
$-C\varepsilon \sigma_n^{-1} \rho(t_n)^2$ and we can apply this estimate even without this condition; so we arrive to
$\boldsymbol{b}_{n,\varepsilon} \ge \sigma_n^{-1}\varepsilon^2 t_n^2 \rho(t_n)^4$ and therefore $]\boldsymbol{b}_\varepsilon \ge -C\varepsilon^2$ provided
\begin{equation}
\sum_n \sigma_n^{-1}\varepsilon^2 t_n^2 \rho(t_n)^4\le C_0.
\label{24-A-25}
\end{equation}
Picking up $\sigma_n =\epsilon_0 t_n \rho(t_n)^2$ we satisfy both (\ref{24-A-24}) and (\ref{24-A-25}).

\item\label{pf-24-A-6-b}
Consider now $t$ such that $t\rho(t)^2 \le \delta^2$. Then
$|W_1 - W|\le C\delta^2$ with $W_1=-\frac{1}{2}\int_{-t}^t V(x)\,dx$. Therefore
$\boldsymbol{b}_\varepsilon =
\boldsymbol{b}_{1,\varepsilon}+  \boldsymbol{b}_{2,\varepsilon}$ with
$\boldsymbol{b}_{1,\varepsilon}=(1-\delta)D^2 +\varepsilon V_1(x)$,
$\boldsymbol{b}_{2,\varepsilon}=\delta D^2 +\varepsilon V_2(x)$. Applying Proposition~\ref{prop-24-A-1} to  $\boldsymbol{b}_{1,\varepsilon}$ and the results of Part~\ref{pf-24-A-6-a} to $\boldsymbol{b}_{2,\varepsilon}$ we conclude that
$\boldsymbol{b}_\varepsilon \ge -\varepsilon^2 (W^2 +C\delta)\implies \lambda(\varepsilon)\ge -\varepsilon^2 (W^2 +C\delta)$.

Similarly, $\boldsymbol{b}_{3,\varepsilon}= (1+\delta)D^2 - V_1(x)\ge \boldsymbol{b}_{\varepsilon}+\boldsymbol{b}_{4,\varepsilon}$ and applying Proposition~\ref{prop-24-A-1} to  $\boldsymbol{b}_{3,\varepsilon}$ and the results of Part~\ref{pf-24-A-6-a} to $\boldsymbol{b}_{4,\varepsilon}$ we conclude that
$\lambda(\varepsilon)\le -\varepsilon^2 (W^2 -C\delta)$.

Since we can take $\delta>0$ arbitrarily small we arrive to (\ref{24-A-23}).
\end{enumerate}
\end{proof}

Let $0<q \le\frac{1}{2}$. Then the integral defining $W$ in (\ref{24-A-23}), diverges (logarithmically, as $q=\frac{1}{2}$).

\begin{proposition}\label{prop-24-A-8}
Let $0<q<\frac{1}{2}$. Then
\begin{enumerate}[label=(\roman*), wide, labelindent=0pt]
\item\label{prop-24-A-8-i}
$\lambda \ge -\varepsilon ^{1/(1-q)}$.
\item\label{prop-24-A-8-ii}
Assume that $V(x)\sim V^0(x)$ as $|x|\to \infty$ where $V^0(x)=V_\pm |x|^{-2q}$ as $\pm x>0$. Let either $V_+<0$ or $V_-<0$ and let $\mu<0$ be the lowest eigenvalue of the operator $\boldsymbol{a}^{0}=D^2+ V^0(x)$. Then
\begin{equation}
\lambda=  \varepsilon^{1/(1-q)}(\mu +o(1)) \qquad \text{as\ \ }
\label{24-A-26}
\varepsilon\to 0.
\end{equation}
\end{enumerate}
\end{proposition}

\begin{proof}
Observe first that for $0<q<\frac{1}{2}$ operator $\boldsymbol{a}^0$ is properly defined and semibounded from below and in the framework of Statement~\ref{prop-24-A-8-ii} it has an infinite number of negative eigenvalues.

\begin{enumerate}[label=(\roman*), wide, labelindent=0pt]
\item\label{pf-24-A-8-i}
Replacing $V$ by $-C|x|^{-2q}$ and using scaling
$x\mapsto c\varepsilon ^{-1/2(1-q)}x$  we arrive to operator
$\varepsilon^{1/(1-q)}\boldsymbol{a}^0\ge -C\varepsilon^{1/(1-q)}$. Thus we arrive to Statement~\ref{prop-24-A-8-i}.

\item\label{pf-24-A-8-ii}
Observe that in the framework of Statement~\ref{prop-24-A-8-ii}
\begin{gather*}
b_\varepsilon \ge b_{1,\varepsilon}+b_{2,\varepsilon}\\
\shortintertext{with}
b_{1,\varepsilon}= (1-\delta) D^2 - \varepsilon (V^0(x) + \delta |x|^{-2q}),
\qquad
b_{1,\varepsilon}=\sigma D^2 - \varepsilon U(x)
\end{gather*}
with arbitrarily small $\sigma>0$ and $U$ supported in $[-t,t]$ with $t=t(\delta)$. Then
$b_{1,\varepsilon}\ge  (\mu_1 -C\delta )\varepsilon^{1/(1-q)}$,
$b_{2,\varepsilon}\ge C(t,\delta)\varepsilon^2$ and therefore
$\lambda(\varepsilon)\ge (\mu_1- 2C\delta)\varepsilon^2$.

Similarly, one can prove that
$\lambda(\varepsilon)\le (\mu_1+ 2C\delta)\varepsilon^2$.
Since we can take $\delta>0$ arbitrarily small we arrive to (\ref{24-A-26}).
\end{enumerate}
\end{proof}

\begin{Problem}\label{Problem-24-A-9}
\begin{enumerate}[label=(\roman*), wide, labelindent=0pt]
\item\label{Problem-24-A-9-i}
Using arguments of the proof of Part~\ref{pf-24-A-6-a} of Proposition~\ref{prop-24-A-6} prove that if $\int_\bR \rho^2(x)\,dx=\infty$ then
$\lambda (\varepsilon)\ge -C\eta$ where $\eta=\eta(\varepsilon)$ is defined from
\begin{equation}
\eta^{\frac{1}{2}}= \varepsilon  \int _{x\colon \varepsilon \rho(x)^2 \ge \eta} \rho^2(x)\,dx
\label{24-A-27}
\end{equation}
which is consistent with $\varepsilon^{1/(1-q)}$ in the framework of Proposition~\ref{prop-24-A-8} but also works for
$\rho(x)= |x|^{-q}|\log |x||^{p}$ with \underline{either} $0<q<\frac{1}{2}$
\underline{or} $q=\frac{1}{2}$, $p\ge -\frac{1}{2}$.

\item\label{Problem-24-A-9-ii}
Derive asymptotics of $\lambda(\varepsilon)$ in the framework of Statement~\ref{Problem-24-A-9-i}; the most interesting and difficult case seems to be $q=\frac{1}{2}$.

\item\label{Problem-24-A-9-iii}
Provide a better  error  estimate in (\ref{24-A-23}) and (\ref{24-A-26}).
\end{enumerate}
\end{Problem}

\section{Examples of vector potential}
\label{sect-24-A-3}

In this Appendix we prove that the results of this Chapter are meaningful.  The only questionable part of their conditions is the existence of vector potentials with the given properties of the scalar intensities $F$ and the doubts are only in the three-dimensional case.  The construction of a conformal asymptotic Euclidean metric tensor\footnote{\label{foot-24-45} I.e., a tensor
$g^{jk}=\updelta ^{jk}(1+\varphi )$ with $|1+\varphi |\ge \epsilon >0$
such that $D ^\alpha \varphi =o(\gamma ^{-|\alpha |})$ as
$|x|\to \infty $ or $|x|\to0$ for all $\alpha $.}%
\index{asymptotic Euclidean metric}%
\index{Euclidean metric!asymptotic}%
\index{metric!asymptotic Euclidean}
$g^{jk}$ and the construction of a scalar potential $V$ are obvious in all cases
and are left to the reader.  Analysis of the two-dimensional case is also obvious.  In what follows $g^{jk}=\updelta ^{jk}$.

In what follows $\Lambda_{(2n)}$ be a block-diagonal $2n \times 2n$-matrix with $n$ diagonal $2\times2$-blocks $\uplambda=\left(\begin{smallmatrix} 0 &1\\-1 &0\end{smallmatrix}\right)$ (and non-diagonal blocks $0$). Further, let $\Lambda_{(2n+1)}$ be a block-diagonal $(2n +1)\times (2n+1)$-matrix with $n$ diagonal $2\times2$-blocks $\uplambda$ and $1\times 1$-block and all non-diagonal blocks equal $0$.

\begin{lemma}\label{lemma-24-A-10}
\begin{enumerate}[label=(\roman*), wide, labelindent=0pt]
\item\label{lemma-24-A-10-i}
Let $d=2n$ and $V_j = (\Lambda x)_j \sigma (|x|)$. Then eigenvalues of $(F_{jk})$ are $\pm i f_1(x),\ldots, \pm if_n(x)$ with
\begin{phantomequation}\label{24-A-28}\end{phantomequation}
\begin{equation}
f_1(x)= 2 \sigma (|x|) + |x| \sigma'(|x|),\qquad
f_2 (x)=\ldots =f_n(x)= 2\sigma (|x|).
\tag*{$\textup{(\ref*{24-A-28})}_{1,2}$}\label{24-A-28-*}
\end{equation}
\item\label{lemma-24-A-10-ii}
Let $d=2n+1$ and $V_j = (\Lambda x)_j \sigma (|x|)$. Then eigenvalues of $(F_{jk})$ are $\pm i f_1(x),\ldots, \pm if_n(x),0$ with $f_2,\ldots,f_n$ defined by \ref{24-A-28-*} and
\begin{equation}
f_1(x)^2= (2\sigma(|x|) +|x'| ^2 |x|^{-1}\sigma '(|x|))^2 +
 x_d^2|x'|^2 |x|^{-4}\sigma(|x|)^{\prime\,2},
\label{24-A-29}
\end{equation}
$x'=(x_1,\ldots, x_{d-1})$.
\end{enumerate}
\end{lemma}

\begin{proof}
\begin{enumerate}[label=(\roman*), wide, labelindent=0pt]
\item\label{pf-24-A-10-i}
Without any loss  of the generality one can assume that $x_j=0$ for $j=3,\ldots,d$ since we can always reach it by a rotation, commuting with $\Lambda$. Then
$F_{jk}= 2\sigma\Lambda_{jk} $ if either $j\ge 3$ or $k\ge 3$ \underline{and} $F_{jk}=(2\sigma + |x|\sigma')\Lambda_{jk}m$ for $j,k\le 2$ which implies Statement~\ref{example-24-A-11-i}.

\item\label{pf-24-A-10-ii}
Without any loss  of the generality one can assume that $x_j=0$ for $j=3,\ldots,d-1$.  Then again $F_{jk}(x)=2\sigma \Lambda_{jk}$ if either $j=3,\ldots, d-1$ or $k=3,\ldots, d-1$ and therefore  again $f_2=\ldots=f_n$ are defined by $\textup{(\ref*{24-A-28})}_{2}$.

Meanwhile $F_{12}= (2\sigma + |x'|^2|x|^{-1}\sigma')$ ,
$F_{13}=x_2x_d|x|^{-1} \sigma'$, $F_{12}=-x_1x_d|x|^{-1} \sigma'$, and
$f_1^2= F_{12}^2+F_{13}^2+F_{23}^2$ which implies (\ref{24-A-29}).
\end{enumerate}
\end{proof}

We start from power singularities:

\begin{example}\label{example-24-A-11}
\begin{enumerate}[label=(\roman*), wide, labelindent=0pt]
\item\label{example-24-A-11-i}
In the framework of Lemma~\ref{lemma-24-A-10}\ref{lemma-24-A-10-i} with
$\sigma =|x|^m$
$f_1= (2+m)|x|^m$. In particular, $|f_1|\asymp |x|^m$ if  $m\ne -2$.

\item\label{example-24-A-11-ii}
In the framework of Lemma~\ref{lemma-24-A-10}\ref{lemma-24-A-10-ii} with
$\sigma =|x|^m$
\begin{equation}
f_1^2= \bigl((2+m |x'|^2|x|^{-2} )^2 + m^2 |x'|^2 x_d^2|x|^{-4}\bigr)|x|^{2m}.
\label{24-A-30}
\end{equation}
In particular, $|f_1|\asymp |x|^m$ if  $m\ne -2$.
\end{enumerate}
\end{example}

\begin{example}\label{example-24-A-12}
In the framework of Lemma~\ref{lemma-24-A-10}\ref{lemma-24-A-10-ii} with
$\sigma =|x|^m$ let us define $V_d = a |x|^{m+1}$. Again, without any loss of the generality we can assume that $x_j=0$, $j=3,\ldots, d-1$. In this case the only $F_{jk}$ to change are $F_{1d},F_{2d}$ (and $F_{d1},F_{d2}$) and therefore $f_2,\ldots, f_n$ are still defined by $\textup{(\ref*{24-A-28})}_{2}$. One can prove easily that
\begin{multline}
f_1^2=\\
\bigl((2+m|x'|^2|x|^{-2})^2 + m^2 x_d^2|x'|^2|x|^{-4}+(m+1)^2a^2 |x'|^2  |x|^{-2}\bigr)|x|^{2m}.
\label{24-A-31}
\end{multline}
In particular, $f_1\asymp |x|^m$ if $a\ne 0$.
\end{example}

\begin{example}\label{example-24-A-13}
\begin{enumerate}[label=(\roman*), wide, labelindent=0pt]
\item\label{example-24-A-13-i}
In the framework of Lemma~\ref{lemma-24-A-10}\ref{lemma-24-A-10-i} with
$\sigma =\langle x\rangle ^m$
$f_1= (2+m|x|^\langle x \rangle^{-2})\langle x \rangle^m$. In particular, $|f_1|\asymp |x|^m$ if  $m> -2$.

\item\label{example-24-A-13-ii}
In the framework of Lemma~\ref{lemma-24-A-10}\ref{lemma-24-A-10-ii} with
$\sigma =\langle x \rangle^m$
\begin{equation}
f_1^2= \bigl((2+m |x'|^2\langle x \rangle^{-2} )^2 +
m^2 |x'|^2 x_d^2\langle x \rangle^{-4}\bigr)\langle x \rangle^{2m}.
\label{24-A-32}
\end{equation}
In particular, $f_1\asymp \langle x \rangle^m$ if  $m> -2$.

\item\label{example-24-A-13-iii}
In the framework of Lemma~\ref{lemma-24-A-10}\ref{lemma-24-A-10-ii} with
$\sigma =\langle x \rangle^m$ and $V_d= a \langle x \rangle^{m+1}$
\begin{multline}
f_1^2=\\
\bigl((2+m|x'|^2\langle x \rangle^{-2})^2 +
m^2 x_d^2|x'|^2\langle x \rangle^{-4}+
(m+1)^2a^2 |x'|^2  \langle x \rangle^{-2}\bigr)|x|^{2m}.
\label{24-A-33}
\end{multline}
In particular, $f_1\asymp \langle x \rangle^m$ if $a\ne 0$.
\end{enumerate}
\end{example}

Consider now power-log singularities.

\begin{example}\label{example-24-A-14}
\begin{enumerate}[label=(\roman*), wide, labelindent=0pt]
\item\label{example-24-A-14-i}
In the framework of Lemma~\ref{lemma-24-A-10}\ref{lemma-24-A-10-i} with
$\sigma =|x| ^m\ell (x)^\beta$, $\ell(x)=|\log |x||+C_0$ (with sufficiently large $C_0$) $f_1 \asymp \sigma $ if  $m\ne  -2$ and
$f_1=\beta |x|^{-2}\ell^{\beta -1}$ if $m=-2$.

\item\label{example-24-A-14-ii}
In the framework of Lemma~\ref{lemma-24-A-10}\ref{lemma-24-A-10-ii} with
$\sigma =|x| ^m\ell (x)^\beta$, $\ell(x)=|\log |x||+C_0$ (with sufficiently large $C_0$) $f_1 \asymp \sigma $ if  $m\ne  -2$.

\item\label{example-24-A-14-iii}
In the framework of Lemma~\ref{lemma-24-A-10}\ref{lemma-24-A-10-ii} with
$\sigma =\langle x \rangle^m$ and $V_d= a | x |^{m+1}\ell (x)^\beta $
 $f_1\asymp \sigma$ if $a\ne 0$.
\end{enumerate}
\end{example}

\begin{example}\label{example-24-A-15}
\begin{enumerate}[label=(\roman*), wide, labelindent=0pt]
\item\label{example-24-A-15-i}
In the framework of Lemma~\ref{lemma-24-A-10}\ref{lemma-24-A-10-i} with
$\sigma =\langle x \rangle^m \ell (x)^\beta$,
$\ell(x)=\log \langle x \rangle +C_0$ (with sufficiently large $C_0$)
$f_1 \asymp \sigma $ if  $m>  -2$.

\item\label{example-24-A-15-ii}
In the framework of Lemma~\ref{lemma-24-A-10}\ref{lemma-24-A-10-ii} with
$\sigma =\langle x \rangle ^m\ell (x)^\beta$,
$\ell(x)=\log \langle x \rangle+C_0$ (with sufficiently large $C_0$) $f_1 \asymp \sigma $ if  $m> -2$.

\item\label{example-24-A-15-iii}
In the framework of Lemma~\ref{lemma-24-A-10}\ref{lemma-24-A-10-ii} with
$\sigma =\langle x \rangle^m \ell (x)^\beta$ and
$V_d= a\langle x \rangle ^{m+1}\ell (x)^\beta $
 $f_1\asymp \sigma$ if $a\ne 0$.
\end{enumerate}
\end{example}

Consider now exponential potentials.

\begin{example}\label{example-24-A-16}
Let $d=2$. Then
\begin{enumerate}[label=(\roman*), wide, labelindent=0pt]
\item\label{example-24-A-16-i}
$V_1=-x_2|x|^m \exp(|x|^\beta)$, $V_2=x_1|x|^m \exp(|x|^\beta)$ with $\beta>0$ provide $f\asymp |x|^{m+\beta-1} \exp(|x|^\beta)$ as $|x|\ge c$.

\item\label{example-24-A-16-ii}
The same example with $\beta<0$ provide $f\asymp |x|^{m+\beta-1} \exp(|x|^\beta)$ as $|x|\le \epsilon$.
\end{enumerate}
\end{example}

For $d=3$ we need to be more crafty.

\begin{example}\label{example-24-A-17}
Let $d=3$.
\begin{enumerate}[label=(\roman*), wide, labelindent=0pt]
\item\label{example-24-A-17-i}
Consider
\begin{phantomequation}\label{24-A-34}\end{phantomequation}
\begin{align}
&V_1=\exp (\nu (x))\cos \bigl(\psi (x)\bigr)|x|^m,
\tag*{$\textup{(\ref*{24-A-34})}_{1}$}\label{24-A-34-1}\\
&V_2=\exp (\nu (x))\sin \bigl(\psi (x)\bigr)|x|^m,
&&V_3=0
\tag*{$\textup{(\ref*{24-A-34})}_{2,3}$}\label{24-A-34-2}
\end{align}
with $\nu(x)=|x|^\beta$, $\beta>0$. Then
\begin{equation}
|\nabla ^\alpha V_j|\le
c_\alpha \exp (a|x| ^\beta )|x|^{(\beta -1)|\alpha |+m}\qquad \forall \alpha
\label{24-A-35}
\end{equation}
as $|x|\gtrsim 1$ provided $|\nabla \psi |\lesssim |x|^{\beta -1}$. Moreover, one can see easily that
\begin{equation}
F\ge (\epsilon_0 b -C) |x|^{\beta -1+m} \exp (|x| ^\beta )
\label{24-A-36}
\end{equation}
provided $|\partial \psi |\ge b |x|^{\beta -1}$ as $|x|\ge c$. One can take
\begin{equation*}
\psi (x',x_3)=\int_0^{x_3} (|x'|^2+y^2)^{(\beta -1)/2}\,dy, \qquad x'=(x_1,x_2)
\end{equation*}
satisfying these restrictions.

\item\label{example-24-A-17-ii}
Similarly, for $\beta<0$ this constructions work for $|x|\le \epsilon$.
\end{enumerate}
\end{example}

Consider now quasihomogeneous case. In what follows $L=(l_1,l_2)$ for $d=2$, $L=(l_1,l_2,l_3)$ for $d=3$
\begin{equation}
[x]_L= \bigl(\sum_j x_j^{2n/l_j}\bigr)^{1/2n}
\label{24-A-37}
\end{equation}
is $L$-quasihomogeneous length, and $n$ is large so functions are smmoth in $\bR^d\setminus 0$.

\begin{example}\label{example-24-A-18}
\begin{enumerate}[label=(\roman*), wide, labelindent=0pt]
\item\label{example-24-A-18-i}
Let $d=2$ and $L=(l_1,l_2)$ with $1=l_1<l_2$. Let
\begin{equation}
V_1 = -a x_2[x]_L^{m},\qquad V_2=x_1[x]_L^m.
\label{24-A-38}
\end{equation}
Then $F\asymp [x]_L^m$ for $[x]_L\ge c$ provided $m\ne -(1+l_2)$ and $a$ is properly chosen.

\item\label{example-24-A-18-ii}
Let $d=2$ and $L=(l_1,l_2)$ with $1=l_1>l_2>0$. Let $V_1,V_2$ are defined by (\ref{24-A-38}). Then
$F\asymp [x]_L^m$ for $[x]_L\le \epsilon $ provided $m\ne-(1+l_2)$ and $a$ is properly chosen.
\end{enumerate}
\end{example}

\begin{example}\label{example-24-A-19}
\begin{enumerate}[label=(\roman*), wide, labelindent=0pt]
\item\label{example-24-A-19-i}
Let $d=3$ and $L=(l_1,l_2,l_3)$ with $1=l_1\le l_2\le l_3$. Let
\begin{align}
&V_1 = 0, &&V_2=x_1[x]_L^m,&&V_3= x_1[x]_L^{m+1},
\label{24-A-39}\\
&V_1 = -x_2[x]_L^m, &&V_2=x_1[x]_L^m,&&V_3= 0
\label{24-A-40}
\end{align}
if $m\ne -1$, $m=-1$ respectively. Then $F\asymp [x]_L^m$ for $[x]_L\ge c$.

\item\label{example-24-A-19-ii}
Let $d=2$ and $L=(l_1,l_2)$ with $1=l_1\ge l_2\ge l_3>0$. Let $V_1,V_2,V_3$ are defined by (\ref{24-A-39}) or (\ref{24-A-39}) if $m\ne -1$, $m=-1$ respectively. Then $F\asymp [x]_L^m$ for $[x]_L\le \epsilon $.
\end{enumerate}
\end{example}

\begin{example}\label{example-24-A-20}
Let $d=3$, $X=(\bR^2\setminus 0)\times \bR/\bZ \ni (x',x_3)=(x_1,x_2,x_3)$, $m\ne -2$ and
\begin{equation}
V_1=x_2|x'|^m,\quad V_2=-x_1|x'|^m,\quad V_3=a|x'|^{m+2}
\label{24-A-41}
\end{equation}
positively homogeneous on $x'$ of degrees $m+1$, $m+1$, $m+2$ respectively\footnote{\label{foot-24-46} Then $F^1$, $F^2$, $F^3$ are positively homogeneous of degrees $m+1$, $m+1$, $m$ respectively.}.
Then  $F^3\ne 0$, $|x'|=\const$ along   integral curves of the vector field
$\frac{1}{F^3} (F^1,F^2,F^3)$ and for irrational $a/\pi$ these curves are not closed.
\end{example}
\end{subappendices}

\chapter*{Comments}

In addition to papers, mentioned in Remark~\ref{rem-23-4-9} I would like also mention S.~Solnyshkin~\cite{solnyshkin:bound}, A.~Sobolev~\cite{sobolev:av:hommag1, sobolev:av:hommag2, sobolev:av:discrete, sobolev:av:asymp}, Y.~Colin de Verdiere~\cite{colindev:boitelles, colindev:specbot}, A.~Morame ~\cite{mohamed:aysmp, mohamed:schromag},
A.~Morame \&J.~Nourrigat~\cite{mohamed:count} and H.~Tamura~\cite{tamura:strong, tamura:bottommag, tamura:mag, tamura:hommag}, M.~Birman \& G.~Raikov~\cite{birman:magnetic},  G.~Raikov~\cite{raikov:hommag1, raikov:hommag2, raikov:strong, raikov:weak, raikov:border, raikov-7, raikov-8, raikov-9}.



\begin{thebibliography}{99}
\bibliographystyle{book}

\providecommand{\bysame}{\leavevmode\hbox to5em{\hrulefill}\thinspace}

\bibitem[Bir]{birman:function}
M. S. Birman, {S}chr{\"{o}}dinger operator: Estimates of the number of bound states as a
function-theoretical problem.
\newblock {AMS Transl. Ser. 2}, 150 (1991).


\bibitem[BR]{birman:magnetic}
M.~S. Birman and G. Raikov. Discrete spectrum in the gaps for perturbations of the magnetic
{S}chr{\"{o}}dinger operator.
\newblock {Adv. Soviet Math.}, 7:75--84 (1991).



\bibitem[CdV1]{colindev:boitelles}
Y. Colin~de Verdiere. Comportement asymptotique du spectre de boitelles mag\-n{\'{e}}\-tiques.
\newblock In {Travaux Conf. EDP Sant Jean de Monts}, volume~2 (1985).

\bibitem[CdV2]{colindev:specbot}
Y. Colin~de Verdiere. L'asymptotique du spectre des boitelles magn{\'{e}}tiques.
\newblock {Commun. Math. Phys.}, 105(2):327--335 (1986).



\bibitem[DS]{danford:schwartz}
N. Danford and J.T. Schwartz. Linear Operators,
Parts I--III.
\newblock Willey Classics Library.


\bibitem[Fr]{friedrichs:hilbert}
K. Friedrichs,
{Perturbations of Spectra in {H}ilbert Space}.
\newblock AMS, Providence, RI (1965).



\bibitem[Ivr1]{ivrii:bms}
\href{http://link.springer.com/article/10.1007/s13373-016-0089-y}{100 years of Weyl's law}.
Bulletin of Mathematical Sciences, (20016). Also \href{http://arxiv.org/abs/1608.03963}{arXiv:math/1608.03963}


\bibitem[Ivr2]{futurebook}
Victor Ivrii, Microlocal Analysis and Sharp Spectral Asymptotics,
 in progress: available online at \newline
\href{http://www.math.toronto.edu/ivrii/futurebook.pdf}{http://www.math.toronto.edu/ivrii/monsterbook.pdf}

\bibitem[MV]{mazya:verb}
V. G. Maz'ya  and I. E. Verbitsky.
Boundedness and compactness criteria for the one-dimensional Schr\"odinger operator. Function spaces, interpolation theory and related topics (Lund, 2000), 369–-382, de Gruyter, Berlin, 2002.


\bibitem[Mar1]{mohamed:aysmp}
A. Mohamed (A. Morame). Comportement asymptotique, avec estimation du reste, des valeurs propres d'une
class d'opd sur {$\bR^n$}.
\newblock {Math. Nachr.}, 140:127--186 (1989).

\bibitem[Mar2]{mohamed:schromag}
A. Mohamed (A. Morame). Quelques remarques sur le spectre de l'op{\'{e}}rateur de {S}chr{\"{o}}dinger
avec un champ magnetique.
\newblock {Commun. Part. Differ. Equat.}, 13(11):1415--1430 (1989).


\bibitem[MN]{mohamed:count}
A. Mohamed (A. Morame) and J. Nourrigat. Encadrement du {$N(\lambda)$} pour un op{\'{e}}rateur de {S}chr{\"{o}}dinger
avec un champ magn{\'{e}}tique et un potentiel {\'{e}}lectrique.
\newblock {J. Math. Pures Appl.}, 70:87--99 (1991).



\bibitem[MR]{rozenblum:melgaard}
M. Melgaard and G. Rozenblum.
{Eigenvalue asymptotics for weakly perturbed Dirac and Schr\"odinger operators with constant magnetic fields of full rank.}.
\newblock {Comm. Partial Differential Equations}, 28:3 (2003).


\bibitem[Rai1]{raikov:steady}
G. Raikov. Spectral asymptotics for the {S}chr{\"{o}}dinger operator with potential which
steadies at infinity.
\newblock {Commun. Math. Phys.}, 124:665--685 (1989).

\bibitem[Rai2]{raikov:hommag1}
G. Raikov. Eigenvalue asymptotics for the {S}chr{\"{o}}dinger operator with homogeneous
magnetic potential and decreasing electric potential {I}: Behavior near the
essential spectrum tips.
\newblock {Commun. Part. Differ. Equat.}, 15(3):407--434 (1990).

\bibitem[Rai3]{raikov:hommag2}
G. Raikov. Eigenvalue asymptotics for the {S}chr{\"{o}}dinger operator with homogeneous
magnetic potential and decreasing electric potential {II}: Strong electric
field approximation.
\newblock In {Proc. Int. Conf. Integral Equations and Inverse Problems,
Varna 1989}, pages 220--224. Longman (1990).

\bibitem[Rai4]{raikov:strong}
G. Raikov. Strong electric field eigenvalue asymptotics for the {S}chr{\"{o}}dinger
operator with the electromagnetic potential.
\newblock {Letters Math. Phys.}, 21:41--49 (1991).

\bibitem[Rai5]{raikov:weak}
G. Raikov. Semi-classical and weak-magnetic-field eigenvalue asymptotics for the
{S}chr{\"{o}}dinger operator with electromagnetic potential.
\newblock {C.R.A.S. Bulg. AN}, 44(1) (1991).

\bibitem[Rai6]{raikov:border}
G. Raikov. Border-line eigenvalue asymptotics for the {S}chr{\"{o}}dinger operator with
electromagnetic potential.
\newblock {Int. Eq. Oper. Theorem.}, 14:875--888 (1991).

\bibitem[Rai7]{raikov-7}
G. Raikov. {Eigenvalue asymptotics for the Schr\"odinger
operator in strong constant magnetic fields}, Commun. P.D.E.
23 (1998), 1583--1620.

\bibitem[Rai8]{raikov-8}
G. Raikov. {Eigenvalue asymptotics for the Dirac
operator in strong constant magnetic fields}, Math. Phys. Electron.
J. 5 (1999), No 2, 22 pp.

\bibitem[Rai9]{raikov-9}
\newblock
G. Raikov. {Eigenvalue asymptotics for the Pauli
operator in strong non-constant magnetic fields}, Ann. Inst. Fourier
49 (1999), 1603--1636.



\bibitem[RaV]{raikov:warzel}
G. Raikov and S. Warzel.
Quasi-classical versus non-classical spectral
asymptotics for magnetic Schro\"odinger operators with decreasing electric potentials.
\newblock {Rev. Math. Phys.}, 14 (2002), 1051--1072.


\bibitem[RT1]{rozenblioum:tash:1}
G.~V. Rozenblioum, and G. Tashchiyan.
\newblock On the spectral properties of the Landau Hamiltonian perturbed by a moderately decaying magnetic field. \newblock Spectral and scattering theory for quantum magnetic systems, 169–-186, Contemp. Math., 500, Amer. Math. Soc., Providence, RI, 2009.

\bibitem[RT2]{rozenblioum:tash:2}
G.~V. Rozenblioum, and G. Tashchiyan.
\newblock On the spectral properties of the perturbed Landau Hamiltonian. Comm. Partial Differential Equations 33 (2008), no. 4--6, 1048–-1081.


\bibitem[Sob1]{sobolev:av:hommag1}
A. V. Sobolev. On the asymptotic for energy levels of a quantum particle in a homogeneous
magnetic field, perturbed by a decreasing electric field 1.
\newblock {J. Soviet Math.}, 35(1):2201--2211 (1986).

\bibitem[Sob2]{sobolev:av:hommag2}
A. V. Sobolev. On the asymptotic for energy levels of a quantum particle in a homogeneous
magnetic field, perturbed by a decreasing electric field 2.
\newblock {Probl. Math. Phys}, 11:232--248 (1986).

\bibitem[Sob3]{sobolev:av:discrete}
A. V. Sobolev. Discrete spectrum asymptotics for the {S}chr{\"{o}}dinger operator with
electric and homogeneous magnetic field.
\newblock {Notices LOMI sem.}, 182:131--141 (1990).

\bibitem[Sob4]{sobolev:av:asymp}
A. V. Sobolev. On the asymptotics of discrete spectrum for the {S}chr{\"{o}}dinger operator in
electric and homogeneous magnetic field.
\newblock {Operator Theory: Advances and Applications}, 46:27--31 (1990).



\bibitem[Sol]{solnyshkin:bound}
S. N. Solnyshkin. Asymptotic of the energy of bound states of the {S}chr{\"{o}}dinger operator in
the presence of electric and homogeneous magnetic fields.
\newblock {Select. Math. Sov.}, 5(3):297--306 (1986).


\bibitem[Tam1]{tamura:strong}
H. Tamura. The asymptotic formulas for the number of bound states in the strong coupling
limit.
\newblock {J. Math. Soc. Japan}, 36:355--374 (1984).

\bibitem[Tam2]{tamura:bottommag}
H. Tamura.  Eigenvalue asymptotics below the bottom of essential spectrum for magnetic
{S}chr{\"{o}}dinger operators.
\newblock In {Proc. Conf. on Spectral and Scattering Theory for
Differential Operators, Fujisakara-so, 1986}, pages 198--205, Tokyo. Seizo
It{\^{o}} (1986).

\bibitem[Tam3]{tamura:mag}
H. Tamura.  Asymptotic distribution of eigenvalues for {S}chr{\"{o}}dinger operators with
magnetic fields.
\newblock {Nagoya Math. J.}, 105(10):49--69 (1987).

\bibitem[Tam4]{tamura:hommag}
H. Tamura. symptotic distribution of eigenvalues for {S}chr{\"{o}}dinger operators with
homogeneous magnetics fields.
\newblock {Osaka J. Math.}, 25:633--647 (1988).

\bibitem[Tam5]{tamura:hommag2}
Asymptotic distribution of eigenvalues for {S}chr{\"{o}}dinger operators with
homogeneous magnetic fields {II}.
\newblock {Osaka J. Math.}, 26:119--137 (1989).

 \end{thebibliography}
\end{document}